\documentclass{article}

\usepackage{fullpage}
\usepackage{amsthm,amsmath,amsfonts,amssymb}

\usepackage{bbm}
\usepackage{mathrsfs}
\usepackage{algorithm}
\usepackage{algpseudocode}
\usepackage{comment}
\usepackage{graphicx}
\usepackage{float}
\usepackage{makecell}
\usepackage[colorlinks]{hyperref}

\newtheorem{remark}{Remark}
\newtheorem{definition}{Definition}
\newtheorem{theorem}{Theorem}
\newtheorem{corollary}{Corollary}
\newtheorem{lemma}{Lemma}
\newtheorem{proposition}{Proposition}
\newtheorem{assumption}{Assumption}

\newcommand{\eps}{\varepsilon}
\newcommand{\real}{\mathbb{R}}
\newcommand{\mprob}{\mathbb{P}}
\newcommand{\E}{\mathbb{E}}
\newcommand{\calM}{\mathcal{M}}
\newcommand{\norm}[1]{\left\lVert #1 \right\rVert}
\newcommand{\Tr}{\ensuremath{\mathrm{Tr}}}
\newcommand{\thetahat}{\ensuremath{\widehat{\theta}}}
\newcommand{\Thetahat}{\ensuremath{\widehat{\Theta}}}
\newcommand{\sigmahat}{\ensuremath{\widehat{\sigma}}}

\newcommand{\Cov}{\ensuremath{\operatorname{Cov}}}
\newcommand{\Var}{\ensuremath{\operatorname{Var}}}

\begin{document}

\begin{center}
{\bf{\LARGE{Node-private community estimation in stochastic block models: Tractable algorithms and lower bounds}}} \\
\vspace*{.25in}
\begin{tabular}{ccc}
{\large{Laurentiu Marchis}} & \hspace*{.5in}  & {\large{Ethan D'souza}}\\
{\large{\texttt{lam223@cam.ac.uk}}} & \hspace*{.5in} & {\large{\texttt{ethandsouza.ldn@gmail.com}}} \\
\mbox{} & \vspace{0.1cm} & \mbox{} \\
{\large{Tom\'{a}\v{s} Fl\'{i}dr}} & \hspace*{.5in}  & {\large{Po-Ling Loh}}\\
{\large{\texttt{tf388@cam.ac.uk}}} & \hspace*{.5in} & {\large{\texttt{pll28@cam.ac.uk}}}
\end{tabular}
\begin{center}
Statistical Laboratory \\
Department of Pure Mathematics and Mathematical Statistics \\
University of Cambridge
\end{center}
\vspace*{.2in}
May 2026
\vspace*{.2in}
\end{center}

\begin{abstract}
We study the classical problem of community recovery in stochastic block models with a fixed number of communities, with a twist: We seek algorithms that are stable with respect to node-wise changes in the graph structure, formally defined as a differential privacy constraint. The algorithms we develop are based on spectral clustering, where we introduce privacy to the community recovery pipeline in the form of directly privatizing the adjacency matrix; private PCA; private convex optimization; private low-rank matrix estimation; and private approximate subspace estimation. Straightforward applications of existing private algorithms lead to a rapid increase in the privacy parameter $\epsilon$ in order to ensure consistent estimation under node differential privacy, in contrast with the simpler setting of edge privacy. To alleviate these issues, we develop novel algorithms based on (1) sampling from an exponential mechanism with a Lipschitz extension and (2) a general framework for constructing smooth projections from the space of undirected graphs to the space of bounded-degree graphs, which can then be combined with various edge-private algorithms. Importantly, the methods we develop are all computable in polynomial-time as a function of the number of nodes in the graph. We also develop novel lower bounds on the growth rate of $\epsilon$ required in order to achieve consistent community estimation under node privacy. On a technical note, our paper highlights the complications that arise when analyzing private algorithms under the non-standard scaling $\epsilon \rightarrow \infty$ and proposes some solutions. We also provide a novel application of the HGR maximal correlation from information theory in the context of accuracy amplification in PAC learning, which may be of independent interest.
\end{abstract}


\section{Introduction}

Recent years have seen a massive increase of interest in privacy from the statistics community. Although the framework of differential privacy was introduced 20 years ago~\cite{dwork2006calibrating, Dwo06, dwork2008differential}, with numerous papers in theoretical computer science developing methods for privatizing various algorithms and studying the tradeoffs between privacy and accuracy (i.e., utility) of the algorithms, privacy has only recently begun to gain more traction in the statistics community, e.g., see the papers~\cite{dwork2009differential, wasserman2010statistical, duchi2018minimax, avella2021privacy, cai2021cost, dong2022gaussian, avella2023differentially}. Indeed, rather than viewing privacy as an aspect of a \emph{mechanism} or \emph{algorithm}, one can more simply consider privacy as a constraint on statistical estimator computed from data. This leads to a trove of new problems in theoretical statistics, where the goal is to pinpoint minimax rates of estimation for various problems, with the minimum taken over the subclass of \emph{private} estimators~\cite{duchi2013local}.

Despite many fascinating advances in this area, a topic that has received relatively little attention in statistics is private estimation in the presence of \emph{discrete-structured} problems. In this paper, we attempt to advance this frontier in the context of estimating the underlying combinatorial structure in an undirected random graph. Specifically, we focus on stochastic block models, which have received a great deal of attention from social scientists, statisticians, engineers, physicists, and computer scientists alike (e.g., see \cite{holland1983stochastic, rohe2011spectral, decelle2011asymptotic, zhao2012consistency, LeiRin15, hajek2016achieving, moitra2016robust,  gao2017achieving, abbe2018community, xu2020optimal} and the references cited therein).

A principled analysis of differentially private algorithms applied to realizations of a random graph immediately faces the fundamental question of how neighboring datasets should be defined. Indeed, standard definitions of differential privacy involve the behavior of an algorithm when one observation in a data matrix is removed/replaced. However, defining a direct analog for graphical data leads to some ambiguity. One possibility is to consider two graphs as ``neighbors" if they differ in a single edge (this is known as \emph{edge} privacy), whereas an alternative possibility is to define two graphs as neighbors if they differ in the connections to any single node (this is known as \emph{node} privacy). If we consider the protection promised by a privacy-preserving data curator to an individual as a bound on the extent to which the individual's data will be reflected in the output of a private algorithm, and treat the nodes of a graph as individuals, edge privacy corresponds to protecting information about the relationship between \emph{any pair of individuals}, while node privacy corresponds to protecting information that \emph{any individual chooses to release about their relationship with anyone else}.

In this paper, we focus on the setting of node privacy, which we believe to be slightly more realistic and also (as will be seen in the sequel) much more mathematically challenging. Given a realization of a random graph drawn from a stochastic block model (SBM) with an equal number of nodes in each community, we study the problem of recovering the unobserved community labels in a manner that respects node differential privacy.


\subsection{Main contributions}

In more detail, the main contributions of our paper are as follows: First, we propose the first computationally-feasible algorithms for node-private community estimation in stochastic block models that we are aware of in the literature. Our algorithms involve methodological innovations based on (1) Lipschitz extensions which are computable in polynomial-time, with corresponding exponential mechanisms which are efficient to sample from; and (2) a general method to construct a smooth projection from the space of all input graphs to those of bounded degree, which is again computationally feasible and may be combined with a host of edge-private algorithms. We emphasize that unlike some of the recent literature on computationally feasible graph-private estimation, our polynomial-time methods only involve solving linear programs rather than sum-of-squares optimization problems, the latter of which are not practically implementable.

The algorithms we present are applicable in slightly different settings, involving pure vs.\ approximate differential privacy, weighted vs.\ unweighted graphs, and different scaling on the expected degree in terms of the number $n$ of nodes in the graph. In the utility analysis, we focus on the required scaling of $\epsilon$ in order to drive the fraction of misclassified nodes to 0 as $n \rightarrow \infty$ (i.e., consistency). Importantly, we provide two novel lower bound arguments proving that $\epsilon \rightarrow \infty$ is necessary for consistent community recovery in a node-private setting.

From a technical viewpoint, our analysis exposes some intricacies of studying the $\epsilon \rightarrow \infty$ regime, including in applying composition theorems and group privacy (we obtain tighter results by passing through zero-concentrated differential privacy) and applications of the Gaussian mechanism. These contributions may be interesting in their own right, as other statistical settings surely exist beyond stochastic block models where the condition $\epsilon \rightarrow \infty$ is needed to ensure accurate estimation.


\subsection{Related work}

Graph differential privacy was first studied in~\cite{NisEtal07, hay2009accurate, karwa2011private, gupta2012iterative, karwa2012differentially, blocki2012johnson}, with the distinction between edge and node privacy being introduced in~\cite{hay2009accurate}, and node-private mechanisms subsequently developed in~\cite{blocki2013differentially, kasiviswanathan2013analyzing, chen2013recursive, raskhodnikova2015efficient}. \cite{kasiviswanathan2013analyzing, chen2013recursive} introduced the notion of Lipschitz extensions, which was later used in~\cite{borgs2015private, borgs2018revealing, borgs2018private}. \cite{blocki2013differentially} and \cite{kasiviswanathan2013analyzing} provided fairly general methods for converting from edge-private algorithms to node-private algorithms. However, although our work is heavily inspired by these papers, in particular high-probability upper bounds on local sensitivity, their goal was to privatize rather different problems such as functions of a degree distribution, and directly applying the methods to our setting of stochastic block models leads to much worse results. Furthermore, the technical lemmas in these earlier papers do not apply when $\epsilon \rightarrow \infty$, so we must carefully refine their analysis to make it suitable for our purposes.

Turning to stochastic block models, only a handful of papers have studied estimation under edge privacy~\cite{HehEtal22, mohamed2022differentially, CheEtal23, chakraborty2024prime, nguyen2024differentially, koskela2025price}. \cite{chen2024private} studied node privacy, but with the goal of estimating the $k \times k$ matrix of connection probabilities between nodes in different communities. The most closely related work is the recent paper \cite{klopp2026node}, who similarly studied node privacy for SBM recovery. However, the algorithms they propose have runtime exponential in the number of nodes $n$, whereas our focus is on \emph{computationally feasible} algorithms that are polynomial-time in $n$. \cite{klopp2026node} have some methodological overlap in their use of an extension lemma from differential privacy~\cite{borgs2018private}. However, it is very important in our paper to only use Lipschitz extensions that can be computed efficiently from the data. Furthermore, while our lower bounds involving $\epsilon = \Omega(\log n)$ roughly agree with the conclusion of \cite{klopp2026node}, our lower bound derivations involve quite different arguments and are also applicable to approximate privacy settings, as well as a wider range of accuracy guarantees/high-probability regimes.


\subsection{Organization}

The remainder of the paper is organized as follows: We begin in Section~\ref{SecBackground} by providing some background on the stochastic block model and fixing notation for the specific types of SBMs we will study. We also rigorously state the notion of node differential privacy. Section~\ref{SecEdge} provides some further insight on the challenges of node differential privacy in comparison to edge differential privacy. We also introduce several algorithms that will serve as building blocks for the methods we develop later in the paper. In Section~\ref{SecPPCA}, we provide our first main methodological contribution, which combines the notion of Lipschitz extensions with a private PCA algorithm. Section~\ref{SecProj} provides our second main methodological contribution, constructing a smooth projection from the space of arbitrary graphs to the space of graphs with bounded degree. We then explain how to combine the smooth projection with the algorithms mentioned in Section~\ref{SecEdge}, and rigorously derive privacy and utility guarantees. Section~\ref{SecLower} complements our analysis of specific algorithms with lower bounds that hold for all private algorithms. We conclude in Section~\ref{SecDiscussion} with a discussion of interesting open questions inspired by our work. Additional proof details may be found in the supplementary appendices.


\subsection{Notation}

For a matrix $A$, we write $\text{row}(A)$ and $\text{col}(A)$ to denote the subspaces spanned by the rows and columns of $A$, respectively. We write $\text{rank}(A)$ to denote the rank of $A$. We write $\|A\|_2$ and $\|A\|_F$ to denote the spectral and Frobenius norms. We write $\|A\|_\infty$ to denote the maximum absolute row sum and $\|A\|_1$ to denote the maximum absolute column sum. Let $\sigma_i(A)$ denote the $i^{\text{th}}$ largest singular value of $A$, and let $\sigma_{\min}(A)$ denote the minimum non-zero singular value. Let $\sigma_{\max}(A)$ denote the maximum singular value. Let $\lambda_i(A)$ denote the $i^{\text{th}}$ largest eigenvalue of $A$. We write $\max(A) = \max_{i,j} A_{ij}$ to denote the maximum entry. For two matrices $A$ and $B$, we write $A \otimes B$ to denote the tensor product and $A \odot B$ to denote the Hadamard product.

We write $J_n$ to denote the $n \times n$ matrix with every entry equal to 1. We write $\textbf{1}$ to denote the all-1's vector. We write $E_k$ to denote the set of $k \times k$ permutation matrices.
For an integer $n$, we write $[n]$ to denote the set $\{1, \dots, n\}$.
For a real number $x \in \real$ and truncation parameters $a < b$, we write $[x]_a^b$ to denote the function $\max\left\{a, \min\{x, b\}\right\}$.
We write $D_\alpha(P||Q) = \frac{1}{\alpha - 1}\log\left(\mathbb{E}_{x \sim Q}\left[\left(\frac{P(x)}{Q(x)}\right)^\alpha\right]\right)$ to denote the $\alpha$-R\'{e}nyi divergence between distributions $P$ and $Q$.

For functions $f(n)$ and $g(n)$, we write $f(n) \precsim g(n)$ and $f(n) = O(g(n))$ to mean that $f(n) \le c_1g(n)$, if $n \geq c_2$, for some universal constants $c_1, c_2 \in (0, \infty)$, and define $f(n) \succsim g(n)$ and $f(n) = \Omega(n)$ analogously. We write $f(n) \asymp g(n)$ and $f(n) = \Theta(n)$ when $f(n) \precsim g(n)$ and $f(n) \succsim g(n)$ hold simultaneously. We use the notation $\widetilde{O}$, $\widetilde{\Omega}$, and $\widetilde{\Theta}$ to suppress factors of $\log(n)$ in the relative scaling. We write $f(n) = \mathrm{poly}(n)$ to denote the fact that $f(n) = \Theta(n^\alpha)$ for some constant $\alpha > 0$. We use the notation $c, C, c', c_i$, etc., to denote absolute positive constants, whose exact value may change between results in our paper.


\section{Background and problem setup}
\label{SecBackground}

In this section, we discuss the setup of our problem. We formally define the stochastic block model, followed by different notions of differential privacy that will be relevant to our paper.

\subsection{Stochastic block models}

Let $\mathcal{G}_n$ denote the set of unweighted, undirected graphs
on $n$ nodes and let $\mathcal{G}_{n, D}$ denote the set of all graphs in $\mathcal{G}_n$ with maximum degree $D$. 
Let $\mathcal{W}_n$ denote the set of weighted, undirected graphs on $n$ nodes, where the weights range over $\mathbb{R}$. Also, let $\mathcal{W}_{n, D}$ denote the set of all weighted graphs in $\mathcal{W}_n$ with maximum degree $D$: For every node, there are at most $D$ edges connected to it (we interpret a $0$ weight as no edge), and all those edges have non-zero, real-valued weights.

We now describe two probabilistic models for generating random graphs with respect to an underlying community assignment $\theta = \{\theta_i\}_{i = 1}^n$, with $\theta_i \in [k]$, for all $i \in [n]$. We will assume throughout that the underlying graph has equal-sized communities, so for each $j \in [k]$, we have $\left|\{i \in [n]: \theta_i = j\}\right| = \frac{n}{k}$. For more background on stochastic block models, we refer the reader to the survey~\cite{abbe2018community}.

\begin{definition}[Unweighted SBM]
\label{DefSBM}
For a probability matrix $B \in (0,1]^{k \times k}$, we write $G \sim SBM(n, k, B, \theta)$ to denote a random graph in $\mathcal{G}_n$ generated in the following manner: We sample each entry $A_{ij}$ of the $n \times n$ adjacency matrix, with $i < j$, as an independent Bernoulli random variable with mean $B_{\theta_i, \theta_j}$. We then set $A_{ji} = A_{ij}$ and define $G$ to be the graph with adjacency matrix $A$. With a slight abuse of notation, we also write $A \sim SBM(n, k, B, \theta)$.

We assume that $B \in (0, 1]^{k \times k}$ is invertible. Let $\lambda_B$ denote the smallest absolute non-zero eigenvalue of $B$, and let $d = n\cdot\mathop{\max}(B)$. Let $P = J(B\otimes J_{\frac{n}{k}})J^T$, for some permutation matrix $J \in \{0, 1\}^{n \times n}$, so that $\mathbb{E}[A] = P - \mathrm{diag}(P)$.
\end{definition}

\begin{definition}[Weighted SBM]
\label{DefWeightSBM}
For a probability matrix $B \in (0,1]^{k \times k}$ and a weight distribution $\mathscr{W}$, we write $G \sim WSBM(n, k, B, \theta, \mathscr{W})$ to denote a random graph in $\mathcal{W}_n$ generated as follows: We first sample $A \sim SBM(n, k, B, \theta)$, and then independently generate $W \sim \mathscr{W}$. We then define $G$ to be the graph with weighted adjacency matrix $W \odot A$.

We assume that the weight distribution $\mathscr{W}$ is such that distributions of different entries are independent, and symmetric around zero (after centering). Let the matrix of means of the weights be $B_w \in \mathbb{R}^{k \times k}$, so $\mathbb{E}[W] = J(B_w \otimes J_{\frac{n}{k}})J^T$, for some permutation matrix $J \in \{0, 1\}^{n \times n}$. Let $P_w = J((B \odot B_w)\otimes J_{\frac{n}{k}})J^T$. Then $\mathbb{E}[W \odot A] = P_w - \mathrm{diag}(P_w)$, and $P_w = \mathbb{E}[W]\odot P$. We will assume $B_w \succeq 0$ or $B_w \preceq 0$, and $B_w \odot B$ is invertible. Let $d = n\cdot\mathop{\max}(B)$.
\end{definition}

Examples of weighted SBMs are the zero-inflated Gaussian model~\cite{oliveira2025counting} or homogeneous weighted SBM~\cite{xu2020optimal}, previously studied in the literature. Note that while we assume invertibility of the probability matrix $B$ in Definition~\ref{DefSBM}, invertibility is only assumed for $B_w \odot B$ in Definition~\ref{DefWeightSBM}, so the presence of weights may enable accurate community recovery when it would have been impossible to recover communities if edge weights were unobserved.

We will operate under the following assumptions on the parameters of the SBM:

\begin{assumption}\label{AssSBM}
Suppose $k \asymp 1$. Also assume $B = a_n B_0$, where $B_0$ does not depend on $n$, and $\frac{25k\log(n)}{n\cdot \mathrm{min}_j(B_0)_{jj}} \le a_n = o(1)$.
\end{assumption}

Importantly, under Assumption~\ref{AssSBM}, the maximal degree of $A \sim SBM(n,k,B,\theta)$ is $O(d)$, w.h.p.\ (cf.\ Appendix~\ref{AppDegree}).
Some of the results later in the paper are specific to the case $k=2$, in which case we will further impose the following homogeneity assumption:

\begin{assumption}
\label{assump_same_B11B22}
Assume $k = 2$ and $(B_0)_{11} = (B_0)_{22} > (B_0)_{12} = (B_0)_{21}$.
\end{assumption}

We use the following two error metrics: the overall misclassification error and the
worst-case community-wise misclassification error, both normalized as in
\cite{LeiRin15}. We note that our definitions are twice the corresponding misclassification
rates used in \cite{HehEtal22}. We adopt this convention because several of our
arguments invoke results from \cite{LeiRin15} directly.

\begin{definition}
[Misclassification rates]
\label{DefLoss}
Let $S_k$ be permutation group on $k$ elements. Let $\theta, \thetahat \in [k]^n$. Let $C_j = \{i: \ \theta_i = j\}$, for all $j \in [k]$. Then the overall misclassification error for the estimated groups $\hat{\theta}$ is defined by
\begin{align*}
\mathcal{L}(\hat{\theta}, \theta) = \mathop{\min}\limits_{\sigma \in S_k}\frac{2}{n}\sum_{i = 1}^n\mathbbm{1}_{\{\sigma(\hat{\theta}_i) \neq \theta_i\}},
\end{align*}
and the worst-case misclassification error is defined by
\begin{align*}
\widetilde{\mathcal{L}}(\hat{\theta}, \theta) = \mathop{\min}\limits_{\sigma \in S_k}\mathop{\max}_{j \in [k]}\frac{2}{|C_j|}\sum_{i \in C_j}\mathbbm{1}_{\{\sigma(\hat{\theta}_i) \neq j\}}.
\end{align*}
\end{definition}

\begin{remark}
\label{remark:avg<=worst}
One can easily see that $0 \leq \mathcal{L}(\hat{\theta}, \theta) \leq  \widetilde{\mathcal{L}}(\hat{\theta}, \theta) \leq 2$. Additionally, if $\theta$ has balanced communities of size $\frac{n}{k}$, then $\widetilde{\mathcal{L}}(\hat{\theta}, \theta) \leq  k\mathcal{L}(\hat{\theta}, \theta)$. Finally, it is easy to check that $\mathcal{L}$ satisfies the triangle inequality.
\end{remark}

We will write $\mathbf{M}_{n\times k}$ to denote the set of $n \times k$ membership (``one-hot encoding") matrices, such that there is exactly one $1$ in every row, and $0$'s otherwise. Each $\theta \in [k]^n$ can naturally be mapped to an element of $\mathbf{M}_{n \times k}$. Note that we then have the convenient alternative formula
\begin{equation*}
\mathcal{L}(\thetahat, \theta) = \min_{J \in E_k} \|\Thetahat J - \theta\|_0,
\end{equation*}
explaining the reason for including a factor of 2 in Definition~\ref{DefLoss}.


\subsection{Differential privacy}

In this section, we provide an overview of the main privacy concepts we will leverage in this paper, with additional background results given in Appendix~\ref{AppPrivacy}. We begin with the classical notion of $(\epsilon, \delta)$-differential privacy, which will be denoted by $(\epsilon, \delta)$-DP:
\begin{definition}
A randomized algorithm/mechanism $\mathcal{A}$ satisfies $(\epsilon, \delta)$-DP differential privacy, for $\epsilon, \delta \geq 0$, if for all pairs of datasets $X$ and $X'$ differing in one element and for all $S$ in the range of $\mathcal{A}$, we have $\mathbb{P}(\mathcal{A}(X) \in S) \leq e^{\epsilon}\mathbb{P}(\mathcal{A}(X') \in S) + \delta$.
\end{definition}  

We also introduce the notion of zCDP, which has much better group privacy properties compared to $(\epsilon, \delta)$-DP in a regime where $\epsilon$ increases (cf.\ Lemmas~\ref{LemGroup} and~\ref{LemGroupZ}). Some of our technical arguments will start from zCDP, apply group privacy, and then convert back to $(\epsilon, \delta)$-DP.

\begin{definition}[$\rho$-zCDP from \cite{bun2016conc}]
A randomized mechanism $\mathcal{A}$ is $\rho$-zCDP if for any two datasets $X$ and $X'$ differing in one element, and for every $\alpha > 1$, we have
\begin{align*}
D_\alpha\left(\mathcal{A}(X)||\mathcal{A}(X')\right) \leq \rho\alpha.
\end{align*}
\end{definition}

In order to define privacy for graphs, we first define two notions of graph adjacency~\cite{kasiviswanathan2013analyzing}:

\begin{definition}
Two graphs $G, G' \in \mathcal{G}_n$ (or $ \mathcal{W}_n$) are edge adjacent ($G \sim_{edge} G'$) if they differ in at most one edge (for weighted graphs, they differ in one weight).
\end{definition}

\begin{definition}
For $G, G' \in \mathcal{G}_n$, we define $d_{node}(G, G')$ to be the minimum number nodes that need to be rewired (i.e., the corresponding row/column of the adjacency matrix must be changed) to obtain $G$ from $G'$. Similarly, for $G, G' \in \mathcal{W}_n$, we define $d_{node}(G, G')$ to be the minimum number nodes whose weights have to be altered to obtain $G$ from $G'$.
\end{definition}

\begin{definition}
Two graphs $G, G' \in \mathcal{G}_n$ (or $ \mathcal{W}_n$) are node adjacent ($G \sim_{node} G'$) if $d_{node}(G, G') \leq 1$.
\end{definition}

\begin{definition}
We say a randomized algorithm $\mathcal{A}$, evaluated on graph inputs in $\mathcal{G}_n$ or $\mathcal{W}_n$, is $(\epsilon, \delta)$-edge DP (or $(\epsilon, \delta)$-node DP) if for any $G \sim_{edge} G'$ (or $G \sim_{node} G'$) and measurable sets $S$, we have
\begin{align*}
\mathbb{P}(\mathcal{A}(G) \in S) \leq e^{\epsilon}\mathbb{P}(\mathcal{A}(G') \in S) + \delta.
\end{align*}
Similarly, we say that $\mathcal{A}$, evaluated on graph inputs in $\mathcal{G}_n$ or $\mathcal{W}_n$, is $\rho$-edge zCDP if for any $G \sim_{edge} G'$ and every $\alpha > 1$, we have
\begin{align*}
D_\alpha\left(\mathcal{A}(G)||\mathcal{A}(G')\right) \leq \rho\alpha.
\end{align*}
\end{definition}

\begin{definition}
We say a randomized algorithm $\mathcal{A}$, evaluated on graph inputs in $\mathcal{G}_{n, D}$ or $\mathcal{W}_{n, D}$, is $(\epsilon, \delta)_D$-node DP if for any $G, G' \in \mathcal{G}_{n,D}$ such that $G \sim_{node} G'$, and measurable sets $S$, we have
\begin{align*}
\mathbb{P}(\mathcal{A}(G) \in S) \leq e^{\epsilon}\mathbb{P}(\mathcal{A}(G') \in S) + \delta.
\end{align*}
Similarly, we say that $\mathcal{A}$, evaluated on graph inputs in $\mathcal{G}_{n, D}$ or $\mathcal{W}_{n, D}$, is $(\rho)_D$-zCDP, if for any $G \sim_{node} G'$ and every $\alpha > 1$, we have
\begin{align*}
D_\alpha\left(\mathcal{A}(G)||\mathcal{A}(G')\right) \leq \rho\alpha.
\end{align*}
\end{definition}


\section{Node privacy via edge privacy}
\label{SecEdge}

Note that an edge-private algorithm is automatically node-private after applying group privacy of size $n$ (cf.\ Lemma~\ref{LemGroup}). To motivate our methods, we first provide a baseline comparison to the naive approach of directly combining existing edge-private methods with group privacy. Our subsequent goal will be to slow the growth of $\epsilon$ as a function of $n$. When $d$ is set to the minimal value allowed by our theoretical results, we will reduce the required growth rate on $\epsilon$ to a value comparable to $d$.

\subsection{Off-the-shelf application of edge-private algorithms}

We first discuss two edge-private algorithms proposed in the literature~\cite{HehEtal22, CheEtal23}. These methods illustrate that naively using group privacy of size $n$ to guarantee node privacy at level $\epsilon$ requires $\epsilon$ to grow substantially with $n$, in order to maintain consistency.


\subsubsection{Edge-flipping mechanism}

The first method we consider is based on the edge-DP method from~\cite{HehEtal22}. Roughly speaking, we apply the randomized response mechanism for binary data to the entries of the adjacency matrix and then apply spectral clustering to the result. A detailed statement of the algorithm is provided in Algorithm~\ref{alg:EF_Spec_Clus} in Appendix~\ref{AppSCFlip}. Note that \cite{HehEtal22} consider the notion of $\epsilon$-relationship DP (cf.\ Definition $4$ in \cite{HehEtal22}). By inspecting the definition, one can see that it immediately implies $(\epsilon, 0)$-edge DP.

Let us show the effect of naively extending edge DP to node DP. Since Algorithm \ref{alg:EF_Spec_Clus} satisfies pure edge DP, applying group privacy for groups of size $n$ simply scales the privacy parameter by $n$, so $\thetahat(G) = \mathcal{A}_{EF}^{(\epsilon/n)}(G)$ satisfies $(\epsilon, 0)$-node DP. 
By Lemma \ref{lemma:Edge_Flip_Ut}, adapted from \cite{HehEtal22}, if $G \sim SBM(n,k,B,\theta)$ and $d \asymp \sqrt{n}\log^2(n)$, we have
\begin{equation*}
\mathbb{P}\left(\widetilde{\mathcal{L}}\left(\thetahat(G), \theta\right) \leq\frac{Cg_{\epsilon/n}}{na_n^2}\right) \geq 1 - \frac{1}{n},
\end{equation*}
where $g_{\epsilon} = \frac{e^{\epsilon} + 1}{e^{\epsilon} - 1}\left(a_n \cdot\max(B_0) + \frac{1}{e^{\epsilon} - 1}\right)$, provided $\frac{g_\epsilon/n}{na_n^2} \precsim 1$. We can consider two cases:
\begin{enumerate}
\item If $\frac{n}{\sqrt{\log(n)}} < \epsilon < n$, then $g_{\epsilon/n} \precsim \frac{n}{\epsilon}\left(a_n + \frac{n}{\epsilon}\right) \asymp \frac{d}{\epsilon} + \frac{n^2}{\epsilon^2}$. Hence,
\begin{align*}
\widetilde{\mathcal{L}}\left(\thetahat(G), \theta\right) \precsim \frac{n}{d\epsilon} + \frac{n^3}{d^2\epsilon^2} \precsim \frac{n^3} {d^2\epsilon^2} \asymp \frac{n^2}{\log^4(n)\epsilon^2} < \frac{1}{\log^3(n)} = o(1),
\end{align*}
with probability at least $1 - \frac{1}{n}$. In this regime, one needs $\epsilon=\widetilde{\Theta}(n)$.
\item If $n\log^2(n) < \epsilon$, then $g_{\epsilon/n} \precsim a_n + e^{-\frac{\epsilon}{n}}$. Hence, 
\begin{align*}
\widetilde{\mathcal{L}}\left(\thetahat(G), \theta\right) \precsim \frac{1}{d} + \frac{n}{d^2e^{\epsilon/n}} < \frac{1}{d} + \frac{1}{d^2} \precsim \frac{1}{d} = o(1),
\end{align*}
with probability at least $1 - \frac{1}{n}$. Thus, in this regime, one needs $\epsilon=\widetilde{\Omega}(n)$.
\end{enumerate}
As explained in Remark \ref{remark:EF_D=n} below, our new methods substantially reduce the required scaling on $\epsilon$ to $\widetilde{\Omega}(\sqrt{n})$ and $\Omega(\log^3(n))$, respectively.




\subsubsection{Private convex optimization algorithm}

Another method for SBM community estimation under $(\epsilon, \delta)$-edge DP was proposed in~\cite{CheEtal23}. However, the algorithm is applicable only in the case $k=2$ under Assumption~\ref{assump_same_B11B22}. Furthermore, the algorithm from~\cite{CheEtal23} relies quite heavily on knowledge of the parameters of the SBM, since one must first rescale and recenter the adjacency matrix to obtain a matrix $Y$. The algorithm then proceeds by finding a projection of $Y$ onto a certain subset of positive semidefinite matrices, adds noise, and rounds the entries of the leading eigenvector.  A detailed statement of the algorithm is provided in Algorithm~\ref{alg:TC_Optimiz} in Appendix~\ref{AppSCOpt}.

Theorem 5.3 in \cite{CheEtal23} provides a high-probability error bound on the fraction of misclassified nodes. We have a similar result, Lemma \ref{lemma:ChenPriv+Ut}, which adapts their result to zCDP, for easier group privacy. Everything is then converted to $(\epsilon, \delta)$-node DP, in the proof of Theorem \ref{theorem:Two_Comm_Main_Thm}. A naive node-DP extension of Algorithm~\ref{alg:TC_Optimiz} from edge-DP follows by applying group privacy of size $n$, or equivalently, taking $D = n$ in Theorem~\ref{theorem:Two_Comm_Main_Thm} (cf.\ Remark~\ref{remark:TC_D=n}). This yields $\epsilon=\widetilde{\Omega}(n^2)$. In contrast, our approach reduces the requirement on $\epsilon$ to $\Omega(\log^7(n))$, when $D \asymp d$, and in the smallest-degree regime allowed by the theorem, $d\asymp \log^3(n)$.



\subsection{Other potential approaches}

In fact, as discussed in Section~\ref{SecLower}, the requirement that $\epsilon \rightarrow \infty$ is a drawback not only of the algorithms discussed in the last section, but rather \emph{any} node-private algorithm that recovers the community assignments consistently. In Section~\ref{SecProj}, we will show how to salvage these algorithms with a proper choice of privacy parameter and a degree truncation preprocessing step that still guarantees the correct level of privacy. Accordingly, we now outline a few alternative algorithms for edge privacy based on private low-rank matrix/subspace estimation that will be analyzed in the sequel. Although these algorithms were not proposed in the literature for the purpose of graph estimation, they can be very naturally applied to an observed adjacency matrix and combined with a spectral clustering step to obtain estimates of the underlying communities. Rigorous privacy and utility guarantees will be provided in Section~\ref{SecProj}.


\subsubsection{Private low-rank matrix estimation}

The next algorithm we consider is based on the matrix estimation approach from \cite{hardt2014noisy}. The idea is to estimate the adjacency matrix privately using a private low-rank matrix estimation procedure, and then perform spectral clustering. A detailed statement of the algorithm is provided in Algorithm~\ref{alg:Matrix_Estimation_hardt2014} in Appendix~\ref{AppSCLR}.

We briefly discuss naively extending edge DP to node DP, via group privacy of size $n$. As explained in Remark \ref{remark:ME_D=n} below, this can be done by taking $D = n$ in Theorem \ref{theorem:Matrix_Mech_Main_Thm}. This results in a rapid growth on $\epsilon$, namely $\epsilon = \widetilde{\Omega}(n^2)$. On the other hand, our degree truncation approach gives $\epsilon = \widetilde{\Omega}(n)$, in the smallest-degree regime allowed by the theorem, $d \asymp \sqrt{n}\log^5(n)$.


\subsubsection{Private approximate subspace estimation}

We now propose an algorithm based on a differentially private approximate subspace estimator~\cite{SinSte21, KamEtal22}. For technical reasons, we use the version of the algorithm appearing in~\cite{KamEtal22} (see Algorithm 2 in that paper). However, we perform the aggregation using a private ``ball-finding" algorithm from \cite{mahpud2022differentially}, designed to privately output the center of a ball which, with high probability, contains at least half the input points. Since this ball-finding algorithm provides zCDP, we write the other privatization step, namely the Gaussian noise addition, in zCDP, as well. Note that we can view the DP approximate subspace algorithm as follows: Given an input matrix $X \in \real^{f \times \ell}$, with rows representing individual data points, the algorithm finds a matrix of orthonormal vectors $V \in \real^{f \times k}$, such that $\text{col}(V)$ approximates $\text{row}(X)$, and $V$ is zCDP. A detailed statement of the algorithm is provided in Algorithm~\ref{alg:Subspace_Estimation_GoodC} in Appendix~\ref{AppSCSubspace}.

Suppose we naively extend edge DP to node DP in this case. We consider unweighted case, since the message for the weighted one is analogous. Performing group privacy of size $n$ is equivalent to taking $D = n$ in Theorem \ref{theorem:GoodCent_Main_Thm} (cf.\ Remark \ref{remark:SENW_D=n}). We obtain $\epsilon = \widetilde{\Omega}(n^2)$, while our degree truncation method gives $\epsilon=\widetilde{\Omega}(n)$, in the smallest-degree regime allowed by the theorem, $d \asymp \sqrt{n}\log^{13/2}(n)$.


\section{Private PCA and Lipschitz extensions}
\label{SecPPCA}

We now present our first main algorithm, which takes the private PCA algorithm in~\cite{ChaEtal13} as a starting point. \cite{ChaEtal13} proposed a method for private principal component analysis based on applying the exponential mechanism (cf.\ Lemma~\ref{LemExponential}). Specifically, given a data matrix $X \in \real^{n \times p}$, we sample a matrix $V \in \real^{p \times q}$ in the set $\mathbb{O}(p,q)$ of $p\times q$ matrices with orthonormal columns from the matrix Bingham distribution, which has density function proportional to $\exp\left(\frac{\epsilon}{2} \text{tr}(V^T X^TX V)\right)$. \cite{ChaEtal13} proved that this method satisfies $(\epsilon, 0)$-DP by showing that it is an instantiation of the exponential mechanism
with score function $\text{score}_X(V) = \text{tr}(V^TX^TX V)$, and derived accuracy guarantees in the case $q=1$. Importantly, the privacy analysis depends on an $O(1)$ sensitivity bound on $|\text{score}_X(V) - \text{score}_{X'}(V)|$, where $X$ and $X'$ are data matrices differing in one row.

\subsection{Lipschitz extension}

Two key observations allow us to leverage the private PCA algorithm for SBMs. Indeed, it is relatively straightforward to calculate a sensitivity bound of the form
\begin{align*}
\sup_{\|v\|_2 = 1} |v^T(A^TA - (A')^TA') v| = O(1),
\end{align*}
when $A$ and $A'$ are the adjacency matrices of two graphs differing by a single edge, so the same analysis would easily yield an $\epsilon$-DP algorithm in the $k=2$ case by sampling unit vectors from a distribution proportional to $\exp\left(c\epsilon \cdot \text{score}_A(v)\right)$, where $\text{score}_A(v) := v^TA^2v$. However, if we allow the graphs to be node-adjacent, the sensitivity bound inflates to $O(n)$, which creates an undesirable dependence of $\epsilon$ on $n$ in order to achieve comparable accuracy to the private PCA analysis.

On the other hand, we can borrow an idea of a ``Lipschitz extension" to circumvent the issue~\cite{naor2015}. We say that a function $f: (X,d) \rightarrow (Y,e)$ is $c$-Lipschitz if for all $x_1,x_2 \in X$, we have $e(f(x_1),f(x_2) \leq c d(x_1,x_2)$.
For a function $f:(E,d) \rightarrow (Y,e)$ which is $c$-Lipschitz on $E$, with $E \subseteq X$, we say that $\hat{f}: (X,d) \rightarrow (Y,e)$ is a \emph{Lipschitz extension of $f$ with stretch $s$} if $\hat{f}=f$ on $E$ and $\hat{f}$ is $sc$-Lipschitz. Importantly, for our purposes, we will take $X = \mathcal{G}_n$ and $E = \mathcal{G}_{n,D}$.

We have the following result, proved in Appendix~\ref{AppLemPCAScore}:

\begin{lemma}
\label{LemPCAScore}
Suppose $A$ and $A'$ are adjacency matrices of node-adjacent graphs, both with maximal degree $D$. Then
\begin{align*}
\sup_{V \in \mathbb{O}(p,q)} |\Tr(V^TA^2V) - \Tr(V^T(A')^2V)| \le qD^2.
\end{align*}
\end{lemma}

We now show how to obtain a suitable Lipschitz extension to the score function $\text{score}_{A^2}(V) = \Tr(V^TA^2V)$ in a computationally feasible manner. We leverage a strategy used in \cite{kasiviswanathan2013analyzing, borgs2015private, chen2024private}. For an adjacency matrix $A \in \{0, 1\}^{n \times n}$ and a matrix $V \in \mathbb{O}(n,q)$, define
\begin{equation*}
s_{A^2}(V) = \text{score}_{A^2}(V) + \Tr(A^2 J_n) = \Tr\left(A^2(VV^T + J_n)\right),
\end{equation*}
where $J_n \in \real^{n \times n}$ denotes the all-ones matrix. In particular, the entries of $VV^T + J_n$ are always nonnegative:
\begin{equation*}
|e_i^T(VV^T)e_j| \le \|e_i\|_2 \|e_j\|_2 \|VV^T\|_2 = 1,
\end{equation*}
so the entries of $VV^T$ are all at least $-1$.
We define $\widehat{s}_{A^2}(V)$ to be the solution to the linear program
\begin{align}
\label{EqnLipLP}
& \max_{C \in \real^{n \times n}} s_C(V) \notag \\
& \text{s.t. } C = C^T, \notag \\
& \qquad 0 \le C_{ij} \le (A^2)_{ij} \quad \forall (i,j), \notag \\
& \qquad \|C\|_\infty \le D^2 \quad \forall 1 \le i \le n.
\end{align}
In particular, note that when $A$ is the adjacency matrix of a graph of degree bounded by $D$, the matrix $C = A^2$ is the feasible set for the linear program, since $\|A^2\|_\infty \le \|A\|_\infty^2 \le D^2$. Since the entries of $VV^T+J_n$ are nonnegative, it follows that $\widehat{s}_{A^2}(V) = s_{A^2}(V)$. Then
\begin{align}
\label{EqnExpPCA}
f_{A^2}(V) &\propto \exp\left(\frac{\epsilon}{2(q+2)D^2} \widehat{s}_{A^2}(V)\right) = \exp\left(\frac{\epsilon}{2(q+2)D^2} s_{A^2}(V)\right) \notag \\
&= \exp\left(\frac{\epsilon}{2(q+2)D^2} (\Tr(V^TA^2V) + \Tr(A^2J_n))\right) \propto \exp\left(\frac{\epsilon}{2(q+2)D^2} \Tr(V^TA^2V)\right).
\end{align}
Furthermore, w.h.p., the degree of $A$ drawn from an SBM is at most $2d < 3d \leq D$. Hence, for the utility analysis, it suffices to analyze the sampling mechanism~\eqref{EqnExpPCA}. For general adjacency matrices $A$, we have $\widehat{s}_{A^2}(V) \le s_{A^2}(V)$. The following privacy guarantees are proved in Appendices \ref{AppLemExpPCA} and \ref{AppExpPCA_Cor}:

\begin{lemma}
\label{LemExpPCA}
The mechanism that samples $V \in \mathbb{O}(n,q)$ from a distribution $f_{A^2}(V) \propto \exp\left(\frac{\epsilon}{2(q+2)D^2} \widehat{s}_{A^2}(V)\right)$, given an observed adjacency matrix $A$, satisfies $(\epsilon, 0)$-node DP. 
\end{lemma}

\begin{corollary}
\label{cor:ExpPCA_Cor}
The mechanism that samples $V \in \mathbb{O}(n,q)$ from a distribution $f_{A^2, M}(V) \propto \exp\left(\frac{\epsilon}{2(q+2)D^2} (\widehat{s}_{A^2}(V) + V^TMV)\right)$, given an observed adjacency matrix $A$, satisfies $(\epsilon, 0)$-node DP, for any fixed $M \in \mathbb{R}^{n \times n}$. 
\end{corollary}

\subsection{Algorithm details}

We now present the algorithm formally.

\begin{algorithm}[h!]
\caption{Private Eigenvector Deflation via Lipschitz Extensions}
\label{AlgEF}
\begin{algorithmic}[1]
\Function{$\mathcal{A}_{EigDefLE}^{(\epsilon)}$}{adjacency matrix $A \in \{0, 1\}^{n \times n}$ based on $G \in \mathcal{G}_n$, number of communities $k$, truncation parameter $D$, privacy parameter $\epsilon$}
\State Compute the Lipschitz extension $\widehat{s}_{A^2}(v)$.
\State Set $A_1 = A^2$ and let $\widehat{s}_{1}(v) = \widehat{s}_{A^2}(v)$.
\For{$i = 1, \dots, k - 1$}
    \State Sample a unit vector $\hat{v}_i$ from the distribution $f_i(v) \propto \exp\left(\frac{\epsilon}{6D^2} \widehat{s}_{i}(v)\right)$.
\If{$\|A\|_\infty \le D$}
\State{Sample from the Bingham distribution with density proportional to $\exp\left(\frac{\epsilon}{6D^2} v^T A_i v\right)$}
\Else
\State{Apply rejection sampling to $f_i(v)$ with reference distribution $g(v) \propto \left(\frac{\epsilon}{6D^2} v^T A_i v\right)^{-n/2}$}
\EndIf
        \State Set $\hat{\sigma}_i = \min\left\{\left[\hat{v}_i^T A_i \hat{v}_i\right]_{-D^2}^{D^2} + \mathrm{Lap}\left(\frac{2D^2}{\epsilon}\right), n^2\right\}$.
    \State Define the deflated adjacency matrix $A_{i+1} = A_i - \hat{\sigma}_i \hat{v}_i \hat{v}_i^T$.
    \State Define $\widehat{s}_{i+1}(v) = \widehat{s}_{A^2}(v) - v^T \left(\sum_{j=1}^i \sigmahat_j \hat{v}_j \hat{v}_j^T\right) v$.
    \EndFor
\State Sample a unit vector $\hat{v}_k$ from the distribution $f_k(v) \propto \exp\left(\frac{\epsilon}{6D^2} \widehat{s}_{k}(v)\right)$.
\State Return $[\hat{v}_1, \dots, \hat{v}_k]$.
\EndFunction
\end{algorithmic}
\end{algorithm}

\subsubsection{Privacy analysis}

We use adaptive composition (cf.\ Lemma~\ref{lemma:adapt_comp}) to justify the privacy of Algorithm~\ref{AlgEF}. In particular, we claim that the algorithm satisfies $(2k\epsilon, 0)$-node DP.

Consider iteration $i+1$. We claim that if we fix the values $\{(\sigmahat_j, \hat{v}_j)\}_{j=1}^i$, releasing $(\sigmahat_{i+1}, \hat{v}_{i+1})$ satisfies $(2\epsilon, 0)$-node DP. Indeed, the release of $\sigmahat_{i+1}$ certainly satisfies $(\epsilon, 0$)-DP privacy due to the truncation step. We now claim that $\widehat{s}_{i+1}(v)$ is a Lipschitz extension of $s_{A_{i+1}}(v) = v^T A_{i+1} v + \Tr(A^2 J_n)$. Indeed, we have
\begin{equation*}
v^T A_{i+1} v + \Tr(A^2 J_n) = v^T\left(A^2 - \sum_{j=1}^i \sigmahat_j \hat{v}_j \hat{v}_j^T\right) v + \Tr(A^2 J_n),
\end{equation*}
and since we are treating $\sum_{j=1}^i \sigmahat_j \hat{v}_j \hat{v}_j^T$ as fixed, we can obtain a Lipschitz extension of this expression by taking a Lipschitz extension of $v^T A^2 v + \Tr(A^2 J_n)$ and subtracting $v^T \left(\sum_{j=1}^i \sigmahat_j \hat{v}_j \hat{v}_j^T\right) v$, which is exactly $\widehat{s}_{i+1}(v)$. Thus, sampling a unit vector from the distribution proportional to $\exp\left(\frac{\epsilon}{6D^2} \widehat{s}_{i+1}(v)\right)$ indeed satisfies $(\epsilon, 0)$-DP. Adaptive composition finishes the proof.



\subsubsection{Computational feasibility}
\label{SecSampling}

Although the computation of the Lipschitz extension function $\widehat{s}_{A^2}(V)$, for any pair $(A,V)$, only involves solving a linear program~\eqref{EqnLipLP}, we also need to consider the computational cost of sampling vectors from the proposed distributions in Algorithm~\ref{AlgEF}. As noted in equation~\eqref{EqnExpPCA}, if the observed adjacency matrix $A$ has bounded degree, sampling form the exponential mechanism is equivalent to sampling from the matrix Bingham distribution. \cite{ChaEtal13} suggest using Markov Chain Monte Carlo techniques popular in statistics~\cite{hoff2009simulation} to sample from the matrix Bingham distribution. However, they also acknowledge that such sampling techniques do \emph{not} come with theoretical guarantees regarding rates of convergence of the Markov chain to the stationary distribution. The subsequent papers~\cite{kapralov2013differentially, ge2021efficient} provide algorithms for sampling from the Bingham distribution in the case $q=1$ with runtime guarantees that are polynomial in $n$ (in the case of \cite{ge2021efficient}, the \emph{expected} runtime is polynomial). The very recent paper~\cite{yun2025high} provides an efficient method for approximate sampling from the matrix Bingham distribution for arbitrary $q$, although the guarantee derived in that paper is that the approximation error tends to 0 as $n \rightarrow \infty$, and one would require a nonasymptotic result in order to be compatible with the sorts of utility bounds we would hope to derive.

Nonetheless, we show how to leverage existing sampling algorithms to provide a computationally feasible algorithm for sampling from the exponential mechanism when $q=1$, described in lines 5--10 of Algorithm~\ref{AlgEF} (and applied similarly in line 15). We first check if $\|A\|_\infty \le D$. If so, the distribution agrees with a Bingham distribution, and we apply a polynomial-time sampling algorithm~\cite{kapralov2013differentially, ge2021efficient}. On the other hand, if $\|A\|_\infty > D$, we use rejection sampling. Straightforward arguments show that the expected runtime of rejection sampling is exponential for the Bingham distribution~\cite{kapralov2013differentially, kent2013new}, when samples are drawn from the angular central Gaussian envelope $g(v)$ (i.e., samples from a multivariate Gaussian distribution are rescaled to the unit sphere). Although the score function of the exponential mechanism is proportional to $\exp(c\widehat{s}_A(v))$ rather than $\exp(c s_A(v))$, we know the Lipschitz extension satisfies $\widehat{s}_A(v) \le s_A(v)$, so a slight modification of the argument in \cite{kent2013new} shows that $\sup_{\|v\|_2 = 1} \left|\frac{f_A(v)}{g(v)}\right| \le M$ for some $M$ scaling exponentially with $n$ (note that since $\widehat{s}_A(v) \ge -k n \|v\|_2^2$, the normalizing factor can be lower-bounded by $\exp(-cn^2)$ times the surface area of the sphere), and the expected runtime is then bounded by $M$. Since an adjacency matrix $A$ sampled from an SBM satisfies $\|A\|_\infty \le D$, with high probability, we conclude that the sampling algorithm is polynomial-time, with high probability, over the randomness of the inputs. We emphasize that sampling is more challenging in this setting than in that of \cite{chen2024private}, where the goal is to estimate the $k \times k$ probability matrix $B$ rather than the membership matrix $\mathbf{M}_{n \times k}$; in the former case, one may discretize the space and sample efficiently once it is possible to evaluate Lipschitz extensions.


\subsection{Utility analysis}

We now turn to analyzing the accuracy of our Lipchitz extension-based algorithm for graphs drawn from a stochastic block model.

\subsubsection{Two communities ($k = 2$)}
\label{SecTwoPCA}

We carefully analyze the accuracy in the case when $k = 2$. For the results in this section, we assume the matrix $B$ satisfies the homogeneity condition in Assumption~\ref{assump_same_B11B22}.

The challenge in applying our private eigenvector extraction approach to stochastic block models is that we only have utility guarantees for $q=1$. However, suppose $A \sim SBM(n,2,B, \theta)$, and we are in the homogeneous case where $(B_0)_{11} = (B_0)_{22} > (B_0)_{12} = (B_0)_{21}$. Then the top eigenvector of $P = \E[A] + \mathrm{diag}(P) = \E[A] + B_{11}I_n$ is known to be $\frac{\textbf{1}_n}{\sqrt{n}}$, with eigenvalue equal to the common expected degree $\sigma_1(P) = \frac{\mathbf{1}_n^T P\mathbf{1}_n}{n} \asymp d$. If we can accurately extract the second eigenvector of $P$, sorting the components of the eigenvector will lead to an accurate clustering, as rigorized in the utility analysis below.

Consider an algorithm which first constructs a private estimator $\sigmahat$ of $\sigma_1(P)$, and then constructs the matrix $A^2 - \frac{\sigmahat^2}{n} \textbf{1}_n\textbf{1}_n^T$. Since the average node degree $\frac{\textbf{1}_n^T A\textbf{1}_n}{n}$ has sensitivity at most $2$, we can obtain an $(\epsilon, 0)$-DP estimate of $\sigma_1(P)$ by taking $\sigmahat = \left[\frac{\textbf{1}_n^T A\textbf{1}_n}{n} + \mathrm{Lap}\left(\frac{2}{\epsilon}\right)\right]_{0}^n$.
Furthermore, we will still have $\sigmahat \asymp d$, w.h.p., assuming $d \succsim \log(n)$ and $\epsilon \asymp \frac{1}{\mathrm{poly}(n)}$. We now use adaptive composition (cf.\ Lemma~\ref{LemBasic}): If we have an algorithm that is $(\epsilon, 0)$-DP for $A^2 - \frac{\sigma^2}{n} \textbf{1}_n \textbf{1}_n^T$, for any value of $\sigma$, then the composition of the two algorithms will be $(2\epsilon, 0)$-DP. Indeed, for a fixed $\sigma$, we can define the score function $s'_A(v) = v^T(A^2 - \frac{\sigma^2}{n} \textbf{1}_n \textbf{1}_n^T) v + \Tr(A^2 J_n)$. If we let $\widehat{s}_A(v)$ be the Lipschitz extension of $v^TA^2v + \Tr(A^2 J_n)$ discussed earlier, clearly, $\widehat{s}'_A(v) := \widehat{s}_A(v) - \frac{\sigmahat^2}{n} v^T \textbf{1}_n \textbf{1}_n^Tv$ is an extension of $s_A'(v)$ with the same Lipschitz bound $(q+2)D^2$ (with respect to perturbations in $A$). Putting the results together, we conclude that the algorithm which samples a vector from
\begin{equation*}
\exp\left(\frac{\epsilon}{2(q+2)D^2} \widehat{s}'_A(v)\right) \propto \exp\left(\frac{\epsilon}{2(q+2)D^2} \left(\widehat{s}_A(v) - \frac{\sigmahat^2}{n} v^T \textbf{1}_n \textbf{1}_n^T v\right)\right)
\end{equation*}
satisfies $(2\epsilon, 0)$-node DP. Note in particular that if $\|A\|_\infty \leq 2d$, we have $\widehat{s}_A(v) = v^T A^2 v$, so we are back in the setting of sampling from a Bingham distribution. A detailed description of the algorithm can be found in Appendix~\ref{AppPCALip}.

We now have our main theorem, proved in Appendix~\ref{AppLipPCA}:

\begin{theorem}
\label{theorem:Lip_PCA_Main}
Algorithm \ref{alg:privPCALipExt} is $(2\epsilon, 0)$-node DP. Let $A \sim SBM(n, 2, B, \theta)$, with parameters satisfying Assumptions \ref{AssSBM} and \ref{assump_same_B11B22}. Let $\hat{u}_2$ be as in Algorithm \ref{alg:privPCALipExt}, using $\gamma \asymp 1$. If in addition, $D \in [3d, n]$, $d \in (\sqrt{n}\log^2(n), o(n))$, and $D^2\log(n) < \epsilon < nD^2$, then
\begin{align*}
\widetilde{\mathcal{L}}\left(\mathcal{A}_{PLE}^{(\epsilon)}(G), \theta\right) \precsim \frac{1}{\sqrt{d}} + \frac{D^2}{\epsilon} < \frac{1}{\log(n)},
\end{align*}
with probability at least $1 - \frac{1}{\mathrm{poly}(n)}$.
\end{theorem}

\begin{remark}
\label{remark:PCA_D=n}
Consider taking $D = n$ in Theorem \ref{theorem:Lip_PCA_Main}. This corresponds to the naive group-privacy baseline, and it results in $\epsilon = \widetilde{\Omega}(n^2)$. Using the Lipschitz extension method, for $D \asymp d$, we obtain $\epsilon = \widetilde{\Omega}(n)$ in the smallest-degree regime allowed by the theorem, $d \asymp \sqrt{n}\log^2(n)$.
\end{remark}

\begin{remark}
Theorem \ref{theorem:Lip_PCA_Main} relies on Lemma \ref{cor:ExpPCA_Cor}. The latter is based on Lemma \ref{LemExpPCA}, which considers the sensitivity for privacy in terms of the difference of node-adjacent matrices. One could similarly derive an analog of Theorem \ref{theorem:Lip_PCA_Main} for weighted graphs with bounded weights, but we do not include such an analysis to avoid excessive technicalities.
\end{remark}

\subsubsection{More communities ($k > 1$)}

Unfortunately, the utility analysis for private PCA in the case $k > 1$ was left unaddressed in \cite{ChaEtal13}, and it appears to still be an open question. It is not clear how to extend the geometric argument in the previous subsection to analyze the accuracy of a matrix $V \in \mathbb{O}(n,k)$ sampled from the Bingham distribution. As mentioned in Section~\ref{SecSampling}, rigorous guarantees for computationally feasible sampling algorithms when $q > 1$ are also still largely absent from the literature.


More directly, we could combine the utility analysis for the exponential mechanism when $q=1$ with the analysis of \cite{LiaEtal24} for error propagation in eigenvector deflation to obtain utility bounds on each of the eigenvectors $\{\hat{v}_1, \dots, \hat{v}_k\}$. This could then be translated into bounds on the misclassification error rate for an SBM using similar techniques to the ones employed in previous subsections. We do not include the details here, to avoid excessive technicality.

\section{Privacy from smooth projections}
\label{SecProj}

The second approach we present combines ideas from the papers~\cite{blocki2013differentially, kasiviswanathan2013analyzing} on general methods for converting algorithms $\mathcal{A}$ which are private on bounded-degree graphs to algorithms which are private on \emph{all} input graphs, where accuracy guarantees are only required on ``typical" inputs, which fall in the bounded-degree subset (and these are the inputs we encounter, w.h.p.). Both papers suggest constructing a mapping $T_D: \mathcal{G}_n \rightarrow \mathcal{G}_{n,2D}$ with certain smoothness properties, then applying the algorithm which is private on bounded-degree graphs to $T_D(G)$. In particular, \cite{kasiviswanathan2013analyzing} suggest applying $\mathcal{A}$ to $T_D(G)$ with privacy parameters scaling with a bound on the local sensitivity of $T$, which is well-behaved provided the map $T$ is sufficiently smooth.

Lemma~\ref{lemma:hp_removal_d} implies that for stochastic block models, a realization of the graph $G$ has degree bounded by a $2d$, w.h.p. Since $d = o(n)$ under our assumptions, we can obtain nontrivial private algorithms for \emph{bounded-degree} graphs by replacing the privacy parameters of an edge-private algorithm (cf.\ Section~\ref{SecEdge}) by $(O(\frac{\epsilon}{d}), O(\frac{\delta}{d}))$.
However, the method of \cite{kasiviswanathan2013analyzing} is not directly applicable to our setting for two reasons: (i) It is not entirely clear what combination of a degree truncation function and local sensitivity bound leads to a nontrivial result, and (ii) whereas the theory of \cite{kasiviswanathan2013analyzing} tracks the degradation of privacy parameters over the separate steps of the algorithm where the privacy parameter $\epsilon$ is treated as a constant, in the private SBM setting we are considering, $\epsilon$ increases with $n$ (cf.\ Section~\ref{SecLower} below). Consequently, we must be much more careful in tracking the growth of $\epsilon$ when we compose the truncation function (step 1) with a private algorithm on the bounded-degree space (step 2). We circumvent this issue by leveraging an algorithm appearing in~\cite{blocki2013differentially}, which outputs a projection map \emph{and} an upper bound on the distance between the original graph and its projection, which can be used to privately calibrate the parameters fed back into the second step of the algorithm.

We begin by presenting a result from \cite{blocki2013differentially} regarding the construction of an efficient projection map $T_D : \mathcal{G}_n \rightarrow \mathcal{G}_{n, 2D}$, a low-sensitivity distance function $d_{T_D} : \mathcal{G}_n \rightarrow \mathbb{R}$, and a privatized upper bound $\hat{L}_{\epsilon, \delta}$. In the following linear program, we have a variable $x_u$ for each node $u$. In addition, for each edge $(u, v)$, we have a variable $w_{uv}$, which represents whether the edge $(u, v)$ remains in $T_D(G)$. We write $a_{uv} = 1$ if $(u, v)$ is in $G$, and $a_{uv} = 0$ otherwise.

\begin{definition} [Degree truncation function]
\label{def:smooth_proj}
\mbox{}
\begin{enumerate}
\item Solve the following linear program in $n + \binom{n}{2}$ variables:
\begin{equation}
\label{eq:Blocki_LP_Eq}
\begin{array}{ll@{}ll}
\min\sum\limits_{u} x_u &\\
\mathrm{subject \ to}& \displaystyle x_u \geq 0, \ \forall u,\\                                                &w_{uv} \geq 0, \ \forall u, v,\\
    &a_{uv} \geq w_{uv} \geq a_{uv} - x_u - x_v, \ \forall u, v,\\
    & \sum_{v \neq u}w_{uv} \leq D, \ \forall u.
\end{array}
\end{equation}
Denote the optimal solution by $(x^*, w^*)$, and define $T_D(G)$ to be the graph obtained by removing any edge $(u, v)$ from $G$ such that $x_u^* > \frac{1}{4}$ or $x_v^* \geq \frac{1}{4}$.
\item Define $d_{T_D}(G) = 4\sum_{u}x_u^*$. 
\item For parameters $\epsilon, \delta \ge 0$, compute the associated random quantity
\begin{equation*}
\hat{L}_{\epsilon, \delta}(G) = \max\left\{\frac{1}{2}, 5 + 2 d_{T_D}(G) + \mathrm{Lap}\left(\frac{8}{\epsilon}\right) + \frac{8 \log(1/\delta)}{\epsilon}\right\}
\end{equation*}
(noting that we suppress the dependence on $D$ in the notation for $\hat{L}_{\epsilon, \delta}$).
\end{enumerate}
\end{definition}

Note that in our applications of $\hat{L}_{\epsilon_1, \delta_1}$, we will take $\epsilon_1 \asymp 1$ and $\delta \asymp \frac{1}{\mathrm{poly}(n)}$, so $\hat{L}_{\epsilon_1, \delta_1}(G) = \Theta(\log(n))$ if $G \in \mathcal{G}_{n, 2D}$. We state this formally as an assumption:

\begin{assumption}
\label{ASS_eps_1_delt_1}
We consider $\epsilon_1 \asymp 1$ and $\delta_1 \asymp \frac{1}{\mathrm{poly}(n)}$.
\end{assumption}

\begin{remark}
\label{remark:No_trunc}
Note that $T_D(G) = G$ if $G \in \mathcal{G}_{n, D}$. By taking $w_{uv}^* = a_{uv}$, for all nodes $u, v$, setting $x_u = 0$ for all nodes $u$ gives a feasible solution to the LP~\eqref{eq:Blocki_LP_Eq}. However, $0$ is a lower bound on the objective $\sum_{u}x_u$. Hence, $x_u^* = 0$, for all nodes $u$, is optimal. Thus, because we remove an edge $(u, v)$ only if either the optimal $x_u^* > \frac{1}{4}$, or the optimal $x_v^* \geq \frac{1}{4}$, we obtain $T_D(G) = G$.
\end{remark}

When the input graph $G$ is weighted, we use the following alternative algorithm:
\begin{definition} [Degree truncation function (weighted graphs)]
\label{def:weight_smooth_proj}
For $G \in \mathcal{W}_n$, we define $T_{w,D} : \mathcal{W}_n \rightarrow \mathcal{W}_{n, 2D}$ as follows: 
\begin{enumerate}
    \item Consider the map $\mathcal{B} : \mathcal{W}_n \rightarrow \mathcal{G}_n$ that turns non-zero weights to $1$ and keeps $0$ weights at $0$, i.e., binarizes the weighted graph, treating each non-zero weight as $1$. Compute $\mathcal{B}(G)$.
    \item Compute $(T_{D}(\mathcal{B}(G)), d_{T_D}(\mathcal{B}(G)))$ as in step 1 of Definition~\ref{def:smooth_proj}.
    \item Where there is an edge $(i, j)$ in $T_{D}(\mathcal{B}(G))$, add back the corresponding non-zero weight that existed for the edge $(i, j)$ in $G$. The final weighted graph is $T_{w,D}(G)$.
    \item Define $\hat{L}_{\epsilon, \delta}(G)$ as in step 3 of Definition~\ref{def:smooth_proj}.
\end{enumerate}
\end{definition}

The main idea, stated rigorously in Theorems~\ref{theorem:generic_red} and~\ref{theorem:generic_red_pure} below, is that we can obtain a node-DP algorithm by applying an algorithm which is $(\epsilon', \delta')_{2D}$-node DP to the degree-truncated graph $T_D(G)$. Privacy is guaranteed whenever the parameters $(\epsilon', \delta')$ are appropriately defined functions of our target privacy parameters $(\epsilon, \delta)$ and $\hat{L}_{\epsilon_1, \delta_1}(G)$; and we have accuracy guarantees that hold with high probability when $G$ is an SBM, if $D$ is chosen appropriately as a function of $d$. Importantly (cf.\ Lemma~\ref{lemma:Lhat}), the quantity $\hat{L}_{\epsilon_1, \delta_1}(G)$ is a high-probability upper bound on the local sensitivity of $T_D$ that is differentially private.




In the subsections that follow, we first provide the main technical result underlying the private degree truncation step (Section~\ref{section:Private Degree Truncation Blackbox}). We then combine the method with the methods discussed in Section~\ref{SecEdge} earlier (Sections~\ref{SecEF}--\ref{SecDeflatePCA}). Table~\ref{tab1} summarizes the assumptions required by each of the methods, along with comments on their strengths and weaknesses. In particular, we highlight the fact that the first two rows are applicable to relatively sparse graphs $d \asymp \text{poly}(\log n)$), whereas the latter rows require $d = \widetilde{\Omega}(\sqrt{n})$. Another key difference is applicability to graphs with edge weights. Furthermore, Algorithm~\ref{alg:privPCALipExt} is the only algorithm we propose which satisfies pure $(\epsilon, 0)$-node DP.
We note that by Remark \ref{remark:avg<=worst}, the utility bounds in Theorems \ref{theorem:Lip_PCA_Main}, \ref{theorem:Edge_Flip_Main_Thm}, \ref{theorem:Matrix_Mech_Main_Thm}, \ref{theorem:GoodCent_Main_Thm}, and \ref{theorem:GoodCent_Main_Thm_Weighted} hold for the overall misclassification error $\mathcal{L}$ as well as for $\widetilde{\mathcal{L}}$.

\begin{table}[!h]
    \centering
    \small  
    \setlength{\tabcolsep}{2.5pt}  
    \renewcommand{\arraystretch}{1.2} 
    \begin{tabular}{|c|c|c|c|}
        \hline
        \textbf{Method and Assumptions} & \textbf{Upper Bound} & \textbf{Additional Comments}\\
        \hline
        \makecell{Algorithm \ref{alg:EF_Spec_Clus} applied to the symmetric edge-flip mechanism, \\ $D\log^2(n) \precsim \epsilon$ and $d \succsim \log(n)$} & 
        \makecell{$O\left(\frac{1}{d}\right)$} & 
        \makecell{Does not work for weighted graphs. \\ Best rate in the given setting} \\
        \hline
        \makecell{Algorithm \ref{alg:TC_Optimiz} for $k = 2$, $D^2\log(n) < \epsilon$, $d > \log^3(n)$} & 
        \makecell{$\widetilde{O}\left(\frac{1}{\sqrt{d}} + \frac{D^2}{\epsilon}\right)$} & 
        \makecell{Does not work for weighted graphs \\ with unbounded weights} \\ 
        \hline \hline
        \makecell{Algorithm \ref{alg:EF_Spec_Clus} applied to the symmetric edge-flip mechanism, \\ $\frac{D}{\sqrt{\log(n)}} < \epsilon < D$, $d > \sqrt{n}\log^2(n)$} &  
        \makecell{$\widetilde{O}\left(\frac{nD^2}{d^2\epsilon^2}\right)$} & 
        \makecell{Does not work for weighted graphs} \\
        \hline
        \makecell{Algorithm \ref{alg:Matrix_Estimation_hardt2014}, $D^2\log(n) < \epsilon$, $d > \sqrt{n}\log^5(n)$} & \makecell{$\widetilde{O}\left(\frac{1}{\sqrt{n}} + \frac{D^2}{\epsilon}\right)$} & 
        \makecell{Does not work for weighted graphs \\ with unbounded weights} \\
        \hline
        \makecell{Algorithm \ref{alg:Subspace_Estimation_GoodC}, $D^2\log(n) < \epsilon < nD^2\log^3(n)$, $d \succsim \sqrt{n}\log^{13/2}(n)$} & 
        \makecell{$\widetilde{O}\left(\frac{D}{\sqrt{\epsilon}}\right)$} & 
        \makecell{Works for weighted graphs \\ (e.g., $N(0, a)$, \\ or $a$-sub-Gaussian weights) \\} \\ 
        \hline \hline
        \makecell{Algorithm \ref{alg:privPCALipExt}, $D^2\log(n) < \epsilon < nD^2$, \\ $d \in (\sqrt{n}\log^2(n), o(n))$} & 
        \makecell{$O\left(\frac{1}{\sqrt{d}} + \frac{D^2}{\epsilon}\right)$} & 
        \makecell{Pure $\epsilon$-node DP. Proved for $k=2$. \\ Does not work for weighted graphs \\ with unbounded weights. \\} \\ 
        \hline
    \end{tabular}
    \caption{Bounds on $\mathcal{L}(\hat{\theta},\theta)$, where $\theta$ are the true cluster assignments and $\hat{\theta}$ are the estimated assignments. We assume $k \asymp 1$, $\frac{\log(n)}{n} \precsim a_n = o(1)$, $d \asymp a_n n$, and $D \in [3d, n]$ (for Algorithms \ref{alg:Matrix_Estimation_hardt2014} and \ref{alg:Subspace_Estimation_GoodC}, we allow $D \succsim  3d/\log(n)$). The approximate-DP mechanisms are $\left(O(\epsilon), 1/\mathrm{poly}(n)\right)$-node DP.}
    \label{tab1}
\end{table}


\subsection{Private degree truncation}
\label{section:Private Degree Truncation Blackbox}

We now formally present our general black-box method for privately truncating the degree of graphs. We use the generic reduction lemma from \cite{kasiviswanathan2013analyzing} in order to guarantee the privacy of $\mathcal{A}(T_D(G))$ over $G \in \mathcal{G}_n$, provided $\mathcal{A}$ is private over $\mathcal{G}_{n, 2D}$. Note that the result, as stated in \cite{kasiviswanathan2013analyzing}, assumes $\mathcal{A}$ is pure $(\epsilon, 0)$-node private over $\mathcal{G}_{n, 2D}$. We state and prove this result for the approximate DP case, which we need in order for it to apply to most of the algorithms we will leverage from Section~\ref{SecEdge}. We also state and prove a version in the pure-DP case, which is tighter than then one in \cite{kasiviswanathan2013analyzing}. 

The following theorem, proved in Appendix~\ref{AppThemGenRed}, extends the result of \cite{kasiviswanathan2013analyzing} to make it valid for approximate DP and a regime where $\epsilon \rightarrow \infty$.

\begin{theorem}
\label{theorem:generic_red}
Suppose $\mathcal{A}^{(\epsilon, \delta)}$ is $(2\epsilon, \delta)_{2D}$-node DP, with $2\epsilon < \log(1/\delta)$. Let $2\epsilon_2 < \log\left(\frac{1}{\delta_2\cdot\mathrm{poly}(n)}\right)$. The algorithm which outputs the pair $(\hat{L}_{\epsilon_1, \delta_1}(G), \mathcal{A}^{(\epsilon_2'(G), \delta_2'(G))}(T_D(G)))$, where $(\epsilon_2'(G), \delta_2'(G)) = \left(\frac{\epsilon_2}{\hat{L}_{\epsilon_1, \delta_1}(G)}, \frac{\delta_2}{\hat{L}_{\epsilon_1, \delta_1}(G)}\right)$, satisfies $\left(\epsilon_1 + 2\epsilon_2, e^{\epsilon_1}\left(\frac{1}{\mathrm{poly}(n)} + \delta_1\right)\right)$-node DP.
\end{theorem}

When $\mathcal{A}$ is pure node-DP, we have a slightly simpler version.
Unlike the analysis of \cite{kasiviswanathan2013analyzing}, our result does not require $\hat{L}_{\epsilon_1,\delta_1}$ and $\mathcal{A}$ to share the same $\epsilon$ parameter; it applies even when they have different privacy parameters. The proof is provided in Appendix~\ref{AppThemGenRedPure}.

\begin{theorem}
\label{theorem:generic_red_pure}
Suppose $\mathcal{A}^{(\epsilon)}$ is $(\epsilon, 0)_{2D}$-node DP. Then the algorithm which outputs the pair $(\hat{L}_{\epsilon_1, \delta_1}(G), \mathcal{A}^{(\epsilon_2'(G))}(T_D(G)))$, where
$\epsilon_2'(G) = \frac{\epsilon_2}{\hat{L}_{\epsilon_1, \delta_1}(G)}$, satisfies $(\epsilon_1 + \epsilon_2, e^{\epsilon_1}\delta_1)$-node DP.
\end{theorem}

Importantly, in the analysis of SBMs below, we will exploit the result from Lemma \ref{lemma:hp_removal_d} to conclude that w.h.p., when $D \ge 3d$, we have $d_{T_{D}}(G) = 0$. Then $\hat{L}_{\epsilon_1, \delta_1}(G) = \max\left\{\frac{1}{2}, 5 + \mathrm{Lap}\left(\frac{8}{\epsilon_1}\right) + \frac{8\log(1/\delta_1)}{\epsilon_1}\right\} \asymp \log n$. We emphasize that naively applying the techniques of \cite{kasiviswanathan2013analyzing} would lead to $\hat{L}_{\epsilon_1, \delta_1} \asymp n$, which would lead to much worse utility bounds (essentially replacing $\epsilon$ by $\frac{\epsilon}{n}$, which is what we sought to avoid in Section~\ref{SecEdge}).


\subsection{Boosting strategy}

In our analysis of some of the algorithms below, we will obtain a method for outputting a private estimate $\widetilde{\theta}(G)$ of community assignments satisfying a bound of the form
\begin{equation*}
\mprob(\widetilde{\mathcal{L}}(\widetilde{\theta}(G), \theta) > \xi) \le 0.2,
\end{equation*}
where $\xi = o(1)$. Here, we present an algorithm which can boost this to a consistency guarantee which holds with probability $1-o(1)$. Details are provided in Algorithm~\ref{alg:Boost}.

\begin{algorithm}[h!]
\caption{Graph-based boosting}
\label{alg:Boost}
\begin{algorithmic}[1]
\Function{$\mathrm{GraphBoost}$}{odd split count $T$, upper bound on the loss $\xi$, input graph $G \in \mathcal{G}_n$ or $\mathcal{W}_n$, privacy parameters $\epsilon, \delta$, number of communities $k$, $(\epsilon, \delta)$-node DP mechanism $\widetilde{\theta}$}
    \State Construct subgraphs $\{G_1, \dots, G_T\}$ of $G$, as follows: Independently for each $j$, select the edges of $G$ independently with probability $\frac{1}{T}$ each to obtain $G_j$.
    \State Compute the community estimates $\{\widetilde{\theta}(G_j)\}_{j=1}^T$. Denote their one-hot encoding matrices by $\{\widetilde{\Theta}_j\}_{j=1}^T$.
    \State Select an index $j^* \in [T]$ such that $\sum_{j=1}^T 1\left\{\mathcal{L}(\widetilde{\theta}(G_{j^*}), \widetilde{\theta}(G_j) \le 2\xi\right\} \ge \frac{T+1}{2}$. (If no such index $j^*$ exists, output $\perp$; if there is more than one choice for $j^*$, pick one at random.)
    \State For $j \in [T]$, let $J_{\widetilde{\Theta}_j}$ denote the permutation that optimally aligns $\widetilde{\Theta}_j$ with $\widetilde{\Theta}_{j^*}$, i.e., $J_{\widetilde{\Theta}_{j}} \in \arg\min_{J \in E_k} \|\widetilde{\Theta}_j J - \widetilde{\Theta}_{j^*}\|_0$.
    \State Compute each row of $\Thetahat$ by taking a majority vote among the corresponding rows of $\{\widetilde{\Theta}_j J_{\widetilde{\Theta}_j}\}_{j=1}^T$. Define $\thetahat(G)$ to be the corresponding community assignments.
    \State Return $\thetahat(G)$.
\EndFunction
\end{algorithmic}
\end{algorithm}

The following result, proved in Appendix~\ref{AppPropBoost}, shows that $T \asymp \log(n)$ leads to a $1-\frac{c}{\log(n)}$ probability of success: 

\begin{proposition}
\label{prop:Boost_result}
Suppose Algorithm \ref{alg:Boost} is applied with $T \asymp \log(n)$ an odd integer. Suppose $\widetilde{\theta}$ is an $(\epsilon, \delta)$-node DP algorithm satisfying the guarantee
\begin{equation}
\label{EqnConstBd}
\mprob(\widetilde{\mathcal{L}}(\widetilde{\theta}(A), \theta) > \xi) \le 0.2,
\end{equation}
for some $\xi < \frac{1}{8k}$, when $A \sim SBM(n, k, \frac{B}{T}, \theta)$. Then the output of Algorithm \ref{alg:Boost} satisfies $(\epsilon T, \delta T)$-node DP and the utility bound
\begin{equation*}
\mprob(\widetilde{\mathcal{L}}(\thetahat(G), \theta) > \xi T) \le \frac{c}{\log(n)},
\end{equation*}
when $G \sim SBM(n, k, B, \theta)$.
\end{proposition}

\begin{remark}
We emphasize a subtlety in the analysis of the boosting argument. Although it is a standard tool in learning theory to boost a constant-probability classifier to a high-probability classifier by taking a majority vote over independent repetitions~\cite{shalev2014understanding}, it is less clear how to boost an SBM probability when we only have a single observation $G$ of the graph. Using the observation that ``thinning" an SBM with independent probability $\frac{1}{T}$ for each edge still results in an SBM with rescaled parameters, our boosting algorithm first constructs graphs $\{G_1, \dots, G_T\}$ which marginally have the same distribution, which is a thinned version of the original SBM. However, since the $G_j$'s are all subgraphs of $G$, the subgraphs are weakly dependent, so we cannot apply a Chernoff bound on indicator random variables corresponding to accurate classification of individual $G_j$'s. As our proof reveals, we circumvent this issue by bounding pairwise correlations using the Hirschfeld-Gebelein-R\'{e}nyi (HGR) maximal correlation (cf.\ Appendix~\ref{AppMaxCor}), which conveniently satisfies data processing and tensorizes over the distribution of edges in the $G_j$'s. This is the first application of HGR correlation to PAC learning we are aware of in the literature, and may be an interesting innovation in its own right.
\end{remark}


\subsection{Analysis of edge-flipping mechanism}
\label{SecEF}

We have the following result, proved in Appendix~\ref{AppEFMain}, based on the black-box method from Theorem \ref{theorem:generic_red_pure}:

\begin{theorem}
\label{theorem:Edge_Flip_Main_Thm}
Suppose $G \sim SBM(n, k, B, \theta)$, with parameters satisfying Assumptions \ref{AssSBM} and \ref{ASS_eps_1_delt_1}. Suppose the parameter of the degree truncation function (cf.\ Definition \ref{def:smooth_proj}) satisfies $D \in [3d, n]$. Let $\mathcal{A}_{EF}^{\left(\epsilon/(4D\hat{L}_{\epsilon_1, \delta_1}(G))\right)}$ denote Algorithm \ref{alg:EF_Spec_Clus} with $\gamma \asymp 1$. Then $\thetahat(G) = \mathcal{A}_{EF}^{\left(\epsilon/(4D\hat{L}_{\epsilon_1, \delta_1}(G))\right)}(T_{D}(G))$ is $(\epsilon_1 + \epsilon, O(\delta_1))$-node DP. Furthermore:
\begin{enumerate}
\item If $\frac{D}{\sqrt{\log(n)}} < \epsilon < D$ and $d > \sqrt{n}\log^2(n)$, we have
\begin{align*}
\widetilde{\mathcal{L}}\left(\thetahat(G), \theta\right) \precsim \frac{nD^2\log^2(n)}{d^2\epsilon^2} < \frac{1}{\log(n)},
\end{align*}
with probability at least $1 - \frac{1}{\mathrm{poly}(n)}$.
\item If $D\log^2(n) \precsim \epsilon$, we have
\begin{align*}
\widetilde{\mathcal{L}}\left(\thetahat(G), \theta\right) \precsim \frac{1}{d} \precsim \frac{1}{\log(n)},
\end{align*}
with probability at least $1 - \frac{1}{\mathrm{poly}(n)}$.
\end{enumerate}
\end{theorem}

Note that since Algorithm \ref{alg:EF_Spec_Clus} relies of edge flipping, i.e., a binary operation, the method described in Theorem \ref{theorem:Edge_Flip_Main_Thm} is incompatible with weighted graphs.

\begin{remark}
\label{remark:EF_D=n}
We consider the case when $D = n$ in the context of Theorem \ref{theorem:Edge_Flip_Main_Thm}. That is, we consider the baseline case, i.e., the one when we do not truncate edges, and extend edge-private methods to node-private methods via group privacy. In either of the two cases in Theorem \ref{theorem:Edge_Flip_Main_Thm}, we have $\epsilon = \widetilde{\Omega}(n)$. Thus, if we consider our degree-truncation method with $D \asymp d$, we can reduce the rate at which $\epsilon$ grows. In the smallest-degree regimes allowed by the theorem, $d \asymp \sqrt{n}\log^2(n)$ in case $1$ and $d \asymp \log(n)$ in case $2$, we obtain $\epsilon = \widetilde{\Omega}(\sqrt{n})$ in the first case and $\epsilon = \Omega(\log^3(n))$ in the second case, respectively.
\end{remark}


\subsection{Analysis of private low-rank matrix estimation}
\label{SecLR}

The following result, proved in Section~\ref{AppMatrixMech}, is based on Lemma \ref{lemma:d_tight_priv}, Lemma \ref{lemma:Initial_Conc_bd_matrix_estim}, and the black-box methods from Section \ref{section:Private Degree Truncation Blackbox}:

\begin{theorem}
\label{theorem:Matrix_Mech_Main_Thm}
Let $T \asymp \log(n)$ be an odd integer. Let $G \sim SBM(n, k, B, \theta)$, with parameters satisfying Assumption \ref{AssSBM} applied to $\frac{a_n}{T}$, and suppose the parameter of the degree truncation function (cf.\ Definition~\ref{def:smooth_proj}) satisfies $D \in \left[\frac{3d}{T}, n\right]$. Suppose Assumption \ref{ASS_eps_1_delt_1} holds. Let $d > \sqrt{n}\log^5(n)$, $D^2\log(n) < \epsilon$, $\delta \in \left(\frac{e^{-\Theta(\epsilon/T)}}{\mathrm{poly}(n)}, \frac{e^{-2\epsilon/T}}{\mathrm{poly}(n)}\right)$, and $\widetilde{\theta}(\cdot) = \mathcal{A}_{ME}^{\left(\epsilon/(4TD\hat{L}_{\epsilon_1, \delta_1}(\cdot)), \delta/\hat{L}_{\epsilon_1, \delta_1}(\cdot)\right)}(T_{D}(\cdot))$, where Algorithm~\ref{alg:Matrix_Estimation_hardt2014} is applied with $\gamma \asymp 1$. Then $\thetahat = \mathrm{GraphBoost}\left(T, \xi, G, \frac{3\epsilon}{T}, \frac{1}{\mathrm{poly}(n)}, k, \widetilde{\theta}\right)$, where $\xi = \frac{1}{\sqrt{n}\log^3(n)} + \frac{D^2}{\epsilon\log(n)}$, satisfies $(3\epsilon, 1/\mathrm{poly}(n))$-node DP, and
\begin{align*}
\widetilde{\mathcal{L}}\left(\thetahat(G), \theta\right) \precsim \frac{1}{\sqrt{n}\log^2(n)} + \frac{D^2}{\epsilon} \precsim \frac{1}{\log(n)},
\end{align*}
with probability at least $1 - \frac{c}{\log(n)}$.
\end{theorem}

\begin{remark}
\label{remark:ME_D=n}
Taking $D = n$ in Theorem \ref{theorem:Matrix_Mech_Main_Thm} corresponds to the naive approach of applying the edge-private algorithm directly and converting edge privacy to node privacy via group privacy. In this case, the condition
$D^2\log(n)<\epsilon$ forces $\epsilon = \widetilde{\Omega}(n^2)$ for consistency. The degree-truncation step improves this dependence by allowing us to take $D \asymp d/T \succsim \sqrt{n}\log^4(n)$. Then, in the smallest-degree regime allowed by the theorem, $d \asymp \sqrt{n}\log^5(n)$, we have $\epsilon = \widetilde{\Omega}(n)$, giving an improved growth condition.
\end{remark}

\begin{remark}
The privacy of Theorem \ref{theorem:Matrix_Mech_Main_Thm} relies on Lemma \ref{lemma:hardt2014_conc_bound}. The latter quantifies the sensitivity for privacy in terms of the operator norm of two edge-adjacent matrices, so one could also derive an analog of Theorem \ref{theorem:Matrix_Mech_Main_Thm} for weighted graphs with bounded weights. However, to avoid excessive technicalities, we do not include such an analysis.
\end{remark}


\subsection{Analysis of private approximate subspace estimation}
\label{section:Approximate DP: Private approximate subspace estimation}

Fix $1 \asymp \zeta \in (0, 1/3)$ and assume $3k \geq \log(2/\zeta)$. We also have the following assumption: 

\begin{assumption}
\label{ASS_SE_t_m}
Suppose there exist constants $C^{(1)} \asymp 1$ and $\beta_0 \in \left(\frac{1}{2}, 1\right)$ such that
\begin{align*}
2 \leq \frac{C^{(1)}\sqrt{n\log(n)\log(1/\delta)}}{\epsilon} \leq n^{\beta_0}.
\end{align*}
\end{assumption}
We will choose the number of chunks to be
\begin{align}
\label{eq:t}
t := t_{\epsilon, \delta} = \frac{C^{(1)}\sqrt{n\log(n)\log(1/\delta)}}{\epsilon}.
\end{align}
Under Assumption \ref{ASS_SE_t_m}, we then have $m = \frac{n}{t} \asymp \epsilon\sqrt{\frac{n}{\log(n)\log(1/\delta)}}$ and 
\begin{align*}
(C^{(1)})^2n^{1 - 2\beta_0}\log(n) \leq \frac{\epsilon^2}{\log(1/\delta)} \leq \frac{(C^{(1)})^2n\log(n)}{4}.
\end{align*}

In what follows, we work under Assumption \ref{ASS_SE_t_m} and equation~\eqref{eq:t}.
We split the analysis into unweighted and weighted graphs. The following result is proved in Appendix~\ref{AppGoodUnwt}:

\begin{theorem}
\label{theorem:GoodCent_Main_Thm}
Let $T \asymp \log(n)$ be an odd integer. Let $G \sim SBM(n, k, B, \theta)$, with parameters satisfying Assumption \ref{AssSBM} applied to $\frac{a_n}{T}$, and suppose the parameter of the degree truncation function (cf.\ Definition~\ref{def:smooth_proj}) satisfies $D \in \left[\frac{3d}{T}, n\right]$. Suppose Assumption \ref{ASS_eps_1_delt_1} holds. Let $d > \sqrt{n}T\log^{11/2}(n)$, $D^2\log(n) < \epsilon < nD^2\log^3(n)$, $\delta \in \left(\frac{e^{-\Theta(\epsilon/T)}}{\mathrm{poly}(n)}, \frac{e^{-2\epsilon/T}}{\mathrm{poly}(n)}\right)$, and $\widetilde{\theta}(\cdot) = \mathcal{A}_{SE}^{\left(\epsilon/(5TD\hat{L}_{\epsilon_1, \delta_1}(\cdot)), \delta/\hat{L}_{\epsilon_1, \delta_1}(\cdot)\right)}(T_{D}(\cdot))$, where Algorithm~\ref{alg:Subspace_Estimation_GoodC} is applied with $\gamma \asymp 1$, with $t = t_{\frac{\epsilon}{5DT\hat{L}_{\epsilon_1, \delta_1}(\cdot)}, \frac{\delta}{\hat{L}_{\epsilon_1, \delta_1}(\cdot)}}$. Then $\thetahat = \mathrm{GraphBoost}\left(T, \xi, G, \frac{3\epsilon}{T}, \frac{1}{\mathrm{poly}(n)}, k, \widetilde{\theta}\right)$, where $\xi = \frac{D}{\log^{3/2}(n)\sqrt{\epsilon}}$, satisfies $(3\epsilon, 1/\mathrm{poly}(n))$-node DP, and
\begin{align*}
\widetilde{\mathcal{L}}\left(\thetahat(G), \theta\right) \precsim \frac{D}{\log^{1/2}(n)\sqrt{\epsilon}} \precsim \frac{1}{\log(n)},
\end{align*}
with probability at least $1 - \frac{c}{\log(n)}$.
\end{theorem}

\begin{remark}
\label{remark:SENW_D=n}
Taking $D=n$ in Theorem \ref{theorem:GoodCent_Main_Thm} corresponds to the naive group-privacy baseline. This forces $\epsilon = \widetilde{\Omega}(n^2)$ from $D^2\log(n)<\epsilon$ and $d \succsim \sqrt{n}\log^{13/2}(n)$. With degree truncation, we may instead take $D\asymp d/T$, since the boosting procedure splits the graph
into $T\asymp \log(n)$ subgraphs. Then consistency requires only $\epsilon=\widetilde{\Omega}(n)$, in the smallest-degree regime allowed by the theorem, $d \asymp \sqrt{n}\log^{13/2}(n)$.
\end{remark}

The following result is proved in Appendix~\ref{AppGoodCWt}:

\begin{theorem}
\label{theorem:GoodCent_Main_Thm_Weighted}
Suppose $G \sim WSBM(n, k, B, \theta, \mathscr{W})$, with parameters satisfying Assumptions~\ref{AssSBM} and \ref{ASS_eps_1_delt_1}. Suppose $\mathscr{W}$ is such that the distribution of each edge weight is $s^2$-sub-Gaussian, with $s^2 \asymp a_n$, and suppose the parameter of the degree truncation function (cf.\ Definition~\ref{def:smooth_proj}) satisfies $D \in [3d, n]$. Consider equation~\eqref{eq:t} with $t = t_{\frac{\epsilon}{5D(\hat{L}_{\epsilon_1, \delta_1}(\mathcal{B}(G)) + 1)}, \frac{\delta}{\hat{L}_{\epsilon_1, \delta_1}(\mathcal{B}(G)) + 1)}}$. Let $d > \sqrt{n}\log^{13/2}(n)$, $D^2 < \epsilon < nD^2\log^2(n)$, and $\delta \in \left(\frac{e^{-\Theta(\epsilon)}}{\mathrm{poly}(n)}, \frac{e^{-2\epsilon}}{\mathrm{poly}(n)}\right)$. Then $\thetahat(G) = \mathcal{A}_{SE}^{(\epsilon/(5D(\hat{L}_{\epsilon_1, \delta_1}(\mathcal{B}(G)) + 1)), \delta/(\hat{L}_{\epsilon_1, \delta_1}(\mathcal{B}(G)) + 1))}(T_{w,D}(G))$ is $(3\epsilon, 1/\mathrm{poly}(n))$-node DP, and
\begin{align*}
\widetilde{\mathcal{L}}\left(\thetahat(G), \theta\right) \precsim \frac{D}{\log^3(n)\sqrt{\epsilon}} < \frac{1}{\log^3(n)},
\end{align*}
with probability at least $0.9 - 3\zeta - \frac{1}{\mathrm{poly}(n)}$.
\end{theorem}

\begin{remark}
\label{remark:SEW_D=n}
Similar to Remark \ref{remark:SENW_D=n}, taking $D = n$ in Theorem \ref{theorem:GoodCent_Main_Thm_Weighted} results in $\epsilon = \widetilde{\Omega}(n^2)$. The degree-truncation method reduces this to $\epsilon = \widetilde{\Omega}(n)$, for $D \asymp d$, and $d \asymp \sqrt{n}\log^{13/2}$, in the smallest-degree regime allowed by the theorem.
\end{remark}

\begin{remark}
\label{remark:weighted_boost}
Similar to Theorem \ref{theorem:GoodCent_Main_Thm}, we could apply Algorithm \ref{alg:Boost} to Algorithm \ref{alg:Subspace_Estimation_GoodC} to boost the error probability to $1 - \frac{c}{\log(n)}$. It is possible to extend Proposition \ref{prop:Boost_result} to accommodate weighted graphs, and then apply it in the context of Theorem \ref{theorem:GoodCent_Main_Thm_Weighted}. We do not include the details to avoid excessive technicality. At a high level, this would require rescaling $d$ and $\epsilon$ by $T \asymp \log(n)$, leading to rates for $\widetilde{\mathcal{L}}$ and conditions on $(\epsilon,\delta)$ similar to those in Theorem \ref{theorem:GoodCent_Main_Thm}. The main difference would be the condition on $d$, which would become $d > \sqrt{n}T\log^{13/2}(n) \asymp \sqrt{n}\log^{15/2}(n)$.
\end{remark}

\begin{remark}
We compare Theorems \ref{theorem:GoodCent_Main_Thm} and \ref{theorem:GoodCent_Main_Thm_Weighted}, for unweighted and weighted graphs. Up to logarithmic factors, they have the same conditions and convergence rates; the extra logarithmic factors in Theorem \ref{theorem:GoodCent_Main_Thm} come from the boosting step in Algorithm \ref{alg:Boost}. As noted in Remark \ref{remark:weighted_boost}, the error probability in Theorem \ref{theorem:GoodCent_Main_Thm_Weighted} could also be boosted, yielding the same conditions and rates as in Theorem \ref{theorem:GoodCent_Main_Thm}, except for the growth condition on $d$.

However, a crucial difference lies in the privacy guarantee: the weighted approach allows arbitrary perturbations of the weighted adjacency matrix, rather than only binary edge flips. Another difference comes from the model setup. In the unweighted case, we need $B$ to be invertible (cf.\ Definition~\ref{DefSBM}). For the weighted SBM, we require $B \odot B_w$ to be invertible, with either $B_w \succeq 0$ or $B_w \preceq 0$. Thus, the advantage of the weighted setting is that we do not need any additional spectral separation condition on $B$; it is enough to ensure that the Hadamard product $B \odot B_w$ is invertible. The conditions $B_w \succeq 0$ or $B_w \preceq 0$ ensure that the weighting aligns with the signal from the edge probabilities. 
\end{remark}


\subsection{Analysis of private optimization algorithm}
\label{section:Two-community Algorithm}

We now analyze the private convex optimization algorithm, applicable in the case when $k = 2$. We work under Assumption \ref{assump_same_B11B22} on the entries of the probability matrix $B$.

The following result is proved in Appendix~\ref{AppThmTwoCom}:

\begin{theorem}
\label{theorem:Two_Comm_Main_Thm}
Let $G \sim SBM(n, k, B, \theta)$, with parameters satisfying Assumption \ref{AssSBM} and $k=2$, and suppose the parameter of the degree truncation function (cf.\ Definition~\ref{def:smooth_proj}) satisfies $D \in [3d, n]$. Suppose Assumptions~\ref{assump_same_B11B22}, and \ref{ASS_eps_1_delt_1} hold. Let $\delta \in \left(\frac{e^{-\Theta(\epsilon)}}{\mathrm{poly}(n)}, \frac{e^{-2\epsilon}}{\mathrm{poly}(n)}\right)$, $d > \log^3(n)$, and $D^2\log(n) < \epsilon$. Then $\thetahat(G) = \mathcal{A}_{TC}^{(\epsilon/(4D\hat{L}_{\epsilon_1, \delta_1}(G)), \delta/\hat{L}_{\epsilon_1, \delta_1}(G))}(T_{D}(G))$ is $(3\epsilon, 1/\mathrm{poly}(n))$-node DP. In addition, we have 
\begin{align*}
\mathcal{L}\left(\thetahat(G), \theta\right) \precsim \frac{1}{\sqrt{d}} + \frac{D^2}{\epsilon} \precsim \frac{1}{\log(n)},
\end{align*}
with probability at least $1 - \frac{1}{\mathrm{poly}(n)}$.   
\end{theorem}

\begin{remark}
\label{remark:TC_D=n}
Consider $D = n$ in Theorem \ref{theorem:Two_Comm_Main_Thm}, which corresponds to the naive group-privacy baseline. This
forces $\epsilon = \widetilde{\Omega}(n^2)$. With degree truncation, we may instead take $D \asymp d$, so the bound becomes $\Omega(\log^7(n))$, in the smallest-degree regime allowed by the theorem, $d \asymp \log^3(n)$.
\end{remark}

\begin{remark}
Theorem \ref{theorem:Two_Comm_Main_Thm} relies on Lemma \ref{lemma:ChenPriv+Ut}. The latter considers the sensitivity for privacy in terms of the operator norm of two edge-adjacent matrices. Thus, one could derive an analog of Theorem~\ref{theorem:Two_Comm_Main_Thm} for weighted graphs with bounded weights. We do not include such an analysis to avoid excessive technicalities
\end{remark}


\subsection{Analysis of private PCA}
\label{SecDeflatePCA}

Consider $D \in [3d, n]$. We first describe an algorithm based on the private PCA algorithm presented in Section~\ref{SecPPCA}, which is $(\epsilon', 0)_{2D}$-node DP. Since we are only considering bounded-degree graphs, we do not need to worry about taking a Lipschitz extension. By Lemma~\ref{LemPCAScore} with $k=1$, the sensitivity of the score function $s_A(v) = v^T A^2 v$ on graphs in $\mathcal{G}_{n,2D}$ is bounded by $4D^2$. Hence, the exponential mechanism which samples a unit vector $\hat{v}_1$ from the distribution $f(v) \propto \exp\left(\frac{\epsilon}{6D^2} \cdot v^T A^2 v\right)$ is node-private on graphs in $\mathcal{G}_{n,2D}$.

\begin{algorithm}[h!]
\caption{Private Eigenvector Deflation}
\label{alg:Priv_evec_def}
\begin{algorithmic}[1]
\Function{$\mathcal{A}_{EigDef}^{(\epsilon)}$}{adjacency matrix $A \in \{0, 1\}^{n \times n}$ based on $G \in \mathcal{G}_n$, number of communities $k$, truncation parameter $D$, privacy parameter $\epsilon$}
\State Set $A_1 = A^2$.
\For{$i = 1, \dots, k - 1$}
    \State Sample a unit vector $\hat{v}_i$ from the distribution $f_i(v) \propto \exp\left(\frac{\epsilon}{6D^2} \cdot v^T A_i v\right)$.
        \State Set $\hat{\sigma}_i = \left[\hat{v}_i^T A_i \hat{v}_i\right]_{-D^2}^{D^2} + \mathrm{Lap}\left(\frac{2D^2}{\epsilon}\right)$.
    \State Define the deflated adjacency matrix $A_{i+1} = A_i - \hat{\sigma}_i \hat{v}_i \hat{v}_i^T$.
    \EndFor
\State Sample a unit vector $\hat{v}_k$ from the distribution $f_k(v) \propto \exp\left(\frac{\epsilon}{6D^2} \cdot v^T A_k v\right)$.
\State Return $[\hat{v}_1, \dots, \hat{v}_k]$.
\EndFunction
\end{algorithmic}
\end{algorithm}

Importantly, we have chosen the threshold $D^2$ such that when $A_0 = A^2$ comes from $\mathcal{G}_{n, 2D}$, we know that w.h.p., the spectral norm of $A_1$ and all successive $A_i$'s is bounded by $4d^2 < 9d^2 \leq D^2$, so the truncation step in line 5 of Algorithm~\ref{alg:Priv_evec_def} can be ignored. The following result, a consequence of adaptive composition, is proved in Appendix~\ref{AppLemPCADeflate}:

\begin{lemma}
\label{LemPCADeflate}
$\mathcal{A}_{EigDef}^{(\epsilon/(2k))}$ satisfies $(\epsilon, 0)_{2D}$-node DP.
\end{lemma}

Combining the black-box methods (Theorem~\ref{theorem:generic_red_pure}) with Lemma~\ref{LemPCADeflate}, we can easily deduce that the projected algorithm satisfies $(\epsilon_1 + \epsilon, 1/\mathrm{poly}(n))$-node DP. The following result is proved in Appendix~\ref{AppThmPCADeflate}:

\begin{theorem}
\label{ThmPCADeflate}
Let $G \in \mathcal{G}_n$, and $\thetahat(G) = \mathcal{A}_{EigDef}^{(\epsilon/(2k\hat{L}_{\epsilon_1, \delta_1}(G)))}(T_D(G))$. Suppose Assumption \ref{ASS_eps_1_delt_1} holds. Then $\thetahat(G)$ satisfies $(\epsilon_1 + \epsilon, 1/\mathrm{poly}(n))$-node DP.
\end{theorem}

We do not include the details for analyzing the utility of Algorithm~\ref{alg:Priv_evec_def} to avoid excessive technicality. However, we note that the utility analysis for the exponential mechanism in private PCA (cf.\ Section~\ref{SecPPCA}) could be combined directly with the analysis of~\cite{LiaEtal24} for error propagation in eigenvector deflation techniques to obtain utility bounds on each of the eigenvectors $\{\hat{v}_1, \dots, \hat{v}_k\}$. This could then be translated into bounds on the misclassification error rate for an SBM using similar techniques to the ones employed in previous subsections.

\section{Lower bounds}
\label{SecLower}

Finally, we present two strategies to obtain lower bounds on $\epsilon$ for our private estimation problem. The first strategy is specific to the case $k=2$ and a particular misclassification error rate, while the second strategy is applicable to general $k$ and a broader range of misclassification rates. We note that in order to match the algorithmic setup used in the methods in Section~\ref{SecProj}, we are particularly interested in obtaining lower bounds on $\epsilon$ when $\delta \approx \frac{1}{\text{poly}(n)}$ for approximate DP, since our earlier analysis requires this scaling for consistent estimation.


We begin with a lower bound that applies to the case $k = 2$, and is based on an argument presented in~\cite{CheEtal23}, adapted from edge to node privacy. We present lower bounds for both approximate and pure DP, proved in Appendices~\ref{AppPropLBApprox} and~\ref{AppPropLBPure}, respectively:

\begin{proposition}
\label{prop:LB_Chen_eps_delt}
Let $\xi \geq\frac{1}{n}$. Suppose there exists an $(\epsilon, \delta)$-node DP algorithm such that for any balanced $\theta$, on input $G \sim SBM(n, k, B, \theta)$ with $k=2$, we have 
\begin{align*}
\mathbb{P}(\mathcal{L}(\mathcal{M}(G), \theta) \leq \xi) \geq 1 - \eta,
\end{align*}
where the randomness is over both the algorithm and stochastic block models. Then
\begin{align*}
\epsilon \geq \frac{\log\left(\max\left\{1, \frac{(1 - \eta - 4\delta\xi n e^{4\epsilon\xi n})\left(\frac{(4e\xi)^{-\xi n}}{2} - 1\right)}{\eta}\right\}\right)}{4\xi n}.
\end{align*}
In particular, if $\eta = O\left(\frac{1}{\log(n)}\right)$, $\xi \leq \frac{c_0}{n}$, $\delta \asymp \frac{1}{\mathrm{poly}(n)}$, and $\delta e^{4c_0\epsilon} < c_1$, with $c_0, c_1 > 0$, we have $\epsilon \succsim \log(n)$.
\end{proposition}

\begin{corollary}
\label{cor:LB_Chen_pure}
Let $\eta \leq \frac{1}{2}$ and $\xi \geq\frac{1}{n}$. Suppose there exists an $(\epsilon, 0)$-node DP algorithm such that for any balanced $\theta$, on input $G \sim SBM(n, k, B, \theta)$ with $k=2$, we have  
\begin{align*}
\mathbb{P}(\mathcal{L}(\mathcal{M}(G), \theta) \leq \xi) \geq 1 - \eta,
\end{align*}
where the randomness is over both the algorithm and stochastic block models. Then
\begin{align*}
\epsilon \succsim \log\left(\frac{1}{\xi}\right) + \frac{\log\left(\frac{1}{4\eta}\right)}{\xi n}.
\end{align*}
In particular, if $\xi, \eta \asymp \frac{1}{\mathrm{poly}(n)}$, we have $\epsilon \succsim \log(n)$. Moreover, if $\eta \asymp \frac{1}{\mathrm{poly}(n)}$, and $\xi \asymp \frac{1}{\log(n)}$, then $\epsilon \succsim \log(\log(n))$.
\end{corollary}


We now turn to an alternative lower bound that is somewhat more general both in terms of its ability to handle $k \ge 2$ and its additional flexibility in terms of the consistency rate $\xi$. It will be applicable to symmetric mechanisms, defined as follows:

\begin{definition}
\label{def:Symm_Mech}
A random mechanism $\calM: \{0, 1\}^{n\times n} \rightarrow [k]^n$ is \textit{symmetric} if for any $\hat{\theta} \in [k]^n$, $A\in \{0,1\}^{n\times n}$, and any a deterministic permutation matrix $L \in \{0, 1\}^{n \times n}$, we have
\begin{align*}
    \mathbb{P}(\calM(A) = \hat{\theta}) = \mprob(\calM(LAL^T) = L\hat{\theta}).
\end{align*}
\end{definition}

This assumption is natural when we do not have any prior information on the nodes, and makes the mechanism easier to analyze. The following result, proved in Appendix~\ref{AppLemSym}, shows that any mechanism $\mathcal{M}$ can easily be converted into one that satisfies Definition \ref{def:Symm_Mech}.

\begin{proposition}
\label{lemma:symmetrization}
Let $\calM: \mathbb{R}^{n\times n} \rightarrow [k]^n$ be a random mechanism. Then the following mechanism is symmetric: Sample a permutation matrix $J_n \in \{0, 1\}^{n \times n}$ uniformly from the set $E_n$ of $n!$ permutation matrices, independently from the randomness in $\calM$, and for the input matrix $A$, output $J_n^T\calM(J_nAJ_n^T)$.
\end{proposition}

\begin{remark}
Consider the four mechanisms from Theorems \ref{theorem:Edge_Flip_Main_Thm}, \ref{theorem:Matrix_Mech_Main_Thm}, \ref{theorem:GoodCent_Main_Thm}, and \ref{theorem:Two_Comm_Main_Thm}, together with the truncation operator $T_{D} : \mathcal{G}_n \rightarrow \mathcal{G}_{n, 2D}$. If we consider the modification from Proposition~\ref{lemma:symmetrization}, i.e., uniform sampling and application of a permutation matrix $J_n$, privacy is clearly  preserved, since $J_n$ is independent of the data, and post-processing holds. Furthermore, the utility analysis will not be affected since we can condition on $J_n$ and note that the definition of $\mathcal{L}$ involves taking a minimum over all possible permutations of estimated labels.

\end{remark}


The proof of the following result and its corollary are contained in Appendices~\ref{AppPropStable} and~\ref{AppCorLB}, respectively:

\begin{proposition}
\label{prop:LBStable}
Suppose $\frac{n}{k} \ge 6$ and $\theta$ is balanced. Let $\calM: \{0,1\}^{n\times n}\to \mathbf{M}_{n\times k}$ be a symmetric $(\eps, \delta)$-node DP mechanism. For any adjacency matrix $A\in \{0,1\}^{n\times n}$ and any $\xi \in (0,1/3)$, we have
$$
\mathbb{P}(\widetilde{\mathcal{L}}(\mathcal{M}(A), \theta)\leq \xi) \leq \frac{1}{(1-\xi)(1+(k-1)e^{-2\eps})} + \frac{2(k - 1)\delta}{1-\xi}.
$$
\end{proposition}

\begin{corollary}
\label{corr:lb1_lim1}
Suppose $\frac{n}{k} \ge 6$, $\theta$ is balanced, and $k \asymp 1$. Let $\calM: \{0,1\}^{n\times n}\to \mathbf{M}_{n\times k}$ be a symmetric $(\epsilon, \delta)$-node DP mechanism. Suppose $A \sim SBM(n, k, B, \theta)$. If $\mprob(\widetilde{\mathcal{L}}(\mathcal{M}(A), \theta) \le \xi) \ge 1 - \eta$, with $\xi \in (0, 1/3)$, then
\begin{equation*}
\epsilon \ge \frac{1}{2}\log\left(\frac{1-\eta-\xi - 2k \delta}{\eta + \xi + 2k\delta}\right) + \frac{1}{2} \log(k-1).
\end{equation*}
In particular, if $\eta, \xi, \delta = O\left(\frac{1}{\mathrm{poly}(n)}\right)$, we have $\epsilon \succsim \log(n)$. Moreover, if $\eta = O\left(\frac{1}{\log(n)}\right)$ and $\xi, \delta = O\left(\frac{1}{\mathrm{poly}(n)}\right)$, we have $\epsilon \succsim \log(\log(n))$.
\end{corollary}

\begin{remark}
The lower bound on $\epsilon$ in Corollary~\ref{corr:lb1_lim1} shows that letting $k \rightarrow \infty$ or allowing the error probabilities $\eta$ or $\delta$ to be $\Theta(1)$ would not preclude a scenario where $\xi = o(1)$ (i.e., consistency) but $\epsilon$ stays constant. Furthermore, if $\eta = O\left(\frac{1}{\log(n)}\right)$ and $\xi \asymp \frac{1}{\log(n)}$, as in the utility guarantees for some of our algorithms above, the lower bound produced by Corollary~\ref{corr:lb1_lim1} becomes $\epsilon \succsim \log(\log(n))$.
\end{remark}

Our results point to a general requirement of $\epsilon \succsim \log(n)$, under certain assumptions on $(\xi, \eta, \delta)$. This is also the lower bound achieved by Klopp and Zadik~\cite{klopp2026node}, who further established matching upper bounds with a privacy mechanism based on the exponential mechanism. However, we emphasize that the methods studied in our paper are all \emph{polynomial-time}. In fact, if one were to combine an exponential-time alternative in \cite{CheEtal23} with our degree truncation techniques, one could derive a consistent community estimation technique requiring only $\epsilon \asymp \log(n)$ in the $k=2$ case.

Notably, both of our lower bounds and indeed, all existing lower bounds in the literature, do not involve any explicit dependence on the expected degree $d$. Since we have different scaling requirements on $\epsilon$ for different algorithms, generally depending on whether $d \asymp \log(n)$ or $d = \widetilde{\Theta}(\sqrt{n})$, an important open question is whether one could derive lower bounds on $\epsilon$ which also depend on $d$.

\section{Discussion}
\label{SecDiscussion}

We have presented several new methods for community estimation in stochastic block models that preserve node-level differential privacy. A central idea is that under standard assumptions on the underlying SBM parameters, the observed graph will have maximum degree of order the average maximum degree $d$, which is somewhat smaller than the number of nodes in the graph. Although a private algorithm must be private even for worst-case inputs, creating substantial technical difficulties, utility analysis is conducted w.h.p., in which case it suffices to consider graphs which lie in the bounded-degree set.

Our first approach leverages a Lipschitz extension idea combined with a method for differentially private PCA previously proposed in the literature. Graphs with bounded degree have adjacency matrices with low sensitivity with respect to the PCA score function. We show how to extend the score function to graphs of arbitrary degree, preserving the sensitivity bound while maintaining computational feasibility. Our second approach is based on constructing a smooth projection from the space of all graphs to the space of bounded-degree graphs, with a calculable high-probability bound on the local sensitivity of the projection. This can be combined with an edge-private SBM estimation procedure; we present several options and the resulting conditions on the privacy parameters required for consistent community recovery.

While we believe that our paper makes substantial progress toward the problem of node-private SBM recovery, several open questions remain. Given the gap between the assumptions we have imposed on $\epsilon$ to ensure consistent estimation for all of our algorithms and the lower bounds on $\epsilon$ we have derived, an important question is to understand which, if any, of our algorithms give the best possible rates in terms of $\epsilon$. We note that \cite{klopp2026node} derive matching upper and lower bounds, but the algorithm they analyze in their upper bound is not computationally feasible. Thus, the open question is whether one can establish larger lower bounds on $(\epsilon, \delta)$ if one is restricted to \emph{polynomial-time} algorithms. Another related question is whether one can establish lower bounds for private estimation in terms of a quantity which depends more explicitly on the expected degree $d$, which might then match our upper bounds in certain settings.

Our Lipschitz extension method based on private PCA proceeds by extracting one eigenvector at a time and using eigenvector deflation. This is due to two bottlenecks: One is that the theoretical analysis of private PCA seems to rely heavily on the fact that we are only extracting $q=1$ eigenvector. Furthermore, existing polynomial-time methods for exactly sampling from a Bingham distribution work only when $q=1$. If it were possible to overcome both of these issues, one could plausibly obtain a private exponential mechanism-based approach with fewer conditions on eigengaps between successive eigenvalues of the population-level matrix $P$.

Another important question concerns a well-studied topic in the analysis of SBMs, namely differences between consistency (as we have studied in our paper) and \emph{exact recovery}, where one wishes to correctly classify every single node, w.h.p. Various papers in the SBM literature, e.g., \cite{gao2017achieving}, demonstrate how an initial consistent community assignment can be further refined to an exact assignment, in parameter regimes where exact recovery is possible. It is unclear whether the second refinement step could be privatized; if so, it could be combined with our methods to obtain a nice approach for private exact recovery in SBMs.

Finally, one might be interested in analyzing community estimation in SBMs under alternative notions of privacy. One that naturally comes to mind is where $G'$ is defined as adjacent to $G$ if $G'$ is obtained from $G$ by \emph{removing} a single node entirely from the graph. Analyzing such a notion of privacy seems rather beyond the scope of this paper, and would likely give rise to other methodological and theoretical innovations.


\section*{Acknowledgments}

The authors thank the Summer Research in Mathematics (SRIM) program at Cambridge for providing a stimulating environment for seeding this research in the summer of 2024. The research was partly supported by funds from the Leverhulme Trust.


\appendix

\section{Auxiliary Results}
\label{Appendix_Aux_Res}

\subsection{Linear algebraic results}


\begin{lemma}[Eckart-Young-Mirsky Theorem~\cite{blum2020foundations}]
\label{LemEYM}
Suppose $A \in \real^{m \times n}$, and let $A = \sum_{i=1}^r \sigma_i u_i v_i^T$ denote the SVD. For $1 \le k \le r$, let $A_k := \sum_{i=1}^k \sigma_i u_i v_i^T$. Then
\begin{equation*}
A_k \in \arg\min_{B \in \real^{m \times n}: \text{rank}(B) \le k} \|A - B\|,
\end{equation*}
where $\|\cdot \|$ denotes the spectral or Frobenius norm.
\end{lemma}

We now state a few results concerning singular values of sums of matrices.

\begin{lemma}[Lemma 2.1 in \cite{SinSte21}]
\label{lemma:l2.1SinSte}
Let $A, B \in \mathbb{R}^{f \times \ell}$. Then for $1 \leq i$ and $j \leq \min\{f, \ell\}$, we have 
\begin{align*}
\sigma_{i + j - 1}(A + B) \leq \sigma_i(A) + \sigma_j(B).
\end{align*}
\end{lemma}

The next few lemmas concern matrix perturbation bounds.

\begin{lemma}[Weyl's inequality, Lemma 2.2 in \cite{LiaEtal24}]
\label{LemWeyl}
Let $A, A^* \in \real^{n \times n}$ be symmetric matrices. Let $\sigma_j$ and $\sigma_j^*$ denote the $j^{\text{th}}$ eigenvalues of $A$ and $A^*$, respectively. Then
\begin{equation*}
|\sigma_j - \sigma_j^*| \le \|A - A^*\|_2.
\end{equation*}
\end{lemma}

Now we have a result concerning the operator norm of the Hadamard product. We provide a short proof for completeness.

\begin{lemma}[\cite{johnson1990matrix}]
\label{lemma:Oper_Hadam}
Let $A, B \in \mathbb{R}^{n \times n}$, with $A^T = A$ and $A \succeq 0$ or $A \preceq 0$. Then
\begin{align*}
    ||A \odot B||_2 \leq \mathop{\max}\limits_{i, j}|A_{ij}|\cdot||B||_2.
\end{align*}
\end{lemma}

\begin{proof}
First suppose $A\succeq 0$ and $A = A^T$. Let $r = \text{rank}(A)$. Then there exist vectors
$u_1,\dots,u_n\in \mathbb{R}^r$ such that $A_{ij} = u_i^Tu_j$. Define a linear map $V:\mathbb{R}^n \rightarrow \mathbb{R}^{rn}$ by
\begin{align*}
Ve_i = e_i\otimes u_i,
\end{align*}
where $e_i$ is the $i^{\text{th}}$ standard basis vector. Then $A\odot B=V^T(B\otimes I_r)V$, so
\begin{equation*}
\|A\odot B\|_2 \leq \|V^T\|_2\,\|B\otimes I_r\|_2\,\|V\|_2 = ||V||_2^2||B||_2.
\end{equation*}
We have
\begin{align*}
V^TV = \mathrm{diag}(\|u_1\|_2^2,\dots,\|u_n\|_2^2)
= \mathrm{diag}(A_{11},\dots,A_{nn}),    
\end{align*}
so $||V||_2^2 = \mathop{\max}\limits_{i} A_{ii} \leq \mathop{\max}\limits_{i, j}|A_{ij}|$. This implies $||A \odot B||_2 \leq \mathop{\max}\limits_{i, j}|A_{ij}|\cdot||B||_2$. The case when $A \preceq 0$ works in the same way, by considering $-A$. This completes the proof.
\end{proof}

\begin{lemma}[Davis-Kahan Theorem, Lemmas 2.3 and 2.4 in \cite{SinSte21}]
\label{lemma:l2.3SinSte}
Consider the SVD 
\begin{align*}
X = \begin{bmatrix}
U & U_{\perp}
\end{bmatrix}\cdot \begin{bmatrix}
\Sigma_1 & 0\\
0 & \Sigma_2
\end{bmatrix}\cdot \begin{bmatrix}
V^T\\
V_{\perp}^T
\end{bmatrix} \in \mathbb{R}^{f \times \ell},
\end{align*}
where $U \in \mathbb{R}^{f \times k}$ and $U_{\perp} \in \mathbb{R}^{f \times (f - k)}$ are orthonormal, $\Sigma_1 \in \mathbb{R}^{k \times k}$ and $\Sigma_2 \in \mathbb{R}^{(f - k) \times (\ell - k)}$ are diagonal, and $V \in \mathbb{R}^{f \times k}$ and $V_{\perp} \in \mathbb{R}^{f \times (f - k)}$ are orthonormal. Let $Y \in \mathbb{R}^{f \times \ell}$ be a perturbation matrix and let $\hat{X} = X + Y$, where $\hat{X}$ has SVD
\begin{align*}
\hat{X} = \begin{bmatrix}
\hat{U} & \hat{U}_{\perp}
\end{bmatrix}\cdot \begin{bmatrix}
\hat{\Sigma_1} & 0\\
0 & \hat{\Sigma_2}
\end{bmatrix}\cdot \begin{bmatrix}
\hat{V}^T\\
\hat{V}_{\perp}^T
\end{bmatrix} \in \mathbb{R}^{f \times \ell}.
\end{align*}
In the above, $\hat{U}, \hat{U}_{\perp}, \hat{\Sigma_1}, \hat{\Sigma_2}, \hat{V}, \hat{V}_{\perp}$, have the same structure as $U, U_{\perp}, \Sigma_1, \Sigma_2, V, V_{\perp}$, respectively. Let $z_{21} = U_{\perp}U_{\perp}^TYVV^T$, $z_{21} = UU^TYV_{\perp}V_{\perp}^T$, $a_1 = \sigma_{min}(U^T\hat{X}V)$, and $b_1 = ||U_{\perp}^T\hat{X}V_{\perp}||_2$. Suppose $\sigma_1 \geq \dots \geq \sigma_k$ are the singular values of $U^T\hat{U}$. Let $\Theta(U, \hat{U}) \in \mathbb{R}^{k \times k}$ be a diagonal matrix such that $\Theta_{ii}(U, \hat{U}) = \cos^{-1}(\sigma_i)$, for all $i$. If $a_1^2 > b_1^2 + \min\{z_{12}^2, z_{21}^2\}$, then
\begin{align*}
||\sin(\Theta)(U, \hat{U})||_2 \leq \frac{a_1||z_{21}||_2 + b_1||z_{12}||_2}{a_1^2 - b_1^2 - \min\{z_{12}^2, z_{21}^2\}}.
\end{align*}
Furthermore, we have the inequality
\begin{align*}
||\sin(\Theta)(U, \hat{U})||_2 \leq \left\|\hat{U}\hat{U}^T - UU^T\right\|_2 \leq 2||\sin(\Theta)(U, \hat{U})||_2.
\end{align*}
\end{lemma}

We now state the Davis-Kahan Theorem for square matrices:

\begin{lemma}[Theorem 2 in \cite{YuEtal15}]
\label{lemma:thm3YuEtal_square}
Suppose $A, \hat{A} \in \mathbb{R}^{p \times p}$ are symmetric, with eigenvalues $\lambda_1 \geq \dots \geq \lambda_p$ and $\hat{\lambda}_1 \geq \dots \geq \hat{\lambda}_p$, respectively. Fix $1 \leq r \leq s \leq p$, and assume $\min\{\lambda_{r - 1} - \lambda_r, \lambda_s - \lambda_{s + 1}\} > 0$, where we define $\lambda_0 = \infty$ and $\lambda_{p + 1} = -\infty$. Let $d = s - r + 1$, and let $V = [v_r, v_{r + 1}, \dots, v_s] \in \mathbb{R}^{p \times d}$ and $\hat{V} = [\hat{v}_r, \hat{v}_{r + 1}, \dots, \hat{v}_s] \in \mathbb{R}^{p \times d}$ have orthonormal columns satisfying $Av_j = \lambda_jv_j$ and $\hat{A}\hat{v}_j = \hat{\lambda}_j\hat{v}_j$, for $j = r, r + 1, \dots, s$. Then
\begin{align*}
||\sin(\Theta)(V, \hat{V})||_F \leq \frac{2\min\{\sqrt{d}||A - \hat{A}||_2, ||A - \hat{A}||_F\}}{\min\{\lambda_{r - 1} - \lambda_r, \lambda_s - \lambda_{s + 1}\}}.
\end{align*}
Moreover, there exists an orthogonal matrix $\hat{Q} \in \mathbb{R}^{d \times d}$ such that
\begin{align*}
||\hat{V}\hat{Q} - V||_F \leq \frac{2^{3/2}\min\{\sqrt{d}||A - \hat{A}||_2, ||A - \hat{A}||_F\}}{\min\{\lambda_{r - 1} - \lambda_r, \lambda_s - \lambda_{s + 1}\}}.
\end{align*}
\end{lemma}

We now state a generalization of the Davis-Kahan Theorem for non-square matrices:

\begin{lemma}[Theorem 3 in \cite{YuEtal15}]
\label{lemma:thm3YuEtal}
Suppose $A, \hat{A} \in \mathbb{R}^{p \times q}$ have singular values $\sigma_1 \geq \dots \geq \sigma_{\min\{p, q\}}$ and $\hat{\sigma}_1 \geq \dots \geq \hat{\sigma}_{\min\{p, q\}}$, respectively. Fix $1 \leq r \leq s \leq \mathrm{rank}(A)$ and assume $\min\{\sigma_{r - 1}^2 - \sigma_r^2, \sigma_s^2 - \sigma_{s + 1}^2\} > 0$, where we define $\sigma_0^2 = \infty$ and $\sigma_{\mathrm{rank}(A) + 1}^2 = -\infty$. Let $d = s - r + 1$, and let $V = [v_r, v_{r + 1}, \dots, v_s] \in \mathbb{R}^{q \times d}$ and $\hat{V} = [\hat{v}_r, \hat{v}_{r + 1}, \dots, \hat{v}_s] \in \mathbb{R}^{q \times d}$ have orthonormal columns satisfying $Av_j = \sigma_ju_j$ and $\hat{A}\hat{v}_j = \hat{\sigma}_j\hat{u}_j$, for $j = r, r + 1, \dots, s$. Then
\begin{align*}
||\sin(\Theta)(V, \hat{V})||_F \leq \frac{2\left(2\sigma_1 + ||A - \hat{A}||_2\right)\min\{\sqrt{d}||A - \hat{A}||_2, ||A - \hat{A}||_F\}}{\min\{\sigma_{r - 1}^2 - \sigma_r^2, \sigma_s^2 - \sigma_{s + 1}^2\}}.
\end{align*}
Moreover, there exists an orthogonal matrix $\hat{Q} \in \mathbb{R}^{d \times d}$ such that
\begin{align*}
||\hat{V}\hat{Q} - V||_F \leq \frac{2^{3/2}\left(2\sigma_1 + ||A - \hat{A}||_2\right)\min\{\sqrt{d}||A - \hat{A}||_2, ||A - \hat{A}||_F\}}{\min\{\sigma_{r - 1}^2 - \sigma_r^2, \sigma_s^2 - \sigma_{s + 1}^2\}}.
\end{align*}
\end{lemma}

Next, we have a result from \cite{VuLei13}:

\begin{lemma}[Proposition 2.2 in Vu \& Lei~\cite{VuLei13}]
\label{lemma:prop2.2VuLei}
Let $\mathbb{V}_{p, q} = \left\{A \in \mathbb{R}^{p \times q}| \ A^TA = I_q\right\}$. If $V_1, V_2 \in \mathbb{V}_{p, q}$, then
\begin{align*}
\frac{1}{2}\inf\limits_{Q \in \mathbb{V}_{q, q}}||V_1 - V_2Q||_F^2 \leq ||\sin(\Theta)(V_1, V_2)||_F^2 \leq \inf\limits_{Q \in \mathbb{V}_{q, q}}||V_1 - V_2Q||_F^2.
\end{align*}
\end{lemma}


\subsection{Random Gaussian vectors}

Recall the notion of a sub-Gaussian norm of a random variable:
\begin{definition}[Definition 2.5 in \cite{SinSte21}]
Let $X$ be a sub-Gaussian random variable. The sub-Gaussian norm of X, denoted by $||X||_{\psi^2}$ is defined as
\begin{align*}
||X||_{\psi^2} = \inf\left\{t > 0| \ \mathbb{E}\left[e^{X^2/t^2}\right] \leq 2\right\}.
\end{align*}
\end{definition}
\begin{remark}
Note that if $X \sim N(0, \sigma^2)$, then $||X||_{\psi^2} = C\sigma$, for some absolute constant $C > 0$.
\end{remark}

\begin{lemma}[Corollary 2.7 in \cite{SinSte21}]
\label{lemma:cor2.7SinSte}
Let $A$ be an $f \times \ell$ matrix, whose columns $A_i$ are independent random vectors in $\mathbb{R}^f$ from $N(0, \Sigma)$, with $\Sigma \succ 0$. Then for any $t \geq 0$, we have
\begin{align*}
(\sqrt{\ell} - CK^2(\sqrt{f} + t))\sqrt{\sigma_n(\Sigma)} \leq \sigma_n(A) \leq (\sqrt{\ell} + CK^2(\sqrt{f} + t))\sqrt{\sigma_n(\Sigma)},
\end{align*}
and
\begin{align*}
\sigma_1(A) \leq (\sqrt{\ell} + CK^2(\sqrt{f} + t))\sqrt{\sigma_1(\Sigma)},
\end{align*}
with probability at least $1 - 2e^{-t^2}$, where $K := \max\limits_{i, j}||A_{ij}||_{\psi^2}$.
\end{lemma}

Now we have a lemma regarding uniform concentration of Gaussian vectors:
\begin{lemma}[Lemma 2.10 in \cite{SinSte21}]
\label{lemma:lem2.10SinSte}
Let $X_1, \dots, X_q \sim N(0, \Sigma)$ be vectors in $\mathbb{R}^d$, where $\Sigma$ is the projection of $I_d$ onto a subspace of $\mathbb{R}^d$ of rank $k$. Then 
\begin{align*}
\mathbb{P}\left(\forall i, ||X_i||_2^2 \leq k + 2\sqrt{kt} + 2t\right) \geq 1 - qe^{-t}.
\end{align*}  
\end{lemma}


\subsection{Concentration inequalities}

We begin by stating a concentration bound for sums of indicator variables:
\begin{lemma}[Theorem $3$ in \cite{janson2016large}]
\label{lemma:Bern_Conc_Janson_2016}
Let $X = \sum_{i = 1}^nX_i$, where $X_i \overset{i.i.d.}{\sim} \mathrm{Ber}(p_i)$. Let $\sigma^2 = \mathrm{Var}(X)$. If $t \geq 0$, then
\begin{align*}
\max\left\{\mathbb{P}(X - \mathbb{E}[X] \geq t), \mathbb{P}(X - \mathbb{E}[X] \leq -t)\right\} \leq e^{-\frac{t^2}{2\sigma^2}\left(1 - \frac{t}{3\sigma^2}\right)}.
\end{align*}
\end{lemma}

We now present a concentration result for hypergeometric random variables:

\begin{lemma}[\cite{hoeffding1963probability}]
\label{lemma:hypergeom_conc}
Let $X \sim \mathrm{HG}(N, K, s)$. Then
\begin{align*}
\mathbb{P}(X - \mathbb{E}[X] \geq ts), \mathbb{P}(X - \mathbb{E}[X] \leq -ts) \leq e^{-2t^2s},
\end{align*}
for any $t > 0$.
\end{lemma}

\begin{remark}
Consider Lemma \ref{lemma:hypergeom_conc}. Let $X \sim \mathrm{HG}(N, K, s)$. Then $\mathbb{P}(X = k)$ is the probability of obtaining $k$ successes in $s$ draws, without replacement, from a population of size $N$, with $K$ available successes.
\end{remark}

Now we state a result concerning sub-Gaussian matrices:

\begin{definition}[\cite{wainwright2019high}]
A zero-mean, symmetric random matrix $Q \in \mathbb{R}^{m \times m}$ is sub-Gaussian with
matrix parameter $V \in \mathbb{R}^{m \times m}$, where $V$ is symmetric positive-definite, if
\begin{align*}
\mathbb{E}\left[e^{\lambda Q}\right] := \sum_{i = 0}^\infty\frac{\lambda^i}{i!}\mathbb{E}\left[Q^i\right] \preceq e^{\frac{\lambda^2V}{2}}, 
\end{align*}
for all $\lambda \in \mathbb{R}$.
\end{definition}

Next, we have a concentration result for sums of random matrices:
\begin{lemma}[Hoeffding bound for random matrices \cite{wainwright2019high}]
\label{lemma:Matrix_Hoeff} 
Let $\{Q_i\}_{i = 1}^n$ be a sequence of zero-mean, $m \times m$, independent, symmetric random matrices that satisfy the sub-Gaussian condition with parameters $\{V_i\}_{i = 1}^n$. For all $\zeta > 0$, we have the upper tail bound
\begin{align*}
\mathbb{P}\left(\left\|\sum_{i = 1}^nQ_i\right\|_2 \geq \zeta\right) \leq 2me^{-\frac{\zeta^2}{2\left\|\sum_{i = 1}^nV_i\right\|_2}}.
\end{align*}
\end{lemma}

We also have a result regarding the product of a random variable with a deterministic matrix:
\begin{lemma}[Exercise 6.7 in \cite{wainwright2019high}]
\label{lemma:gB_Ex6.7} 
Let $g \in \mathbb{R}$ be a zero-mean $s^2$-sub-Gaussian variable, with a distribution symmetric around $0$. Assume $B \in \mathbb{R}^{m \times m}$ is symmetric and deterministic. Then $gB$ is sub-Gaussian with matrix parameter $c^2s^2B^2$, where $c$ is a universal constant.
\end{lemma}


\subsection{Privacy results}
\label{AppPrivacy}

Following \cite{dwork2010boosting}, we formally define adaptive composition of differentially private mechanisms. Suppose $\{\mathcal{A}_i\}_{i=T}$ are mechanisms applied to a common database $\mathcal{X}^n$, such that $\mathcal{A}_1: \mathcal{X}^n \rightarrow \mathcal{Y}_1$, and for $i > 1$, we have $\mathcal{A}_i: \mathcal{Y}_1 \times \cdots \times \mathcal{Y}_{i-1} \times \mathcal{X}^n \rightarrow \mathcal{Y}_i$. The composition of $(\mathcal{A}_1, \dots, \mathcal{A}_T)$ is the algorithm that maps $x \in \mathcal{X}^n$ to
\begin{equation*}
\mathcal{A}_T\Big(\mathcal{A}_1(x), \mathcal{A}_2\big(\mathcal{A}_1(x), x\big), \dots, x\Big).
\end{equation*}
Assuming that each $\mathcal{A}_i$ is $(\epsilon_i, \delta_i)$-DP, in the sense that $\mathcal{A}_i(a_1, \dots, a_{i-1}, \cdot)$ is $(\epsilon_i, \delta_i)$-DP for any fixed $(a_1, \dots, a_{i-1})$, we wish to study the privacy properties of the composed mechanism.

\begin{lemma}[Basic composition \cite{DwoRot14}]
\label{LemBasic}
Suppose each $\mathcal{A}_i$ satisfies $(\epsilon, \delta)$-DP, where $\epsilon > 0$ and $\delta \in [0, 1]$. Then the composed mechanism satisfies $(T\epsilon, T\delta)$-DP.
\end{lemma}

If we introduce an extra slack term $\tilde{\delta}$, we can achieve a better bound in terms of the growth of $\epsilon$:

\begin{lemma}[Adaptive composition \cite{kairouz2015comp}]
\label{lemma:adapt_comp}  
Suppose each $\mathcal{A}_i$ satisfies $(\epsilon, \delta)$-DP, where $\epsilon > 0$ and $\delta \in [0, 1]$. Suppose $\tilde{\delta} \in (0, 1]$. Then the composed mechanism satisfies $(\tilde{\epsilon}_{\tilde{\delta}}, T\delta + \tilde{\delta})$-DP, where
\begin{align*}
\tilde{\epsilon}_{\tilde{\delta}} = T\epsilon(e^\epsilon - 1) + \epsilon\sqrt{2T\log(1/\tilde{\delta})}.
\end{align*}
In particular, $\tilde{\epsilon}_{\tilde{\delta}} = O(T\epsilon(e^\epsilon - 1) + \epsilon\sqrt{T})$.
\end{lemma}

\begin{lemma}[Group privacy \cite{vadhan2017complexity}]
\label{LemGroup}
If $\mathcal{A}$ is $(\epsilon, \delta)$-DP, then $\mathcal{A}$ is $(T\epsilon, T\delta e^{(T - 1)\epsilon})$-DP, on groups of size at most $T$.
\end{lemma}

We now state a few results related to zero-concentrated differential privacy:

\begin{lemma}[zCDP to DP from \cite{bun2016conc}]
\label{lemma:zCDP_to_DP}
If $\mathcal{A}$ is $\rho$-zCDP, then it is $\left(\rho + \sqrt{4\rho\log(1/\delta)}, \delta\right)$-DP, for any $\delta \in (0, 1)$. 
\end{lemma} 
\begin{lemma}[Composition of zCDP from \cite{bun2016conc}]
\label{lemma:Comp_zCDP}
Let $\mathcal{M} : \mathcal{X}^n \rightarrow \mathcal{Y}$ and $\mathcal{M}' : \mathcal{X}^n \rightarrow \mathcal{Z}$ be randomized algorithms. Suppose $\mathcal{M}$ satisfies $\rho$-zCDP and $\mathcal{M}'$ satisfies $\rho'$-zCDP. Define $\mathcal{M}' : \mathcal{X}^n \rightarrow \mathcal{Y} \times \mathcal{Z}$ by $\mathcal{M}''(X) = (\mathcal{M}(X), \mathcal{M}'(X))$. Then $\mathcal{M}''$ satisfies $(\rho + \rho')$-zCDP.    
\end{lemma}

\begin{remark}
Note that by \cite{whitehouse2023fully}, the same composition guarantee from Lemma \ref{lemma:Comp_zCDP} holds for fully-adaptive composition of zCDP mechanisms.
\end{remark}

\begin{lemma}[Postprocessing of zCDP from \cite{bun2016conc}]
Let $\mathcal{M} : \mathcal{X}^n \rightarrow \mathcal{Y}$ and $f : \mathcal{Y} \rightarrow \mathcal{Z}$ be randomized algorithms. Suppose $\mathcal{M}$ satisfies $\rho$-zCDP. Define $\mathcal{M}' : \mathcal{X}^n \rightarrow \mathcal{Z}$ by $\mathcal{M}'(X) = f(\mathcal{M}(X))$. Then $\mathcal{M}'$ satisfies $\rho$-zCDP.
\end{lemma}

\begin{lemma}[Group privacy of zCDP from \cite{bun2016conc}]
\label{LemGroupZ}
\label{lemma:Group_zCDP}
If $\mathcal{A}$ is $\rho$-zCDP, it is $T^2\rho$-zCDP on groups of size $T$. 
\end{lemma} 

\begin{lemma}[Gaussian mechanism for zCDP from \cite{bun2016conc}]
\label{lemma:Gauss_zCDP}
Let $f : \mathcal{X}^n \rightarrow \mathbb{R}^d$ have $\ell_2$-sensitivity $\Delta_f$. Consider the mechanism $\mathcal{M} : \mathcal{X}^n \rightarrow \mathbb{R}^d$, which on input $X$, releases a sample from $N(f(X), \sigma^2I_d)$, where $\sigma^2 = \frac{\Delta_f^2}{2\rho}$. Then $\mathcal{M}$ satisfies $\rho$-zCDP.
\end{lemma}

We also mention the exponential mechanism~\cite{mcsherry2007mechanism}:

\begin{lemma}[Exponential mechanism~\cite{DwoRot14}]
\label{LemExponential}
Consider a utility function $u: \mathbb{N}^{|\mathcal{X}|} \times \mathcal{R} \rightarrow \real$, which maps database/output pairs to utility scores, and suppose
\begin{equation*}
\Delta u := \max_{r \in \mathcal{R}} \max_{x,y: \|x-y\|_1 \le 1} |u(x,r) - u(y,r)| < \infty.
\end{equation*}
The mechanism $\mathcal{M}(x; u, \mathcal{R})$, which selects and outputs an element $r \in \mathcal{R}$ with probability proportional to $\exp\left(\frac{\epsilon u(x,r)}{2\Delta u}\right)$, satisfies $(\epsilon, 0)$-DP.
\end{lemma}


\subsection{Properties of SBMs}
\label{AppDegree}

We will use the following result about the deviation of the adjacency matrix in a random graph from its mean:

\begin{lemma}[Theorem 5.2 in \cite{LeiRin15}]
\label{lemma:thm5.2LeiRin}
Let $A$ be the adjacency matrix of a random graph on $n$ nodes in which edges occur independently. Suppose $n \cdot \max\limits_{i, j}\mathbb{E}[A]_{ij} \leq d$, for $d \geq c_0\log(n)$. Then for any $r > 0$, there exists a constant $C = C(c_0, r)$ such that
\begin{align*}
||A - \mathbb{E}[A]||_2 \leq C\sqrt{d},
\end{align*}
with probability at least $1 - n^{-r}$.
\end{lemma}

We will also use the following result, proved easily using basic concentration inequalities:

\begin{lemma}
\label{lemma:hp_removal_d}
Suppose $A \sim SBM(n, k, B, \theta)$, with parameters as in Assumption~\ref{AssSBM}.
For each $i \in [n]$, let $d_i = \|Ae_i\|_1$ be the degree of node $i$. Then
\begin{align*}
\mprob\left(d_i < 2d, \quad \mathrm{for \ all} \  i\right) \ge 1 - \frac{1}{\mathrm{poly}(n)}.
\end{align*}
\end{lemma}

\begin{proof}
Let $\Omega_0 = \bigcap_{i = 1}^n\left\{d_i < 2\mathbb{E}[d_i]\right\}$. Let $j(i) \in [k]$ be the community containing node $i$. For each $i$, we have
\begin{align*}
&\mathbb{E}[d_i] = \left(\frac{n}{k} - 1\right)a_n (B_0)_{j(i)j(i)} + \sum_{\ell \neq j(i)}\frac{na_n (B_0)_{j(i)\ell}}{k},\\
&\sigma_i^2 = \text{Var}(d_i) = \left(\frac{n}{k} - 1\right)a_n (B_0)+{j(i)j(i)}(1 - a_n (B_0)_{j(i)j(i)}) \\
& \qquad \qquad \qquad + \sum_{\ell \neq j(i)}\frac{n}{k}a_n (B_0)_{j(i)\ell}(1 - a_n (B_0)_{j(i)\ell}).
\end{align*}
Then
\begin{align*}
\mathbb{P}(\Omega_0^c) \leq \sum_{i = 1}^n\mathbb{P}(d_i \geq 2\mathbb{E}[d_i]) \leq \sum_{i = 1}^n\mathbb{P}(d_i - \mathbb{E}[d_i] \geq \mathbb{E}[d_i]) \le \sum_{i=1}^n e^{-\frac{\mathbb{E}[d_i]^2}{2\sigma_i^2}\left(1 - \frac{\mathbb{E}[d_i]}{3\sigma_i^2}\right)},
\end{align*}
using Lemma~\ref{lemma:Bern_Conc_Janson_2016} to bound each summand, since $d_i$ is a sum of i.i.d.\ Bernoulli random variables.
Under Assumption~\ref{AssSBM}, we have $\frac{\mathbb{E}[d_i]}{\sigma_i^2} \rightarrow 1$, as $n \rightarrow \infty$, for each $i$. Hence, for $n$ sufficiently large, we have
$e^{-\frac{\mathbb{E}[d_i]^2}{2\sigma_i^2}\left(1 - \frac{\mathbb{E}[d_i]}{3\sigma_i^2}\right)} < e^{-\frac{\mathbb{E}[d_i]}{12}}$, for each $i$. Also, since $a_n \geq \frac{C_1k\log(n)}{n\cdot \text{min}_{j}{(B_0)_{jj}}}$, we have
\begin{equation*}
\mathbb{E}[d_i] \geq \left(\frac{n}{k} - 1\right)a_n \cdot\text{min}_{j}{(B_0)_{jj}} \geq C_1\log(n) - \frac{C_1k\log(n)}{n} \geq C_1\log(n) - C_1k,
\end{equation*}
for $n$ sufficiently large, implying that 
\begin{align*}
\mathbb{P}(\Omega_0^c) < ne^{-\frac{C_1\log(n)}{12} + \frac{C_1k}{12}} \precsim n^{1 - \frac{C_1}{12}} + n^{1 - \frac{C_1}{24}} = o(1),
\end{align*}
since $C_1 > 24$. Finally, for $i \in [n]$, we have 
\begin{align*}
\mathbb{E}[d_i] = \left(\frac{n}{k} - 1\right)a_n (B_0)_{j(i)j(i)} + \sum_{l \neq j(i)}\frac{na_n (B_0)_{j(i)l}}{k} \leq (n - 1)a_n \cdot\text{max}(B_0) \leq d,
\end{align*}
implying that 
\begin{align*}
\mprob\left(d_i < 2d, \quad \text{for all } i\right) \geq 1 - \frac{1}{\mathrm{poly}(n)},
\end{align*}
as required.
\end{proof}


\subsection{Maximal correlation}
\label{AppMaxCor}

Recall that the Hirschfeld-Gebelein-R\'{e}nyi (HGR) maximal correlation is defined for random variables $(X,Y)$ as
\begin{equation*}
\rho_m(X,Y) = \max_{(f(X), g(Y)) \in \mathcal{S}} \E[f(X) g(Y)],
\end{equation*}
where $\mathcal{S}$ consists of pairs of random variables $f(X)$ and $g(Y)$ satisfying $\E[f(X)] = 0 = \E[g(Y)]$ and $\E[f^2(X)] = 1 = \E[g^2(Y)]$ \cite{anantharam2013maximal, polyanskiy2025information}. Clearly, by taking $f$ and $g$ to be appropriate linear transformations, we have $\rho_m(X,Y) \le \rho(X,Y)$, where $\rho$ is the usual Pearson correlation. Furthermore, when $X$ and $Y$ are Bernoulli, we have $\rho_m(X,Y) = \rho(X,Y)$, since all functions of Bernoulli variables are linear.

It is known that $\rho_m$ tensorizes~\cite{witsenhausen1975sequences}, in the sense that the maximal correlation of two product distributions is bounded by the maximal correlation between pairs of coordinates:
\begin{lemma}
If $(X_1, Y_1)$ and $(X_2, Y_2)$ are independent, then
\begin{equation*}
\rho_m\left((X_1, X_2), (Y_1, Y_2)\right) = \max\left\{\rho_m(X_1, Y_1), \rho_m(X_2, Y_2)\right\}.
\end{equation*}
\end{lemma}
Furthermore, it is evident from the definition that $\rho_m$ satisfies a data processing inequality, so $\rho_m(f(X), g(Y)) \le \rho_m(X,Y)$.

\begin{lemma}
\label{LemHGRBer}
Suppose we have independent random variables $(Z, R_1, R_2)$, such that $Z \sim \mathrm{Ber}(q)$ and $R_1, R_2 \sim \mathrm{Ber}(p)$. Then
\begin{equation*}
\rho_m(Z R_1, Z R_2) = \frac{p(1-q)}{1-pq}.
\end{equation*}
\end{lemma}

\begin{proof}
Since the maximal correlation $\rho_m$ is equal to the Pearson correlation $\rho$ for Bernoulli varaibles, we may simply calculate
\begin{align*}
\rho(ZR_1, ZR_2) & = \frac{\Cov(ZR_1, ZR_2)}{\sqrt{\Var(ZR_1) \Var(ZR_2)}} = \frac{\E[Z^2 R_1 R_2] - \E[ZR_1] \E[ZR_2]}{pq(1-pq)} \\
& = \frac{qp^2 - p^2q^2}{pq(1-pq)} = \frac{p(1-q)}{1-pq},
\end{align*}
as wanted.
\end{proof}


\section{Details of privacy algorithms}
\label{AppPrivate}

In this appendix, we provide detailed statements of some of the privacy algorithms that were omitted in the main text due to space constraints.


\subsection{Two-community private PCA with Lipschitz extension} 
\label{AppPCALip}

We first describe the details of the algorithm studied in Section~\ref{SecTwoPCA} (see Algorithm~\ref{alg:privPCALipExt}).

\begin{algorithm}[h!]
\caption{Private PCA via Lipschitz Extension}
\label{alg:privPCALipExt}
\begin{algorithmic}[1]
\Function{$\mathcal{A}_{PLE}^{(\epsilon)}$}{adjacency matrix $A \in \{0, 1\}^{n \times n}$ based on $G \in \mathcal{G}_n$, privacy parameter $\epsilon$, truncation parameter $D$, approximation error $\gamma$}
    \State Sample $L_\epsilon \sim \mathrm{Lap}\left(\frac{2}{\epsilon}\right)$, and compute $\sigmahat = \left[\frac{\textbf{1}_n^T A\textbf{1}_n}{n} + L_\epsilon\right]_{0}^n$.
    \State Sample $\hat{u}^{0}_2$ from $\exp\left(\frac{\epsilon}{6D^2}\widehat{s}'_{A^2}(v)\right)$ w.r.t. the uniform measure on the unit sphere, where $\widehat{s}'_{A^2}(v) := \widehat{s}_{A^2}(v) - \frac{\sigmahat^2}{n} v^T J_nv$.
    \State Set $\hat{u}_2 = \frac{(I_n - J_n/n)\hat{u}^{0}_2}{||(I_n - J_n/n)\hat{u}^{0}_2||_2}$, so that $\hat{u}_2^T\mathbf{1}_n = 0$.
    \State Construct $\hat{U} = \left[\frac{\mathbf{1}_n}{\sqrt{n}}, \hat{u}_2\right] \in \mathbb{O}(n, 2)$.
    \State Let $\hat{\Theta} \in \mathbf{M}_{n\times 2}$ and $\hat{X} \in \mathbb{R}^{n \times 2}$ be such that 
    \begin{align*}
        ||\hat{\Theta}\hat{X} - \hat{U}||_F^2 \leq (1 + \gamma)\mathop{\inf}\limits_{\Theta' \in \mathbf{M}_{n\times 2}, X' \in \mathbb{R}^{2 \times 2}}||\Theta'X' - \hat{U}||_F^2.
    \end{align*}
    \State Let $\hat{\theta} = \{\hat{\theta}_i\}_{i = 1}^n \in [2]^n$ be given by $\hat{\theta}_i = \mathop{\arg\max}\limits_{j \in [2]}\hat{\Theta}_{ij}$, for all $i \in [n]$.
    \State Return $\hat{\theta}$.
\EndFunction
\end{algorithmic}
\end{algorithm}


\subsection{Clustering by edge flipping}
\label{AppSCFlip}

We now provide the details of the spectral clustering algorithm from \cite{HehEtal22}, which applies spectral clustering after flipping the entries of the observed adjacency matrix.

\begin{definition}[Symmetric edge-flip mechanism~\cite{HehEtal22}]
\label{def:Edge_Flip}
Let $\epsilon > 0$ and let $A \in \{0, 1\}^{n \times n}$ be any adjacency matrix. For $i \in [n]$, let $A_{i*}$ denote the $i^{\text{th}}$ row of $A$, and let $\mathcal{M}_i : \{0, 1\}^n \rightarrow \{0, 1\}$ be such that
\begin{align*}
[\mathcal{M}_i(x)]_j = 
\begin{cases}
 0 & if \quad i \geq j,\\    
  1 - x_j & if \quad  i < j, \quad \textrm{with probability} \quad \frac{1}{1 + e^\epsilon},\\
  x_j & if \quad  i < j, \quad \textrm{with probability} \quad \frac{e^\epsilon}{1 + e^\epsilon}.
\end{cases}
\end{align*}
Let
\begin{align*}
\mathcal{T}(A) &= \begin{bmatrix}
           \mathcal{M}_1(A_{1*}) \\
           \vdots \\
           \mathcal{M}_n(A_{n*})
         \end{bmatrix}.
\end{align*}
Then the $n \times n$ symmetric edge-flipping mechanism is $\mathcal{M}_\epsilon(A) = \mathcal{T}(A) + \mathcal{T}(A)^T$.
\end{definition}
Next, we present the private spectral clustering algorithm from \cite{HehEtal22} (see Algorithm~\ref{alg:EF_Spec_Clus}).

\begin{algorithm}
\caption{EF Spectral Clustering (k-means) from \cite{HehEtal22}}
\label{alg:EF_Spec_Clus}
\begin{algorithmic}[1]
\Function{$\mathcal{A}_{EF}^{(\epsilon)}$}{adjacency matrix $A \in \{0, 1\}^{n \times n}$ based on $G \in \mathcal{G}_n$, number of communities $k$, approximation error $\gamma$, privacy budget $\epsilon$}
    \State Construct the edge-flipped $\mathcal{M}_\epsilon(A)$.
    \State Let $A_{\downarrow} = \begin{cases}
 \mathcal{M}_\epsilon(A) & if \quad \epsilon = \infty,\\    
  \mathcal{M}_\epsilon(A) - (e^\epsilon + 1)^{-1}(\textbf{1}\textbf{1}^T- I_n) & if \quad  \epsilon < \infty.
\end{cases}$
    \State Let $\hat{u}_1, \dots, \hat{u}_k \in \mathbb{R}^n$ be the $k$ leading eigenvectors (by absolute value) of $A_{\downarrow}$.
    \State Let $\hat{\Theta} \in \mathbf{M}_{n\times k}$ and $\hat{X} \in \mathbb{R}^{n \times k}$ be such that 
    \begin{align*}
        ||\hat{\Theta}\hat{X} - U_{\downarrow}||_F^2 \leq (1 + \gamma)\mathop{\inf}\limits_{\Theta' \in \mathbf{M}_{n\times k}, X' \in \mathbb{R}^{k \times k}}||\Theta'X' - U_{\downarrow}||_F^2,
    \end{align*}
    where $U_{\downarrow} = [\hat{u}_1, \dots, \hat{u}_k] \in \mathbb{R}^{n \times k}$.
    \State Let $\hat{\theta} = \{\hat{\theta}_i\}_{i = 1}^n \in [k]^n$ be given by $\hat{\theta}_i = \mathop{\arg\max}\limits_{j \in [k]}\hat{\Theta}_{ij}$, for all $i \in [n]$.
    \State Return $\hat{\theta}$.
\EndFunction
\end{algorithmic}
\end{algorithm}


\subsection{Clustering by convex relaxation}
\label{AppSCOpt}

We now describe the private estimation algorithm based on optimization from \cite{CheEtal23} (see Algorithm~\ref{alg:TC_Optimiz}).

\begin{algorithm}[h!]
\caption{Two-community Private Weak Recovery for SBM}
\label{alg:TC_Optimiz}
\begin{algorithmic}[1]
\Function{$\mathcal{A}_{TC}^{(\epsilon, \delta)}$}{adjacency matrix $A \in \{0, 1\}^{n \times n}$ based on $G \in \mathcal{G}_n$, probability matrix $B$ of the SBM, privacy parameters $\epsilon, \delta$}
\State Construct the rescaled, recentered matrix $Y = \frac{2}{n(B_{11} - B_{12})} \left(A - \frac{(B_{11} + B_{12})}{n} J_n\right)$.
\State Solve the convex optimization problem
\begin{equation*}
\widehat{X} \in \arg\min_{X \in \real^{n \times n}: X \succeq 0, X_{ii} = \frac{1}{n} \; \forall i} \|Y-X\|_F^2.
\end{equation*}
\State Compute the leading unit eigenvector $\widetilde{x}$ of $\widehat{X} + N$, where $N \in \real^{n \times n}$ is a matrix of i.i.d.\ entries distributed as $N\left(0, \frac{96\log(1/\delta)}{n^2\epsilon^2(B_{11} - B_{12})}\right)$.
\State For each $i \in [n]$, assign node $i$ to community $\text{sign}(\widetilde{x}_i)$. Call this $\hat{\theta} \in [2]^n$.
\State Return $\hat{\theta}$.
\EndFunction
\end{algorithmic}
\end{algorithm}


\subsection{Clustering by low-rank matrix estimation}
\label{AppSCLR}

We now provide the private spectral clustering method based on private low-rank matrix estimation, as studied in \cite{hardt2014noisy, d2025tight} (see Algorithm~\ref{alg:Matrix_Estimation_hardt2014}).

\begin{algorithm}[h!]
\caption{Spectral Clustering via Matrix Estimation}
\label{alg:Matrix_Estimation_hardt2014}
\begin{algorithmic}[1]
\Function{$\mathcal{A}_{ME}^{(\epsilon, \delta)}$}{adjacency matrix $A \in \{0, 1\}^{n \times n}$ based on $G \in \mathcal{G}_n$, iteration count $L$, number of communities $k$, approximation error $\gamma$, privacy parameters $\epsilon, \delta$}
    \State Start running PPM from \cite{hardt2014noisy}: Let $X_0$ be a random orthonormal basis.
    \For{$l = 1, \dots, L$}
    \State $Y_l = AX_{l - 1} +G_l$, where $G_l \sim N\left(0, \frac{4kL\log(1/\delta)}{\epsilon^2}\right)^{n \times 2k}$.
    \State Compute the QR factorization $Y_l = X_lR_l$.
    \EndFor
    \State Compute the rank-$2k$ matrix $\hat{A}_{(2k)} = X_{L - 1}Y_L^T \in \mathbb{R}^{n \times n}$. This ends PPM from \cite{hardt2014noisy}.
    \State Compute the SVD: $\hat{A}_{(2k)} = U_{(2k)}D_{(2k)}V_{(2k)}^T$, with $U_{(2k)} \in \mathbb{R}^{n \times 2k}$. Extract the first $k$ columns of $U_{(2k)}$, and denote them by $U_{(k)} \in \mathbb{R}^{n \times k}$.
    \State Let $\hat{\Theta} \in \mathbf{M}_{n\times k}$ and $\hat{X} \in \mathbb{R}^{n \times k}$ be such that 
    \begin{align*}
        ||\hat{\Theta}\hat{X} - U_{(k)}||_F^2 \leq (1 + \gamma)\mathop{\inf}\limits_{\Theta' \in \mathbf{M}_{n\times k}, X' \in \mathbb{R}^{k \times k}}||\Theta'X' - U_{(k)}||_F^2.
    \end{align*}
    \State Define $\hat{\theta} = \{\hat{\theta}_i\}_{i = 1}^n \in [k]^n$ by $\hat{\theta}_i = \mathop{\arg\max}\limits_{j \in [k]}\hat{\Theta}_{ij}$, for all $i \in [n]$.
    \State Return $\hat{\theta}$.
\EndFunction
\end{algorithmic}
\end{algorithm}


\subsection{Clustering by approximate subspace estimation}
\label{AppSCSubspace}

Finally, we provide the details of the spectral clustering algorithm which uses a private approximate subspace method (some details are omitted, which relate to the algorithmic details in the earlier papers~\cite{SinSte21, KamEtal22, mahpud2022differentially}). See Algorithm~\ref{alg:Subspace_Estimation_GoodC}.

\begin{algorithm}[h!]
\caption{Spectral Clustering via Subspace Estimation}
\label{alg:Subspace_Estimation_GoodC}
\begin{algorithmic}[1]
\Function{$\mathcal{A}_{SE}^{(\epsilon, \delta)}$}{adjacency matrix $A \in \{0, 1\}^{n \times n}$ based on $G \in \mathcal{G}_n$ (or weighted adjacency matrix $W \odot A \in \mathbb{R}^{n \times n}$ based on $G \in \mathcal{W}_n$), number of communities $k \asymp 1$, approximation error $\gamma$, number of Gaussian points $q = C'k$ (can choose $0 < C' \asymp 1$ as large as we like), privacy parameters $\epsilon, \delta$, failure error $1 \asymp \zeta \in (0, 1/3)$ s.t. $3k \geq \log(2/\zeta)$, parameters $\rho = \frac{\epsilon^2}{32q\log(1/\delta)}$, $R_{max} = \sqrt{n}\log(n)$, $r_{min} = \frac{1}{n^3} $}
    \State Randomly divide the rows of $A$ (or $W \odot A$) into $t$ chunks $\{A_1, \dots, A_t\}$ (or $\{(W \odot A)_1, \dots, (W \odot A)_t\}$). Let $m = \frac{n}{t}$, so $A_j, (W \odot A)_j \in \mathbb{R}^{m \times n}$. That is, sample uniformly from the set of permutations of $[n]$, and split these permuted indexes sequentially into $t$ groups of size $m$,
    \State Form $\Pi_1, \dots, \Pi_t$. $\Pi_j = U_j^{(k)}(U_j^{(k)})^T \in \mathbb{R}^{m \times m}$, where $A_jA_j^T = U_jD_jU_j^T$ (similarly for $(W \odot A)_j(W \odot A)_j^T$), and $U_j^{(k)} = [U_j^1, \dots, U_j^k] \in \mathbb{R}^{m \times k}$.
    \State Form $\Pi^{(1)}, \dots, \Pi^{(t)} \in \mathbb{R}^{n \times n}$, where $\Pi^{(j)}$ is the projection onto the \emph{row} space of $\Pi_jA_j$ (or $\Pi_j(W \odot A)_j$). Write $\Pi^{(j)} = U^{(j), (k)}(U^{(j), (k)})^T$, where $U^{(j), (k)}$ consists of the first $k$ columns from the matrix of orthonormal eigenvectors from the SVD of $(\Pi_jA_j)^T\Pi_jA_j$ (or $\Pi_j(W \odot A)_j^T(\Pi_j(W \odot A)_j)$).
    \State Sample reference points $z_1, \dots, z_q \stackrel{i.i.d.}{\sim} N(0, I_n)$. Construct $z_i^j = \Pi^{(j)} z_i$.
    \State Let $X_i = \{z_i^{j, grid}\}_{j = 1}^t$, where $z_i^{j, grid}$ is the closest point (in $\ell_2$-norm) on the grid $\left[-\frac{R_{max}}{\sqrt{n}}, \frac{R_{max}}{\sqrt{n}}\right]^n$ with spacing $r_{min}$, to $z_i^{j}$.  
    \State For each $i$, run $\mathrm{GoodCenter}_{X_i, 0, R_{max}, \frac{r_{min}}{2}, \frac{\zeta}{q}, \rho}$ to obtain $\theta^*_i$. 
    \State Compute $r = \frac{\sqrt{\log(n)}}{n^2\sqrt{n}} + \left(\frac{\log^{1/4}(1/\delta)}{\log^{5/2}(n)\sqrt{\epsilon}} + \frac{\sqrt{\log(1/\delta)}}{\log^5(n)\epsilon}\right)\log(n)$.
    \State For each $i$, truncate all $z_i^j$'s to the ball $B(\theta^*_i, r)$. Let $\hat{z}_i = \frac{1}{t} \sum_{j=1}^t z_i^j + N(0, \sigma^2 I_n)$, where $\sigma = \frac{2r}{t\sqrt{2\rho}}$.
    \State Let $\hat{Z} = [\hat{z}_1, \dots, \hat{z}_q] \in \mathbb{R}^{n \times q}$. Write the SVD of $\hat{Z} = \hat{U}\hat{D}\hat{V}^T$. 
    Extract $\hat{U}^{(k)} = [\hat{U}_1, \dots, \hat{U}_k] \in \mathbb{R}^{n \times k}$. Form the projection $\hat{\Pi}$. 
    \State Let $\hat{\Theta} \in \mathbf{M}_{n\times k}$ and $\hat{X} \in \mathbb{R}^{n \times k}$ be such that 
    \begin{align*}
        ||\hat{\Theta}\hat{X} - \hat{U}^{(k)}||_F^2 \leq (1 + \gamma)\mathop{\inf}\limits_{\Theta' \in \mathbf{M}_{n\times k}, X' \in \mathbb{R}^{k \times k}}||\Theta'X' - \hat{U}^{(k)}||_F^2.
    \end{align*}
    \State Let $\hat{\theta} = \{\hat{\theta}_i\}_{i = 1}^n \in [k]^n$ be given by $\hat{\theta}_i = \mathop{\arg\max}\limits_{j \in [k]}\hat{\Theta}_{ij}$, for all $i \in [n]$.
    \State Return $\hat{\theta}$.
\EndFunction
\end{algorithmic}
\end{algorithm}

We also present Algorithm 6 from the supplement of \cite{mahpud2022differentially} and its corresponding utility guarantee (see Algorithm~\ref{alg:GoodC_zCDP}).

\begin{algorithm}[h!]
\caption{Noisy Average and Radius (GoodCenter)}
\label{alg:GoodC_zCDP}
\begin{algorithmic}[1]
\Function{GoodCenter}{set of $t$ points $X \in \mathbb{R}^{n \times t}$, parameters $\theta_0$, $R_{max}$, and $r_{min}$, such that $X \subset B(\theta_0, R_{max})$, and $r_{opt} \geq r_{min}$, where $r_{opt}$ is the radius of the minimum enclosing ball of $X$. Failure parameter $\zeta \in (0, 1)$, privacy parameter $\rho$.}
\State Set $T = \lceil\log_2(R_{max}/r_{min})\rceil + 1$, $Y = \sqrt{\frac{2T\log(4T/\zeta)}{\rho}}$.
\State Set $\sigma_{count}^2, \sigma_{sum}^2 = \frac{T}{\rho}$.
\State Initialize $X_0 \leftarrow X$, $\theta^0 \leftarrow \theta_0$, $n_{cur} \leftarrow n$, and $r_{cur} \leftarrow R_{max}$.
\For{$s = 0, \dots, S - 1$}
    \State $X^s \leftarrow X^s \cap B(\theta^s, r_{cur})$.
    \State Sample $\Delta_{sum} \sim N(0, 4r_{cur}^2\sigma_{sum}^2I_n)$.
    \State $\tilde{\mu}^s \leftarrow \left(\mathop{\sum}\limits_{x \in X^s} x + \Delta_{sum}\right)/n_{cur}$.
    \State $\Delta_{count} = N(0, \sigma_{count}^2)$.
    \If{$\left|X^s \setminus B\left(\tilde{\mu}^s, \frac{r_{cur}}{2}\right)\right| + \Delta_{count} \geq Y$}
    \State Return $B(\theta^s, r_{cur})$.
    \EndIf
    \State Update: $r_{cur} \leftarrow \frac{r_{cur}}{2}$, $n_{cur} \leftarrow n_{cur} - 2Y$, $\theta^{s + 1} \leftarrow \tilde{\mu}^s$ .
    \EndFor
\State Return $B(\theta^S, r_{cur})$.
\EndFunction
\end{algorithmic}
\end{algorithm}

\begin{lemma}[Corollary A.3 in \cite{mahpud2022differentially}]
\label{lemma:cor_a3_supp} 
Let $X = [X_1, \dots, X_t]^{n \times t}$ and $\zeta \in (0, 1)$. Suppose the $t$ points lie on a grid in $[-B, B]^{n \times n}$ of step size $\tau$. Then $\mathrm{GoodCenter}_{X, 0, B\sqrt{n}, \frac{\tau}{2}, \zeta, \rho}$ (Algorithm \ref{alg:GoodC_zCDP}) is $\rho$-zCDP (with respect to a change among the $t$ points). In addition, for
\begin{align*}
t \succsim \sqrt{\frac{\log(Bn/\tau)}{\rho}}\left(\sqrt{n} + \sqrt{\log(Bn/(\tau\zeta))}\right).    
\end{align*}
$\mathrm{GoodCenter}_{X, 0, B\sqrt{n}, \frac{\tau}{2}, \zeta, \rho}$ returns, with probability at least $1 - \zeta$, a ball $B(\theta^*, r^*)$ satisfying: For $X' = X \setminus B(\theta^*, r^*)$, it holds $|X'| \geq \frac{t}{2}$. For $B(\theta_{opt}(X'), r_{opt}(X'))$ the smallest $\ell_2$-ball in $\mathbb{R}^n$ enclosing $X'$, we have $||\theta^* - \theta_{opt}(X')||_2 \leq 6r_{opt}(X')$.  
\end{lemma}

Now we state and prove a matrix estimation result based on \cite{hardt2014noisy}:

\begin{lemma}[Adapted from Corollary 4.6 in \cite{hardt2014noisy}]
\label{lemma:hardt2014_conc_bound}
Let $A \in \mathbb{R}^{n \times n}$ be an adjacency matrix, and let $\gamma_1 = 1 - \frac{\sigma_{k + 1}(A)}{\sigma_k(A)}$. There exists an efficient algorithm (cf.\ $\hat{A}_{(2k)}$ in Algorithm \ref{alg:Matrix_Estimation_hardt2014}) that, given $A \in \mathbb{R}^{n \times n}$ and $k$, returns a rank-$2k$ matrix $\hat{A}_{(2k)} \in \mathbb{R}^{n \times n}$ such that, with probability at least $0.9$, we have
\begin{align*}
    \left\|A - \hat{A}_{(2k)}\right\|_2 \leq \sigma_{k + 1}(A) + O\left(\frac{\sigma_1(A)\sqrt{(k/\gamma_1)n\log(n)\log(\log(n)/\gamma_1)}}{\sigma_k(A) - \sigma_{k + 1}(A)}\cdot \frac{\sqrt{\log(1/\delta)}}{\epsilon}\right).
\end{align*}
In addition, the algorithm is $\frac{\epsilon^2}{4\log(1/\delta)}$-edge zCDP.
\end{lemma}
\begin{proof}
We start with the privacy guarantee. Consider the mechanism PPM from \cite{hardt2014noisy} (see Algorithm \ref{alg:Matrix_Estimation_hardt2014}), executed with each $G_{l}$ sampled independently as $G_l \sim N(0, \sigma^2)^{n \times p}$, with $\sigma^2 = \frac{4kL\log(1/\delta)}{\epsilon^2}$. Let $A$ and $A'$ be edge-adjacent and differing in entries $(i, j)$ and $(j, i)$. Then $(A - A')(A - A')^T$ is the diagonal matrix with a $1$ in entries $(i, i)$ and $(j, j)$, and $0$ everywhere else. Hence, we have $||A - A'||_2 \leq 1$. At step $l$ of the algorithm, for fixed $X_{l - 1}$, the $\ell_2$-sensitivity of $AX_{l - 1, s}$, where $X_{l - 1, s}$, the $s^{\text{th}}$ column of $X_{l - 1}$, is $1$. Since we have $2k$ columns and $L$ iterations in total, using the Gaussian mechanism from Lemma \ref{lemma:Gauss_zCDP} and the composition of zCDP mechanisms for Lemma \ref{lemma:Comp_zCDP} (and the adaptive version), we obtain the desired privacy guarantee.

For the utility analysis, choose $L \asymp \log(n)/\gamma_1$. Following the lines of the proof of Corollary 4.6 in \cite{hardt2014noisy}, we obtain the desired utility.
\end{proof}

\section{Proofs for Section~\ref{SecPPCA}}
\label{AppPPCA}

\subsection{Proof of Lemma~\ref{LemPCAScore}}
\label{AppLemPCAScore}

Note that
\begin{equation*}
\Tr(V^TA^2V) = \|AV\|_F^2, \qquad \Tr(V^T(A')^2V) = \|A'V\|_F^2.
\end{equation*}
Thus,
\begin{equation*}
|\Tr(V^TA^2V) - \Tr(V^T(A')^2V)| = \left|\|AV\|_F^2 - \|A'V\|_F^2\right|.
\end{equation*}
Furthermore,
\begin{equation*}
\|AV\|_F^2 \le \|A\|_2^2 \|V\|_F^2 = q\|A\|_2^2 \le q\|A\|_1 \|A\|_\infty \le qD^2,
\end{equation*}
and similarly for $\|A'V\|_F^2$. Since both quantities are nonnegative, we see that
\begin{equation*}
\left|\|AV\|_F^2 - \|A'V\|_F^2\right| \le qD^2,
\end{equation*}
as claimed.


\subsection{Proof of Lemma~\ref{LemExpPCA}}
\label{AppLemExpPCA}

We will show that the global sensitivity of $\widehat{s}_{A^2}(v)$ is bounded by $(q+2)D^2$. Consider two node-adjacent adjacency matrices $(A,A')$, and note that any matrix $C$ which is feasible in the LP corresponding to $\widehat{s}_{A^2}(v)$ can be made into a matrix $C'$ feasible in the LP corresponding to $\widehat{s}_{A'^2}(v)$ by zeroing out at most one row and column. Note that
\begin{equation*}
s_C(V) - s_{C'}(V) = \Tr(V^T(C-C')V) + \Tr((C-C')J_n).
\end{equation*}
The latter term is the sum of the entries in a single row/column of $C-C'$, and since $\|C\|_\infty \le D^2$, it must be bounded by $2D^2$. Furthermore, consider $W=W^T$ and $Z \succeq 0$, with spectral decomposition $W = HLH^T = \mathop{\sum}\limits_{i = 1}^n L_{ii}H_iH_i^T$. Since $H_iZH_i^T \geq 0$, for all $i \in [n]$, and 
\begin{align*}
\Tr(WZ) = \sum_{i = 1}^n \Tr(L_{ii}ZH_iH_i^T) = \sum_{i = 1}^n L_{ii}H_iZH_i^T \leq ||W||_2 \sum_{i}\Tr(HZH^T) = ||W||_2\Tr(Z), 
\end{align*}
we have
\begin{align*}
\Tr(V^T(C-C')V) & = \Tr((C-C')VV^T) \le \Tr(VV^T) \|C-C'\|_2 \\
& \le \Tr(V^TV) \|C-C'\|_\infty \le q \|C\|_\infty \le qD^2.
\end{align*}
Altogether, we have
\begin{equation*}
s_C(V) - s_{C'}(V) \le (q+2)D^2,
\end{equation*}
so by taking a supremum over feasible matrices $C$, we clearly have
\begin{equation*}
\widehat{s}_{A^2}(V) \le \widehat{s}_{A'^2}(V) + (q+2)D^2.
\end{equation*}
A similar inequality holds in the opposite direction, implying the desired sensitivity bound.

Applying the exponential mechanism to the score function $\widehat{s}_{A^2}(V)$ then gives the desired result.

\subsection{Proof of Corollary \ref{cor:ExpPCA_Cor}}
\label{AppExpPCA_Cor}

Fix $M \in \mathbb{R}^{n \times n}$. By the proof of Lemma \ref{LemExpPCA}, we have $
|\widehat{s}_{A^2}(V) - \widehat{s}_{A'^2}(V)| \le (q+2)D^2$. Since
\begin{equation*}
|\widehat{s}_{A^2}(V) + V^TMV - \widehat{s}_{A'^2}(V) - V^TMV| \le (q+2)D^2,
\end{equation*}
the conclusion follows.


\subsection{Proof of Theorem~\ref{theorem:Lip_PCA_Main}}
\label{AppLipPCA}

Regarding the privacy guarantee, note that $\frac{\textbf{1}_n^T A\textbf{1}_n}{n}$ has sensitivity at most $2$, so $\sigmahat$ is $(\epsilon, 0)$-DP. Next, for any fixed $\sigma^2$, sampling from $\exp\left(\frac{\epsilon}{6D^2} (\widehat{s}_{A^2}(v) - \frac{\sigma^2}{n}v^TJ_nv)\right)$ is $(\epsilon, 0)$-node DP, by Corollary \ref{cor:ExpPCA_Cor}. Hence, by adaptive composition, the vector $\hat{u}_2^0$ is $(2\epsilon, 0)$-node DP, and by post-processing, the vector $\hat{u}_2$ is $(2\epsilon, 0)$-node DP, as well. Finally, by post-processing again, the matrix $\hat{U}$ and all derived quantities satisfy $(2\epsilon, 0)$-node DP.

Regarding utility, we have the following lemma, proved in Appendix~\ref{AppLemRegEvent}:

\begin{lemma}
\label{lemma:reg_event_PCA}
Under the assumptions of Theorem \ref{theorem:Lip_PCA_Main}, there exists an event $\Omega$ with $\mathbb{P}(\Omega)\ge 1 - \frac{1}{\mathrm{poly}(n)}$, such that on $\Omega$, we have $||A - P||_2 \precsim \sqrt{d}$, $|\sigmahat^2-\sigma_1^2(P)| \precsim d\sqrt{d}$, and $\widehat{s}_{A^2}(v) = s_{A^2}(v)$, for every $v$.
\end{lemma}

To reason about $\hat{u}_2^0$ and $\hat{u}_2$, we use following lemma, proved in Appendix~\ref{AppTopEvec}:

\begin{lemma}
\label{lemma:conc_top_evec_PCA}
Let $M = A^2+n^2I_n-\frac{\sigmahat^2}{n}J_n$, and let $u_2$ be a leading unit eigenvector of $M$. Then the vector $\hat u_2^0$ sampled in Algorithm \ref{alg:privPCALipExt} satisfies $\mathbb{P}\left(|(\hat u_2^0)^Tu_2|^2 > 1 - \frac{D^2}{\epsilon}\right) \geq 1 - \frac{1}{\mathrm{poly}(n)}$.
\end{lemma}

Let $1 - \rho^2 = \frac{D^2}{\epsilon}$, as in the proof of Lemma \ref{lemma:conc_top_evec_PCA}, and let $\Omega$ be as in Lemma \ref{lemma:reg_event_PCA}. Let $u_{2, P}$ be the second orthonormal eigenvector of $P$. We will control $|u_{2, P}^T\hat{u}_2|$. First, let $u_2^*$ be the top eigenvector of $A^2 - \frac{\sigma_1^2(P)}{n}J_n$. We check the eigengap condition using the following result, proved in Appendix~\ref{AppAsigma}:

\begin{lemma}
\label{lemma:Asigma_A_sigmahat}
Let $A_1 = A^2 - \frac{\sigma_1^2}{n} \textbf{1}_n\textbf{1}_n^T$ and $A_0 = A^2 - \frac{\sigmahat_1^2}{n} \textbf{1}_n\textbf{1}_n^T$. Let $u_2^*$ be the top eigenvector of $A_1$, and let $u_2$ be the top eigenvector of $A_0$. Then
\begin{equation*}
\mathop{\inf}_{q: q^2 = 1}||qu_2 - u_2^*||_2 \precsim \frac{|\sigmahat_1^2 - \sigma_1^2|}{\lambda_1(A_1) - \lambda_2(A_1)}.
\end{equation*}
\end{lemma}

On $\Omega$, we have $|\sigmahat^2 - \sigma_1^2(P)| \precsim d\sqrt{d}$. Also, by Lemma \ref{LemWeyl}, we have 
\begin{align*}
&\lambda_1\left(A^2 - \frac{\sigma_1^2(P)}{n}J_n\right) \geq \sigma_2^2(P) - ||A^2 - P^2||_2 \succsim d^2,\\
&\lambda_1\left(A^2 - \frac{\sigma_1^2(P)}{n}J_n\right) - \lambda_2\left(A^2 - \frac{\sigma_1^2(P)}{n}J_n\right) \geq \sigma_2^2(P) - 2||A^2 - P^2||_2 \succsim d^2.
\end{align*}
Thus, by Lemma \ref{lemma:Asigma_A_sigmahat}, we have on $\Omega$ that $||u_2 - q_1u_2^*||_2 \leq\frac{c_1}{\sqrt{d}}$, for some $0 < c_1 \asymp 1$, with $q_1^2 = 1$. If $q_1 = -1$, we can replace $u_2^*$ by $-u_2^*$, since we can pick any top unit singular vector of $A^2 - \frac{\sigma_1^2(P)}{n}J_n$. We thus have $||u_2 - u_2^*||_2 \leq\frac{c_1}{\sqrt{d}}$. We now derive the following result, proved in Appendix~\ref{AppASigmaP}:

\begin{lemma}
\label{lemma:A_sigma_P_sigma}
Let $\frac{\mathbf{1}_n}{\sqrt{n}}$ be the top eigenvector of $P$, with top eigenvalue given by $\sigma_1$. Let $A_1 = A^2 - \frac{\sigma_1^2}{n} \textbf{1}_n\textbf{1}_n^T$ and $P_1 = P^2 - \frac{\sigma_1^2}{n} \textbf{1}_n\textbf{1}_n^T$. Let $u_2^*$ be the top eigenvector of $A_1$, and let $u_{2, P}$ be the top eigenvector of $P_1$. Then
\begin{equation*}
\mathop{\inf}_{q: q^2 = 1}||qu_2^* - u_{2, P}||_2 \precsim \frac{||A^2 - P^2||_2}{\sigma_2^2(P)}.
\end{equation*}
\end{lemma}

By Lemma \ref{lemma:A_sigma_P_sigma}, we have, on $\Omega$, that
\begin{align*}
\mathop{\inf}_{q: q^2 = 1}||qu_2^* - u_{2, P}||_2 \precsim \frac{||A^2 - P^2||_2}{\sigma_2^2(P)} \precsim \frac{d\sqrt{d}}{d^2} \precsim \frac{1}{\sqrt{d}}.
\end{align*}
Hence, on $\Omega$, we have $||u_2^* - q_2u_{2, P}||_2 \leq \frac{c_2}{\sqrt{d}}$, for some $c_2 \asymp 1$ and $q_2^2 = 1$. If $q_2 = 1$, we can pick $-u_{2, P}$ instead of $u_{2, P}$. Hence, we have $||u_2^* +u_{2, P}||_2 \leq \frac{c_2}{\sqrt{d}}$. Thus, since
\begin{equation*}
||(I_n - J_n/n)\hat{u}_2^0||_2 \leq ||I_n - J_n/n||_2||\hat{u}_2^0||_2 \leq 1,
\end{equation*}
we have 
\begin{align*}
\mathbb{P}\left(|\hat{u}_2^Tu_{2, P}| \leq \rho'\right) & \leq \mathbb{P}\left(|u_{2, P}^T(I_n - J_n/n)\hat{u}_2^0| \leq \rho'\right) \leq \mathbb{P}\left(|(\hat{u}_2^0)^Tu_{2, P}| - |\hat{u}_2^TJ_nu_{2, P}|/n \leq \rho'\right) \\
& = \mathbb{P}(|(\hat{u}_2^0)^Tu_{2, P}| \leq \rho'),
\end{align*}
since $u_{2, P}$ and $\frac{\mathbf{1}_n}{\sqrt{n}}$ are the top orthonormal eigenvectors of $P$, so $u_{2, P}^TJ_n = 0$. Now, for $\rho' = \rho - (c_1 + c_2)/\sqrt{d}$, we have
\begin{align*}
\mathbb{P}\left(|\hat{u}_2^Tu_{2, P}| \leq  \rho'\right) & \leq \mathbb{P}\left(|u_2^T\hat{u}_2^0| \leq \rho' + |(u_2^* - u_2)^T\hat{u}_2^0| + |(u_2^* + u_{2, P})^T\hat{u}_2^0|\right)\\
&\leq \mathbb{P}\left(|u_2^T\hat{u}_2^0| \leq \rho' + ||u_2^* - u_2||_2 + ||u_2^* + u_{2, P}||_2\right)\\
&\leq \mathbb{P}\left(|u_2^T\hat{u}_2^0| \leq \rho' + ||u_2^* - u_2||_2 + ||u_2^* + u_{2, P}||_2|\Omega\right)\mathbb{P}(\Omega) + \frac{1}{\mathrm{poly}(n)}\\
&\leq \mathbb{P}\left(|u_2^T\hat{u}_2^0| \leq \rho' +(c_1 + c_2)/\sqrt{d}|\Omega\right)\mathbb{P}(\Omega) + \frac{1}{\mathrm{poly}(n)}\\
&\leq \mathbb{P}\left(|u_2^T\hat{u}_2^0| \leq \rho' +(c_1 + c_2)/\sqrt{d}\right) + \frac{1}{\mathrm{poly}(n)}\\
&= \mathbb{P}\left(|u_2^T\hat{u}_2^0| \leq \rho\right) + \frac{1}{\mathrm{poly}(n)} \leq \frac{1}{\mathrm{poly}(n)},
\end{align*}
by Lemma \ref{lemma:conc_top_evec_PCA}. Let $U = \left[\mathbf{1}_n/\sqrt{n}, u_{2, P}\right]$. Then 
\begin{align*}
||\hat{U}\hat{U}^T - UU^T||_2^2 &= ||\hat{u}_2\hat{u}_2^T - u_{2, P}u_{2, P}^T||_2^2 \leq ||\hat{u}_2\hat{u}_2^T - u_{2, P}u_{2, P}^T||_F^2 = \Tr((\hat{u}_2\hat{u}_2^T - u_{2, P}u_{2, P}^T)^2)\\
&= 2 - 2(\hat{u}_2^Tu_{2, P})^2 \leq 2(1 - (\rho')^2) \leq 2\left(1 - \rho^2-\frac{(c_1 + c_2)^2}{d} + \frac{2\rho(c_1 + c_2)}{\sqrt{d}}\right)\\
& \precsim \frac{D^2}{\epsilon} + \frac{1}{\sqrt{d}} < \frac{1}{\log(n)},
\end{align*}
with probability at least $1 - \frac{1}{\mathrm{poly}(n)}$, since $1 > \rho^2 = 1 - \frac{D^2}{\epsilon} > 1 - \frac{1}{\log(n)}$. Now note that by Lemma \ref{lemma:l2.3SinSte} and Lemma \ref{lemma:prop2.2VuLei}, we have
\begin{align*}
\frac{1}{2}\inf\limits_{Q \in \mathbb{V}_{2, 2}}||\hat{U} - UQ||_F^2 & = \frac{1}{2}||\hat{U} - UQ'||_F^2 \leq ||\sin(\Theta)(\hat{U}, U)||_F^2 \leq 2||\sin(\Theta)(\hat{U}, U)||_2^2 \\
& \leq 2||\hat{U}\hat{U}^T - UU^T||_2^2 \precsim \frac{1}{\sqrt{d}} + \frac{D^2}{\epsilon},   
\end{align*}
where $Q' \in \mathbb{R}^{2 \times 2}$ and $Q'^TQ' = I_2$. The rest of the proof follows the lines of argument from the proof of Theorem 5.1 in \cite{HehEtal22}. We know from \cite{LeiRin15} that $U = \Theta X$, for some $\Theta \in \mathbf{M}_{n\times 2}$ and $X$ satisfying $||X_{j*} - X_{l*}||_2 = \sqrt{\frac{4}{n}}$, for $j \neq l$. Then $UQ' = \Theta XQ' = \Theta X'$, with $X' = XQ'$, and
\begin{equation*}
||X'_{j*} - X'_{l*}||_2 = ||(e_j- e_l)^TXQ'||_2 = ||X_{j*} - X_{l*}||_2 = \sqrt{\frac{4}{n}},
\end{equation*}
for $j \neq l$. Here, $e_j$ is the $j^{\text{th}}$ standard basis vector. For all $j \in [2]$, choose $\delta_j = \sqrt{\frac{4}{n}}$ and define $S_j$ as in Lemma 5.3 of \cite{LeiRin15}. We want to show that $(16 + 8\gamma)||\hat{U} - UQ'||_F^2 < \frac{\delta_j^2n}{2}$, for all $j \in [k]$. Since $\frac{\delta_j^2n}{2} > 1$, it is enough to prove that $(16 + 8\gamma)||\hat{U} - UQ'||_F^2 \leq 1$. Let $c_3 \asymp 1$ be the absolute constant such that
\begin{align*}
c_3 ||\hat{U} - UQ'||_F^2 \leq \frac{1}{\sqrt{d}} + \frac{D^2}{\epsilon}.
\end{align*}
Then
\begin{align*}
(16 + 8\gamma)||\hat{U} - UQ'||_F^2 \leq \frac{16 + 8\gamma}{c_3} \left(\frac{1}{\sqrt{d}} + \frac{D^2}{\epsilon}\right) < \frac{16 + 8\gamma}{c_3\log(n)} \leq 1, 
\end{align*}
as desired, for $n$ large enough, since $\gamma, c_3 \asymp 1$. Thus, for each
community $j \in [2]$, the set of nodes that are possibly misclassified by Algorithm \ref{alg:privPCALipExt} must be a subset of $S_j$, and since $\frac{\delta_j^2n}{2} > 1$, we have
\begin{align*}
\widetilde{\mathcal{L}}\left(\mathcal{A}_{PLE}^{(\epsilon)}(G), \theta\right) &\leq \mathop{\max}_{j \in [2]}\frac{2|S_j|}{n} \leq \frac{2}{n}\sum_{j = 1}^2|S_j| \leq \sum_{j = 1}^2|S_j|\delta_j^2 \leq (16 + 8\gamma)||\hat{U} - UQ'||_F^2\\
&\leq \frac{16 + 8\gamma}{c_3} \left(\frac{1}{\sqrt{d}} + \frac{D^2}{\epsilon}\right) \precsim \frac{1}{\sqrt{d}} + \frac{D^2}{\epsilon} < \frac{1}{\log(n)},
\end{align*}
with probability at least $1 - \frac{1}{\mathrm{poly}(n)}$, as desired, since $\gamma, c_3 \asymp 1$.


\subsection{Proofs of supporting lemmas for Theorem~\ref{theorem:Lip_PCA_Main}}

We now provide the proofs of the lemmas stated above in the proof of Theorem~\ref{theorem:Lip_PCA_Main}.

\subsubsection{Proof of Lemma~\ref{lemma:reg_event_PCA}}
\label{AppLemRegEvent}

Note that by Lemma \ref{lemma:hp_removal_d}, there exists an event $\Omega_0$ with $\mathbb{P}(\Omega_0) \geq 1 - \frac{1}{\mathrm{poly}(n)}$, such that on $\Omega_0$, the degree of every node in $A$ is at most $2d$. Hence, since $D \geq 3d$, we have $\widehat{s}_{A^2}(v) = s_{A^2}(v)$ on $\Omega_0$. Next, by Lemma \ref{lemma:thm5.2LeiRin}, since $a_n \succsim \frac{\log(n)}{n}$, there is an event $\Omega_1$, with $\mathbb{P}(\Omega_1) \geq 1 - \frac{1}{n}$, such that
\begin{equation*}
||A - \mathbb{E}[A]||_2 = ||A - P + \text{diag}(P)||_2 \precsim \sqrt{na} \asymp \sqrt{d}
\end{equation*}
on $\Omega_1$. Next, for $\Omega_2 = \left\{|L_\epsilon| \leq \frac{2\log(n)}{\epsilon}\right\}$, we have
\begin{align*}
\mathbb{P}\left(\Omega_2^c\right) = \mathbb{P}\left(|\mathrm{Lap}(1)| > \log(n)\right) \leq \frac{2}{n}.
\end{align*}
Hence, $|L_\epsilon| \leq \frac{2\log(n)}{\epsilon} < \frac{2}{D^2} \leq \frac{1}{n}$ on $\Omega_2$, since $\epsilon > D^2\log(n) \geq d^2\log(n) > n\log(n)$. Let $\Omega = \Omega_0 \cap \Omega_1 \cap \Omega_2$, so $\mathbb{P}(\Omega) \geq 1 - \frac{1}{\mathrm{poly}(n)}$. Next, on $\Omega$, we have
\begin{align*}
\frac{\textbf{1}_n^T A\textbf{1}_n}{n} + L_\epsilon & = \frac{\textbf{1}_n^T P\textbf{1}_n}{n} + \frac{\textbf{1}_n^T(A - P)\textbf{1}_n}{n} + L_\epsilon \\
& \leq \sigma_1(P) + ||A - P||_2 + \frac{1}{n} \precsim d + \sqrt{d} + \frac{1}{n} \precsim d. 
\end{align*}
Also,
\begin{align*}
\frac{\textbf{1}_n^T A\textbf{1}_n}{n} + L_\epsilon & = \frac{\textbf{1}_n^T P\textbf{1}_n}{n} + \frac{\textbf{1}_n^T(A - P)\textbf{1}_n}{n} + L_\epsilon \\
& \geq \sigma_1(P) - ||A - P||_2 - \frac{1}{n} \succsim d - \sqrt{d} - \frac{1}{n} \succsim d.
\end{align*}
Since $0 < d = o(n)$, we obtain $\sigmahat = \frac{\textbf{1}_n^T A\textbf{1}_n}{n} + L_\epsilon \asymp d$ on $\Omega$. Moreover, we have
\begin{align*}
|\sigmahat^2 - \sigma_1^2(P)| = |\sigmahat - \sigma_1(P)|\cdot|\sigmahat + \sigma_1(P)| \precsim d\sqrt{d},
\end{align*}
completing the proof.


\subsubsection{Proof of Lemma~\ref{lemma:conc_top_evec_PCA}}
\label{AppTopEvec}

Let $\Omega$ be as in Lemma \ref{lemma:reg_event_PCA}. Note that $M \succeq 0$, since $n^2I_n -\frac{\sigmahat^2}{n}J_n \succeq 0$ by the definition of $\sigmahat$. Define $\mathcal{U}_s = \{v: |v^Tu_2| \ge s\}$, for $s \in [0, 1]$. We parallel the utility analysis in Section 4.2 of Chaudhuri et al.~\cite{ChaEtal13}. We define $q_F(v) = \frac{\epsilon}{6D^2}v^T Mv$. The eigenvalues of $M$ are $\sigma_1(M) \geq \cdots \ge \sigma_n(M)$. For $s \in [0, 1]$, we then have (cf. \cite{ChaEtal13})
\begin{align*}
q_F(v) & \ge \frac{\epsilon}{6D^2} \left(s^2 \sigma_1(M) + (1 - s^2) \sigma_n(M)\right), & \text{for } v \in \mathcal{U}_s, \\
q_F(v) & \le \frac{\epsilon}{6D^2} \left(s^2 \sigma_1(M) + (1-s^2)\sigma_2(M)\right), & \text{for } v \in \mathcal{U}^c_s.
\end{align*}
Let $1 - \rho^2 = \frac{D^2}{\epsilon}$. For any $\tau \in [0, 1]$ (we choose this value later), we have
\begin{align}
\label{eq:Urho^c}
\mprob(\hat{u}_2^0 \in \mathcal{U}^c_\rho) & = \mathbb{E}_{A, \sigmahat}[\mprob(\hat{u}_2^0 \in \mathcal{U}^c_\rho|A, \sigmahat)] \notag \\
& = \mathbb{E}_{A, \sigmahat}[\mprob(\hat{u}_2^0 \in \mathcal{U}^c_\rho|A, \sigmahat)\cdot \mathbbm{1}_{\Omega}] + \mathbb{E}_{A, \sigmahat}[\mprob(\hat{u}_2^0 \in \mathcal{U}^c_\rho|A, \sigmahat)\cdot \mathbbm{1}_{\Omega^c}] \notag\\
&\leq \mathbb{E}_{A, \sigmahat}[\mprob(\hat{u}_2^0 \in \mathcal{U}^c_\rho|A, \sigmahat)\cdot \mathbbm{1}_{\Omega}] + \frac{1}{\mathrm{poly}(n)} \notag \\
& \leq \mathbb{E}_{A, \sigmahat}\left[\frac{\mprob(\hat{u}_2^0 \in \mathcal{U}^c_\rho|A, \sigmahat)}{\mprob(\hat{u}_2^0 \in \mathcal{U}_\tau|A, \sigmahat)}\cdot \mathbbm{1}_{\Omega}\right] + \frac{1}{\mathrm{poly}(n)}.
\end{align}
Let $\mu$ be the uniform measure on the unit sphere. By Lemma \ref{lemma:reg_event_PCA}, we have $\widehat{s}_{A^2}(v) = s_{A^2}(v)$ on $\Omega$, for every $v$. We can then write
\begin{align*}
\mathbb{E}_{A, \sigmahat}\left[\frac{\mathbb{P}(\hat{u}_2^0 \in \mathcal{U}_\rho^c|A, \sigmahat)}{\mathbb{P}(\hat{u}_2^0 \in \mathcal{U_\tau}|A, \sigmahat)}\cdot \mathbbm{1}_{\Omega}\right] &= \mathbb{E}_{A, \sigmahat}\left[\frac{\int_{\mathcal{U}^c_\rho}\mathrm{exp}(\frac{\epsilon}{6D^2}\widehat{s}'_{A^2}(v))d\mu}{\int_{\mathcal{U}_\tau}\mathrm{exp}(\frac{\epsilon}{6D^2}\widehat{s}'_{A^2}(v))d\mu}\cdot \mathbbm{1}_{\Omega}\right] \\
& = \mathbb{E}_{A, \sigmahat}\left[\frac{\int_{\mathcal{U}^c_\rho}\mathrm{exp}(\frac{\epsilon}{6D^2}(\widehat{s}_{A^2}(v) - \frac{\sigmahat^2}{n}v^TJ_nv))d\mu}{\int_{\mathcal{U}_\tau}\mathrm{exp}(\frac{\epsilon}{6D^2}(\widehat{s}_{A^2}(v) - \frac{\sigmahat^2}{n}v^TJ_nv))d\mu}\cdot \mathbbm{1}_{\Omega}\right]\\
&= \mathbb{E}_{A, \sigmahat}\left[\frac{\int_{\mathcal{U}^c_\rho}\mathrm{exp}(\frac{\epsilon}{6D^2}(s_{A^2}(v) - \frac{\sigmahat^2}{n}v^TJ_nv))d\mu}{\int_{\mathcal{U}_\tau}\mathrm{exp}(\frac{\epsilon}{6D^2}(s_{A^2}(v) - \frac{\sigmahat^2}{n}v^TJ_nv))d\mu}\cdot \mathbbm{1}_{\Omega}\right]\\
&= \mathbb{E}_{A, \sigmahat}\left[\frac{\int_{\mathcal{U}^c_\rho}\mathrm{exp}(\frac{\epsilon}{6D^2}(v^TA^2v - \frac{\sigmahat^2}{n}v^TJ_nv))d\mu}{\int_{\mathcal{U}_\tau}\mathrm{exp}(\frac{\epsilon}{6D^2}(v^TA^2v - \frac{\sigmahat^2}{n}v^TJ_nv))d\mu}\cdot \mathbbm{1}_{\Omega}\right]\\
&= \mathbb{E}_{A, \sigmahat}\left[\frac{\int_{\mathcal{U}^c_\rho}\mathrm{exp}(\frac{\epsilon}{6D^2}(v^TA^2v - \frac{\sigmahat^2}{n}v^TJ_nv + n^2))d\mu}{\int_{\mathcal{U}_\tau}\mathrm{exp}(\frac{\epsilon}{6D^2}(v^TA^2v - \frac{\sigmahat^2}{n}v^TJ_nv + n^2))d\mu}\cdot \mathbbm{1}_{\Omega}\right]\\
&= \mathbb{E}_{A, \sigmahat}\left[\frac{\int_{\mathcal{U}^c_\rho}\mathrm{exp}(\frac{\epsilon}{6D^2}(v^TA^2v - \frac{\sigmahat^2}{n}v^TJ_nv + n^2||v||_2^2))d\mu}{\int_{\mathcal{U}_\tau}\mathrm{exp}(\frac{\epsilon}{6D^2}(v^TA^2v - \frac{\sigmahat^2}{n}v^TJ_nv + n^2||v||_2^2))d\mu}\cdot \mathbbm{1}_{\Omega}\right]\\
&= \mathbb{E}_{A, \sigmahat}\left[\frac{\int_{\mathcal{U}^c_\rho}\mathrm{exp}(\frac{\epsilon}{6D^2}v^TMv)d\mu}{\int_{\mathcal{U}_\tau}\mathrm{exp}(\frac{\epsilon}{6D^2}v^TMv)d\mu}\cdot \mathbbm{1}_{\Omega}\right],
\end{align*}
where we used the independence of $(A, \sigmahat)$ from $\mu$. We also used the fact that, under $\mu$, the input $v$ satisfies $||v||_2 = 1$, and adding an extra $n^2$ term does not change the ratio of the two densities. Therefore, we have
\begin{align*}
\mathbb{E}_{A, \sigmahat}\left[\frac{\mprob(\hat{u}_2^0 \in \mathcal{U}^c_\rho|A, \sigmahat)}{\mprob(\hat{u}_2^0 \in \mathcal{U}_\tau|A, \sigmahat)}\cdot \mathbbm{1}_{\Omega}\right] &\leq \mathbb{E}_{A, \sigmahat}\left[\frac{\exp\left(\frac{\epsilon}{6D^2} \left(\rho^2 \sigma_1(M) + (1-\rho^2)\sigma_2(M)\right)\right) \mu(\mathcal{U}^c_\rho)}{\exp\left(\frac{\epsilon}{6D^2} \left(\tau^2 \sigma_1(M) + (1-\tau^2) \sigma_n(M)\right)\right) \mu(\mathcal{U}_\tau)}\cdot \mathbbm{1}_{\Omega}\right]\\
&\leq \mathbb{E}_{A, \sigmahat}\Bigg[\exp\left(\frac{-\epsilon}{6D^2} \left(\sigma^2 \sigma_1(M) - \rho^2 \sigma_1(M) - (1-\rho^2) \sigma_2(M)\right)\right) \\
& \qquad \qquad \cdot \exp\left(\frac{n-1}{2} \log\frac{2}{1-\tau}\right)\cdot \mathbbm{1}_{\Omega}\Bigg],
\end{align*}
where $\rho < \tau < 1$, to make the upper bound less than $1$. We also used Lemma $6$ in~\cite{ChaEtal13}. For the choice $1 - \tau^2 = \frac{1-\rho^2}{2} \left(1-\frac{\sigma_2(M)}{\sigma_1(M)}\right)$, we have 
\begin{align*}
&\mathbb{E}_{A, \sigmahat}\left[\frac{\mprob(\hat{u}_2^0 \in \mathcal{U}^c_\rho|A, \sigmahat)}{\mprob(\hat{u}_2^0 \in \mathcal{U}_\tau|A, \sigmahat)}\cdot \mathbbm{1}_{\Omega}\right]\\
&\leq \mathbb{E}_{A, \sigmahat}\Bigg[\exp\Bigg(\frac{-\epsilon}{12D^2} \cdot (1-\rho^2)(\sigma_1(M) - \sigma_2(M)) + \frac{n-1}{2} \log\left(\frac{8\sigma_1(M)}{(1-\rho^2)(\sigma_1(M) - \sigma_2(M))}\right)\Bigg)\cdot \mathbbm{1}_{\Omega}\Bigg].
\end{align*}

Now we analyze the eigenvalues/singular values of $M$. By Lemma \ref{lemma:reg_event_PCA}, we have on $\Omega$ that
\begin{align*}
||A^2 - P^2||_2 \leq (||A||_2 + ||P||_2)||A - P||_2 \leq (||A - P||_2 + 2\sigma_1(P))||A - P||_2 \precsim d\sqrt{d}.
\end{align*}
By Lemma \ref{lemma:l2.1SinSte}, we have
\begin{align*}
\sigma_2(M) & = \sigma_2\left(A^2 + n^2I_n -\frac{\sigmahat^2}{n}J_n\right) \\
& \leq \sigma_2\left(P^2 - \frac{\sigma^2_1(P)}{n}J_n\right) + n^2 + ||A^2 - P^2||_2 + |\sigmahat^2 - \sigma^2_1(P)|. 
\end{align*}
Also, 
\begin{align*}
\sigma_1^2(A) + n^2 + \sigmahat^2 \geq \sigma_1\left(M\right) \geq \sigma_1\left(P^2 - \frac{\sigma^2_1(P)}{n}J_n\right) + n^2 - ||A^2 - P^2||_2 - |\sigmahat^2 - \sigma^2_1(P)|.
\end{align*}
But since $\mathbf{1}_n$ is the top eigenvector of $P$, the eigenvalues/singular values of $P^2 - \frac{\sigma^2_1(P)}{n}J_n$ are $\sigma^2_2(P)$, and $n - 1$ zeroes. Hence, on $\Omega$, we have 
\begin{align*}
\sigma_1\left(M\right) - \sigma_2\left(M\right) &\geq \sigma_2^2(P) - 2||A^2 - P^2||_2 - 2|\sigmahat^2 - \sigma^2_1(P)| \succsim d^2 - d\sqrt{d} \succsim d^2,
\end{align*}
and 
\begin{align*}
\sigma_1(M) = \sigma_1\left(A^2 + n^2I_n -\frac{\sigmahat^2}{n}J_n\right) \precsim d^2 + n^2 + d^2 \precsim n^2.
\end{align*}
Hence, $\sigma_1(M) \leq n^2\log(n)$ and $\sigma_1(M) - \sigma_2(M) \geq \frac{d^2}{\log(n)}$, for $n$ large enough, on $\Omega$, by assumption. Thus, since $d > \sqrt{n}$, we obtain 
\begin{align*}
& \mathbb{E}_{A, \sigmahat}\left[\frac{\mprob(\hat{u}_2^0 \in \mathcal{U}^c_\rho|A, \sigmahat)}{\mprob(\hat{u}_2^0 \in \mathcal{U}_\tau|A, \sigmahat)}\cdot \mathbbm{1}_{\Omega}\right] \\
& \qquad \leq \mathbb{E}_{A, \sigmahat}\left[\exp\left(\frac{-\epsilon}{12\log(n)D^2} \cdot (1-\rho^2)d^2 + \frac{n-1}{2} \log\left(\frac{8n\log^2(n)}{1-\rho^2}\right)\right)\cdot \mathbbm{1}_{\Omega}\right].
\end{align*}
Next, since $\epsilon < nD^2$, we have
\begin{align*}
\log\left(\frac{8n\log^2(n)}{1 - \rho^2}\right) = \log\left(\frac{8n\log^2(n)\epsilon}{D^2}\right) < \log(8\log^2(n)n^2) < \frac{\log^3(n)}{12} = \frac{\epsilon\log^3(n)(1 - \rho^2)}{12D^2}.
\end{align*}
Since $d^2 > n\log^4(n)$, we obtain
\begin{align*}
(n - 1)\log\left(\frac{8n\log^2(n)}{1 - \rho^2}\right) < \frac{\epsilon d^2(1 - \rho^2)}{12\log(n)D^2},
\end{align*}
so since $1 - \rho^2 < \frac{1}{\log(n)}$, we have $\log(n) \leq \frac{n - 1}{2}\log\left(\frac{8n\log^2(n)}{1 - \rho^2}\right)$, implying that
\begin{align*}
\log(n) + \frac{n - 1}{2}\log\left(\frac{8n\log^2(n)}{1 - \rho^2}\right) < \frac{\epsilon d^2(1 - \rho^2)}{12\log(n)D^2}.
\end{align*}
Thus, 
\begin{align*}
\mathbb{E}_{A, \sigmahat}\left[\frac{\mprob(\hat{u}_2^0 \in \mathcal{U}^c_\rho|A, \sigmahat)}{\mprob(\hat{u}_2^0 \in \mathcal{U}_\tau|A, \sigmahat)}\cdot \mathbbm{1}_{\Omega}\right] \leq \frac{1}{n}.
\end{align*}
Using inequality~\eqref{eq:Urho^c}, we obtain $\mprob(|u_2^T\hat{u}_2^0|^2 > \rho^2) \geq 1 - \frac{1}{\mathrm{poly}(n)}$, as required.


\subsubsection{Proof of Lemma~\ref{lemma:Asigma_A_sigmahat}}
\label{AppAsigma}

Note that
\begin{equation*}
\|A_0 - A_1\|_2 = \left\|(\sigmahat_1^2 - \sigma_1^2) \frac{\textbf{1}_n\textbf{1}_n^T}{n}\right\|_2 = |\sigmahat_1^2 - \sigma_1^2|.
\end{equation*}
Combined with the Davis-Kahan Theorem (cf.\ Lemma \ref{lemma:thm3YuEtal_square}), we then have
\begin{align*}
\mathop{\inf}_{q: q^2 = 1}||qu_2 - u_2^*||_2 = \mathop{\inf}_{q: q^2 = 1}||qu_2 - u_2^*||_F \precsim \frac{|\sigmahat_1^2 - \sigma_1^2|}{\lambda_1(A_1) - \lambda_2(A_1)}.
\end{align*}
This completes the proof.


\subsubsection{Proof of Lemma~\ref{lemma:A_sigma_P_sigma}}
\label{AppASigmaP}

Note that, since $\sigma_1^2$ is the top eigenvalue of $P$, with associated eigenvector $\frac{\mathbf{1}_n}{\sqrt{n}}$, we have $\lambda_1(P_1) = \sigma_1(P_1) = \sigma_2^2(P)$. Since the rank of $P$ is $2$, we have $\lambda_2(P_1) = \sigma_2(P_1) = \sigma_3^2(P) = 0$. Hence, by Lemma \ref{lemma:thm3YuEtal_square}, we have
\begin{align*}
\mathop{\inf}_{q: q^2 = 1}||qu_2^* - u_{2, P}||_2 = \mathop{\inf}_{q: q^2 = 1}||qu_2^* - u_{2, P}||_F \precsim \frac{||A^2 - P^2||_2}{\sigma_2^2(P)},
\end{align*}
completing the proof.


\section{Proofs for Section~\ref{SecProj}}
\label{AppProj}

\subsection{Proofs for Section~\ref{section:Private Degree Truncation Blackbox}}
\label{AppTrunc}

We first show that $\hat{L}_{\epsilon, \delta}(G)$ is pure node DP and upper-bounds $LS_{T_D}(G) = \mathop{\max}\limits_{G' \sim_{node} G}d_{node}(T_D(G), T_D(G'))$, with high probability. We will use the following result:

\begin{lemma}[Claim $17$ in \cite{blocki2013differentially}]
\label{lemma:Efficient_T}
The maps $T_D : \mathcal{G}_n \rightarrow \mathcal{G}_{n, 2D}$ and $d_{T_D} : \mathcal{G}_n \rightarrow \mathbb{R}$ obtained from solving the linear program~\eqref{eq:Blocki_LP_Eq} satisfy:
\begin{enumerate}
    \item $d_{T_D}(G) = 0$, for every $G \in \mathcal{G}_{n, D}$,
    \item $d_{T_D}(G) \geq d_{node}(G, T_D(G))$, for every $G \in \mathcal{G}_n$,
    \item $|d_{T_D}(G) - d_{T_D}(G')| \leq 4$, for every $G' \sim_{node} G$, and $G', G \in \mathcal{G}_{n}$.
\end{enumerate}
Moreover, $T_D(G) = G$, for $G \in \mathcal{G}_{n, D}$.
\end{lemma}

\begin{lemma}
\label{lemma:Lhat}
Let $\delta \in (0, 1)$, and let $T_D$, $d_{T_D}$, and $\hat{L}_{\epsilon, \delta}$ be as in Definition~\ref{def:smooth_proj}. Then $\hat{L}_{\epsilon, \delta}$ is $(\epsilon, 0)$-node DP and
\begin{equation*}
\mathbb{P}(\hat{L}_{\epsilon, \delta}(G) \geq LS_{T_D}(G)) \geq 1 - \delta, \qquad \text{for every } G \in \mathcal{G}_n.
\end{equation*}
\end{lemma}

\begin{proof}
Note that since $|d_{T_D}(G) - d_{T_D}(G')| \leq 4$ by Lemma~\ref{lemma:Efficient_T}, for every $G' \sim_{node} G$, and $G', G \in \mathcal{G}_n$, the global sensitivity of $d_{T_D}$ is bounded above by $4$. Hence, applying Laplace noise with parameter $\frac{8}{\epsilon}$ guarantees $(\epsilon, 0)$-node DP, and with post-processing, $\hat{L}_{\epsilon, \delta}$ clearly satisfies $(\epsilon, 0)$-node DP.

Next, we prove a deterministic upper bound on $LS_{T_D}(G) = \mathop{\max}\limits_{G' \sim_{node} G}d_{node}(T_D(G), T_D(G'))$. By the triangle inequality, for every $G' \sim_{node} G$, we have
\begin{align*}
d_{node}(T_D(G), T_D(G')) &\leq d_{node}(G, T_D(G)) + d_{node}(G, G') + d_{node}(G', T_D(G'))\\
&\leq d_{node}(G, T_D(G)) + 1 + d_{node}(G', T_D(G')) \\
& \leq d_{T_D}(G) + 1 + d_{T_D}(G') \leq 5 + 2d_{T_D}(G),
\end{align*}
where we used the fact that $d_{T_D}(G) \geq d_{node}(G, T_D(G))$ and $|d_{T_D}(G) - d_{T_D}(G')| \leq 4$, for every $G' \sim_{node} G$. Thus, we have $LS_{T_D}(G) \leq 5 + 2d_{T_D}(G)$ for every $G \in \mathcal{G}_n$. In particular, this implies that
\begin{align*}
\mathbb{P}(\hat{L}_{\epsilon, \delta}(G) < LS_{T_D}(G)) &\leq \mathbb{P}\left(5 + 2 d_{T_D}(G) + \mathrm{Lap}\left(\frac{8}{\epsilon}\right) + \frac{8 \log(1/\delta)}{\epsilon} < LS_{T_D}(G)\right)\\
&\leq  \mathbb{P}\left(\mathrm{Lap}\left(\frac{8}{\epsilon}\right) < -\frac{8\log(1/\delta)}{\epsilon}\right) \leq \delta,
\end{align*}
as required.
\end{proof}


\subsubsection{Proof of Theorem~\ref{theorem:generic_red}}
\label{AppThemGenRed}

Fix $G \sim_{node} G'$, with $G, G' \in \mathcal{G}_n$. To simplify notation, let $\mathcal{A}^{(G)} := \mathcal{A}^{(\epsilon_2'(G), \delta_2'(G))}$. Define
\begin{align*}
A_G & := (\hat{L}_{\epsilon_1, \delta_1}(G), \mathcal{A}^{(G)}(T_D(G))), \\
A_{G'} & := (\hat{L}_{\epsilon_1, \delta_1}(G'), \mathcal{A}^{(G')}(T_D(G'))), \\
A_m & := (\hat{L}_{\epsilon_1, \delta_1}(G), \mathcal{A}^{(G)}(T_D(G'))).
\end{align*}
Let $P_G$, $P_{G'}$, and $P_m$ be the distributions of $A_G$, $A_{G'}$, and $A_m$, respectively. 
Let $E$ be an arbitrary measurable set in the image of $(\hat{L}_{\epsilon_1, \delta_1}(\cdot), \mathcal{A}^{(\cdot)}(T_D(\cdot)))$. By the privacy of $\hat{L}_{\epsilon_1, \delta_1}$ we have
\begin{align*}
P_{G'}(E) = \mathbb{P}\left((\hat{L}_{\epsilon_1, \delta_1}(G'), \mathcal{A}^{(G')}(T_D(G'))) \in E\right) \leq e^{\epsilon_1}P_m(E). 
\end{align*}
Let $F = \left\{\hat{L}_{\epsilon_1, \delta_1}(G) \geq LS_{{T_D}}(G) \right\}$. We have from Lemma \ref{lemma:Lhat} that $\mathbb{P}(F^c) \leq \delta_1$. Let $y_{\epsilon_1, \delta_1}$ be the pdf of $Y_{\epsilon_1, \delta_1}(G) = 5 + 2 d_{T_D}(G) + \mathrm{Lap}\left(\frac{8}{\epsilon_1}\right) + \frac{8 \log(1/\delta_1)}{\epsilon_1}$, so $\hat{L}_{\epsilon_1, \delta_1}(G) = \max\left\{\frac{1}{2}, Y_{\epsilon_1, \delta_1}(G)\right\}$. If $LS_{T_D}(G) = 0$, then $T_D(G) = T_D(G')$, and
\begin{equation*}
P_m(E|F) = P_G(E|F) \leq e^{2\epsilon_2} P_G(E|F) + \frac{1}{\mathrm{poly}(n)}.
\end{equation*}
Now assume $LS_{T_D}(G) \geq 1$. Then
\begin{align*}
& P_{m}(E|F) = \frac{\mathbb{P}\left(\left\{(\hat{L}_{\epsilon_1, \delta_1}(G), \mathcal{A}^{(G)}(T_D(G')))\in E\right\} \bigcap F\right)}{\mathbb{P}(F)}\\
& \quad = \int_{LS_{{T_D}}(G)}^{\infty}\frac{\mathbb{P}\left((\max\{1/2, x\}, \mathcal{A}^{(\epsilon_2/\max\{1/2, x\}, \delta_2/\max\{1/2, x\})}(T_D(G'))) \in E|Y_{\epsilon_1, \delta_1}(G) = x\right)y_{\epsilon_1, \delta_1}(x)}{\mathbb{P}(F)}dx\\
& \quad = \int_{LS_{{T_D}}(G)}^{\infty}\frac{\mathbb{P}\left((x, \mathcal{A}^{(\epsilon_2/x, \delta_2/x)}(T_D(G'))) \in E\right)y_{\epsilon_1, \delta_1}(x)}{\mathbb{P}(F)}dx,
\end{align*}
where we dropped the conditioning because $Y_{\epsilon_1, \delta_1}(G)$ is independent of $\mathcal{A}^{(\epsilon_2/x, \delta_2/x)}(T_D(G'))$. Also, since $\max\left\{\frac{1}{2}, x\right\} \geq LS_{T_D}(G) \geq 1$, we must have $x \geq LS_{T_D}(G) \geq 1$. Now fix $x \geq LS_{{T_D}}(G)$. Since $2\epsilon_2 < \log(1/(\delta_2 \cdot \mathrm{poly}(n)))$ and $x \geq LS_{T_D}(G) \geq 1$, we have $\frac{2\epsilon_2}{x} < \log\left(\frac{x}{\delta_2}\right)$. Since $\mathcal{A}^{(\epsilon_2/x, \delta_2/x)}$ is $\left(\frac{2\epsilon_2}{x}, \frac{\delta_2}{x}\right)_{2D}$-node DP, it is clearly $\left(\frac{2\epsilon_2}{LS_{T_D}(G)}, \frac{\delta_2}{LS_{T_D}(G)}\right)_{2D}$-node DP. By group privacy of size $LS_{T_D}(G)$, we obtain $\left(2\epsilon_2, \delta_2e^{2\epsilon_2}\right)_{2D}$-node DP on groups of size $LS_{T_D}(G)$. Since $\delta_2 < e^{-2\epsilon_2}/\mathrm{poly}(n)$, this is $(2\epsilon_2, 1/\mathrm{poly}(n))_{2D}$-node DP. Now, when changing $G'$ to $G$, we change at most $LS_{T_D}(G)$ nodes to go from $T_D(G')$ to $T_D(G)$. Hence, we have
\begin{align*}
& P_{m}(E|F) \leq \int_{LS_{{T_D}}(G)}^{\infty}\frac{\left(e^{2\epsilon_2}\mathbb{P}\left((x, \mathcal{A}^{(\epsilon_2/x, \delta_2/x)}(T_D(G))) \in E\right) + 1/\mathrm{poly}(n)\right)y_{\epsilon_1, \delta_1}(x)}{\mathbb{P}(F)}dx\\
& \quad = e^{2\epsilon_2}\int_{LS_{{T_D}}(G)}^{\infty}\frac{\mathbb{P}\left((x, \mathcal{A}^{(\epsilon_2/x, \delta_2/x)}(T_D(G))) \in E\right)y_{\epsilon_1, \delta_1}(x)}{\mathbb{P}(F)}dx + \frac{1}{\mathrm{poly}(n)}\\
& \quad = e^{2\epsilon_2}\int_{LS_{{T_D}}(G)}^{\infty}\frac{\mathbb{P}\left((\max\{1/2, x\}, \mathcal{A}^{(\epsilon_2/\max\{1/2, x\}, \delta_2/\max\{1/2, x\})}(T_D(G))) \in E|Y_{\epsilon_1, \delta_1}(G) = x\right)y_{\epsilon_1, \delta_1}(x)}{\mathbb{P}(F)}dx\\
& \qquad + \frac{1}{\mathrm{poly}(n)}\\
&\quad = \frac{e^{2\epsilon_2}\mathbb{P}\left(\left\{(\hat{L}_{\epsilon_1, \delta_1}(G), \mathcal{A}^{(G)}(T_D(G)))\in E\right\} \cap F\right)}{\mathbb{P}(F)} + \frac{1}{\mathrm{poly}(n)} \\
& \quad = e^{2\epsilon_2}P_{G}(E|F) + \frac{1}{\mathrm{poly}(n)}.
\end{align*}
Note now that $P_{G}(F) = P_{m}(F)$, since in both cases, we have $\hat{L}_{\epsilon_1, \delta_1}(G)$, not $\hat{L}_{\epsilon_1, \delta_1}(G')$. Thus, we have
\begin{align*}
P_{m}(E \cap F) \leq e^{2\epsilon_2}P_{G}(E \cap F) + \frac{1}{\mathrm{poly}(n)}.
\end{align*}
Hence, we obtain
\begin{align*}
P_{G'}(E) &\leq e^{\epsilon_1}(P_m(E \cap F) + \delta_1) \leq e^{\epsilon_1}(e^{2\epsilon_2}P_G(E \cap F) + 1/\mathrm{poly}(n)) + e^{\epsilon_1}\delta_1\\
&\leq e^{\epsilon_1 + 2\epsilon_2}P_G(E) + e^{\epsilon_1}(1/\mathrm{poly}(n) + \delta_1).
\end{align*}
Thus, $(\hat{L}_{\epsilon_1, \delta_1}(G), \mathcal{A}^{(G)}(T_D(G)))$ is $(\epsilon_1 + 2\epsilon_2, e^{\epsilon_1}(1/\mathrm{poly}(n) + \delta_1))$-node DP, as required.

\subsubsection{Proof of Theorem~\ref{theorem:generic_red_pure}}
\label{AppThemGenRedPure}

Using the same notation and steps as in the proof of Theorem \ref{theorem:generic_red}, we have by the privacy of $\hat{L}_{\epsilon_1, \delta_1}$ that
\begin{align*}
P_{G'}(E) = \mathbb{P}\left((\hat{L}_{\epsilon_1, \delta_1}(G'), \mathcal{A}^{\epsilon_2'(G')}(T_D(G'))) \in E\right) \leq e^{\epsilon_1}P_m(E). 
\end{align*}
Let $F = \left\{\hat{L}_{\epsilon_1, \delta_1}(G) \geq LS_{{T_D}}(G)\right\}$. Let $y_{\epsilon_1, \delta_1}$ be the pdf of
\begin{equation*}
Y_{\epsilon_1, \delta_1}(G) = 5 + 2 d_{T_D}(G) + \mathrm{Lap}\left(\frac{8}{\epsilon_1}\right) + \frac{8 \log(1/\delta_1)}{\epsilon_1}.
\end{equation*}
If $LS_{T_D}(G) = 0$, then $T_D(G) = T_D(G')$ and $P_m(E|F) = P_G(E|F) \leq e^{2\epsilon_2} P_G(E|F)$. Now assume $LS_{T_D}(G) \geq 1$. We have 
\begin{align*}
P_{m}(E|F) &= \frac{\mathbb{P}\left(\left\{(\hat{L}_{\epsilon_1, \delta_1}(G), \mathcal{A}^{\left(\epsilon_2'(G)\right)}(T_D(G')))\in E\right\} \cap F\right)}{\mathbb{P}(F)}\\
&= \int_{LS_{T_D}(G)}^{\infty}\frac{\mathbb{P}\left((\max\{1/2, x\}, \mathcal{A}^{(\epsilon_2/\max\{1/2, x\})}(T_D(G'))) \in E|Y_{\epsilon_1, \delta_1}(G) = x\right)y_{\epsilon_1, \delta_1}(x)}{\mathbb{P}(F)}dx\\
&= \int_{LS_{T_D}(G)}^{\infty}\frac{\mathbb{P}\left((x, \mathcal{A}^{(\epsilon_2/x)}(T_D(G'))) \in E\right)y_{\epsilon_1, \delta_1}(x)}{\mathbb{P}(F)}dx,     
\end{align*}
where we dropped the conditioning because $Y_{\epsilon_1, \delta_1}(G)$ is independent of $\mathcal{A}^{(\epsilon_2/x)}(T_D(G'))$. Also, since $\max\left\{\frac{1}{2}, x\right\} \geq LS_{T_D}(G) \geq 1$, we must have $x \geq LS_{T_D}(G) \geq 1$. Now, $\mathcal{A}^{(\epsilon_2/x)}(T_D(G'))$ is $\left(\frac{\epsilon_2}{x}, 0\right)_{2D}$-node DP. Hence, since $x \geq LS_{T_D}(G)$, we know that $\mathcal{A}^{(\epsilon_2/x)}(T_D(G'))$ is $\left(\frac{\epsilon_2}{LS_{T_D}(G)}, 0\right)_{2D}$-node DP. Next, when changing $G'$ to $G$, we change at most $LS_{T_D}(G)$ nodes to go from $T_D(G')$ to $T_D(G)$. Given that $\mathcal{A}^{(\epsilon_2/x)}$ is $\left(\frac{\epsilon_2}{LS_{T_D}(G)}, 0\right)_{2D}$-node DP, the privacy parameter changes to $\frac{LS_{T_D}(G)\epsilon_2}{LS_{T_D}(G)} = \epsilon_2$, using group privacy. Hence, we have
\begin{align*}
& P_{m}(E|F) \leq e^{\epsilon_2}\int_{LS_{T_D}(G)}^{\infty}\frac{\mathbb{P}\left((x, \mathcal{A}^{(\epsilon_2/x)}(T_D(G))) \in E\right)y_{\epsilon_1, \delta_1}(x)}{\mathbb{P}(F)}dx\\
& \quad = e^{\epsilon_2}\int_{LS_{T_D}(G)}^{\infty}\frac{\mathbb{P}\left((\max\{1/2, x\}, \mathcal{A}^{(\epsilon_2/\max\{1/2, x\})}(T_D(G))) \in E|Y_{\epsilon_1, \delta_1}(G) = x\right)y_{\epsilon_1, \delta_1}(x)}{\mathbb{P}(F)}dx\\
& \quad = \frac{e^{\epsilon_2}\mathbb{P}\left(\left\{(\hat{L}_{\epsilon_1, \delta_1}(G), \mathcal{A}^{\left(\epsilon_2'(G)\right)}(T_D(G)))\in E\right\} \cap F\right)}{\mathbb{P}(F)} = e^{\epsilon_2}P_{G}(E|F).
\end{align*}
Now note that $P_{G}(F) = P_{m}(F)$, since in both cases, we have $\hat{L}_{\epsilon_1, \delta_1}(G)$, not $\hat{L}_{\epsilon_1, \delta_1}(G')$. Thus, we have
\begin{align*}
P_{m}(E \cap F) \leq e^{\epsilon_2}P_{G}(E \cap F).
\end{align*}
Hence, using Lemma \ref{lemma:Lhat}, we obtain 
\begin{align*}
P_{G'}(E) \leq e^{\epsilon_1}(P_m(E \cap F) + \delta_1) \leq e^{\epsilon_1 + \epsilon_2}P_G(E \cap F) + e^{\epsilon_1}\delta_1 \leq e^{\epsilon_1 + \epsilon_2}P_G(E) + e^{\epsilon_1}\delta_1.
\end{align*}
Thus, $(\hat{L}_{\epsilon_1, \delta_1}(G), \mathcal{A}^{\left(\epsilon_2'(G)\right)}(T_D(G)))$ is $(\epsilon_1 + \epsilon_2, e^{\epsilon_1}\delta_1)$-node DP, as required.


\subsubsection{Proof of Proposition~\ref{prop:Boost_result}}
\label{AppPropBoost}

By basic composition (cf.\ Lemma~\ref{LemBasic}), if each $\widetilde{\theta}(G_j)$ satisfies $(\epsilon, \delta)$-node DP, the algorithm $\thetahat(G)$ satisfies $(T\epsilon, T\delta)$-node DP. Importantly, we also have the inequality
\begin{equation}
\label{EqnComErr}
\mprob\left(\widetilde{\mathcal{L}}(\thetahat(G), \theta) \le \xi T\right) \ge \mprob\left(\sum_{j=1}^T 1\left\{\widetilde{\mathcal{L}}(\widetilde{\theta}(G_j), \theta) \le \xi\right\} \ge \frac{T+1}{2}\right).
\end{equation}
Indeed, suppose at least $\frac{T+1}{2}$ subgraph assignments have misclassification rate bounded by $\xi$. WLOG, reindex the subgraphs so that the estimates $\widetilde{\theta}(G_1), \dots, \widetilde{\theta}(G_\ell)$ are all within $\widetilde{\mathcal{L}}$-distance $\xi$ of $\theta$, where $\ell = \frac{T+1}{2}$. We claim that this implies the aggregated estimator $\thetahat(G)$ has misclassification rate bounded by $\xi T$. Note that the boosting algorithm will certainly output an index $j^*$ satisfying the criterion, since any of the $\{G_j\}_{j=1}^\ell$ are within $\mathcal{L}$-distance $2\xi$ of all the rest, since $\mathcal{L} \le \widetilde{\mathcal{L}}$ and $\mathcal{L}$ satisfies the triangle inequality. Furthermore, there must exist $\ell^* \in [\ell]$ such that $\mathcal{L}(\widetilde{\theta}(G_{j^*}), \widetilde{\theta}(G_{\ell^*})) \le 2\xi$. In addition, since $\widetilde{\mathcal{L}}(\widetilde{\theta}(G_{\ell^*}), \theta) \leq \xi$ by assumption, we have $\mathcal{L}(\widetilde{\theta}(G_{j^*}), \theta) \leq 3\xi$. We now prove the following lemma:

\begin{lemma}
\label{LemAlign}
Suppose $\theta$ has all communities of size $m = \frac{n}{k}$. Let $\thetahat_1, \thetahat_2 \in [k]^n$, and define $\Theta, \Theta_1, \Theta_2 \in \mathbf{M}_{n \times k}$ accordingly. Suppose $\mathcal{L}(\thetahat_1, \theta) \le \eta$ and $\widetilde{\mathcal{L}}(\thetahat_2, \theta) \le \xi$. Let $J_i \in \arg\min_{J \in E_k} \|\Thetahat_i J - \Theta\|_0$ for $i \in \{1, 2\}$. If $\frac{\xi + 2\eta k + \xi k}{2} < 1$, then the unique minimizer of $\min_{J \in E_k} \|\Theta_2 J - \Theta_1\|_0$ is given by $J = J_2 J_1^{-1}$.
\end{lemma}

\begin{proof}
By assumption, after applying $J_1$, the community assignment $\thetahat_1$ has at most $\frac{\eta n}{2}$ misclassified nodes, whereas $\thetahat_2$ has at most $\frac{\xi m}{2}$ misclassified nodes in each true community. Thus, under the permutation $J_2 J_1^{-1}$, the disagreement between $\thetahat_1$ and $\thetahat_2$ is at most $\frac{(\eta + \xi)n}{2}$. Now consider any other permutation, which must misclassify at least one community relative to $\Theta_1$. The number of misclassified nodes in that community must be at least
\begin{equation*}
m - \frac{\xi m}{2} - \frac{\eta n}{2} = m\left(1 - \frac{\xi}{2} - \frac{\eta k}{2}\right).
\end{equation*}
We can check that
\begin{equation*}
m\left(1 - \frac{\xi}{2} - \frac{\eta k}{2}\right) > \frac{(\eta + \xi)n}{2}
\end{equation*}
under the stated condition, concluding the proof.
\end{proof}

Note that $\widetilde{\mathcal{L}}(\widetilde{\theta}(G_j), \theta) \le \xi$ for all $j \in [\ell]$ by assumption and $\mathcal{L}(\widetilde{\theta}(G_j^*), \theta) \le 3 \xi$. Applying Lemma~\ref{LemAlign} with $\eta = 3\xi$, we see that we indeed have the condition
\begin{equation*}
\frac{\xi + 2\eta k + \xi k}{2} = \frac{\xi(7k + 1)}{2} < 1,
\end{equation*}
assuming $\xi < \frac{1}{8k}$, implying that every permutation minimizing $\|\widetilde{\Theta}_j J - \widetilde{\Theta}_{j^*}\|_0$ agrees with the composition of the optimal alignment of $\widetilde{\Theta}_j$ to $\Theta$ and the inverse optimal alignment of $\widetilde{\Theta}_{j^*}$ to $\Theta$. Therefore, after a global relabeling, the majority vote over the good estimates may be analyzed as if each $\widetilde{\Theta}_j$ were optimally aligned directly to $\Theta$.


For each $j$, let $J_{\widetilde{\Theta}_j}$ denote the optimal alignment of $\widetilde{\Theta}_j$ to $\Theta$. Consider any fixed community. By assumption, for each $j \in [\ell]$, the matrix $\widetilde{\Theta}_j J_{\widetilde{\Theta}_j} - \Theta$ has at most $\xi n$ nonzero entries, of which at most $\xi m$ correspond to nodes in that community, where $m = \frac{n}{k}$. Thus, among $\{\widetilde{\Theta}_j\}_{j=1}^{\ell}$, at most $\ell(\xi m)$ nodes are misclassified by \emph{any} of the $\ell$ estimates. The remaining nodes in the community would certainly be classified correctly by a majority vote. Since this holds for all communities, we conclude that
\begin{equation*}
\widetilde{\mathcal{L}}(\thetahat(G), \theta) \le \ell \xi < \xi T,
\end{equation*}
establishing inequality~\eqref{EqnComErr}.


Finally, note that if $G \sim SBM(n, k, B, \theta)$, we have $G_j \sim SBM(n, k, \frac{B}{T}, \theta)$ for each $j$. On the other hand, if we define $X_j = 1\left\{\widetilde{\mathcal{L}}(\widetilde{\theta}(G_j), \theta) \le \xi\right\}$, the $X_j$'s may not be independent due to the fact that the thinned subgraphs all share the same underlying edges from $G$. Thus, we cannot apply a standard Chernoff bound (e.g., Lemma~\ref{lemma:Bern_Conc_Janson_2016}). Nonetheless, we may bound the variance of $\sum_{j=1}^T X_j$ and apply Chebyshev's inequality. Indeed, we can consider the maximal correlation $\rho_m$ (cf.\ Appendix~\ref{AppMaxCor}). For $\ell \neq m$, we have
\begin{align*}
\Cov(X_\ell, X_m) & \le \rho(X_\ell, X_m) \sqrt{\Var(X_\ell) \Var(X_m)} \le \frac{\rho(X_\ell, X_m)}{4} \le \frac{\rho_m(X_\ell, X_m)}{4} \\
& \le \max_{i \neq j} \frac{\rho_m(Y_{ij}, Y_{ij}')}{4},
\end{align*}
where for each $i \neq j$, we define $Y_{ij} = Z_{ij} R_{ij}$ and $Y_{ij} = Z_{ij} R'_{ij}$, where $Z_{ij} \sim Ber(B_{ij})$ and then we independently sample $R_{ij}, R_{ij}' \sim Ber\left(\frac{1}{T}\right)$. By Lemma~\ref{LemHGRBer}, we have
\begin{equation*}
\rho_m(Y_{ij}, Y_{ij}') = \frac{\frac{1}{T}(1 - B_{ij})}{1-\frac{B_{ij}}{T}}.
\end{equation*}
In particular, for $n$ sufficiently large, we have $\Cov(X_\ell, X_m) \le \frac{1}{2T}$. It follows that
\begin{equation*}
\Var\left(\frac{1}{T}\sum_{j=1}^T X_j\right) = \frac{1}{T^2}\left(\sum_{j=1}^T \Var(X_j) + \sum_{\ell \neq m} \Cov(X_\ell, X_m)\right) \le \frac{3}{4T},
\end{equation*}
so we can bound
\begin{equation*}
\mprob\left(\frac{1}{T} \sum_{i=1}^T X_i \le \frac{1}{2}\right) \le \frac{c}{T} = o(1),
\end{equation*}
as wanted.


\subsection{Proof of Theorem~\ref{theorem:Edge_Flip_Main_Thm}}
\label{AppEFMain}

To establish privacy, we first show that $\mathcal{A}_{EF}^{(\epsilon/(4D))}$ (i.e. Algorithm \ref{alg:EF_Spec_Clus} applied with the symmetric edge-flip mechanism $\mathcal{M}_{\frac{\epsilon}{4D}}$) is $(\epsilon, 0)_{2D}$–node DP.
Indeed, a result from Hehir et al.~\cite{HehEtal22} immediately implies that $\mathcal{M}_{\frac{\epsilon}{4D}}$ is $\left(\frac{\epsilon}{4D}, 0\right)$-edge DP. Note that \cite{HehEtal22} proves $\mathcal{M}_{\frac{\epsilon}{4D}}$ is $\frac{\epsilon}{4D}$-relationship DP (cf. Theorem $3.1$ in \cite{HehEtal22}), but this notion implies $\left(\frac{\epsilon}{4D}, 0\right)$-edge DP. 

By group privacy for pure $(\epsilon, 0)$-edge DP mechanisms, $\mathcal{M}_{\frac{\epsilon}{4D}}$ is $(\epsilon, 0)$-edge DP for $4D$ edges that can change. When changing one node in graphs belonging to $\mathcal{G}_{n, 2D}$, at most $4D$ edges can be flipped. The desired privacy guarantee then follows by post-processing. Combined with Theorem \ref{theorem:generic_red_pure}, we conclude that $\thetahat(G)$ is $(\epsilon_1 + \epsilon, e^{\epsilon_1}\delta_1)$-node DP, which is $(\epsilon_1 + \epsilon, O(\delta_1))$-node DP, since $\epsilon_1 \asymp 1$ and $\delta_1 \asymp \frac{1}{\mathrm{poly}(n)}$. This is the desired privacy guarantee.

For the utility guarantee, we begin by deriving a utility result with fixed privacy parameters:

\begin{lemma}[Adapted from \cite{HehEtal22}]
\label{lemma:Edge_Flip_Ut}
Let $B_0 \in (0, 1]^{k \times k}$ be invertible. Let $k \asymp 1$, $a_n \cdot\max(B_0) \geq \frac{\log(n)}{n}$, and $\epsilon \in (0, \infty)$. Let $\lambda_{B_0}$ be the smallest absolute non-zero eigenvalue of $B_0$. Let $G \in \mathcal{G}_{n}$ be sampled from an SBM with $k$ balanced communities and probability matrix  $a_n B_0$. Let $\theta = \{\theta_i\}_{i = 1}^n$ be the true community assignments, and let $\thetahat(G) = \mathcal{A}_{EF}^{(\epsilon)}(G)$ be the output of Algorithm \ref{alg:EF_Spec_Clus} with $\gamma \asymp 1$. Let 
\begin{align*}
g_{\epsilon} = \frac{e^{\epsilon} + 1}{e^{\epsilon} - 1}\left(a_n \cdot\max(B_0) + \frac{1}{e^{\epsilon} - 1}\right).
\end{align*}
Then there exist $c_1, C_1 \asymp 1$ such that if $\frac{(2 + \gamma)k^3ng_\epsilon}{n^2a_n ^2\lambda_{B_0}^2} < c_1^{-1}$, we have
\begin{align*}
\widetilde{\mathcal{L}}\left(\thetahat(G), \theta\right) \leq \frac{C_1g_\epsilon}{na_n^2},
\end{align*}
with probability at least $1 - \frac{1}{n}$.
\end{lemma}

\begin{proof}
With the given assumptions, we have the existence of $c_1, C_1 \asymp 1$, by Theorem 5.1 from \cite{HehEtal22}. Because $\frac{(2 + \gamma)k^3ng_\epsilon}{n^2a_n ^2\lambda_{B_0}^2} < c_1^{-1}$, we have again by Theorem 5.1 from \cite{HehEtal22} that
\begin{align*}
\widetilde{\mathcal{L}}\left(\mathcal{A}_{EF}^{(\epsilon)}(G), \theta\right) \leq \frac{c_1(2 + \gamma)k^3g_\epsilon}{na_n^2\lambda_{B_0}^2} \leq \frac{C_1g_\epsilon}{na_n^2},
\end{align*}
with probability at least $1 - \frac{1}{n}$, because $k, \gamma, c_1, \lambda_{B_0} \asymp 1$. This completes the proof.
\end{proof}

On the other hand, the privacy parameters used in the algorithm are \emph{random} quantities $\frac{\epsilon}{4D\hat{L}_{\epsilon_1, \delta_1}(G)}$. We first exploit the result from Lemma \ref{lemma:hp_removal_d} to guarantee that $d_{T_{D}}(G) = 0$, w.h.p., for $D \geq 3d$. Then $\hat{L}_{\epsilon_1, \delta_1}(G) = \max\left\{\frac{1}{2}, 5 + \mathrm{Lap}\left(\frac{8}{\epsilon_1}\right) + \frac{8\log(1/\delta_1)}{\epsilon_1}\right\}$, which, w.h.p., lies in an interval $I$ with bounds scaling as $\log(n)$. Finally, we will apply Lemma~\ref{lemma:Edge_Flip_Ut} with $\frac{\epsilon}{4Dx}$, for $x \in I$, to obtain the desired guarantee.

Let $X_{\epsilon_1, \delta_1} := 5 + \mathrm{Lap}\left(\frac{8}{\epsilon_1}\right) + \frac{8\log(1/\delta_1)}{\epsilon_1}$, $I := \left[5 + \frac{4\log(1/\delta_1)}{\epsilon_1}, 5 + \frac{16\log(1/\delta_1)}{\epsilon_1}\right]$, and $I_{max} := 5 + \frac{16\log(1/\delta_1)}{\epsilon_1}$. For $F_1 = \left\{X_{\epsilon_1, \delta_1} \in I\right\}$, we have
\begin{align*}
 \mathbb{P}\left(F_1^c\right) \leq \mathbb{P}\left(\mathrm{Lap}\left(\frac{8}{\epsilon_1}\right) < -\frac{4\log(1/\delta_1)}{\epsilon_1}\right) + \mathbb{P}\left(\mathrm{Lap}\left(\frac{8}{\epsilon_1}\right) > \frac{8\log(1/\delta_1)}{\epsilon_1}\right) \leq 2\sqrt{\delta_1}.
\end{align*}
Let
\begin{align*}
g_{\frac{\epsilon}{4D}} = \frac{e^{\frac{\epsilon}{4D}} + 1}{e^{\frac{\epsilon}{4D}} - 1}\left(a_n \cdot\max(B_0) + \frac{1}{e^{\frac{\epsilon}{4D}} - 1}\right).
\end{align*}
Let $c_1$ and $C_1$ be as in Lemma \ref{lemma:Edge_Flip_Ut}. Let $f_{X_{\epsilon_1, \delta_1}}$ be the pdf of $X_{\epsilon_1, \delta_1}$, and let $\Omega_0$ be as in the proof of Lemma \ref{lemma:hp_removal_d}. We have $\mathbb{P}(\Omega_0) \geq 1 - \frac{1}{\mathrm{poly}(n)}$. In addition, on $\Omega_0$, we do not have to truncate the graph, since the degree of every node is less than $2d$, so $T_{D}(G) = G$ and $d_{T_{D}}(G) = 0$. Hence, since $X_{\epsilon_1, \delta_1} \geq 5 > \frac{1}{2}$ on $F_1$, we have
\begin{align*}
&\mathbb{P}\left(\widetilde{\mathcal{L}}\left(\mathcal{A}_{EF}^{\left(\epsilon/(4D\hat{L}_{\epsilon_1, \delta_1}(G))\right)}(T_{D}(G)), \theta\right) \leq \frac{C_1g_{\frac{\epsilon}{4D\cdot I_{max}}}}{na_n^2}\right)\\
& \quad \geq \mathbb{P}\left(\left\{\widetilde{\mathcal{L}}\left(\mathcal{A}_{EF}^{\left(\epsilon/(4D\hat{L}_{\epsilon_1, \delta_1}(G))\right)}(T_{D}(G)), \theta\right) \leq \frac{C_1g_{\frac{\epsilon}{4D\cdot I_{max}}}}{na_n^2}\right\} \bigcap \Omega_0\right)\\
& \quad = \mathbb{P}\left(\left\{\widetilde{\mathcal{L}}\left(\mathcal{A}_{EF}^{\left(\epsilon/(4D\max\{1/2, X_{\epsilon_1, \delta_1}\})\right)}(T_{D}(G)), \theta\right) \leq \frac{C_1g_{\frac{\epsilon}{4D\cdot I_{max}}}}{na_n^2}\right\} \bigcap \Omega_0\right) \\
& \quad \geq \mathbb{P}\left(\left\{\widetilde{\mathcal{L}}\left(\mathcal{A}_{EF}^{\left(\epsilon/(4DX_{\epsilon_1, \delta_1})\right)}(G), \theta\right) \leq \frac{C_1g_{\frac{\epsilon}{4D\cdot I_{max}}}}{na_n^2}\right\} \bigcap \Omega_0 \cap F_1\right)\\
& \quad \geq \int_I\mathbb{P}\left(\widetilde{\mathcal{L}}\left(\mathcal{A}_{EF}^{\left(\epsilon/(4DX_{\epsilon_1, \delta_1})\right)}(G), \theta\right) \leq \frac{C_1g_{\frac{\epsilon}{4D\cdot I_{max}}}}{na_n^2}\middle\vert X_{\epsilon_1, \delta_1} = x\right)f_{X_{\epsilon_1, \delta_1}}(x)dx - \frac{1}{\mathrm{poly}(n)}\\
& \quad = \int_I\mathbb{P}\left(\widetilde{\mathcal{L}}\left(\mathcal{A}_{EF}^{\left(\epsilon/(4Dx)\right)}(G), \theta\right) \leq \frac{C_1g_{\frac{\epsilon}{4D\cdot I_{max}}}}{na_n^2}\right)f_{X_{\epsilon_1, \delta_1}}(x)dx - \frac{1}{\mathrm{poly}(n)}\\
& \quad \geq \int_I\mathbb{P}\left(\widetilde{\mathcal{L}}\left(\mathcal{A}_{EF}^{\left(\epsilon/(4Dx)\right)}(G), \theta\right) \leq \frac{C_1g_{\frac{\epsilon}{4Dx}}}{na_n^2}\right)f_{X_{\epsilon_1, \delta_1}}(x)dx - \frac{1}{\mathrm{poly}(n)},
\end{align*}
since the function $g_{\frac{\epsilon}{4Dx}}$ is increasing in $x > 0$. We also dropped the conditioning on $\left\{X_{\epsilon_1, \delta_1} = x\right\}$, since the randomness in $X_{\epsilon_1, \delta_1}$ only depends on the $\mathrm{Lap}\left(\frac{8}{\epsilon_1}\right)$ noise, and this is independent of $\widetilde{\mathcal{L}}\left(\mathcal{A}_{EF}^{(\epsilon/(4Dx))}(G), \theta\right)$. Assume $\frac{(2 + \gamma)k^3ng_{\frac{\epsilon}{4D\cdot I_{max}}}}{n^2a_n^2\lambda_{B_0}^2} < c_1^{-1}$ (we will check this holds in both cases of the theorem hypothesis). Since $g_{\frac{\epsilon}{4Dx}}$ is increasing in $x > 0$, we have $\frac{(2 + \gamma)k^3ng_{\frac{\epsilon}{4Dx}}}{n^2a_n^2\lambda_{B_0}^2} < c_1^{-1}$. Hence, by Lemma \ref{lemma:Edge_Flip_Ut}, we have 
\begin{align*}
\widetilde{\mathcal{L}}\left(\mathcal{A}_{EF}^{(\epsilon/(4Dx))}(G), \theta\right) \leq \frac{C_1g_{\frac{\epsilon}{4Dx}}}{na_n^2},
\end{align*}
with probability at least $1 - \frac{1}{n}$. Thus, we have
\begin{align*}
\mathbb{P}\left(\widetilde{\mathcal{L}}\left(\thetahat(G), \theta\right) \leq \frac{C_1g_{\frac{\epsilon}{4D\cdot I_{max}}}}{na_n^2}\right) &\geq \int_I\left(1 - \frac{1}{n}\right)f_{X_{\epsilon_1, \delta_1}}(x)dx - \frac{1}{\mathrm{poly}(n)}\\
&\geq 1 - \frac{1}{n} - \mathbb{P}(F_1^c) - \frac{1}{\mathrm{poly}(n)}\\
&\geq 1 - 2\sqrt{\delta_1} - \frac{1}{\mathrm{poly}(n)} \geq 1 - \frac{1}{\mathrm{poly}(n)},
\end{align*}
since $\delta_1 \asymp \frac{1}{\mathrm{poly}(n)}$.
Now we split the analysis into two cases, and we check that $\frac{(2 + \gamma)k^3ng_{\frac{\epsilon}{4D\cdot I_{max}}}}{n^2a_n^2\lambda_{B_0}^2} < c_1^{-1}$ in each case.

Take $\frac{D}{\sqrt{\log(n)}} < \epsilon < D$, and $d > \sqrt{n}\log^2(n)$. Since $\delta_1 \asymp \frac{1}{\mathrm{poly}(n)}$ and $\epsilon_1 \asymp 1$, we have $I_{max} \asymp\log(n)$. Then $\frac{\epsilon}{DI_{max}} < \frac{1}{I_{max}} = o(1)$, and 
\begin{align*}
g_{\frac{\epsilon}{4D\cdot I_{max}}} = \frac{e^{\frac{\epsilon}{4D\cdot I_{max}}} + 1}{e^{\frac{\epsilon}{4D\cdot I_{max}}} - 1}\left(a_n\cdot\max(B_0) + \frac{1}{e^{\frac{\epsilon}{4D\cdot I_{max}}} - 1}\right) \precsim \frac{D\log(n)}{\epsilon}\left(a_n + \frac{D\log(n)}{\epsilon}\right).
\end{align*}
Therefore, since $\epsilon < D \leq n < n\log(n)$, we have
\begin{align*}
\frac{(2 + \gamma)k^3ng_{\frac{\epsilon}{4D\cdot I_{max}}}}{n^2a_n^2\lambda_{B_0}^2} & \asymp \frac{g_{\frac{\epsilon}{4D\cdot I_{max}}}}{na_n^2} \precsim \frac{D\log(n)}{d\epsilon} + \frac{nD^2\log^2(n)}{d^2\epsilon^2} \\
& \precsim \frac{nD^2\log^2(n)}{d^2\epsilon^2} < \frac{D^2}{\log^2(n)\epsilon^2} < \frac{1}{\log(n)} < c_1^{-1},
\end{align*}
for $n$ large enough. So $\widetilde{\mathcal{L}}\left(\thetahat(G),  \theta\right) \precsim \frac{nD^2\log^2(n)}{d^2\epsilon^2} < \frac{1}{\log(n)}$, with probability at least $1 - \frac{1}{\mathrm{poly}(n)}$.

Now let $C'$ be such that $I_{max} \leq \frac{C'\log(n)}{4}$, and assume $C'D\log^2(n) < \epsilon$. Then $g_{\frac{\epsilon}{4D\cdot I_{max}}} \precsim a_n + e^{-\frac{\epsilon}{4D\cdot I_{max}}}$, so
\begin{equation*}
\frac{(2 + \gamma)k^3ng_{\frac{\epsilon}{4D\cdot I_{max}}}}{n^2a_n^2\lambda_{B_0}^2} \asymp \frac{g_{\frac{\epsilon}{4D\cdot I_{max}}}}{na_n^2} \precsim \frac{1}{d} + \frac{n}{d^2e^{\epsilon/(C'D\log(n))}} < \frac{1}{d} + \frac{1}{d^2} \precsim \frac{1}{d} < c_1^{-1},
\end{equation*}
for $n$ large enough, since $d \succsim \log(n)$. So $\widetilde{\mathcal{L}}\left(\thetahat(G), \theta\right) \precsim \frac{1}{d} \precsim \frac{1}{\log(n)}$, with probability at least $1 - \frac{1}{\mathrm{poly}(n)}$. This completes the proof.


\subsection{Proofs for Section~\ref{SecLR}}
\label{AppLR}

\subsubsection{Proof of Theorem~\ref{theorem:Matrix_Mech_Main_Thm}}
\label{AppMatrixMech}



To establish privacy, we begin with the following result:
\begin{lemma}
\label{lemma:d_tight_priv}
Let $2\epsilon < \log(1/\delta)$. Then Algorithm \ref{alg:Matrix_Estimation_hardt2014} with privacy parameters $\left(\frac{\epsilon}{4D}, \delta\right)$ is $(2\epsilon, \delta)_{2D}$-node DP. 
\end{lemma}

\begin{proof}
By Lemma \ref{lemma:hardt2014_conc_bound}, we know that $\hat{A}_{(2k)}$ (using privacy parameters $\left(\frac{\epsilon}{4D}, \delta\right)$) satisfies $\frac{\epsilon^2}{64D\log(1/\delta)}$-edge zCDP. By postprocessing, Algorithm \ref{alg:Matrix_Estimation_hardt2014} satisfies $\frac{\epsilon^2}{64D\log(1/\delta)}$-edge zCDP. When changing one node in graphs belonging to $\mathcal{G}_{n, 2D}$, at most $4D$ edges can be flipped. Hence, by group privacy (cf.\ Lemma \ref{lemma:Group_zCDP}), Algorithm \ref{alg:Matrix_Estimation_hardt2014} satisfies $\left(\frac{\epsilon^2}{4\log(1/\delta)}\right)_{2D}$-zCDP. Finally, by Lemma \ref{lemma:zCDP_to_DP}, we obtain $\left(\frac{\epsilon^2}{4\log(1/\delta)} + \epsilon, \delta\right)_{2D}$-DP. Since $\delta < e^{-2\epsilon}$, this is also $(9\epsilon/8, \delta)_{2D}$-DP, and the conclusion follows.
\end{proof}

To see the benefit of zCDP, suppose we used Lemma \ref{lemma:hardt2014_conc_bound} and we converted it immediately to $(\epsilon, \delta)$-DP. We could use group privacy with parameters $\left(\frac{\epsilon}{4D}, \frac{\delta}{4D}\right)$ after, to obtain a final guarantee of $(\epsilon, \delta e^{\epsilon})$-DP, provided $2\epsilon < \log(1/\delta)$. Hence, zCDP presents an improvement, especially when $\epsilon$ grows quickly with $n$.

Thus, by Lemma \ref{lemma:d_tight_priv}, we know that if $\frac{2\epsilon}{T} < \log(1/\delta)$, then $\mathcal{A}_{ME}^{(\epsilon/(4TD), \delta)}$ is $(2\epsilon/T, \delta)_{2D}$-node DP. Since $\frac{2\epsilon}{T} < \log(1/(\delta\cdot\mathrm{poly}(n)))$, Theorem \ref{theorem:generic_red} implies that $\widetilde{\theta}(G_j) = \mathcal{A}_{ME}^{\left(\epsilon/(4TD\hat{L}_{\epsilon_1, \delta_1}(G_j)), \delta/\hat{L}_{\epsilon_1, \delta_1}(G_j)\right)}(T_{D}(G_j))$ is $(\epsilon_1 + 2\epsilon/T, 1/\mathrm{poly}(n))$-node DP. We also used the fact that $\epsilon_1 \asymp 1$ and $\delta_1 \asymp \frac{1}{\mathrm{poly}(n)}$. Since $T \asymp \log(n)$, we know that $\epsilon/T$ grows with $n$. Since $\epsilon_1 \asymp 1$, we know that $\epsilon_1 \leq \epsilon/T$, for $n$ large enough, and thus, $\widetilde{\theta}(G_j)$ is $(3\epsilon/T, 1/\mathrm{poly}(n))$-node DP. Finally, by Proposition \ref{prop:Boost_result}, we know that $\thetahat(G)$ is $(3\epsilon, 1/\mathrm{poly}(n))$-node DP.

For the utility guarantee, let us start by focusing on $G_1 \sim SBM\left(n, k, \frac{B}{T}\right)$. We will then use Proposition \ref{prop:Boost_result} to boost the accuracy. Let $X_{\epsilon_1, \delta_1} := 5 + \mathrm{Lap}\left(\frac{8}{\epsilon_1}\right) + \frac{8\log(1/\delta_1)}{\epsilon_1}$, $I := \left[5 + \frac{4\log(1/\delta_1)}{\epsilon_1}, 5 + \frac{16\log(1/\delta_1)}{\epsilon_1}\right]$, and $I_{max} := 5 + \frac{16\log(1/\delta_1)}{\epsilon_1}$. Then, for $F_1 = \left\{X_{\epsilon_1, \delta_1} \in I\right\}$, we have
\begin{align*}
 \mathbb{P}\left(F_1^c\right) \leq \mathbb{P}\left(\mathrm{Lap}\left(\frac{8}{\epsilon_1}\right) < -\frac{4\log(1/\delta_1)}{\epsilon_1}\right) + \mathbb{P}\left(\mathrm{Lap}\left(\frac{8}{\epsilon_1}\right) > \frac{8\log(1/\delta_1)}{\epsilon_1}\right) \leq 2\sqrt{\delta_1}.
\end{align*}

Let
\begin{align*}
h_{\epsilon, \delta} = \frac{T}{d} + \frac{nT^2\log^2(n)\log(1/\delta)}{d^2\epsilon^2} + \frac{n^2T^4\log^4(n)\log^2(1/\delta)}{d^4\epsilon^4}.
\end{align*}
Since we are working with $G_1 \sim SBM\left(n, k, \frac{B}{T}\right)$, we need to scale $d$ accordingly, as well.

We first derive the following result, proved in Appendix~\ref{AppInitConc}, which bounds the accuracy of the algorithm with fixed privacy parameters $(\epsilon, \delta)$:

\begin{lemma}
\label{lemma:Initial_Conc_bd_matrix_estim}
Let $G \sim SBM(n, k, B, \theta)$, with parameters satisfying Assumption~\ref{AssSBM}. Let $\theta = \{\theta_i\}_{i = 1}^n$ be the true community assignments, and let
$\thetahat(G) = \mathcal{A}_{ME}^{(\epsilon, \delta)}(G)$ be the output of Algorithm \ref{alg:Matrix_Estimation_hardt2014} with $\gamma \asymp 1$. Then there exists $c_1 \asymp 1$ such that if
\begin{equation*}
(16 + 8\gamma)\left(\frac{1}{d} + \frac{n\log^2(n)\log(1/\delta)}{d^2\epsilon^2} + \frac{n^2\log^4(n)\log^2(1/\delta)}{d^4\epsilon^4}\right) < c_1,
\end{equation*}
we have
\begin{align*}
\widetilde{\mathcal{L}}\left(\thetahat(G), \theta\right) \precsim \frac{1}{d} + \frac{n\log^2(n)\log(1/\delta)}{d^2\epsilon^2} + \frac{n^2\log^4(n)\log^2(1/\delta)}{d^4\epsilon^4},
\end{align*}
with probability at least $0.9 - \frac{1}{\mathrm{poly}(n)}$.
\end{lemma}

In order to handle the random privacy parameters $\left(\frac{\epsilon}{4TD\hat{L}_{\epsilon_1, \delta_1}(G)}, \frac{\delta}{\hat{L}_{\epsilon_1, \delta_1}(G)}\right)$, we exploit the result from Lemma \ref{lemma:hp_removal_d} to guarantee that w.h.p., for $D \geq \frac{3d}{T}$, we have $d_{T_{D}}(G_1) = 0$. Then note that $\hat{L}_{\epsilon_1, \delta_1}(G_1) = \max\left\{\frac{1}{2}, 5 + \mathrm{Lap}\left(\frac{8}{\epsilon_1}\right) + \frac{8\log(1/\delta_1)}{\epsilon_1}\right\}$, and w.h.p., we concentrate this in an interval $I$ with bounds scaling as $\log(n)$. The utility based on $G_1$ follows from Lemma~\ref{lemma:Initial_Conc_bd_matrix_estim} with $\left(\frac{\epsilon}{4TDx}, \frac{\delta}{x}\right)$, for $x \in I$. The final utility guarantee for the boosted version applied to $G \sim SBM(n, k, B, \theta)$, follows from Proposition \ref{prop:Boost_result}.

To this end, let $c_1$ and $C_1$ be as in Lemma \ref{lemma:Initial_Conc_bd_matrix_estim}. Next, let $f_{X_{\epsilon_1, \delta_1}}$ be the pdf of $X_{\epsilon_1, \delta_1}$. Also, let $\Omega_0$ be as in the proof of Lemma \ref{lemma:hp_removal_d} applied to $\frac{a_n}{T}$. The guarantee on this event holds because we are under Assumption \ref{AssSBM} applied to $\frac{a_n}{T}$. We have $\mathbb{P}(\Omega_0) \geq 1 - \frac{1}{\mathrm{poly}(n)}$. In addition, on $\Omega_0$, we do not have to truncate the graph, since the degree of every node is less than $\frac{2d}{T}$. So on $\Omega_0$, we have $T_{D}(G_1) = G_1$ and $d_{T_{D}}(G_1) = 0$. Hence, since $X_{\epsilon_1, \delta_1} \geq 5 > \frac{1}{2}$ on $F_1$, we have
\begin{align*}
&\mathbb{P}\left(\widetilde{\mathcal{L}}\left(\widetilde{\theta}(G_1), \theta\right) \leq C_1h_{\frac{\epsilon}{4TDI_{max}}, \frac{\delta}{I_{max}}}\right)\\
&\geq \mathbb{P}\left(\left\{\widetilde{\mathcal{L}}\left(\widetilde{\theta}(G_1), \theta\right) \leq C_1h_{\frac{\epsilon}{4TDI_{max}}, \frac{\delta}{I_{max}}}\right\} \cap \Omega_0\right)\\
&\geq \mathbb{P}\left(\left\{\widetilde{\mathcal{L}}\left(\mathcal{A}_{ME}^{\left(\epsilon/(4TD\max\{1/2, X_{\epsilon_1, \delta_1}\}), \delta/\max\{1/2, X_{\epsilon_1, \delta_1}\}\right)}(G_1), \theta\right) \leq C_1h_{\frac{\epsilon}{4TDI_{max}}, \frac{\delta}{I_{max}}}\right\} \cap \Omega_0 \cap F_1\right)\\
&= \mathbb{P}\left(\left\{\widetilde{\mathcal{L}}\left(\mathcal{A}_{ME}^{\left(\epsilon/(4TDX_{\epsilon_1, \delta_1}), \delta/X_{\epsilon_1, \delta_1}\right)}(G_1), \theta\right) \leq C_1h_{\frac{\epsilon}{4TDI_{max}}, \frac{\delta}{I_{max}}}\right\} \cap \Omega_0 \cap F_1\right)\\
&\geq \int_I\mathbb{P}\left(\widetilde{\mathcal{L}}\left(\mathcal{A}_{ME}^{\left(\epsilon/(4TDX_{\epsilon_1, \delta_1}), \delta/X_{\epsilon_1, \delta_1}\right)}(G_1), \theta\right) \leq C_1h_{\frac{\epsilon}{4TDI_{max}}, \frac{\delta}{I_{max}}}\middle\vert X_{\epsilon_1, \delta_1} = x\right)f_{X_{\epsilon_1, \delta_1}}(x)dx \\
& \qquad \qquad - \frac{1}{\mathrm{poly}(n)}\\
&= \int_I\mathbb{P}\left(\widetilde{\mathcal{L}}\left(\mathcal{A}_{ME}^{\left(\epsilon(4TDx), \delta/x\right)}(G_1), \theta\right) \leq C_1h_{\frac{\epsilon}{4TDI_{max}}, \frac{\delta}{I_{max}}}\right)f_{X_{\epsilon_1, \delta_1}}(x)dx - \frac{1}{\mathrm{poly}(n)}\\
&\geq \int_I\mathbb{P}\left(\widetilde{\mathcal{L}}\left(\mathcal{A}_{ME}^{\left(\epsilon/(4TDx), \delta/x\right)}(G_1), \theta\right) \leq C_1h_{\frac{\epsilon}{4TDx}, \frac{\delta}{x}}\right)f_{X_{\epsilon_1, \delta_1}}(x)dx - \frac{1}{\mathrm{poly}(n)},
\end{align*}
since the function $h_{\frac{\epsilon}{4TDx}, \frac{\delta}{x}}$ is increasing in $x > 0$. We also dropped the conditioning on $\left\{X_{\epsilon_1, \delta_1} = x\right\}$, since the randomness in $X_{\epsilon_1, \delta_1}$ only depends on the $\mathrm{Lap}\left(\frac{8}{\epsilon_1}\right)$ noise, and this is independent of $\widetilde{\mathcal{L}}\left(\mathcal{A}_{ME}^{(\epsilon/(4TDx), \delta/x}(G_1), \theta\right)$. Now, assume $(16 + 8\gamma)h_{\frac{\epsilon}{4TDI_{max}}, \frac{\delta}{I_{max}}} < c_1$ (we will later check that this holds). Since $h_{\frac{\epsilon}{4TDx}, \frac{\delta}{x}}$ is increasing in $x > 0$, we have $(16 + 8\gamma)h_{\frac{\epsilon}{4TDx}, \frac{\delta}{x}} < c_1$. By Lemma \ref{lemma:Initial_Conc_bd_matrix_estim}, we obtain
\begin{align*}
\widetilde{\mathcal{L}}\left(\mathcal{A}_{ME}^{\left(\epsilon/(4TDx), \delta/x\right)}(G_1), \theta\right) \leq C_1h_{\frac{\epsilon}{4TDx}, \frac{\delta}{x}},
\end{align*}
with probability at least $0.9 -\frac{1}{\mathrm{poly}(n)}$. Thus, we have
\begin{align*}
\mathbb{P}\left(\widetilde{\mathcal{L}}\left(\widetilde{\theta}(G_1), \theta\right) \leq C_1h_{\frac{\epsilon}{4TDI_{max}}, \frac{\delta}{I_{max}}}\right) &\geq \int_I\left(0.9 - \frac{1}{\mathrm{poly}(n)}\right)f_{X_{\epsilon_1, \delta_1}}(x)dx - \frac{1}{\mathrm{poly}(n)}\\
&\geq 0.9 - \frac{1}{\mathrm{poly}(n)} - \mathbb{P}(F_1^c) - \frac{1}{\mathrm{poly}(n)}\\
&\geq 0.9 - 2\sqrt{\delta_1} - \frac{1}{\mathrm{poly}(n)} \geq 0.9 - \frac{1}{\mathrm{poly}(n)},
\end{align*}
since $\delta_1 \asymp \frac{1}{\mathrm{poly}(n)}$.
We now check that $(16 + 8\gamma)h_{\frac{\epsilon}{4TDI_{max}}, \frac{\delta}{I_{max}}} < c_1$. 

Recall that $d > \sqrt{n}\log^5(n)$ and $\epsilon > D^2\log(n)$. Since $\delta_1 \asymp \frac{1}{\mathrm{poly}(n)}$, $\epsilon_1 \asymp 1$, we obtain $I_{max} \asymp \log(n)$. Since $\delta > \frac{e^{-\Theta(\epsilon/T)}}{\mathrm{poly}(n)}$ and $\frac{\epsilon}{T}$ is growing with $n$, we obtain
\begin{align*}
\frac{D^2T^4n\log^4(n)\log(\log(n)/\delta)}{d^2\epsilon^2} & \precsim \frac{D^2T^4n\log^5(n)\epsilon}{d^2\epsilon^2T} = \frac{D^2T^3n\log^5(n)}{d^2\epsilon} \\
& \asymp \frac{D^2n\log^8(n)}{d^2\epsilon} < \frac{D^2}{\epsilon\log^2(n)} < \frac{1}{\log^3(n)}.
\end{align*}
Thus, we have
\begin{align*}
h_{\frac{\epsilon}{4TDI_{max}}, \frac{\delta}{I_{max}}} &= \frac{T}{d} + \frac{D^2T^4I_{max}^2n\log^2(n)\log(I_{max}/\delta)}{d^2\epsilon^2} + \frac{D^4T^8I_{max}^4n^2\log^4(n)\log^2(I_{max}/\delta)}{d^4\epsilon^4}\\
&\precsim \frac{T}{d} + \frac{D^2T^4n\log^4(n)\log(\log(n)/\delta)}{d^2\epsilon^2} \precsim \frac{1}{\log^3(n)}.
\end{align*}
Since $\gamma, c_1 \asymp 1$, we obtain $(16 + 8\gamma)h_{\frac{\epsilon}{4TDI_{max}}, \frac{\delta}{I_{max}}} < c_1$. Hence, since $T \asymp \log(n)$ and $d > \sqrt{n}\log^5(n)$, we have
\begin{align*}
\widetilde{\mathcal{L}}\left(\widetilde{\theta}(G_1), \theta\right) \precsim \frac{T}{d} + \frac{D^2T^4n\log^3(n)\log(\log(n)/\delta)}{d^2\epsilon^2} \precsim \frac{1}{\sqrt{n}\log^4(n)} + \frac{D^2}{\epsilon\log^2(n)} \precsim \frac{1}{\log^3(n)},
\end{align*}
with probability at least $0.9 - \frac{1}{\mathrm{poly}(n)}$. So for $n$ large enough, we have
\begin{align*}
\mathcal{L}\left(\widetilde{\theta}(G_1), \theta\right) \leq \widetilde{\mathcal{L}}\left(\widetilde{\theta}(G_1), \theta\right) \leq \frac{1}{\sqrt{n}\log^3(n)} + \frac{D^2}{\epsilon\log(n)} \precsim \frac{1}{\log^2(n)},
\end{align*}
with probability at least $0.89$. Now we apply Algorithm \ref{alg:Boost}, i.e., $\mathrm{GraphBoost}\left(T, \xi, G, \frac{3\epsilon}{T}, \frac{1}{\mathrm{poly}(n)}, k, \widetilde{\theta}\right)$. By Proposition \ref{prop:Boost_result}, we have
\begin{align*}
\mathcal{L}\left(\thetahat(G), \theta\right) \precsim\frac{1}{\sqrt{n}\log^2(n)} + \frac{D^2}{\epsilon} \precsim \frac{1}{\log(n)},
\end{align*}
with probability at least $1 - \frac{c}{\log(n)}$, for some $0 < c \asymp 1$. Since all the communities in $\theta$ have size $\frac{n}{k}$, it is easy to see that $\widetilde{\mathcal{L}}\left(\thetahat(G), \theta\right) \leq k\mathcal{L}\left(\thetahat(G), \theta\right)$, and since $k \asymp 1$, the conclusion follows.


\subsubsection{Proof of Lemma~\ref{lemma:Initial_Conc_bd_matrix_estim}}
\label{AppInitConc}

Recall that $P = J(B \otimes J_{\frac{n}{k}})J^T \in \mathbb{R}^{n \times n}$, and let $\Pi \in \mathbb{R}^{n \times n}$ be the projection onto $\text{col}(P)$. We run Algorithm \ref{alg:Matrix_Estimation_hardt2014} with adjacency matrix $A, k, \epsilon, \delta$, and $L \asymp \log(n)/\gamma_1$, where $\gamma_1 = 1 - \frac{\sigma_{k + 1}(A)}{\sigma_k(A)}$. By Lemma \ref{lemma:hardt2014_conc_bound}, there exists an event $\Omega_1$, with $\mathbb{P}(\Omega_1) \geq 0.9$, such that  
\begin{align*}
\left\|A - \hat{A}_{(2k)}\right\|_2 \leq \sigma_{k + 1}(A) + O\left(\frac{\sigma_1(A)\sqrt{(n/\gamma_1)\log(n)\log(\log(n)/\gamma_1)}}{\sigma_k(A) - \sigma_{k + 1}(A)}\cdot \frac{\sqrt{\log(1/\delta)}}{\epsilon}\right),
\end{align*}
on $\Omega_1$. By Lemma \ref{lemma:thm5.2LeiRin}, since $a_n \succsim \frac{\log(n)}{n}$, there is an event $\Omega_2$, with $\mathbb{P}(\Omega_2) \geq 1 - \frac{1}{n}$, such that $||A - \mathbb{E}[A]||_2 = ||A - P + \text{diag}(P)||_2 \precsim \sqrt{na_n}$ on $\Omega_2$. Let us work on $\Omega_2' = \Omega_1 \cap \Omega_2$. Hence, $||A - P||_2 \precsim \sqrt{na_n} + a_n \precsim \sqrt{na_n}$. By Lemma \ref{lemma:l2.1SinSte}, we have
\begin{align*}
\sigma_1(A) = \sigma_1(A - P + P) \leq ||A - P||_2 + \sigma_1(P) \precsim \sqrt{na_n} + na_n \precsim na_n.    
\end{align*}
Also, 
\begin{align*}
\sigma_k(A) \geq \sigma_k(P) - ||A - P||_2 \succsim na_n - \sqrt{na_n} \succsim na_n.
\end{align*}
Hence, $\sigma_1(A), \sigma_k(A) \asymp na_n$. Next, by Lemma \ref{lemma:l2.1SinSte} again, we have
\begin{align*}
\sigma_{k + 1}(A) \leq \sigma_{k + 1}(P) + ||A - P||_2  = ||A - P||_2 \precsim \sqrt{na_n},
\end{align*}
since $P$ has rank $k$. So $\gamma_1 \succsim 1 - \frac{1}{\sqrt{d}}$. Hence, $\gamma_1 \geq \frac{1}{2}$, for $n$ large enough, because $d \succsim \log(n)$. Therefore, we have
\begin{align*}
    \left\|A - \hat{A}_{(2k)}\right\|_2 \precsim \sqrt{na_n} + \frac{\log(n)\sqrt{n\log(1/\delta)}}{\epsilon},
\end{align*}
on $\Omega_2'$. Write $\Pi = UU^T$, with $U \in \mathbb{R}^{n \times k}$ and $U^TU = I_k$. Let $\hat{\Pi} = U_{(k)}U_{(k)}^T$. Then 
\begin{align*}
& ||\hat{\Pi} - \Pi||_2 = ||U_{(k)}U_{(k)}^T - UU^T||_2 \stackrel{(a)}{\leq} 2||\sin(\Theta)(U_{(k)}, U)||_2 \stackrel{(b)}{\leq} 2||\sin(\Theta)(U_{(k)}, U)||_F\\
& \quad \stackrel{(c)}{\leq} \mathop{\inf}_{Q: Q^TQ = QQ^T = I_k}||U_{(k)}Q - U||_F \stackrel{(d)}{\leq} ||U_{(k)}\hat{Q} - U||_F\\
& \quad \precsim \frac{\sqrt{k}\left(\sigma_1(P) + \left\|\hat{A}_{(2k)} - P\right\|_2\right)\left\|\hat{A}_{(2k)} - P\right\|_2}{\sigma_k^2(P)} \precsim \frac{\left(na_n + \left\|\hat{A}_{(2k)} - P\right\|_2\right)\left\|\hat{A}_{(2k)} - P\right\|_2}{(na_n)^2}.
\end{align*}
Inequality (a) follows from Lemma \ref{lemma:l2.3SinSte}. Inequality (b) holds because the Frobenius norm is greater than the operator norm. Inequality (c) holds because of Lemma \ref{lemma:prop2.2VuLei}. Inequality (d) and the existence of the $k \times k$ orthogonal matrix $\hat{Q}$ hold because of Lemma \ref{lemma:thm3YuEtal}. We also used the facts that $k \asymp 1$ and $\sigma_1(P) \asymp \sigma_{k}(P) \asymp na_n$. Hence, since
\begin{equation*}
\left\|\hat{A}_{(2k)} - P\right\|_2 \leq \left\|\hat{A}_{(2k)} - A\right\|_2 +  ||A - P||_2 \precsim \sqrt{na_n} + \frac{\log(n)\sqrt{n\log(1/\delta)}}{\epsilon},
\end{equation*}
we have with probability at least $0.9 - \frac{1}{\mathrm{poly}(n)}$ that
\begin{align*}
||\hat{\Pi} - \Pi||_2 &\precsim \frac{\left(na_n + \sqrt{na_n} + \frac{\log(n)\sqrt{n\log(1/\delta)}}{\epsilon}\right)\left(\sqrt{na_n} + \frac{\log(n)\sqrt{n\log(1/\delta)}}{\epsilon}\right)}{(na_n)^2}\\
&\precsim \frac{\left(d + \frac{\log(n)\sqrt{n\log(1/\delta)}}{\epsilon}\right)\left(\sqrt{d} + \frac{\log(n)\sqrt{n\log(1/\delta)}}{\epsilon}\right)}{d^2}\\
&\precsim \frac{1}{\sqrt{d}} + \frac{\log(n)\sqrt{n\log(1/\delta)}}{d\epsilon} + \frac{n\log^2(n)\log(1/\delta)}{d^2\epsilon^2}.
\end{align*}
Now note that by Lemma \ref{lemma:l2.3SinSte} and Lemma \ref{lemma:prop2.2VuLei}, since $k \asymp 1$, we have
\begin{align*}
\frac{1}{2}\inf\limits_{Q \in \mathbb{V}_{k, k}}||U_{(k)} - UQ||_F^2 & = \frac{1}{2}||U_{(k)} - UQ'||_F^2 \leq ||\sin(\Theta)(U_{(k)}, U)||_F^2 \\
& \leq k||\sin(\Theta)(U_{(k)}, U)||_2^2 \leq k||\hat{\Pi} - \Pi||_2^2\\
&\precsim \frac{1}{d} + \frac{n\log^2(n)\log(1/\delta)}{d^2\epsilon^2} + \frac{n^2\log^4(n)\log^2(1/\delta)}{d^4\epsilon^4},   
\end{align*}
where $Q' \in \mathbb{R}^{k \times k}$ and $Q'^TQ' = I_k$. The rest of the proof follows the lines of argument from the proof of Theorem 5.1 in \cite{HehEtal22}. We know from \cite{LeiRin15} that $U = \Theta X$, for some $\Theta \in \mathbf{M}_{n\times k}$ and $X$ satisfying $||X_{j*} - X_{l*}||_2 = \sqrt{\frac{2k}{n}}$, for $j \neq l$. Then $UQ' = \Theta XQ' = \Theta X'$, with $X' = XQ'$, and
\begin{equation*}
||X'_{j*} - X'_{l*}||_2 = ||(e_j- e_l)^TXQ'||_2 = ||X_{j*} - X_{l*}||_2 = \sqrt{\frac{2k}{n}},
\end{equation*}
for $j \neq l$. Here, $e_j$ is the $j^{\text{th}}$ standard basis vector. For all $j \in [k]$, choose $\delta_j = \sqrt{\frac{2k}{n}}$ and define $S_j$ as in Lemma 5.3 of \cite{LeiRin15}. We want to show that $(16 + 8\gamma)||U_{(k)} - UQ'||_F^2 < \frac{\delta_j^2n}{k}$, for all $j \in [k]$. Since $\frac{\delta_j^2n}{k} > 1$, it is enough to prove that $(16 + 8\gamma)||U_{(k)} - UQ'||_F^2 \leq 1$. Let $c_1 \asymp 1$ be the absolute constant such that
\begin{align*}
c_1 ||U_{(k)} - UQ'||_F^2 \leq \left(\frac{1}{d} + \frac{n\log^2(n)\log(1/\delta)}{d^2\epsilon^2} + \frac{n^2\log^4(n)\log^2(1/\delta)}{d^4\epsilon^4}\right).
\end{align*}
Then
\begin{align*}
(16 + 8\gamma)||U_{(k)} - UQ'||_F^2 \leq \frac{16 + 8\gamma}{c_1} \left(\frac{1}{d} + \frac{n\log^2(n)\log(1/\delta)}{d^2\epsilon^2} + \frac{n^2\log^4(n)\log^2(1/\delta)}{d^4\epsilon^4}\right) \leq 1, 
\end{align*}
as desired. Thus, for each
community $j \in [k]$, the set of nodes that are possibly misclassified by Algorithm \ref{alg:Matrix_Estimation_hardt2014} must be a subset of $S_j$, and since $\frac{\delta_j^2n}{k} > 1$, we have
\begin{align*}
\widetilde{\mathcal{L}}\left(\mathcal{A}_{ME}^{(\epsilon, \delta)}(G), \theta\right) &\leq \mathop{\max}_{j \in [k]}\frac{k|S_j|}{n} \leq \frac{k}{n}\sum_{j = 1}^k|S_j| \leq \sum_{j = 1}^k|S_j|\delta_j^2 \leq (16 + 8\gamma)||U_{(k)} - UQ'||_F^2\\
&\leq \frac{16 + 8\gamma}{c_1} \left(\frac{1}{d} + \frac{n\log^2(n)\log(1/\delta)}{d^2\epsilon^2} + \frac{n^2\log^4(n)\log^2(1/\delta)}{d^4\epsilon^4}\right)\\
&= C_1\left(\frac{1}{d} + \frac{n\log^2(n)\log(1/\delta)}{d^2\epsilon^2} + \frac{n^2\log^4(n)\log^2(1/\delta)}{d^4\epsilon^4}\right),
\end{align*}
where $C_1 = \frac{16 + 8\gamma}{c_1} \asymp 1$, with probability at least $0.9 - \frac{1}{\mathrm{poly}(n)}$, as desired, since $\gamma, c_1 \asymp 1$.


\subsection{Proofs for Section~\ref{section:Approximate DP: Private approximate subspace estimation} (unweighted case)}

\subsubsection{Proof of Theorem~\ref{theorem:GoodCent_Main_Thm}}
\label{AppGoodUnwt}

Let $G_1, \dots, G_T$ be as in Algorithm \ref{alg:Boost}. Denote $\widetilde{\theta}(G_j) = \mathcal{A}_{SE}^{\left(\epsilon/(5TD\hat{L}_{\epsilon_1, \delta_1}(G_j)), \delta/\hat{L}_{\epsilon_1, \delta_1}(G_j)\right)}(T_{D}(G_j))$, for $j \in [T]$.

We start by analyzing the privacy of Algorithm \ref{alg:Subspace_Estimation_GoodC}. The following result is proved in Appendix~\ref{AppLemGoodC}:

\begin{lemma}
\label{lemma:GoodC_Priv}
Let $2\epsilon < \log(1/\delta)$. Then Algorithm \ref{alg:Subspace_Estimation_GoodC} with privacy parameters $\left(\frac{\epsilon}{5D}, \delta\right)$ is $(2\epsilon, \delta)_{2D}$-node DP (for both $\mathcal{G}_{n, 2D}$ and $\mathcal{W}_{n, 2D}$).
\end{lemma}

By Lemma \ref{lemma:GoodC_Priv}, we know that if $\frac{2\epsilon}{T} < \log(1/\delta)$, then $\mathcal{A}_{SE}^{(\epsilon/(5TD), \delta)}$ is $(2\epsilon/T, \delta)_{2D}$-node DP. Since $\frac{2\epsilon}{T} < \log(1/(\delta\cdot\mathrm{poly}(n)))$, Theorem \ref{theorem:generic_red} implies that $\widetilde{\theta}(G_j)$ is $(\epsilon_1 + 2\epsilon/T, 1/\mathrm{poly}(n))$-node DP. We also used the fact that $\epsilon_1 \asymp 1$ and $\delta_1 \asymp \frac{1}{\mathrm{poly}(n)}$. Since $T \asymp \log(n)$, we know that $\epsilon/T$ grows with $n$. So since $\epsilon_1 \asymp 1$, we have $\epsilon_1 \leq \epsilon/T$ for $n$ sufficiently large, implying that $\widetilde{\theta}(G_j)$ is $(3\epsilon/T, 1/\mathrm{poly}(n))$-node DP. Finally, by Proposition \ref{prop:Boost_result}, we know that $\thetahat(G)$ is $(3\epsilon, 1/\mathrm{poly}(n))$-node DP.

For the utility guarantee, we first focus on $G_1 \sim SBM\left(n, k, \frac{B}{T}\right)$. We will then use Proposition \ref{prop:Boost_result}. The following result, proved in Appendix~\ref{AppLemPihat}, provides a utility guarantee for deterministic privacy parameters:

\begin{lemma}
\label{lemma:Pihat-Pi}
Assume $\frac{\epsilon^2}{\log(1/\delta)} \geq \frac{1}{\log^3(n)}$ and $d > \sqrt{n}\log^{11/2}(n)$. Suppose Assumption \ref{ASS_SE_t_m} and equation \eqref{eq:t} hold. Let $\theta = \{\theta_i\}_{i = 1}^n$ be the true community assignments. Let $\mathcal{A}_{SE}^{(\epsilon, \delta)}(G)$ be the output of Algorithm \ref{alg:Subspace_Estimation_GoodC}. Then there exists an absolute constant $C_1 \asymp 1$ such that
\begin{align*}
\widetilde{\mathcal{L}}\left(\mathcal{A}_{SE}^{(\epsilon, \delta)}(G), \theta\right) \leq \frac{C_1\sqrt{\log(1/\delta)}}{\epsilon\log^4(n)} \precsim \frac{1}{\log^{5/2}(n)},
\end{align*}
with probability at least $0.9 - 3\zeta - \frac{1}{\mathrm{poly}(n)}$.
\end{lemma}

Let $1 \asymp \zeta \in (0, 1/3)$. Let $I := \left[5 + \frac{4\log(1/\delta_1)}{\epsilon_1}, 5 + \frac{16\log(1/\delta_1)}{\epsilon_1}\right]$, $I_{min} := 5 + \frac{4\log(1/\delta_1)}{\epsilon_1}$, and $I_{max} := 5 + \frac{16\log(1/\delta_1)}{\epsilon_1}$. Let $X_{\epsilon_1, \delta_1} := 5 + \mathrm{Lap}\left(\frac{8}{\epsilon_1}\right) + \frac{8\log(1/\delta_1)}{\epsilon_1}$. Then for $F_1 = \left\{X_{\epsilon_1, \delta_1} \in I\right\}$, we have
\begin{align*}
 \mathbb{P}\left(F_1^c\right) \leq \mathbb{P}\left(\mathrm{Lap}\left(\frac{8}{\epsilon_1}\right) < -\frac{4\log(1/\delta_1)}{\epsilon_1}\right) + \mathbb{P}\left(\mathrm{Lap}\left(\frac{8}{\epsilon_1}\right) > \frac{8\log(1/\delta_1)}{\epsilon_1}\right) \leq 2\sqrt{\delta_1}.
\end{align*}
Let $h_{\epsilon, \delta} = \frac{\sqrt{\log(1/\delta)}}{\epsilon\log^4(n)}$ and let $C_1$ be as in Lemma \ref{lemma:Pihat-Pi}. Let $f_{X_{\epsilon_1, \delta_1}}$ be the pdf of $X_{\epsilon_1, \delta_1}$. Also, let $\Omega_0$ be as in the proof of Lemma \ref{lemma:hp_removal_d} applied to $\frac{a_n}{T}$. The guarantee on this event holds because we are under Assumption \ref{AssSBM} applied to $\frac{a_n}{T}$. We have $\mathbb{P}(\Omega_0) \geq 1 - \frac{1}{\mathrm{poly}(n)}$. In addition, on $\Omega_0$, we do not have to truncate the graph, since the degree of every node is less than $\frac{2d}{T}$. Hence, on $\Omega_0$, we have $T_{D}(G_1) = G_1$ and $d_{T_{D}}(G_1) = 0$. Since $X_{\epsilon_1, \delta_1} \geq 5 > \frac{1}{2}$ on $F_1$, we then have
\begin{align*}
&\mathbb{P}\left(\widetilde{\mathcal{L}}\left(\widetilde{\theta}(G_1), \theta\right) \leq C_1h_{\frac{\epsilon}{5TDI_{max}}, \frac{\delta}{I_{max}}}\right)\\
&\geq \mathbb{P}\left(\left\{\widetilde{\mathcal{L}}\left(\widetilde{\theta}(G_1), \theta\right) \leq C_1h_{\frac{\epsilon}{5TDI_{max}}, \frac{\delta}{I_{max}}}\right\} \cap \Omega_0\right)\\
&\geq \mathbb{P}\left(\left\{\widetilde{\mathcal{L}}\left(\mathcal{A}_{SE}^{(\epsilon/(5TD\max\{1/2, X_{\epsilon_1, \delta_1}\}), \delta/\max\{1/2, X_{\epsilon_1, \delta_1}\})}(G_1), \theta\right) \leq C_1h_{\frac{\epsilon}{5TDI_{max}}, \frac{\delta}{I_{max}}}\right\} \cap \Omega_0 \cap F_1\right)\\
&= \mathbb{P}\left(\left\{\widetilde{\mathcal{L}}\left(\mathcal{A}_{SE}^{(\epsilon/(5TD X_{\epsilon_1, \delta_1}), \delta/X_{\epsilon_1, \delta_1})}(G_1), \theta\right) \leq C_1h_{\frac{\epsilon}{5TDI_{max}}, \frac{\delta}{I_{max}}}\right\} \cap \Omega_0 \cap F_1\right)\\
&\geq \int_I\mathbb{P}\left(\widetilde{\mathcal{L}}\left(\mathcal{A}_{SE}^{(\epsilon/(5TDX_{\epsilon_1, \delta_1}), \delta/X_{\epsilon_1, \delta_1})}(G_1), \theta\right) \leq C_1h_{\frac{\epsilon}{5TDI_{max}}, \frac{\delta}{I_{max}}}\middle\vert X_{\epsilon_1, \delta_1} = x\right)f_{X_{\epsilon_1, \delta_1}}(x)dx \\
& \qquad - \frac{1}{\mathrm{poly}(n)}\\
&= \int_I\mathbb{P}\left(\widetilde{\mathcal{L}}\left(\mathcal{A}_{SE}^{(\epsilon/(5TDx), \delta/x)}(G_1), \theta\right) \leq C_1h_{\frac{\epsilon}{5TDI_{max}}, \frac{\delta}{I_{max}}}\right)f_{X_{\epsilon_1, \delta_1}}(x)dx - \frac{1}{\mathrm{poly}(n)}\\
&\geq \int_I\mathbb{P}\left(\widetilde{\mathcal{L}}\left(\mathcal{A}_{SE}^{(\epsilon/(5TDx), \delta/x)}(G_1), \theta\right) \leq C_1h_{\frac{\epsilon}{5TDx}, \frac{\delta}{x}}\right)f_{X_{\epsilon_1, \delta_1}}(x)dx - \frac{1}{\mathrm{poly}(n)},
\end{align*}
since the function $h_{\frac{\epsilon}{5TDx}, \frac{\delta}{x}}$ is increasing in $x > 0$. We also dropped the conditioning on $\left\{X_{\epsilon_1, \delta_1} = x\right\}$, since the randomness in $X_{\epsilon_1, \delta_1}$ only depends on the $\mathrm{Lap}\left(\frac{8}{\epsilon_1}\right)$ noise, and this is independent of $\widetilde{\mathcal{L}}\left(\mathcal{A}_{SE}^{(\epsilon/(5TDx), \delta/x)}(G_1), \theta\right)$. We now use Lemma \ref{lemma:Pihat-Pi}, with $\frac{\epsilon}{5TDx}, \frac{\delta}{x}$, and $x \in [I_{min}, I_{max}]$. For this, we need $d > \sqrt{n}T\log^{11/2}(n)$, and 
\begin{align*}
\max\left\{(C^{(1)})^2n^{1 - 2\beta_0}\log(n), \frac{1}{\log^3(n)}\right\} < \frac{\epsilon^2}{25D^2T^2x^2\log(x/\delta)} \leq \frac{(C^{(1)})^2n\log(n)}{4}
\end{align*}
(Assumption~\ref{ASS_SE_t_m} with $\frac{\epsilon}{5DTx}$), which is equivalent to 
\begin{align*}
\frac{1}{\log^3(n)} < \frac{\epsilon^2}{25D^2T^2x^2\log(x/\delta)} \leq \frac{(C^{(1)})^2n\log(n)}{4},
\end{align*}
since $\beta_0 > \frac{1}{2}$, so for $n$ large enough, we have $(C^{(1)})^2n^{1 - 2\beta_0}\log(n) \leq \frac{1}{\log^3(n)}$. We already have $d > \sqrt{n}T\log^{11/2}(n)$, by assumption. Next, since $\epsilon_1 \asymp 1$ and $\delta_1 \asymp \frac{1}{\mathrm{poly}(n)}$, we have $I_{max}, I_{min} \asymp \log(n)$. Since $\frac{2\epsilon}{T} < \log(1/\delta)$, and $T \asymp \log(n)$, we then obtain
\begin{align*}
\frac{\epsilon^2}{D^2T^2x^2\log(x/\delta)} & \leq \frac{\epsilon^2}{D^2T^2I_{min}^2\log(I_{min}/\delta)} \asymp \frac{\epsilon^2}{D^2\log^4(n)\log(\log(n)/\delta)} \\
& \leq \frac{\epsilon^2}{D^2\log^4(n)\log(1/\delta)} \precsim \frac{\epsilon}{D^2\log^3(n)} < n.
\end{align*}
Hence, for $n$ large enough, we have $\frac{\epsilon^2}{25D^2T^2x^2\log(x/\delta)} \leq \frac{(C^{(1)})^2n\log(n)}{4}$. Now, since $T\log(1/(\delta \cdot \mathrm{poly}(n))) < \Theta(\epsilon)$, we have
\begin{align*}
\frac{\epsilon^2}{D^2T^2x^2\log(x/\delta)} &\geq \frac{\epsilon^2}{D^2T^2I_{max}^2\log(I_{max}/\delta)} \asymp \frac{\epsilon^2}{D^2\log^4(n)\log(\log(n)/\delta)} \\
& \succsim \frac{\epsilon^2}{D^2\log^4(n)(\log(n) + \epsilon/T)} \succsim \frac{\epsilon}{D^2\log^3(n)} > \frac{1}{\log^2(n)},
\end{align*}
since $\frac{\epsilon}{T}$ grows larger than $\log(n)$. So for $n$ large enough, we have $\frac{\epsilon^2}{25D^2T^2x^2\log(x/\delta)} \geq \frac{1}{\log^3(n)}$. Hence, we can use Lemma \ref{lemma:Pihat-Pi}, implying that 
\begin{align*}
\widetilde{\mathcal{L}}\left(\mathcal{A}_{SE}^{(\epsilon/(5TDx), \delta/x)}(G_1), \theta\right) \leq C_1h_{\frac{\epsilon}{5TDx}, \frac{\delta}{x}},
\end{align*}
with probability at least $0.9 - 3\zeta - \frac{1}{\mathrm{poly}(n)}$. Thus, we have
\begin{align*}
& \mathbb{P}\left(\widetilde{\mathcal{L}}\left(\widetilde{\theta}(G_1), \theta\right) \leq C_1h_{\frac{\epsilon}{5TDI_{max}}, \frac{\delta}{I_{max}}}\right) \\
&\qquad \geq \int_I\left(0.9 - 3\zeta - \frac{1}{\mathrm{poly}(n)}\right)f_{X_{\epsilon_1, \delta_1}}(x)dx - \frac{1}{\mathrm{poly}(n)}\\
& \qquad \geq 0.9 - 3\zeta - \frac{1}{\mathrm{poly}(n)} - \mathbb{P}(F_1^c) - \frac{1}{\mathrm{poly}(n)}\\
& \qquad \geq  0.9 - 3\zeta - 2\sqrt{\delta_1} - \frac{1}{\mathrm{poly}(n)} \geq 0.9 - 3\zeta - \frac{1}{\mathrm{poly}(n)},
\end{align*}
since $\delta_1 \asymp \frac{1}{\mathrm{poly}(n)}$. Hence, since $I_{max} \asymp \log(n)$, $D^2\log(n) < \epsilon$, and $T\log(1/(\delta\cdot\mathrm{poly}(n))) < \Theta(\epsilon)$, we obtain
\begin{align*}
\widetilde{\mathcal{L}}\left(\widetilde{\theta}(G_1), \theta\right) & \precsim \frac{DTI_{max}\sqrt{\log(I_{max}/\delta)}}{\epsilon\log^4(n)} \precsim \frac{D\log^2(n)\sqrt{\log(n) + \epsilon/T}}{\epsilon\log^4(n)} \precsim \frac{D}{\log^{5/2}(n)\sqrt{\epsilon}} < \frac{1}{\log^3(n)},
\end{align*}
with probability at least $0.9 - 3\zeta - \frac{1}{\mathrm{poly}(n)}$. For a suitable choice of $\zeta \in (0, 1/3)$, we obtain $0.9 - 3\zeta - \frac{1}{\mathrm{poly}(n)} \geq 0.87$, for $n$ large enough, implying that
\begin{align*}
\mathcal{L}\left(\widetilde{\theta}(G_1), \theta\right) \leq \widetilde{\mathcal{L}}\left(\widetilde{\theta}(G_1), \theta\right) \leq \frac{D}{\log^{3/2}(n)\sqrt{\epsilon}} < \frac{1}{\log^2(n)},
\end{align*}
with probability at least $0.87$. Now we apply Algorithm \ref{alg:Boost}, i.e., $\mathrm{GraphBoost}\left(T, \xi, G, \frac{3\epsilon}{T}, \frac{1}{\mathrm{poly}(n)}, k, \widetilde{\theta}\right)$. By Proposition \ref{prop:Boost_result}, we have
\begin{align*}
\mathcal{L}\left(\thetahat(G), \theta\right) \precsim \frac{D}{\log^{1/2}(n)\sqrt{\epsilon}} < \frac{1}{\log(n)},
\end{align*}
with probability at least $1 - \frac{c}{\log(n)}$, for some $0 < c \asymp 1$. Since all the communities in $\theta$ have size $\frac{n}{k}$, it is easy to see that $\widetilde{\mathcal{L}}\left(\thetahat(G), \theta\right) \leq k\mathcal{L}\left(\thetahat(G), \theta\right)$, and since $k \asymp 1$, the conclusion follows. 

\subsubsection{Proof of Lemma~\ref{lemma:GoodC_Priv}}
\label{AppLemGoodC}

The proofs for the unweighted and weighted graphs are similar, so we discuss both simultaneously. There are two privatization steps in the algorithm: The application of GoodCenter $q$ times and the Gaussian noise addition step, also happening $q$ times. Let us examine each in turn. 

By Lemma \ref{lemma:cor_a3_supp}, one application of GoodCenter in Algorithm \ref{alg:Subspace_Estimation_GoodC} is $\frac{\epsilon^2}{800D^2q\log(1/\delta)}$-zCDP, when changing one of the $t$ points. When changing one edge in the graph (either flip or weight-change), at most $2$ of the $t$ points can change. By the composition of zCDP mechanisms, i.e., Lemma \ref{lemma:Comp_zCDP}, one application of GoodCenter in Algorithm \ref{alg:Subspace_Estimation_GoodC} is $\frac{\epsilon^2}{200D^2q\log(1/\delta)}$-edge zCDP. Since there are $q$ applications, we obtain $\frac{\epsilon^2}{200D^2\log(1/\delta)}$-edge zCDP, by Lemma \ref{lemma:Comp_zCDP}.

Next, recall that when changing one edge, we change at most $2$ of the $t$ chunks. Thus, the $\ell_2$-sensitivity of the average of $t$ points is $\frac{4r}{t}$. Using Lemma \ref{lemma:Gauss_zCDP}, one application of the Gaussian mechanism in Algorithm \ref{alg:Subspace_Estimation_GoodC} is $\frac{\epsilon^2}{200D^2q\log(1/\delta)}$-edge zCDP. Since there are $q$ applications, we obtain $\frac{\epsilon^2}{200D^2\log(1/\delta)}$-edge zCDP, by Lemma \ref{lemma:Comp_zCDP}.

By Lemma \ref{lemma:Comp_zCDP} again, we see that Algorithm \ref{alg:Subspace_Estimation_GoodC} is $\frac{\epsilon^2}{100D^2\log(1/\delta)}$-edge zCDP. When changing one node, provided all changes remain in $\mathcal{G}_{n, 2D}$ or $\mathcal{W}_{n, 2D}$, at most $4D$ connections from that node are changed. For unweighted graphs, there can only be flips, while in the weighted case, we can set edges to $0$, set them to some non-zero weight, or set their non-zero weight to another non-zero value. In either case, at most $4D + 1$ of the $t$ chunks are changed, so at most $5D$ (for $n$ large enough) chunks can change. Hence, by group privacy, using Lemma \ref{lemma:Group_zCDP}, Algorithm \ref{alg:Subspace_Estimation_GoodC} with privacy parameters $\left(\frac{\epsilon}{5D}, \delta\right)$ is $\left(\frac{\epsilon^2}{4\log(1/\delta)}\right)_{2D}$-zCDP. Finally, by Lemma \ref{lemma:zCDP_to_DP}, we obtain $\left(\frac{\epsilon^2}{4\log(1/\delta)} + \epsilon, \delta\right)_{2D}$-DP. Since $\delta < e^{-2\epsilon}$, this is also $(9\epsilon/8, \delta)_{2D}$-DP, and the conclusion follows.


\subsubsection{Proof of Lemma~\ref{lemma:Pihat-Pi}}
\label{AppLemPihat}

The following result, proved in Appendix~\ref{AppLemHPAll}, shows that w.h.p., when dividing $P = [P_1^T, \dots, P_t^T]^T \in \mathbb{R}^{n \times n}$ according to how $A$ is divided in step $1$, all chunks have rank $k$:
\begin{lemma}
\label{lemma:hp_all_types}
Let $\Omega_1$ be the event that $P_1, \dots, P_t$ all have rank $k$. Then
\begin{equation*}
\mprob(\Omega_1) \ge 1 - \frac{1}{\mathrm{poly}(n)}.
\end{equation*}
\end{lemma}

From now on, we work in the event $\Omega_1$. Since each chunk has rank $k$, the projection onto $\text{row}(P_j)$ is $\Pi$, for all $j \in [t]$. Let $P = UD_PU^T$, with $D_P \in \mathbb{R}^{k \times k}$ diagonal and $U \in \mathbb{R}^{n \times k}$ having orthonormal columns. Then $\Pi = UU^T$.


The following result is proved in Appendix~\ref{AppLemPiZ}:

\begin{lemma}
\label{lemma:Pi_Z}
Let $d > \sqrt{n}\log^{11/2}(n)$. Let $E := E_Z + E_{Z^{\perp}}$, i.e., the decomposition of each column of $E$ as a sum of the component in the column span of $Z$ and the corresponding orthogonal complement. Let $Z' = Z + E_Z$. Assume $\xi_0\sqrt{\log(n)} = o(1)$. Then $\sigma_k(Z') \asymp \sqrt{k} \asymp 1$, and $\Pi$ is also the projection onto the columns of $Z'$, with probability at least $0.9 - 3\zeta - \frac{1}{\mathrm{poly}(n)}$.
\end{lemma}

Since $\frac{\epsilon^2}{\log(1/\delta)} \geq \frac{1}{\log^3(n)}$, we have
\begin{align*}
\xi_0 = \frac{\log^{1/4}(1/\delta)}{\log^{5/2}(n)\sqrt{\epsilon}} + \frac{\sqrt{\log(1/\delta)}}{\log^5(n)\epsilon} \asymp \frac{\log^{1/4}(1/\delta)}{\log^{5/2}(n)\sqrt{\epsilon}} \leq \frac{1}{\log^{7/4}(n)},
\end{align*}
so, $\xi_0\sqrt{\log(n)} \precsim \frac{1}{\log^{5/4}(n)} = o(1)$. Hence, we can use Lemma \ref{lemma:Pi_Z}.

Consider the SVD of $\hat{Z} = \hat{U}\hat{D}\hat{V}^T$ and also the SVD of $Z' = U'D'V'^T$. Write $U' = \left[U'^{(k)}, U'^{(k)}_{\perp}\right] \in \mathbb{R}^{n \times n}$ and $V' = \left[V'^{(k)}, V'^{(k)}_{\perp}\right] \in \mathbb{R}^{q \times q}$. Then
\begin{align*}
a_1 &= \sigma_{min}\left((U'^{(k)})^T\hat{Z}V'^{(k)}\right) = \sigma_{min}\left((U'^{(k)})^TZ'V'^{(k)} + (U'^{(k)})^TE_{Z^{\perp}}V'^{(k)}\right)\\
&= \sigma_{min}\left((U'^{(k)})^TZ'V'^{(k)}\right) = \sigma_k(Z') \asymp \sqrt{k} \asymp 1.
\end{align*}
This holds since, by Lemma \ref{lemma:Pi_Z}, the columns of $U'^{(k)}$ are in the span of $Z$, so they are orthogonal to the columns of $E_{Z^{\perp}}$. Now, recall that, under Assumption \ref{ASS_SE_t_m}, we have $\frac{\epsilon^2}{\log(1/\delta)} \precsim n\log(n)$, so $\frac{\log^{1/4}(1/\delta)\sqrt{\log(n)}}{\sqrt{\epsilon}} \succsim \frac{\log^{1/4}(n)}{n^{1/4}} \succsim \frac{1}{n^2\sqrt{n}}$. Hence, we have $\frac{1}{n^2\sqrt{n}} + \xi_0\sqrt{\log(n)} \asymp \xi_0\sqrt{\log(n)}$.

The following result is proved in Appendix~\ref{AppLemZhat}: 

\begin{lemma}
\label{lemma:Zhat_Z}
Let $Z = [\Pi z_1, \dots, \Pi z_q] \in \mathbb{R}^{n \times q}$, and suppose $d > \sqrt{n}\log^{11/2}(n)$. We have $\|\hat{Z} - Z\|_2 \precsim \frac{1}{n^2\sqrt{n}} + \xi_0\sqrt{\log(n)}$, with probability at least $0.9 - 2\zeta - \frac{1}{\mathrm{poly}(n)}$.
\end{lemma}

Using Lemma \ref{lemma:Zhat_Z}, we obtain
\begin{align*}
b_1 &= \left\|(U'^{(k)}_{\perp})^T\hat{Z}V'^{(k)}_{\perp}\right\|_2 = \left\|(U'^{(k)}_{\perp})^TZ'V'^{(k)}_{\perp} + (U'^{(k)}_{\perp})^TE_{Z^{\perp}}V'^{(k)}_{\perp}\right\|_2 = \left\|(U'^{(k)}_{\perp})^TE_{Z^{\perp}}V'^{(k)}_{\perp}\right\|_2\\
&\leq ||E_{Z^{\perp}}||_2 \leq ||E||_2 \precsim \xi_0\sqrt{\log(n)},
\end{align*}
since the columns of $U'^{(k)}_{\perp}$ are orthogonal to the columns of $Z'$. Next, for a matrix $M$, let $\Pi_M$ be the projection matrix onto the columns of $M$. Then
\begin{align*}
z_{12} = \left\|\Pi_{U'^{(k)}}E_{Z^{\perp}}\Pi_{V'^{(k)}_{\perp}}\right\|_2 = 0,
\end{align*}
since the columns of $U'^{(k)}$ are orthogonal to the columns of $E_{Z^{\perp}}$, using Lemma \ref{lemma:Pi_Z}. Also, by Lemma \ref{lemma:Zhat_Z}, again, we have
\begin{align*}
z_{21} = \left\|\Pi_{U'^{(k)}_{\perp}}E_{Z^{\perp}}\Pi_{V'^{(k)}}\right\|_2 \leq ||E_{Z^{\perp}}||_2 \leq ||E||_2 \precsim \xi_0\sqrt{\log(n)}.
\end{align*}
With $a_1, b_1, z_{12}$, and $z_{21}$ in mind, by Lemma \ref{lemma:l2.3SinSte}, we have 
\begin{align*}
\left\|\sin(\Theta)(U'^{(k)}, \hat{U}^{(k)})\right\|_2 \leq \frac{a_1z_{21} + b_1z_{12}}{a_1^2 - b_1^2 - \min\{z_{12}^2, z_{21}^2\}} = \frac{a_1z_{21}}{a_1^2 - b_1^2} \precsim \frac{b_1}{1 - b_1^2}.
\end{align*}
Hence, by Lemma \ref{lemma:l2.3SinSte}, and because the column span of $Z'$ is the same as the column span of $P$ (both having rank $k$ by Lemma \ref{lemma:Pi_Z}), we have
\begin{align*}
||\hat{\Pi} - \Pi||_2^2 & = \left\|\hat{U}^{(k)}(\hat{U}^{(k)})^T - U'^{(k)}(U'^{(k)})^T\right\|_2^2 \leq 4\left\|\sin(\Theta)(U'^{(k)}, \hat{U}^{(k)})\right\|_2^2 \\
& \precsim \frac{b_1^2}{(1 - b_1^2)^2} \precsim \xi_0^2\log(n) \precsim \frac{1}{\log^{5/2}(n)},
\end{align*}
since $b_1 = o(1)$. Now note that by Lemma \ref{lemma:l2.3SinSte} and Lemma \ref{lemma:prop2.2VuLei}, since $k \asymp 1$, we have
\begin{align*}
\frac{1}{2}\inf\limits_{Q \in \mathbb{V}_{k, k}}||\hat{U}^{(k)} - UQ||_F^2 & = \frac{1}{2}||\hat{U}^{(k)} - UQ||_F^2 \leq ||\sin(\Theta)(\hat{U}^{(k)}, U)||_F^2 \\
& \leq k||\sin(\Theta)(\hat{U}^{(k)}, U)||_2^2 \leq k||\hat{\Pi} - \Pi||_2^2 \\
& \precsim \xi_0^2\log(n),   
\end{align*}
where $Q' \in \mathbb{R}^{k \times k}$ and $Q'^TQ' = I_k$. The rest of the proof follows the argument from the proof of Theorem 5.1 in \cite{HehEtal22}. We know from \cite{LeiRin15} that $U = \Theta X$, for some $\Theta \in \mathbf{M}_{n\times k}$ and $X$ satisfying $||X_{j*} - X_{l*}||_2 = \sqrt{\frac{2k}{n}}$, for $j \neq l$. Then $UQ' = \Theta XQ' = \Theta X'$, with $X' = XQ'$ and
\begin{equation*}
||X'_{j*} - X'_{l*}||_2 = ||(e_j- e_l)^TXQ'||_2 = ||X_{j*} - X_{l*}||_2 = \sqrt{\frac{2k}{n}},
\end{equation*}
for $j \neq l$. Here, $e_j$ is the $j^{\text{th}}$ standard basis vector. For all $j \in [k]$, choose $\delta_j = \sqrt{\frac{2k}{n}}$ and define $S_j$ as in Lemma 5.3 of \cite{LeiRin15}. We want to show that $(16 + 8\gamma)||\hat{U}^{(k)} - UQ'||_F^2 < \frac{\delta_j^2n}{k}$, for all $j \in [k]$. Since $\frac{\delta_j^2n}{k} > 1$, it is enough to prove $(16 + 8\gamma)||\hat{U}^{(k)} - UQ'||_F^2 \leq 1$. Let $c_1 \asymp 1$ be the absolute constant such that
\begin{align*}
c_1||\hat{U}^{(k)} - UQ'||_F^2 \leq \xi_0^2\log(n).   
\end{align*}
Then
\begin{align*}
(16 + 8\gamma)||\hat{U}^{(k)} - UQ'||_F^2 \leq \frac{(16 + 8\gamma)\xi_0^2\log(n)}{c_1} \leq 1, 
\end{align*}
as desired, since $\xi_0^2\log(n) = o(1)$ and $\gamma, c_1 \asymp 1$. Thus, for each
community $j \in [k]$, the set of nodes that are possibly misclassified by Algorithm \ref{alg:Subspace_Estimation_GoodC} must be a subset of $S_j$, and since $\frac{\delta_j^2n}{k} > 1$, we have
\begin{align*}
\widetilde{\mathcal{L}}\left(\mathcal{A}_{SE}^{(\epsilon, \delta)}(G), \theta\right) &\leq \mathop{\max}_{j \in [k]}\frac{k|S_j|}{n} \leq \frac{k}{n}\sum_{j = 1}^k|S_j| \leq \sum_{j = 1}^k|S_j|\delta_j^2 \leq (16 + 8\gamma)||\hat{U}^{(k)} - UQ'||_F^2\\
&\leq \frac{(16 + 8\gamma)\xi_0^2\log(n)}{c_1} \leq \frac{C_1\sqrt{\log(1/\delta)}}{\epsilon\log^4(n)} \precsim \frac{1}{\log^{5/2}(n)},
\end{align*}
for $C_1 = \frac{(16 + 8\gamma)}{c_1} \asymp 1$, with probability at least $0.9 - 3\zeta - \frac{1}{\mathrm{poly}(n)}$, as desired, since $\gamma, c_1 \asymp 1$.


\subsubsection{Proof of Lemma~\ref{lemma:hp_all_types}}
\label{AppLemHPAll}

Let $G_j = \{P_j \text{ has rank } k\}$. Then $\mathbb{P}(G_1) = \dots = \mathbb{P}(G_t)$, and
\begin{align*}
p_1 &= \mathbb{P}(G_1) = \mathbb{P}(\text{all of the distinct  types of rows in } P \text{ appear at least once in chunk } 1)\\
&= 1 - \mathbb{P}(\cup_{j = 1}^kM_j^c),
\end{align*}
where
\begin{equation*}
M_j = \{\text{row of type } j \text{ in } P \text{ appears at least once}\}.
\end{equation*}
For $L \subseteq [k]$, let
\begin{equation*}
K_L = \{\text{none of the types in } L \text{ appear at all}\}.
\end{equation*}
By the inclusion-exclusion principle, we have
\begin{align*}
p_1 = 1 - \sum_{\ell = 1}^k (-1)^{\ell - 1}\sum_{|L| = \ell}\mathbb{P}(K_L).
\end{align*}
We first show that $\mprob(\Omega_1) \ge 1 - t(1-p_1)$. Note that $K_L = \emptyset$ for $\ell = k$. There are \(\binom{k}{\ell}\) ways of choosing the set $L$ from $[k]$ such that $|L| = \ell$. Out of the \(\binom{n}{m}\) ways choose the rows of chunk 1, there are $\binom{n - \frac{\ell n}{k}}{m}$ ways such that none of the distinct types in $L$ appear at all, implying that
\begin{align*}
p_1 = 1 - \sum_{\ell = 1}^{k  -1} (-1)^{\ell - 1} \binom{k}{\ell}\frac{\binom{n - \frac{\ell n}{k}}{m}}{\binom{n}{m}}.
\end{align*}
Clearly, we then have 
\begin{align*}
\mathbb{P}(\Omega_1) \geq 1 - \sum_{j = 1}^t\mathbb{P}(G_j^c) \geq 1 - t(1 - p_1),
\end{align*}
as claimed. 

Now recall that $k, \zeta \asymp 1$ and $t \leq n^{\beta_0}$, so $m \geq n^{1 - \beta_0}$. Observe that for $\ell \in \{1, \dots, k - 1\}$ and $g_0 = \frac{\ell}{k} \asymp 1$, we have
\begin{align*}
\frac{\binom{n - g_0n}{m}}{\binom{n}{m}} = \prod_{i = 1}^{g_0n}\frac{n - m - g_0n + i}{n - g_0n + i} = \prod_{i = 1}^{g_0n}\left(1 - \frac{m}{n - g_0n + i}\right).
\end{align*}
Then 
\begin{align*}
\Omega\left(3^{-\frac{g_0m}{1 - g_0}}\right) = \left(1 - \frac{m}{n - g_0n}\right)^{g_0n} \leq \frac{\binom{n - g_0n}{m}}{\binom{n}{m}} \leq \left(1 - \frac{m}{n}\right)^{g_0n} = O\left(2^{-g_0m}\right).
\end{align*}
With the given scaling of $k, t$, and $m$ in terms of $n$, we have $-\log\left(\frac{\binom{n - g_0n}{m}}{\binom{n}{m}}\right) \asymp m \succsim n^{1 - \beta_0}$. Thus, $1 - p_1$ is a sum of $k - 1 \asymp 1$ terms that decay exponentially to $0$, with exponent $\Omega(n^{1 - \beta_0})$. Since $t \leq n^{\beta_0}$ and $\beta_0 \in (1/2, 1)$, we have $t(1 - p_1) \leq \frac{1}{\mathrm{poly}(n)}$.


\subsubsection{Proof of Lemma~\ref{lemma:Pi^j-Pi}}
\label{AppLemPij}

Recall that $z_i \sim N(0, I_n)$. Fix $j \in [t]$. By Lemma \ref{lemma:hp_all_types}, the projection onto the row space of $P_j$ is $\Pi$. We have 
\begin{align*}
||\Pi^{(j)} - \Pi||_2 &= ||U^{(j), (k)}(U^{(j), (k)})^T - UU^T||_2 \\
& \stackrel{(a)}{\leq} 2||\sin(\Theta)(U^{(j), (k)}, U)||_2 \stackrel{(b)}{\leq} 2||\sin(\Theta)(U^{(j), (k)}, U)||_F\\
& \stackrel{(c)}{\leq} \mathop{\inf}_{Q: Q^TQ = QQ^T = I_k}||U^{(j), (k)}Q - U||_F \stackrel{(d)}{\leq} ||U^{(j), (k)}\hat{Q} - U||_F\\
&\precsim \frac{(\sigma_1(P_j) + ||\Pi_jA_j - P_j||_2)||\Pi_jA_j - P_j||_2}{\sigma_k^2(P_j)},
\end{align*}
where we write $\sigma_1(\cdot)$ and $\sigma_k(\cdot)$ to denote the largest and $k^{\text{th}}$ largest singular values, respectively. Inequality (a) follows from Lemma \ref{lemma:l2.3SinSte}. Inequality (b) holds because the Frobenius norm is greater than the operator norm. Inequality (c) holds because of Lemma \ref{lemma:prop2.2VuLei}. Inequality (d) and the existence of the $k \times k$ orthogonal matrix $\hat{Q}$ hold because of Lemma \ref{lemma:thm3YuEtal}. Note that $\Pi_jA_j$ is the best rank-$k$ approximation of $A_j$ in the operator norm, by Lemma~\ref{LemEYM}. Since $P_j$ has rank $k$, we have $||\Pi_jA_j - A_j||_2 \leq ||A_j - P_j||_2 \leq ||A - P||_2$, further implying that
\begin{align*}
||\Pi^{(j)} - \Pi||_2 \precsim \frac{(\sigma_1(P_j) + ||A - P||_2)||A - P||_2}{\sigma_k^2(P_j)}.
\end{align*}
Thus, by Lemma \ref{lemma:thm5.2LeiRin}, since $a_n \succsim \frac{\log(n)}{n}$, there is an event $\Omega_2$, with $\mathbb{P}(\Omega_2) \geq 1 - \frac{1}{n}$, such that
\begin{equation*}
||A - \mathbb{E}[A]||_2 = ||A - P + \text{diag}(P)||_2 \precsim \sqrt{na_n}
\end{equation*}
on $\Omega_2$. Let us work on $\Omega_2' = \Omega_1 \cap \Omega_2$. Then
\begin{align*}
||\Pi^{(j)} - \Pi||_2 \precsim \frac{\sqrt{k}(\sigma_1(P_j) + \sqrt{na_n} + a_n)(\sqrt{na_n} + a_n)}{\sigma_k^2(P_j)}.
\end{align*}

We now have the following result, proved in Appendix~\ref{AppLemSubSample}:

\begin{lemma}
\label{LemSubsampleSigma}
Suppose $m \rightarrow \infty$ as $n \rightarrow \infty$, and $P_m$ is a matrix obtained by randomly subsampling $m$ rows of $P$ without replacement, i.e., $P_m = E_mP$, where $E_m \in \{0, 1\}^{m \times n}$ ($E_m$ selects a uniformly random subset of rows, of size $m$). Suppose $k \asymp 1$.
With probability at least $1 - 2ke^{-\frac{m}{2k^2}}$, we have
\begin{align*}
    \sigma_{k}(P_m) & \succsim \sigma_{k}(D_P) \cdot \sqrt{\frac{m}{n}}, \\
    \sigma_{\max}(P_m) & \precsim \sigma_{\max}(D_P) \cdot \sqrt{\frac{m}{n}},
\end{align*}
where $P = UD_PU^T$ is the spectral decomposition of $P$, with $D_P \in \real^{k \times k}$.
\end{lemma}

By Lemma~\ref{LemSubsampleSigma}, there is an event $\Omega_3$, with probability at least $1 - 2kte^{-\frac{m}{2k^2}}$ (a union bound over all $j \in [t]$), such that we have
\begin{align*}
\sigma_k(P) \cdot \sqrt{\frac{m}{n}} \precsim \sigma_{k}(P_j) \leq \sigma_{1}(P_j) & \precsim \sigma_1(P) \sqrt{\frac{m}{n}}.
\end{align*}
Let us work on $\Omega_3' = \Omega_3 \cap \Omega_2'$. Recall that $P = J(B \otimes J_{\frac{n}{k}})J^T$, where $J$ is a permutation matrix. Then $\sigma_1(P) = \frac{na_n}{k}\sigma_1(B_0)$ and $\sigma_k(P) = \frac{na_n}{k}\sigma_k(B_0) > 0$. Since $B_0$ does not change with $n$, we then have
\begin{align*}
a_n \sqrt{mn} \precsim \sigma_{k}(P_j) \leq \sigma_{1}(P_j) & \precsim a_n \sqrt{mn}.
\end{align*}
Therefore, since $d > \sqrt{n}\log^{11/2}(n)$, we have 
\begin{align*}
||\Pi^{(j)} - \Pi||_2 &\precsim \frac{(a_n \sqrt{mn} + \sqrt{na_n} + a_n)(\sqrt{na_n} + a_n)}{a_n^2mn} \precsim \frac{1}{\sqrt{a_nm}} + \frac{1}{a_nm} \asymp \sqrt{\frac{t}{d}} + \frac{t}{d}\\
&\precsim \frac{1}{\log^{5/2}(n)\rho^{1/4}} + \frac{1}{\log^5(n)\sqrt{\rho}} \asymp \frac{\log^{1/4}(1/\delta)}{\log^{5/2}(n)\sqrt{\epsilon}} + \frac{\sqrt{\log(1/\delta)}}{\log^5(n)\epsilon} = \xi_0,
\end{align*}
for all $j \in [t]$. Now we follow the argument in Corollary 4.6 in Singhal \& Steinke~\cite{SinSte21}.
Let $\hat{\Pi}^{(j)}$ be the projection matrix for the union of the subspace spanned by the columns of $U^{(j), (k)}$ and the subspace spanned
by the columns of $U$. By Lemma \ref{lemma:lem2.10SinSte}, we have $||\hat{\Pi}^{(j)}z_i||_2 \precsim \sqrt{k} + \sqrt{\log(qt)} \precsim \sqrt{\log(n)}$, for all $(i, j)$, on an event $\Omega_{41}$, with $\mathbb{P}(\Omega_{41}) \geq 0.99$. Again by Lemma \ref{lemma:lem2.10SinSte} we have $||\Pi z_i||_2 \precsim \sqrt{k} + \sqrt{\log(qt)} \precsim \sqrt{\log(n)}$, for all $(i, j)$, on an event $\Omega_{42}$, with $\mathbb{P}(\Omega_{42}) \geq 0.99$. We also have, by Lemma \ref{lemma:lem2.10SinSte}, that $||\Pi^{(j)}z_i||_2 \precsim \sqrt{k} + \sqrt{\log(qt)} \precsim \sqrt{\log(n)}$, for all $(i, j)$, on an event $\Omega_{43}$, with $\mathbb{P}(\Omega_{43}) \geq 0.99$. Let us work on $\Omega_4' = \Omega_{41} \cap \Omega_{42} \cap \Omega_{43} \cap \Omega_3'$. Hence, we have
\begin{align*}
||(\Pi^{(j)} - \Pi)z_i||_2 = ||(\Pi^{(j)} - \Pi)\hat{\Pi}^{(j)}z_i||_2 \leq ||\Pi^{(j)} - \Pi||_2||\hat{\Pi}^{(j)}z_i||_2 \precsim \xi_0 \sqrt{\log(n)},
\end{align*}
with probability at least $0.9 - \tau(n)$, where $\tau(n) \le \frac{1}{\mathrm{poly}(n)} + \frac{1}{n} + 2kte^{-\frac{m}{2k^2}}$. Recall that we assumed $m \geq n^{1 - \beta_0}$ and $1 \asymp \beta_0 \in (1/2, 1)$, so $2kte^{-\frac{m}{2k^2}} \leq 2kn^{\beta_0}e^{-\frac{n^{1 - \beta_0}}{2k^2}}$. For $\zeta, k \asymp 1$ as in Lemma \ref{lemma:hp_all_types}, we have $\tau(n) \leq \frac{1}{\mathrm{poly}(n)}$.


\subsubsection{Proof of Lemma~\ref{lemma:Zhat_Z}}
\label{AppLemZhat}

Define $\bar{Z} := \left[\frac{1}{t}\sum_{j = 1}^tz_1^j, \dots, \frac{1}{t}\sum_{j = 1}^tz_q^j\right] \in \mathbb{R}^{n \times q}$,
\begin{align*}
E_0 & := \bar{Z} - Z, \\
E_1 & := \hat{Z} - \bar{Z}, \\
E & := E_0 + E_1 = \hat{Z} - Z.
\end{align*}
We analyze GoodCenter and argue that, with high probability, no truncation is necessary. First, the randomness on $\Omega_4'$ is independent from the randomness for privatizing $\mathrm{GoodCenter}_{X_i, 0, R_{max}, \frac{r_{min}}{2}, \frac{\zeta}{q}, \rho}$, for every $i \in [q]$. Hence, we argue conditionally on $\Omega_4'$. Fix $i \in [q]$. Let $z_i^{grid}$ be the closest point on the grid $\left[-\frac{R_{max}}{\sqrt{n}}, \frac{R_{max}}{\sqrt{n}}\right]^n = \left[-\log(n), \log(n)\right]^n$, with spacing $r_{min}$, to $\Pi z_i$. The following result is proved in Appendix~\ref{AppLemPij}:

\begin{lemma}
\label{lemma:Pi^j-Pi}
Let $d > \sqrt{n}\log^{11/2}(n)$. For all $i \in [q]$ and $j \in [t]$, we have
\begin{align*}
&\left\|\Pi^{(j)}z_i\right\|_2, ||\Pi z_i||_2 \precsim \sqrt{\log(n)},\\
&\|(\Pi^{(j)} - \Pi)z_i\|_2 \precsim \xi_0\sqrt{\log(n)},
\end{align*}
with probability at least $0.9 - \frac{1}{\mathrm{poly}(n)}$, where 
\begin{align*}
\xi_0 := \frac{\log^{1/4}(1/\delta)}{\log^{5/2}(n)\sqrt{\epsilon}} + \frac{\sqrt{\log(1/\delta)}}{\log^5(n)\epsilon}.
\end{align*}
\end{lemma}

First, we know that $\{z_i^{j, grid}\}_{j = 1}^t$, and $z_i^{grid}$ lie on the grid $\left[-\frac{R_{max}}{\sqrt{n}}, \frac{R_{max}}{\sqrt{n}}\right]^n$ with spacing $r_{min}$. By Lemma \ref{lemma:Pi^j-Pi}, we have $||z_i^j||_2, ||\Pi z_i||_2 \precsim \sqrt{\log(n)}$ and $||z_i^{j} - \Pi z_i||_2 \precsim \xi_0\sqrt{\log(n)}$, for all $j \in [t]$. So for $n$ large enough, $\{z_i^j\}_{j = 1}^t$ and $\Pi z_i$ lie inside the surface enclosed by the grid $\left[-\frac{R_{max}}{\sqrt{n}}, \frac{R_{max}}{\sqrt{n}}\right]^n$. Since $z_i^{j, grid}$, and $z_i^{grid}$ are the closest points on the grid to $z_i^j$, and $\Pi z_i$, respectively, and since $z_i^{j, grid}, z_i^{grid}$ lie on the grid themselves, we obtain, for all $j \in [t]$, that $||z_i^j - z_i^{j, grid}||_2, ||\Pi z_i - z_i^{grid}||_2 \precsim r_{min}\sqrt{n}$. Using Lemma \ref{lemma:Pi^j-Pi}, we have
\begin{align}
\label{eq:bdddd}
||z_i^{j, grid} - z_i^{grid}||_2, ||z_i^{j, grid} - z_i^{l, grid}||_2 \precsim r_{min}\sqrt{n} + \xi_0\sqrt{\log(n)},
\end{align}
for all $j, l \in [t]$. Next, for our choice of $t$ in equation~\eqref{eq:t}, we obtain from Lemma \ref{lemma:cor_a3_supp} that with probability at least $1 - \zeta/q$, we have $||z_i^{j, grid} - \theta_i^*||_2 \leq 7r_{opt}(X_i')$, for all $z_i^{j, grid} \in X_i'$, where $X_i'$ is the subset with $|X_i'| \geq \frac{t}{2}$. Hence, since $X_i' \subseteq X_i$, we have, using inequality \eqref{eq:bdddd}, that $||z_i^{j, grid} - \theta_i^*||_2 \precsim r_{min}\sqrt{n} + \xi_0\sqrt{\log(n)}$. Now, for $z_i^{j, grid} \notin X_i'$, we have
\begin{equation*}
||z_i^{j, grid} - \theta_i^*||_2 \leq ||z_i^{j, grid} - z_i^{j', grid}||_2 + ||z_i^{j', grid} - \theta_i^*||_2 \precsim r_{min}\sqrt{n} + \xi_0\sqrt{\log(n)},
\end{equation*}
where $z_i^{j', grid}$ is any point in $X_i'$. Therefore, $||z_i^{j, grid} - \theta_i^*||_2 \precsim r_{min}\sqrt{n} + \xi_0\sqrt{\log(n)}$, for all $j \in [t]$, so
\begin{equation*}
||z_i^{j} - \theta_i^*||_2 \leq ||z_i^{j, grid} - z_i^j||_2 + ||z_i^{j, grid} - \theta_i^*||_2 \precsim r_{min}\sqrt{n} + \xi_0\sqrt{\log(n)}.
\end{equation*}
Thus, for $n$ large enough, we have
\begin{equation*}
||z_i^{j} - \theta_i^*||_2 \leq r_{min}\sqrt{n\log(n)} + \xi_0\log(n) = r.
\end{equation*}
Hence, all $\{z_i^j\}_{j = 1}^t$ lie in $B(\theta_i^*, r)$. Taking a union bound over all $i$, we do not have to truncate at all on an event $\Omega_5$, with $\mathbb{P}(\Omega_5|\Omega_4')\geq 1 - \zeta$. Let us work on $\Omega_5' = \Omega_5 \cap \Omega_4'$. Then $\mathbb{P}(\Omega_5') \geq 0.9 - \zeta - \frac{1}{\mathrm{poly}(n)}$. 

Now note that on the no-truncation event, $E_1$ has independent Gaussian columns, so
\begin{align*}
||E_1||_2 & \precsim \sigma(\sqrt{n} + \sqrt{k} + \sqrt{\log(2/\zeta)}) \precsim \sigma(\sqrt{n} + \sqrt{\log(2/\zeta)}) \\
& \precsim \frac{r\sqrt{n}}{t\sqrt{\rho}} \asymp \frac{r}{\sqrt{\log(n)}} =\frac{1}{n^2\sqrt{n}} + \xi_0\sqrt{\log(n)},
\end{align*}
with probability at least $1 - \zeta$, by Lemma \ref{lemma:cor2.7SinSte}, since $q \asymp k$. Call this event $\Omega_6$ and let us work on $\Omega_6' = \Omega_6 \cap \Omega_5'$. We also have
\begin{align*}
||E_0||_2 &= ||\bar{Z} - Z||_2 = \left\|\left[\frac{1}{t}\sum_{j = 1}^t(\Pi^{(j)} - \Pi)z_1, \dots ,\frac{1}{t}\sum_{j = 1}^t(\Pi^{(j)} - \Pi)z_q\right]\right\|_2 \\
& \leq \sqrt{q}\mathop{\max}\limits_{i \in [q]}\frac{1}{t}\sum_{j = 1}^t\left\|(\Pi^{(j)} - \Pi)z_i\right\|_2 \precsim \xi_0\sqrt{\log(n)}.
\end{align*}
Hence, we have $||E||_2 \precsim \frac{1}{n^2\sqrt{n}} + \xi_0\sqrt{\log(n)}$, as required.


\subsubsection{Proof of Lemma~\ref{lemma:Pi_Z}}
\label{AppLemPiZ}

Recall that $Z = [\Pi z_1, \dots, \Pi z_q]$. Thus, the rank of $Z$ is at most the rank of $\Pi$, i.e., $k$. Hence, $\sigma_{k + 1}(Z) = 0$. Note that the $k^{\text{th}}$ singular value can be written as follows:
\begin{align*}
\sigma_k(Z) = \mathop{\max}\limits_{\mathcal{S}: \dim(\mathcal{S}) = k}\mathop{\min}\limits_{v \in \mathcal{S}, ||v||_2 = 1}||Zv||_2.
\end{align*}
Since $\Pi = UU^T$ and $U$ has $k$ orthonormal columns, we obtain
\begin{align*}
\sigma_k(Z) = \mathop{\max}\limits_{\mathcal{S}: \dim(\mathcal{S}) = k}\mathop{\min}\limits_{v \in \mathcal{S}, ||v||_2 = 1}\left\|[U^Tz_1, \dots, U^Tz_q]v\right\|_2.
\end{align*}
Hence, $\sigma_k(Z)$ is the $k^{\text{th}}$ singular value of a $k \times q$ matrix with independent columns distributed as $N(0, I_k)$, since $U^TU = I_k$. Hence, by Lemma \ref{lemma:cor2.7SinSte}, we have
$\sigma_k(Z) \asymp \sqrt{k} \asymp 1$, on an event $\Omega_7$ with probability at least $1 - \zeta$. This is because we assumed $q = C'k$ (for $C' \asymp 1$ as large as we like), and since $3k \geq \log(2/\zeta)$, we can upper-bound $\sqrt{\log(2/\zeta)}$ by $\sqrt{3k}$, up to absolute constants. Let us work on $\Omega_7' = \Omega_7 \cap \Omega_6'$. Now, by Lemma \ref{lemma:l2.1SinSte}, we have
\begin{equation*}
\sigma_k(Z') \leq \sigma_k(Z) + ||E_Z||_2 \leq \sigma_k(Z) + ||E||_2 \precsim \sqrt{k} \asymp 1,
\end{equation*}
since $\xi_0\sqrt{\log(n)} = o(1)$. By Lemma \ref{lemma:l2.1SinSte}, we have
\begin{equation*}
\sigma_k(Z') \geq \sigma_k(Z) - ||E_Z||_2 \geq \sigma_k(Z) - ||E||_2 \succsim \sqrt{k} \asymp 1.
\end{equation*}
Hence, $\sigma_k(Z') \asymp \sqrt{k} \asymp 1$. By the construction of $Z'$, its column span is in $\text{col}(Z)$, which is in $\text{col}(P)$, since $\Pi$ is the projection onto the columns of $P$. But $\text{col}(P)$ has dimension $k$, and we also showed that $\sigma_k(Z') > 0$. Hence, the dimension of $\text{col}(P)$ is the same as the dimension of $\text{col}(Z')$, i.e., $k$. Hence, the columns of $Z'$ and the columns/rows of $P$ span the same subspace. We conclude that $\Pi$ is also the projection onto $\text{col}(Z')$, as required.


\subsubsection{Proof of Lemma~\ref{LemSubsampleSigma}}
\label{AppLemSubSample}

By Lemma 2.1 of Lei \& Rinaldo~\cite{LeiRin15}, we have $U = \Theta X$, where $\Theta \in \{0,1\}^{n \times k}$ is the membership matrix of nodes and $X = \sqrt{\frac{k}{n}} Z$, for $Z \in \real^{k \times k}$ an orthonormal matrix. In fact, we have $D_P = \frac{n}{k} \cdot Z^TBZ$, for an orthonormal matrix $Z$, so the eigenvalues/singular values of $D_P$ are $\frac{n}{k}$ times the eigenvalues of $B$.

We have, since $||U^T||_2 \leq 1$, that
\begin{align*}
\sigma_{\max}(P_m) &= ||P_m||_2 = ||E_m\Theta XD_PU^T||_2 \\
& \leq ||E_m\Theta X||_2 \cdot \sigma_{max}(D_P) = ||E_m\Theta ||_2 \cdot \sigma_{max}(D_P)\sqrt{\frac{k}{n}}.
\end{align*}
Thus, it remains to upper-bound the maximum singular value of $E_m \Theta$. Also recall that if the singular values of a matrix $V$ are all non-zero, they coincide with the square root of the non-zero eigenvalues of both $V^TV$ and $VV^T$. Note that
\begin{equation*}
(E_m \Theta)^T (E_m \Theta) = \mathrm{diag}(m_1, \dots, m_k),
\end{equation*}
where each $m_i$ is the number of times rows from community $i$ are chosen in the subsample of $m$ rows.
Then $m_i \sim \mathrm{HG}\left(n, \frac{n}{k}, m\right)$, for all $i \in [k]$. Note that $\mathbb{E}[m_i] = \frac{m}{k}$. Let $\Omega_1'' = \left\{\max\{m_1, \dots, m_k\} < \frac{2m}{k}\right\}$ and $\Omega_1''' = \left\{\min\{m_1, \dots, m_k\} > \frac{m}{2k}\right\}$. Therefore, by Lemma \ref{lemma:hypergeom_conc}, we have
\begin{align*}
\mathbb{P}\left((\Omega_1'')^{c}\right) \leq \sum_{i = 1}^k\mathbb{P}\left(m_i - \mathbb{E}[m_i] \geq \frac{m}{k}\right) \leq ke^{-\frac{2m}{k^2}} \leq ke^{-\frac{m}{2k^2}}.
\end{align*}
Similarly, we have
\begin{align*}
\mathbb{P}\left((\Omega_1''')^c\right) \leq \sum_{i = 1}^k\mathbb{P}\left(m_i - \mathbb{E}[m_i] \leq -\frac{m}{2k}\right) \leq ke^{-\frac{m}{2k^2}}.
\end{align*}
Hence, on $\Omega_1'' \cap \Omega_1'''$, we have
\begin{equation*}
\max\{m_1, \dots, m_k\} \precsim\frac{m}{k}.
\end{equation*}
Plugging in gives
\begin{equation*}
\sigma_{\max}(P_m) \precsim \sigma_{\max}(D_P) \cdot \sqrt{\frac{m}{n}},
\end{equation*}
as wanted.

For the lower bound, recall that the $k^{\text{th}}$ singular value can be obtained as follows:
\begin{align*}
\sigma_k(P_m) = \mathop{\max}\limits_{\mathcal{S}: \dim(\mathcal{S}) = k}\mathop{\min}\limits_{v \in \mathcal{S}, ||v||_2 = 1}||E_m\Theta XD_PU^Tv||_2.
\end{align*}
For $\mathcal{S} = \text{col}(U)$, which has dimension $k$, we have
\begin{align*}
\sigma^2_k(P_m) \geq \mathop{\min}\limits_{v \in \mathcal{S}, ||v||_2 = 1}||E_m\Theta XD_PU^Tv||_2^2.
\end{align*}
Write $U = [U_1, \dots, U_k] \in \mathbb{R}^{n \times k}$, and let the minimum be achieved by $v^* = \sum_{l = 1}^ka_l^*U_l$, with $a^* = [a_1^*, \dots, a_k^*]^T$ and $||a^*||_2 = 1$. Then
\begin{align*}
\sigma_{k}^2 (P_m) &\geq ||E_m\Theta XD_PU^Tv^*||_2^2 = ||E_m\Theta XD_Pa^*||_2^2 = (a^*)^TD_PX^T\mathrm{diag}(m_1, \dots, m_k)XD_Pa^*\\
&\geq ||XD_Pa^*||_2^2\min\{m_1, \dots, m_k\} = \frac{k}{n}||D_Pa^*||_2^2\min\{m_1, \dots, m_k\}\\
&\geq \frac{k\sigma^2_{k}(D_P)}{n}\min\{m_1, \dots, m_k\},
\end{align*}
where we used the fact that $X^TX = \frac{k}{n}I_k$ and $||a^*||_2 = 1$. On $\Omega_1'' \cap \Omega_1'''$ we also have 
\begin{equation*}
\min\{m_1, \dots, m_k\} \succsim \frac{m}{k},
\end{equation*}
from which the desired lower bound follows.


\subsection{Proofs for Section~\ref{section:Approximate DP: Private approximate subspace estimation} (weighted case)}

\subsubsection{Proof of Theorem~\ref{theorem:GoodCent_Main_Thm_Weighted}}
\label{AppGoodCWt}

We start with the privacy guarantee. Fix $G \sim_{node} G'$, with $G, G' \in \mathcal{W}_n$. To simplify notation, let $\mathcal{A}^{(G)} := \mathcal{A}_{SE}^{(\epsilon/(5D(\hat{L}_{\epsilon_1, \delta_1}(\mathcal{B}(G)) + 1)), \delta/(\hat{L}_{\epsilon_1, \delta_1}(\mathcal{B}(G)) + 1))}$. Define
\begin{align*}
A_G & := (\hat{L}_{\epsilon_1, \delta_1}(\mathcal{B}(G)), \mathcal{A}^{(G)}(T_{w, D}(G))), \\
A_{G'} & := (\hat{L}_{\epsilon_1, \delta_1}(\mathcal{B}(G')), \mathcal{A}^{(G')}(T_{w, D}(G'))), \\
A_m & := (\hat{L}_{\epsilon_1, \delta_1}(\mathcal{B}(G)), \mathcal{A}^{(G)}(T_{w, D}(G'))).
\end{align*}
Let $P_G$, $P_{G'}$, and $P_m$ be the distributions of $A_G$, $A_{G'}$, and $A_m$, respectively. 
Let $E$ be an arbitrary measurable set in the image of $(\hat{L}_{\epsilon_1, \delta_1}(\mathcal{B}(\cdot)), \mathcal{A}^{(\cdot)}( T_{w, D}(\cdot)))$. By the privacy of $\hat{L}_{\epsilon_1, \delta_1}$ we have
\begin{align*}
P_{G'}(E) = \mathbb{P}\left((\hat{L}_{\epsilon_1, \delta_1}(\mathcal{B}(G')), \mathcal{A}^{(G')}(T_{w,D}(G'))) \in E\right) \leq e^{\epsilon_1}P_m(E). 
\end{align*}
Let $F = \left\{\hat{L}_{\epsilon_1, \delta_1}(\mathcal{B}(G)) > LS_{T_{D}}(\mathcal{B}(G))\right\}$. Since $\mathcal{B}(G) \in \mathcal{G}_n$, we have from Lemma \ref{lemma:Lhat} that $\mathbb{P}(F^c) \leq \delta_1$. 

We need to bound $d_{node}(T_{w,D}(G'), T_{w,D}(G))$. When altering the weights from one node in $G'$, call it $i$, to reach $G$, there are $3$ types of changes that can occur:
\begin{enumerate}
    \item change a non-zero weight to a non-zero weight \label{change:nn}
    \item change a non-zero weight to $0$ \label{change:n0}
    \item change a $0$ weight (same as no edge) to an edge with some non-zero weight. \label{change:0n}
\end{enumerate}
Say $G$ and $G'$ differ in node $i$. The number of nodes that differ between $T_{w,D}(G')$ and $T_{w,D}(G)$ because of changes \ref{change:n0} and \ref{change:0n} is at most $LS_{T_{D}}(\mathcal{B}(G))$. Next, because of change \ref{change:nn}, there can be one additional change. This happens if changes incident from node $i$ are not counted by changes \ref{change:n0} and \ref{change:0n}, after applying $\mathcal{B}(\cdot)$, and before adding the weights back. Therefore, we have
\begin{align}
\label{eq:d_node_LST1}
d_{node}(T_{w,D}(G), T_{w,D}(G')) \leq LS_{T_{D}}(\mathcal{B}(G)) + 1.    
\end{align}

Let $y_{\epsilon_1, \delta_1}$ be the pdf of $Y_{\epsilon_1, \delta_1}(\mathcal{B}(G)) = 5 + 2d_{T_{D}}(\mathcal{B}(G)) + \mathrm{Lap}\left(\frac{8}{\epsilon_1}\right) + \frac{8\log(1/\delta_1)}{\epsilon_1}$, so $\hat{L}_{\epsilon_1, \delta_1}(\mathcal{B}(G)) = \max\left\{\frac{1}{2}, Y_{\epsilon_1, \delta_1}(\mathcal{B}(G))\right\}$. If $LS_{T_D}(\mathcal{B}(G)) = 0$, then $d_{node}(T_{w,D}(G), T_{w,D}(G')) \in \{0, 1\}$. If $d_{node}(T_{w,D}(G), T_{w,D}(G')) = 1$, then 
\begin{align*}
&P_m(E)\\
&= \int_{-\infty}^\infty\mathbb{P}((\hat{L}_{\epsilon_1, \delta_1}(\mathcal{B}(G)), \mathcal{A}^{(G)}(T_{w, D}(G')))\in E|Y_{\epsilon_1, \delta_1}(\mathcal{B}(G)) = x)y_{\epsilon_1, \delta_1}(x)dx\\
&= \int_{-\infty}^\infty\mathbb{P}((\max\{1/2, x\}, \mathcal{A}_{SE}^{(\epsilon/(5D(\max\{1/2, x\} + 1)), \delta/(\max\{1/2, x\} + 1))}(T_{w, D}(G')))\in E|Y_{\epsilon_1, \delta_1}(\mathcal{B}(G)) = x) \\
& \qquad \cdot y_{\epsilon_1, \delta_1}(x)dx\\
&= \int_{-\infty}^\infty\mathbb{P}((\max\{1/2, x\}, \mathcal{A}_{SE}^{(\epsilon/(5D(\max\{1/2, x\} + 1)), \delta/(\max\{1/2, x\} + 1))}(T_{w, D}(G')))\in E)y_{\epsilon_1, \delta_1}(x)dx.
\end{align*}
Since $2\epsilon < \log(1/(\delta \cdot \mathrm{poly}(n)))$, we have $\frac{2\epsilon}{\max\{1/2, x\} + 1} < \log\left(\frac{\max\{1/2, x\} + 1}{\delta}\right)$. Therefore, by Lemma \ref{lemma:GoodC_Priv}, we know that $\mathcal{A}_{SE}^{(\epsilon/(5D(\max\{1/2, x\} + 1)), \delta/(\max\{1/2, x\} + 1))}$ is $\left(\frac{2\epsilon}{\max\{1/2, x\} + 1}, \frac{\delta}{\max\{1/2, x\} + 1}\right)_{2D}$-node DP. Thus, we have
\begin{align*}
&P_m(E)\\
&\leq e^{2\epsilon}\int_{-\infty}^\infty\mathbb{P}((\max\{1/2, x\}, \mathcal{A}_{SE}^{(\epsilon/(5D(\max\{1/2, x\} + 1)), \delta/(\max\{1/2, x\} + 1))}(T_{w, D}(G)))\in E)y_{\epsilon_1, \delta_1}(x)dx + \delta\\
&= e^{2\epsilon}\int_{-\infty}^\infty\mathbb{P}((\max\{1/2, x\}, \mathcal{A}_{SE}^{(\epsilon/(5D(\max\{1/2, x\} + 1)), \delta/(\max\{1/2, x\} + 1))}(T_{w, D}(G)))\in E|Y_{\epsilon_1, \delta_1}(\mathcal{B}(G)) = x) \\
& \qquad \cdot y_{\epsilon_1, \delta_1}(x)dx + \delta\\
&= e^{2\epsilon}\int_{-\infty}^\infty\mathbb{P}((\hat{L}_{\epsilon_1, \delta_1}(\mathcal{B}(G)), \mathcal{A}^{(G)}(T_{w, D}(G)))\in E|Y_{\epsilon_1, \delta_1}(\mathcal{B}(G)) = x) y_{\epsilon_1, \delta_1}(x)dx + \delta = e^{2\epsilon}P_G(E) + \delta\\
&\leq e^{2\epsilon}P_G(E) + \frac{1}{\mathrm{poly}(n)},
\end{align*}
implying that
\begin{align}
\label{eq:case_1}
P_{G'}(E) \leq e^{\epsilon_1}P_m(E) \leq e^{\epsilon_1 + 2\epsilon}P_G(E) +  \frac{e^{\epsilon_1}}{\mathrm{poly}(n)} \leq e^{\epsilon_1 + 2\epsilon}P_G(E) +  e^{\epsilon_1}\left(\frac{1}{\mathrm{poly}(n)} + \delta_1\right).
\end{align}
Next, if $d_{node}(T_{w,D}(G), T_{w,D}(G')) = 0$, then $T_{w,D}(G) = T_{w,D}(G')$ and
\begin{equation*}
P_m(E|F) = P_G(E|F) \leq e^{2\epsilon}P_{G}(E|F) + \frac{1}{\mathrm{poly}(n)}.
\end{equation*}
Now assume $LS_{T_D}(\mathcal{B}(G)) \geq 1$. Then we have
\begin{align*}
&P_{m}(E|F)\\
&= \frac{\mathbb{P}\left(\left\{(\hat{L}_{\epsilon_1, \delta_1}(\mathcal{B}(G)), \mathcal{A}^{(G)}(T_{w, D}(G')))\in E\right\} \bigcap F\right)}{\mathbb{P}(F)}\\
&= \int_{LS_{{T_{D}}}(\mathcal{B}(G))}^{\infty}\frac{\mathbb{P}\left((\max\{1/2, x\}, \mathcal{A}_{SE}^{(\epsilon/(5D(\max\{1/2, x\} + 1)), \delta/(\max\{1/2, x\} + 1))}(T_{w, D}(G'))) \in E|Y_{\epsilon_1, \delta_1}(\mathcal{B}(G)) = x\right)}{\mathbb{P}(F)} \\
& \qquad \cdot y_{\epsilon_1, \delta_1}(x) dx\\
&= \int_{LS_{{T_{D}}}(\mathcal{B}(G))}^{\infty}\frac{\mathbb{P}\left((x, \mathcal{A}_{SE}^{(\epsilon/(5D(x + 1)), \delta/(x + 1))}(T_{w, D}(G'))) \in E\right)y_{\epsilon_1, \delta_1}(x)}{\mathbb{P}(F)}dx,
\end{align*}
where we dropped the conditioning because $Y_{\epsilon_1, \delta_1}(\mathcal{B}(G))$ is independent of $\mathcal{A}_{SE}^{(\epsilon/(5D(x + 1)), \delta/(x + 1))}(T_{w, D}(G'))$. Also, since $\max\left\{\frac{1}{2}, x\right\} \geq LS_{T_D}(\mathcal{B}(G)) \geq 1$, we must have $x \geq LS_{T_D}(\mathcal{B}(G)) \geq 1$. Now fix $x \geq LS_{{T_{D}}}(\mathcal{B}(G))$. Since $2\epsilon < \log(1/(\delta \cdot \mathrm{poly}(n)))$, we have $\frac{2\epsilon}{x + 1} < \log\left(\frac{x + 1}{\delta}\right)$. So, by Lemma \ref{lemma:GoodC_Priv}, we know that $\mathcal{A}_{SE}^{(\epsilon/(5D(x + 1)), \delta/(x + 1))}$ is $\left(\frac{2\epsilon}{x + 1}, \frac{\delta}{x + 1}\right)_{2D}$-node DP, and it is clearly $\left(\frac{2\epsilon}{LS_{T_{D}}(\mathcal{B}(G)) + 1}, \frac{\delta}{LS_{T_{D}}(\mathcal{B}(G)) + 1}\right)_{2D}$-node DP. By group privacy of size $LS_{T_{D}}(\mathcal{B}(G)) + 1$, we obtain $\left(2\epsilon, \delta e^{2\epsilon}\right)_{2D}$-node DP on groups of size $LS_{T_{D}}(\mathcal{B}(G)) + 1$. Since $\delta < e^{-2\epsilon}/\mathrm{poly}(n)$, this is $(2\epsilon, 1/\mathrm{poly}(n))_{2D}$-node DP. Since, when changing $G'$ to $G$, we change at most $LS_{T_{D}}(\mathcal{B}(G)) + 1$ nodes to go from $T_{w, D}(G')$ to $T_{w, D}(G)$, we have
\begin{align*}
&P_{m}(E|F)\\
&\leq \int_{LS_{{T_{D}}}(\mathcal{B}(G))}^{\infty}\frac{\left(e^{2\epsilon}\mathbb{P}\left((x, \mathcal{A}_{SE}^{(\epsilon/(5D(x + 1)), \delta/(x + 1))}(T_{w, D}(G))) \in E\right) + 1/\mathrm{poly}(n)\right)y_{\epsilon_1, \delta_1}(x)}{\mathbb{P}(F)}dx\\
&= e^{2\epsilon}\int_{LS_{{T_{D}}}(\mathcal{B}(G))}^{\infty}\frac{\mathbb{P}\left((x, \mathcal{A}_{SE}^{(\epsilon/(5D(x + 1)), \delta/(x + 1))}(T_{w, D}(G))) \in E\right)y_{\epsilon_1, \delta_1}(x)}{\mathbb{P}(F)}dx + \frac{1}{\mathrm{poly}(n)}\\
&= e^{2\epsilon}\int_{LS_{{T_{D}}}(\mathcal{B}(G))}^{\infty}\frac{\mathbb{P}\left((\max\{1/2, x\}, \mathcal{A}_{SE}^{(\epsilon/(5D(\max\{1/2, x\} + 1)), \delta/(\max\{1/2, x\} + 1))}(T_{w, D}(G))) \in E|Y_{\epsilon_1, \delta_1}(\mathcal{B}(G)) = x\right)}{\mathbb{P}(F)}\\
& \qquad \cdot y_{\epsilon_1, \delta_1}(x) dx + \frac{1}{\mathrm{poly}(n)}\\
&= \frac{e^{2\epsilon}\mathbb{P}\left(\left\{(\hat{L}_{\epsilon_1, \delta_1}(\mathcal{B}(G)), \mathcal{A}^{(G)}(T_{w, D}(G)))\in E\right\} \cap F\right)}{\mathbb{P}(F)} + \frac{1}{\mathrm{poly}(n)} = e^{2\epsilon}P_{G}(E|F) + \frac{1}{\mathrm{poly}(n)}.
\end{align*}
Now note that $P_{G}(F) = P_{m}(F)$, since in both cases, we have $\hat{L}_{\epsilon_1, \delta_1}(\mathcal{B}(G))$, not $\hat{L}_{\epsilon_1, \delta_1}(\mathcal{B}(G'))$. Thus, we have
\begin{align*}
P_{m}(E \cap F) \leq e^{2\epsilon}P_{G}(E \cap F) + \frac{1}{\mathrm{poly}(n)}.
\end{align*}
Hence, we obtain
\begin{align*}
P_{G'}(E) &\leq e^{\epsilon_1}(P_m(E \cap F) + \delta_1) \leq e^{\epsilon_1}(e^{2\epsilon}P_G(E \cap F) + 1/\mathrm{poly}(n)) + e^{\epsilon_1}\delta_1\\
&\leq e^{\epsilon_1 + 2\epsilon}P_G(E) + e^{\epsilon_1}(1/\mathrm{poly}(n) + \delta_1).
\end{align*}
Using inequality \eqref{eq:case_1}, we deduce that $(\hat{L}_{\epsilon_1, \delta_1}(\mathcal{B}(G)), \mathcal{A}^{(G)}(T_{w, D}(G)))$ is $(\epsilon_1 + 2\epsilon, e^{\epsilon_1}(1/\mathrm{poly}(n) + \delta_1))$-node DP. Since $\epsilon_1 \asymp 1$ and $\epsilon$ grows with $n$, we have $\epsilon_1 \leq \epsilon$, for $n$ large enough. Hence, because $\delta_1 \asymp \frac{1}{\mathrm{poly}(n)}$ as well, we obtain the desired privacy guarantee.

For the utility guarantee, we rely on the following utility result for fixed privacy parameters, proved in Appendix~\ref{AppLemPihatWt}:

\begin{lemma}
\label{lemma:Pihat-Pi_weighted}
Assume $\frac{\epsilon^2}{\log(1/\delta)} \geq \frac{1}{\log^3(n)}$, $d > \sqrt{n}\log^{13/2}(n)$, and $s^2 \asymp a_n$. Suppose Assumption \ref{ASS_SE_t_m} and equation~\eqref{eq:t} hold. Let $\theta = \{\theta_i\}_{i = 1}^n$ be the true community assignments. Let $\mathcal{A}_{SE}^{(\epsilon, \delta)}(G)$ be the output of Algorithm \ref{alg:Subspace_Estimation_GoodC}. Then there exists an absolute constant $C_1 \asymp 1$ such that
\begin{align*}
\widetilde{\mathcal{L}}\left(\mathcal{A}_{SE}^{(\epsilon, \delta)}(G), \theta\right) \leq \frac{C_1\sqrt{\log(1/\delta)}}{\epsilon\log^4(n)} \precsim \frac{1}{\log^{5/2}(n)},
\end{align*}
with probability at least $0.9 - 3\zeta - \frac{1}{\mathrm{poly}(n)}$.
\end{lemma}

Let $I := \left[5 + \frac{4\log(1/\delta_1)}{\epsilon_1}, 5 + \frac{16\log(1/\delta_1)}{\epsilon_1}\right]$, $I_{min} := 5 + \frac{4\log(1/\delta_1)}{\epsilon_1}$, and $I_{max} := 5 + \frac{16\log(1/\delta_1)}{\epsilon_1}$. Let $X_{\epsilon_1, \delta_1} := 5 + \mathrm{Lap}\left(\frac{8}{\epsilon_1}\right) + \frac{8\log(1/\delta_1)}{\epsilon_1}$. For $F_1 = \left\{X_{\epsilon_1, \delta_1} \in I\right\}$, we have
\begin{align*}
 \mathbb{P}\left(F_1^c\right) \leq \mathbb{P}\left(\mathrm{Lap}\left(\frac{8}{\epsilon_1}\right) < -\frac{4\log(1/\delta_1)}{\epsilon_1}\right) + \mathbb{P}\left(\mathrm{Lap}\left(\frac{8}{\epsilon_1}\right) > \frac{8\log(1/\delta_1)}{\epsilon_1}\right) \leq 2\sqrt{\delta_1}.
\end{align*}
Let $h_{\epsilon, \delta} = \frac{\sqrt{\log(1/\delta)}}{\epsilon\log^4(n)}$ and let $C_1$ be as in Lemma \ref{lemma:Pihat-Pi_weighted}. Let $f_{X_{\epsilon_1, \delta_1}}$ be the pdf of $X_{\epsilon_1, \delta_1}$. Also, let $\Omega_0$ be as in the proof of Lemma \ref{lemma:hp_removal_d}. We have $\mathbb{P}(\Omega_0) \geq 1 - \frac{1}{\mathrm{poly}(n)}$. In addition, on $\Omega_0$, we do not have to truncate the graph, since the degree of every node in the unweighted $A$ is less than $2d$. Thus, the degree of every node in $W\odot A$ is also less than $2d$. Hence, on $\Omega_0$, because of the definition of $T_{w, D}$, we have $T_{w,D}(G) = G$ and $d_{T_{D}}(\mathcal{B}(G)) = 0$. Hence, since $X_{\epsilon_1, \delta_1} \geq 5 > \frac{1}{2}$ on $F_1$, we have
\begin{align*}
&\mathbb{P}\left(\widetilde{\mathcal{L}}\left(\thetahat(G), \theta\right) \leq C_1h_{\frac{\epsilon}{5D(I_{max} + 1)}, \frac{\delta}{I_{max} + 1}}\right)\\
&\geq \mathbb{P}\left(\left\{\widetilde{\mathcal{L}}\left(\thetahat(G), \theta\right) \leq C_1h_{\frac{\epsilon}{5D(I_{max} + 1)}, \frac{\delta}{I_{max} + 1}}\right\} \cap \Omega_0\right)\\
&\geq \mathbb{P}\left(\left\{\widetilde{\mathcal{L}}\left(\mathcal{A}_{SE}^{(\epsilon/(5D(\max\{1/2, X_{\epsilon_1, \delta_1}\} + 1)), \delta/(\max\{1/2, X_{\epsilon_1, \delta_1}\} + 1))}(G), \theta\right) \leq C_1h_{\frac{\epsilon}{5D(I_{max} + 1)}, \frac{\delta}{I_{max} + 1}}\right\} \cap \Omega_0 \cap F_1\right)\\
&= \mathbb{P}\left(\left\{\widetilde{\mathcal{L}}\left(\mathcal{A}_{SE}^{(\epsilon/(5D(X_{\epsilon_1, \delta_1} + 1)), \delta/(X_{\epsilon_1, \delta_1} + 1))}(G), \theta\right) \leq C_1h_{\frac{\epsilon}{5D(I_{max} + 1)}, \frac{\delta}{I_{max} + 1}}\right\} \cap \Omega_0 \cap F_1\right)\\
&\geq \int_I\mathbb{P}\left(\widetilde{\mathcal{L}}\left(\mathcal{A}_{SE}^{(\epsilon/(5D(X_{\epsilon_1, \delta_1} + 1)), \delta/(X_{\epsilon_1, \delta_1} + 1))}(G), \theta\right) \leq C_1h_{\frac{\epsilon}{5D(I_{max} + 1)}, \frac{\delta}{I_{max} + 1}}\middle\vert X_{\epsilon_1, \delta_1} = x\right)f_{X_{\epsilon_1, \delta_1}}(x)dx \\
& \qquad - \frac{1}{\mathrm{poly}(n)}\\
&= \int_I\mathbb{P}\left(\widetilde{\mathcal{L}}\left(\mathcal{A}_{SE}^{(\epsilon/(5D(x + 1)), \delta/(x + 1))}(G), \theta\right) \leq C_1h_{\frac{\epsilon}{5D(I_{max} + 1)}, \frac{\delta}{I_{max} + 1}}\right)f_{X_{\epsilon_1, \delta_1}}(x)dx - \frac{1}{\mathrm{poly}(n)}\\
&\geq \int_I\mathbb{P}\left(\widetilde{\mathcal{L}}\left(\mathcal{A}_{SE}^{(\epsilon/(5D(x + 1)), \delta/(x + 1))}(G), \theta\right) \leq C_1h_{\frac{\epsilon}{5D(x + 1)}, \frac{\delta}{x + 1}}\right)f_{X_{\epsilon_1, \delta_1}}(x)dx - \frac{1}{\mathrm{poly}(n)},
\end{align*}
since the function $h_{\frac{\epsilon}{5D(x + 1)}, \frac{\delta}{x + 1}}$ is increasing in $x > 0$. We also dropped the conditioning on $\left\{X_{\epsilon_1, \delta_1} = x\right\}$, since the randomness in $X_{\epsilon_1, \delta_1}$ only depends on the $\mathrm{Lap}\left(\frac{8}{\epsilon_1}\right)$ noise, and this is independent of $\widetilde{\mathcal{L}}\left(\mathcal{A}_{SE}^{(\epsilon/(5D(x + 1)), \delta/(x + 1))}(G), \theta\right)$. We want to use Lemma \ref{lemma:Pihat-Pi_weighted}, with $\frac{\epsilon}{5D(x + 1)}, \frac{\delta}{x + 1}$, and $x \in [I_{min}, I_{max}]$. For this, we need $d > \sqrt{n}\log^{13/2}(n)$, and 
\begin{align*}
\max\left\{(C^{(1)})^2n^{1 - 2\beta_0}\log(n), \frac{1}{\log^3(n)}\right\} < \frac{\epsilon^2}{25D^2(x + 1)^2\log((x + 1)/\delta)} \leq \frac{(C^{(1)})^2n\log(n)}{4}
\end{align*}
(Assumption~\ref{ASS_SE_t_m} holds with $\frac{\epsilon}{5D(x + 1)}$), which is equivalent to 
\begin{align*}
\frac{1}{\log^3(n)} < \frac{\epsilon^2}{25D^2(x + 1)^2\log((x + 1)/\delta)} \leq \frac{(C^{(1)})^2n\log(n)}{4},
\end{align*}
since $\beta_0 > \frac{1}{2}$, so for $n$ large enough, we have $(C^{(1)})^2n^{1 - 2\beta_0}\log(n) \leq \frac{1}{\log^3(n)}$. We already have $d > \sqrt{n}\log^{13/2}(n)$, by assumption. Next, since $\epsilon_1 \asymp 1$ and $\delta_1 \asymp \frac{1}{\mathrm{poly}(n)}$, we have $I_{max}, I_{min} \asymp \log(n)$. Hence, since $2\epsilon < \log(1/\delta)$, we obtain
\begin{align*}
\frac{\epsilon^2}{D^2(x + 1)^2\log((x + 1)/\delta)} & \leq \frac{\epsilon^2}{D^2(I_{min} + 1)^2\log((I_{min} + 1)/\delta)} \asymp \frac{\epsilon^2}{D^2\log^2(n)\log(\log(n)/\delta)} \\
& \leq \frac{\epsilon^2}{D^2\log^2(n)\log(1/\delta)} \precsim \frac{\epsilon}{D^2\log^2(n)} < n.
\end{align*}
Hence, for $n$ large enough, we have $\frac{\epsilon^2}{25D^2(x + 1)^2\log((x + 1)/\delta)} \leq \frac{(C^{(1)})^2n\log(n)}{4}$. Since $\log(1/(\delta \cdot \mathrm{poly}(n))) < \Theta(\epsilon)$, we then have
\begin{align*}
\frac{\epsilon^2}{D^2(x + 1)^2\log((x + 1)/\delta)} &\geq \frac{\epsilon^2}{D^2(I_{max} + 1)^2\log((I_{max} + 1)/\delta)} \\
& \asymp \frac{\epsilon^2}{D^2\log^2(n)\log(\log(n)/\delta)} \succsim \frac{\epsilon^2}{D^2\log^2(n)(\log(n) + \epsilon)}\\
&\succsim \frac{\epsilon}{D^2\log^2(n)} > \frac{1}{\log^2(n)},
\end{align*}
since $\epsilon$ grows larger than $\log(n)$. So, for $n$ large enough, we have $\frac{\epsilon^2}{25D^2(x + 1)^2\log((x + 1)/\delta)} \geq \frac{1}{\log^3(n)}$. Hence, we can use Lemma \ref{lemma:Pihat-Pi_weighted}, implying that 
\begin{align*}
\widetilde{\mathcal{L}}\left(\mathcal{A}_{SE}^{(\epsilon/(5D(x + 1)), \delta/(x + 1))}(G), \theta\right) \leq C_1h_{\frac{\epsilon}{5D(x + 1)}, \frac{\delta}{x + 1}},
\end{align*}
with probability at least $0.9 - 3\zeta - \frac{1}{\mathrm{poly}(n)}$. Thus, we have
\begin{align*}
& \mathbb{P}\left(\widetilde{\mathcal{L}}\left(\thetahat(G), \theta\right) \leq C_1h_{\frac{\epsilon}{5D(I_{max} + 1)}, \frac{\delta}{I_{max} + 1}}\right) \\
&\qquad \geq \int_I\left(0.9 - 3\zeta - \frac{1}{\mathrm{poly}(n)}\right)f_{X_{\epsilon_1, \delta_1}}(x)dx - \frac{1}{\mathrm{poly}(n)}\\
& \qquad \geq 0.9 - 3\zeta - \frac{1}{\mathrm{poly}(n)} - \mathbb{P}(F_1^c) - \frac{1}{\mathrm{poly}(n)}\\
& \qquad \geq  0.9 - 3\zeta - 2\sqrt{\delta_1} - \frac{1}{\mathrm{poly}(n)} \geq 0.9 - 3\zeta - \frac{1}{\mathrm{poly}(n)},
\end{align*}
since $\delta_1 \asymp \frac{1}{\mathrm{poly}(n)}$. Hence, since $I_{max} \asymp \log(n)$, $D^2 < \epsilon$, and $\log(1/(\delta\cdot\mathrm{poly}(n))) < \Theta(\epsilon)$, we obtain
\begin{align*}
\widetilde{\mathcal{L}}\left(\thetahat(G), \theta\right) & \precsim \frac{D(I_{max} + 1)\sqrt{\log((I_{max} + 1)/\delta)}}{\epsilon\log^4(n)} \precsim \frac{D\log(n)\sqrt{\log(n) + \epsilon}}{\epsilon\log^4(n)} \\
& \precsim \frac{D}{\log^3(n)\sqrt{\epsilon}} < \frac{1}{\log^3(n)},
\end{align*}
with probability at least $0.9 - 3\zeta - \frac{1}{\mathrm{poly}(n)}$, as required.


\subsubsection{Proof of Lemma~\ref{lemma:Pihat-Pi_weighted}}
\label{AppLemPihatWt}

The proof proceeds similarly to the proof of Lemma \ref{lemma:Pihat-Pi}, and it relies on Lemma \ref{lemma:Pi_Z_weighted} (in the same way Lemma \ref{lemma:Pihat-Pi} relies on Lemma \ref{lemma:Pi_Z}).
We work under Assumption \ref{ASS_SE_t_m} and equation~\eqref{eq:t}. Fix $ 1 \asymp \zeta \in (0, 1/3)$, and assume $3k \geq \log(2/\zeta)$. Also, w.h.p, when dividing $P_w =[P_{w, 1}^T, \dots, P_{w, t}^T]^T \in \mathbb{R}^{n \times n}$ according to how $W \odot A$ is divided in step $1$, all of the $t$ chunks have rank $k$. Let $\Omega_1'$ denote the event that all chunks have rank $k$. The following result is proved in the same way as Lemma \ref{lemma:hp_all_types}:

\begin{lemma}
\label{lemma:hp_all_types_weight}
Let $\Omega_1'$ be the event that $P_{w, 1}, \dots, P_{w, t}$ all have rank $k$. Then
\begin{equation*}
\mprob(\Omega_1') \ge 1 - \frac{1}{\mathrm{poly}(n)}.
\end{equation*}
\end{lemma}

Let us work on $\Omega_1'$. Since each chunk has rank $k$, the projection onto the row space of $P_{w, j}$ is $\Pi_w$, for all $j \in [t]$. Let $P_w = U_wD_wU_w^T$, with $D_w \in \mathbb{R}^{k \times k}$ and $U_w \in \mathbb{R}^{n \times k}$. Then $\Pi_w = U_wU_w^T$.

The following result is proved in Appendix~\ref{AppPijWt}:

\begin{lemma}
\label{lemma:Pi^j-Pi_weight}
Let $d > \sqrt{n}\log^{13/2}(n)$ and $s^2 \asymp a_n$. For all $i \in [q]$ and $j \in [t]$, we have
\begin{align*}
&\left\|\Pi^{(j)}z_i\right\|_2 \precsim \sqrt{\log(n)},\\
&\|(\Pi^{(j)} - \Pi_w)z_i\|_2 \precsim \xi_0\sqrt{\log(n)},
\end{align*}
with probability at least $0.9 - \frac{1}{\mathrm{poly}(n)}$, where 
\begin{align*}
\xi_0 := \frac{\log^{1/4}(1/\delta)}{\log^{5/2}(n)\sqrt{\epsilon}} + \frac{\sqrt{\log(1/\delta)}}{\log^5(n)\epsilon}.
\end{align*}
\end{lemma}

We work on $\Omega_4' = \Omega_{41} \cap \Omega_{42} \cap \Omega_{43} \cap \Omega_3'$.

The following result is proved analogously to Lemma~\ref{lemma:Zhat_Z}, using Lemma \ref{lemma:Pi^j-Pi_weight}:

\begin{lemma}
\label{lemma:Zhat_Z_weighted}
Let $Z_w = [\Pi_w z_1, \dots, \Pi_w z_q] \in \mathbb{R}^{n \times q}$. Suppose $d > \sqrt{n}\log^{13/2}(n)$ and $s^2 \asymp a_n$. We have $\|\hat{Z} - Z_w\|_2 \precsim \frac{1}{n^2\sqrt{n}} + \xi_0\sqrt{\log(n)}$, with probability at least $0.9 - 2\zeta - \tau_w(n)$.
\end{lemma}

The following result is proved analogously to Lemma~\ref{lemma:Pi_Z}, using Lemma \ref{lemma:Zhat_Z_weighted}:

\begin{lemma}
\label{lemma:Pi_Z_weighted}
Let $d > \sqrt{n}\log^{13/2}(n)$ and $s^2 \asymp a_n$. Let $E_w := E_{Z_w} + E_{Z_w^{\perp}}$, i.e., the decomposition of each column of $E_w$ as a sum of the component in the column span of $Z_w$, and the corresponding orthogonal complement. Let $Z_w' := Z_w + E_{Z_w}$. Assume $\xi_0\sqrt{\log(n)} = o(1)$ as $n \rightarrow \infty$. Then $\sigma_k(Z_w') \asymp \sqrt{k} \asymp 1$ and $\Pi_w$ is also the projection onto the columns of $Z_w'$, with probability at least $0.9 - 3\zeta - \frac{1}{\mathrm{poly}(n)}$.
\end{lemma}

\subsubsection{Proof of Lemma~\ref{lemma:Pi^j-Pi_weight}}
\label{AppPijWt}

Fix $j \in [t]$. By Lemma \ref{lemma:hp_all_types_weight}, the projection onto the row space of $P_{w, j}$ is $\Pi_w$. We have  
\begin{align*}
||\Pi^{(j)} - \Pi_w||_2 &= ||U^{(j), (k)}(U^{(j), (k)})^T - U_wU_w^T||_2 \\
& \stackrel{(a)}{\leq} 2||\sin(\Theta)(U^{(j), (k)}, U_w)||_2 \stackrel{(b)}{\leq} 2||\sin(\Theta)(U^{(j), (k)}, U_w)||_F\\
& \stackrel{(c)}{\leq} \mathop{\inf}_{Q: Q^TQ = QQ^T = I_k}||U^{(j), (k)}Q - U_w||_F \stackrel{(d)}{\leq} ||U^{(j), (k)}\hat{Q} - U_w||_F\\
&\precsim \frac{(\sigma_1(P_{w,j}) + ||\Pi_j(W \odot A)_j - P_{w, j}||_2)||\Pi_j(W \odot A)_j - P_{w, j}||_2}{\sigma^2_k(P_{w, j})}.
\end{align*}
Here, $\sigma_1(\cdot)$ and $\sigma_k(\cdot)$ represent the largest and $k^{\text{th}}$ largest singular values, respectively. Inequality (a) follows from Lemma \ref{lemma:l2.3SinSte}. Inequality (b) holds because the Frobenius norm is greater than the operator norm. Inequality (c) holds because of Lemma \ref{lemma:prop2.2VuLei}. Inequality (d) and the existence of the $k \times k$ orthogonal matrix $\hat{Q}$ hold because of Lemma \ref{lemma:thm3YuEtal}. Now, we have
\begin{align*}
||\Pi_j(W \odot A)_j - P_{w, j}||_2 \leq ||\Pi_j(W \odot A)_j - (W \odot A)_j||_2 + ||(W \odot A)_j - P_{w, j}||_2.  
\end{align*}
Note that $\Pi_j(W \odot A)_j$ is the best rank-$k$ approximation of $(W \odot A)_j$ in the operator norm by the Eckart-Young Theorem (cf.\ Lemma~\ref{LemEYM}). Now split $W = [W_1^T, \dots, W_t^T]^T$ accordingly. Since $P_{w, j}$ has rank $k$ on $\Omega_1'$, we have
\begin{align*}
||\Pi_j(W \odot A)_j - (W \odot A)_j||_2 & \leq ||(W \odot A)_j - P_{w, j}||_2 \\
& \leq ||(W \odot A)_j - \mathbb{E}[W_j]\odot A_j||_2 + ||\mathbb{E}[W_j]\odot A_j - P_{w, j}||_2 \\
& \leq ||W \odot A - \mathbb{E}[W]\odot A||_2 + ||\mathbb{E}[W]\odot A - P_w||_2.
\end{align*}
Hence, since $P_w = \mathbb{E}[W]\odot P$, we have
\begin{align*}
||\Pi_j(W \odot A)_j -  P_{w, j}||_2 &\leq 2||W \odot A - \mathbb{E}[W]\odot A||_2 + 2||\mathbb{E}[W]\odot A - P_w||_2\\
&= 2||W \odot A - \mathbb{E}[W]\odot A||_2 + 2||\mathbb{E}[W]\odot (A - P)||_2.
\end{align*}
Since $B_w \succeq 0$ or $B_w \preceq 0$, we have $\mathbb{E}[W] = J(B_w\otimes J_{\frac{n}{k}})J^T \succeq 0$ or $\mathbb{E}[W] = J(B_w\otimes J_{\frac{n}{k}})J^T \preceq 0$. Since $\mathbb{E}[W]$ is also symmetric, by Lemma \ref{lemma:Oper_Hadam}, we have
\begin{equation*}
||\mathbb{E}[W]\odot (A - P)||_2 \leq \mathop{\max}_{l, g}|\mathbb{E}[W_{lg}]|||A - P||_2 \precsim ||A - P||_2.
\end{equation*}
We also assumed that the entries of $B_w$ are constant, so the entries of $\mathbb{E}[W]$ do not change with $n$. Therefore, we have
\begin{align}
\label{eq:W_Pi_j_P_j}
||\Pi_j(W \odot A)_j -  P_{w, j}||_2 \precsim ||W \odot A - \mathbb{E}[W]\odot A||_2 + ||A - P||_2.
\end{align}
For $1 \leq l < v \leq n$, let $Q^{(lv)} \in \{0, 1\}^{n \times n}$ be the $0$ matrix, with the $(l, v)$ and $(v, l)$ entries given by the $(l, v)$ entry in $A$. Then
\begin{align*}
W \odot A - \mathbb{E}[W]\odot A = (W - \mathbb{E}[W])\odot A = \sum_{1 \leq l < v \leq n}(W_{lv} - \mathbb{E}[W_{lv}])Q^{(lv)},  
\end{align*}
where $W_{lv}$ is the $(l, v)$ entry in $W$. Also, for $\zeta_2 > 0$, we have
\begin{align*}
\mathbb{P}\left(||W \odot A - \mathbb{E}[W]\odot A||_2 \geq \zeta_2\right) & = \mathbb{E}_A\left[\mathbb{P}\left(||W \odot A - \ \mathbb{E}[W]\odot A||_2 \geq \zeta_2| A\right)\right] \\
& = \mathbb{E}_A\left[\mathbb{P}_W\left(||W \odot A - \mathbb{E}[W]\odot  A||_2 \geq \zeta_2\right)\right],
\end{align*}
where the last step holds because $W$ and $A$ are independent. So we could drop the conditioning, and inside $\mathbb{P}_W$, the only randomness is in $W$. Hence, we have
\begin{align*}
\mathbb{P}\left(||W \odot A - \mathbb{E}[W]\odot A||_2 \geq \zeta_2\right) = \mathbb{E}_A\left[\mathbb{P}_W\left(\left\|\sum_{1 \leq l < v \leq n}(W_{lv} - \mathbb{E}[W_{lv}])Q^{(lv)}\right\|_2 \geq \zeta_2\right)\right].
\end{align*}
For each fixed $(l, v)$, the random variable $W_{lv} - \mathbb{E}[W_{lv}]$ is zero-mean $s^2$-sub-Gaussian, with a distribution symmetric around $0$. Also, inside $\mathbb{P}_W$, we know that $Q^{(lv)}$ is deterministic and symmetric. Then by Lemma \ref{lemma:gB_Ex6.7}, the matrix $(W_{lv} - \mathbb{E}[W_{lv}])Q^{(lv)}$ is sub-Gaussian with matrix parameter $C_1^2s^2(Q^{(lv)})^2$, where $C_1 \asymp 1$. Since all the $(W_{lv} - \mathbb{E}[W_{lv}])Q^{(lv)}$ matrices are independent and symmetric, we have by Lemma \ref{lemma:Matrix_Hoeff} that
\begin{align*}
\mathbb{P}_W\left(\left\|\sum_{1 \leq l < v \leq n}(W_{lv} - \mathbb{E}[W_{lv}])Q^{(lv)}\right\|_2 \geq \zeta_2\right) \leq 2ne^{-\frac{\zeta_2^2}{2C_1^2s^2\left\|\sum_{1 \leq l < v \leq n}^n(Q^{(lv)})^2\right\|_2}}.
\end{align*}
Note that $(Q^{(lv)})^2$ is either the $0$ matrix or the diagonal matrix with $1$'s in diagonal positions $l$ and $v$, so
\begin{align*}
\left\|\sum_{1 \leq l < v \leq n}(Q^{(lv)})^2\right\|_2 \leq n - 1
\end{align*}
and
\begin{align*}
\mathbb{P}_W\left(\left\|\sum_{1 \leq l < v \leq n}(W_{lv} - \mathbb{E}[W_{lv}])Q^{(lv)}\right\|_2 \geq \zeta_2\right) \leq 2ne^{-\frac{\zeta_2^2}{2C_1^2s^2(n - 1)}}.
\end{align*}
Therefore, by taking $\zeta_2^2 = 4C_1^2s^2(n - 1)\log(n)$, we obtain
\begin{align*}
\mathbb{P}\left(||W \odot A - \mathbb{E}[W]\odot A||_2 \geq 2C_1s\sqrt{(n - 1)\log(n)}\right) \leq \frac{2}{n}.   
\end{align*}
Since $s^2 \asymp a_n$, we obtain $||W \odot A - \mathbb{E}[W]\odot A||_2 \precsim \sqrt{a_n n\log(n)}$ on an event $\Omega_{21}$ which holds with probability at least $1 - \frac{2}{n}$. By Lemma \ref{lemma:thm5.2LeiRin}, since $a_n \succsim \frac{\log(n)}{n}$, there is an event $\Omega_2$, with $\mathbb{P}(\Omega_2) \geq 1 - \frac{1}{n}$, such that $||A - \mathbb{E}[A]||_2 = ||A - P + \text{diag}(P)||_2 \precsim \sqrt{na_n}$ on $\Omega_2$. Let us work on $\Omega_2' = \Omega_1' \cap \Omega_2 \cap \Omega_{21}$. Using inequality~\eqref{eq:W_Pi_j_P_j}, we have
\begin{align*}
||\Pi^{(j)} - \Pi_w||_2 \precsim \frac{(\sigma_1(P_{w, j}) + \sqrt{na_n \log(n)} + a_n)(\sqrt{na_n \log(n)} + a_n)}{\sigma^2_k(P_{w, j})}.
\end{align*}

The following result is analogous to Lemma \ref{LemSubsampleSigma}, for the singular values of matrices subsampled from $P_w$. Suppose $P_{w, m}$ is a matrix obtained by randomly subsampling $m$ rows of $P_w$ without replacement, i.e., $P_{w, m} = E_m P_w$ ($E_m$ selects a uniformly random subset of rows, of size $m$). 
\begin{lemma}
\label{LemSubsampleSigma_weighted}
With probability at least $1 - 2ke^{-\frac{m}{2k^2}}$, we have
\begin{align*}
    \sigma_{k}(P_{w, m}) & \succsim \sigma_{k}(D_w) \cdot \sqrt{\frac{m}{n}}, \\
    \sigma_{\max}(P_{w, m}) & \precsim \sigma_{\max}(D_w) \cdot \sqrt{\frac{m}{n}},
\end{align*}
where $P_w = U_wD_wU_w^T$ is the spectral decomposition of $P_w$, with $D_w \in \real^{k \times k}$.
\end{lemma}

\begin{proof}
Using the same proof as in Lemma 2.1 from \cite{LeiRin15}, we can derive the following result:

\begin{lemma}
\label{lemma:weighted_lemma2.1_lei_rin}
Suppose $G \sim WSBM(n, k, B, \theta, \mathscr{W})$. Let $\Theta \in \{0,1\}^{n \times k}$ be the membership matrix of the nodes. Let $U_wD_wU_w^T$ be the eigen-decomposition of $P_w = \Theta (B \odot B_w)\Theta^T$. Then $U_w = \Theta X$ and $D_w = \frac{n}{k} \cdot Z^T(B \odot B_w)Z$, where $X = \sqrt{\frac{k}{n}} Z \in \mathbb{R}^{k \times k}$, and $Z$ is orthonormal. Moreover, $||X_{i*} - X_{j*}||_2 = \sqrt{\frac{2k}{n}}$, for all $1 \leq i < j \leq k$.
\end{lemma}

We now follow the same lines of the proof of Lemma \ref{LemSubsampleSigma}, using Lemma \ref{lemma:weighted_lemma2.1_lei_rin}.
\end{proof}

By Lemma \ref{LemSubsampleSigma_weighted}, there is an event $\Omega_{31}$, with probability at least $1 - 2kte^{-\frac{m}{2k^2}}$ (a union bound over all $j \in [t]$), such that on $\Omega_{31}$ we have
\begin{align*}
\sigma_k(P_w) \cdot \sqrt{\frac{m}{n}} \precsim \sigma_{k}(P_{w, j}) \leq \sigma_{1}(P_{w, j}) & \precsim \sigma_1(P_w) \cdot \sqrt{\frac{m}{n}}.
\end{align*}
Let us work on $\Omega_3' = \Omega_{31} \cap \Omega_2'$. Recall that $P_w = J((B\odot B_w) \otimes J_{\frac{n}{k}})J^T$, and $J \in \{0, 1\}^{n \times n}$ is a permutation matrix. Then $\sigma_1(P_w) = \frac{na_n}{k}\sigma_1(B_0 \odot B_w)$ and $\sigma_k(P_w) = \frac{na_n}{k}\sigma_k(B_0 \odot B_w)$. Hence, because $B_0 \odot B_w$ does not change with $n$, we have
\begin{align*}
a_n \sqrt{mn} \precsim \sigma_{k}(P_{w, j}) \leq \sigma_{1}(P_{w, j}) & \precsim a_n \sqrt{mn}.
\end{align*}
Therefore, since $d > \sqrt{n}\log^{13/2}(n)$, we have
\begin{align*}
||\Pi^{(j)} - \Pi_w||_2 &\precsim \frac{(a_n \sqrt{mn} + \sqrt{na_n \log(n)} + a_n)(\sqrt{na_n\log(n)} + a_n)}{a_n^2mn} \\
& \precsim \sqrt{\frac{\log(n)}{a_n m}} + \frac{\log(n)}{a_n m} \asymp \sqrt{\frac{t\log(n)}{d}} + \frac{t\log(n)}{d}\\
&\precsim \frac{1}{\log^{5/2}(n)\rho^{1/4}} + \frac{1}{\log^5(n)\sqrt{\rho}} \asymp \frac{\log^{1/4}(1/\delta)}{\log^{5/2}(n)\sqrt{\epsilon}} + \frac{\sqrt{\log(1/\delta)}}{\log^5(n)\epsilon} = \xi_0,
\end{align*}
for all $j \in [t]$. Now we follow the argument in Corollary 4.6 in Singhal \& Steinke~\cite{SinSte21}.
Let $\hat{\Pi}^{(j)}$ be the projection matrix for the union of the subspace spanned by the columns of $U^{(j), (k)}$ and the subspace spanned
by the columns of $U_w$. By Lemma \ref{lemma:lem2.10SinSte} we have $||\hat{\Pi}^{(j)}z_i||_2 \precsim \sqrt{k} + \sqrt{\log(qt)} \precsim \sqrt{\log(n)}$, for all $(i, j)$, on an event $\Omega_{41}$, with $\mathbb{P}(\Omega_{41}) \geq 0.99$. By Lemma \ref{lemma:lem2.10SinSte} again, we have $||\Pi z_i||_2 \precsim \sqrt{k} + \sqrt{\log(qt)} \precsim \sqrt{\log(n)}$, for all $(i, j)$, on an event $\Omega_{41}$, with $\mathbb{P}(\Omega_{41}) \geq 0.99$. We also have, by Lemma \ref{lemma:lem2.10SinSte}, that $||\Pi^{(j)}z_i||_2 \precsim \sqrt{k} + \sqrt{\log(qt)} \precsim \sqrt{\log(n)}$, for all $(i, j)$, on an event $\Omega_{43}$, with $\mathbb{P}(\Omega_{43}) \geq 0.99$. Let us work on $\Omega_4' = \Omega_{41} \cap \Omega_{42} \cap \Omega_{43} \cap \Omega_3'$. Hence, we have
\begin{align*}
||(\Pi^{(j)} - \Pi_w)z_i||_2 = ||(\Pi^{(j)} - \Pi_w)\hat{\Pi}^{(j)}z_i||_2 \leq ||\Pi^{(j)} - \Pi_w||_2||\hat{\Pi}^{(j)}z_i||_2 \precsim \xi_0\sqrt{\log(n)} ,
\end{align*}
with probability at least $0.9 - \tau_w(n)$, where  $\tau_w(n) = t(1 - p_1) + \frac{3}{n} + 2kte^{-\frac{m}{2k^2}}$. Recall that we assumed $m \geq n^{1 - \beta_0}$ and $1 \asymp \beta_0 \in (1/2, 1)$, so $2kte^{-\frac{m}{2k^2}} \leq 2kn^{\beta_0}e^{-\frac{n^{1 - \beta_0}}{2k^2}}$. For $\zeta, k \asymp 1$, as in Lemma \ref{lemma:hp_all_types_weight}, we have $\tau_w(n) \leq \frac{1}{\mathrm{poly}(n)}$.



\subsection{Proof of Theorem~\ref{theorem:Two_Comm_Main_Thm}}
\label{AppThmTwoCom}

Regarding privacy, we first establish the following result:

\begin{lemma}
\label{lemma:TwoBD_Priv}
Let $2\epsilon < \log(1/\delta)$. Then Algorithm \ref{alg:TC_Optimiz} with privacy parameters $\left(\frac{\epsilon}{4D}, \delta\right)$ is $(2\epsilon, \delta)_{2D}$-DP.
\end{lemma}

\begin{proof}
By Lemma \ref{lemma:ChenPriv+Ut}, Algorithm \ref{alg:TC_Optimiz} with privacy parameters $\left(\frac{\epsilon}{4D}, \delta\right)$ is $\frac{\epsilon^2}{64D^2\log(1/\delta)}$-edge zCDP. When changing one node in graphs belonging to $\mathcal{G}_{n, 2D}$, at most $4D$ edges can be flipped. Hence, by group privacy using Lemma \ref{lemma:Group_zCDP}, Algorithm \ref{alg:TC_Optimiz} with privacy parameters $\left(\frac{\epsilon}{4D}, \delta\right)$ is $\left(\frac{\epsilon^2}{4\log(1/\delta)}\right)_{2D}$-zCDP. Finally, by Lemma \ref{lemma:zCDP_to_DP}, we obtain $\left(\frac{\epsilon^2}{4\log(1/\delta)} + \epsilon, \delta\right)_{2D}$-DP. Since, $\delta < e^{-2\epsilon}$, this is also $(9\epsilon/8, \delta)_{2D}$-DP, and the conclusion follows.
\end{proof}

By Lemma \ref{lemma:TwoBD_Priv}, we know that if $2\epsilon < \log(1/\delta)$, then $\mathcal{A}_{TC}^{(\epsilon/(4D)), \delta)}$ is $(2\epsilon, \delta)_{2D}$-node DP. Since $2\epsilon < \log(1/(\delta\cdot\mathrm{poly}(n)))$ Theorem \ref{theorem:generic_red} implies that $\thetahat(G)$ is $(\epsilon_1 + 2\epsilon, 1/\mathrm{poly}(n))$-node DP. We also used the fact that $\epsilon_1 \asymp 1$ and $\delta_1 \asymp \frac{1}{\mathrm{poly}(n)}$. Hence, for $n$ large enough, we obtain $(3\epsilon, 1/\mathrm{poly}(n))$-node DP.

For the utility guarantee, let
\begin{align*}
X_{\epsilon_1, \delta_1} & := 5 + \mathrm{Lap}\left(\frac{8}{\epsilon_1}\right) + \frac{8\log(1/\delta_1)}{\epsilon_1}, \\
I & := \left[5 + \frac{4\log(1/\delta_1)}{\epsilon_1}, 5 + \frac{16\log(1/\delta_1)}{\epsilon_1}\right], \\
I_{max} & := 5 + \frac{16\log(1/\delta_1)}{\epsilon_1}.
\end{align*}
For $F_1 = \left\{X_{\epsilon_1, \delta_1} \in I\right\}$, we have
\begin{align*}
 \mathbb{P}\left(F_1^c\right) \leq \mathbb{P}\left(\mathrm{Lap}\left(\frac{8}{\epsilon_1}\right) < -\frac{4\log(1/\delta_1)}{\epsilon_1}\right) + \mathbb{P}\left(\mathrm{Lap}\left(\frac{8}{\epsilon_1}\right) > \frac{8\log(1/\delta_1)}{\epsilon_1}\right) \leq 2\sqrt{\delta_1}.
\end{align*}
Let $h_{\epsilon, \delta} = \frac{1}{\sqrt{a_n n}} + \frac{\log(1/\delta)}{a_n n\epsilon^2}$.

We will use the following result, which applies to fixed privacy parameters $(\epsilon, \delta)$:

\begin{lemma}[Adapted from Theorem 5.3 in \cite{CheEtal23}]
\label{lemma:ChenPriv+Ut}
Let $G \in \mathcal{G}_n$ be sampled from an unweighted SBM (Definition~\ref{DefSBM}) with $k = 2$. We assume $(B_0)_{11} = (B_0)_{22} > (B_0)_{12} = (B_0)_{21}$. Let $\mathcal{A}_{TC}^{(\epsilon, \delta)}(G)$ be Algorithm \ref{alg:TC_Optimiz} evaluated on $G$. Then there exist $C_1, C_2, C_3 \asymp 1$, such that for $d \geq C_1$, we have  
\begin{align*}
\mathcal{L}\left(\mathcal{A}_{TC}^{(\epsilon, \delta)}(G), \theta\right) \leq C_2\left(\frac{1}{\sqrt{a_n n}} + \frac{\log(1/\delta)}{a_n n\epsilon^2}\right),
\end{align*}
with probability at least $1 - e^{-C_3n}$. In addition, Algorithm \ref{alg:TC_Optimiz} is $\frac{\epsilon^2}{4\log(1/\delta)}$-edge zCDP and runs in polynomial time.
\end{lemma}

\begin{proof}
We prove privacy first. By \cite{CheEtal23}, the function $\arg\min_{X \in \real^{n \times n}: X \succeq 0, X_{ii} = \frac{1}{n} \; \forall i} \|Y-X\|_F^2$ has $\ell_2$-sensitivity given by $\sqrt{\frac{48}{a_n n^2((B_0)_{11} - (B_0)_{12})}}$. Then the privacy follows directly from Lemma \ref{lemma:Gauss_zCDP}.

Regarding utility, following the same steps as in the proof of Theorem 5.3 in \cite{CheEtal23}, we obtain the desired guarantee. The polynomial-time guarantee also follows from the proof of Theorem 5.3 in \cite{CheEtal23}.
\end{proof}

Let $C_1, C_2$, and $C_3$ be as in Lemma \ref{lemma:ChenPriv+Ut} and let $f_{X_{\epsilon_1, \delta_1}}$ be the pdf of $X_{\epsilon_1, \delta_1}$. Also, let $\Omega_0$ be as in the proof of Lemma \ref{lemma:hp_removal_d}. We have $\mathbb{P}(\Omega_0) \geq 1 - \frac{1}{\mathrm{poly}(n)}$. In addition, on $\Omega_0$, we do not have to truncate the graph, since the degree of every node is less than $2d$, so $T_{D}(G) = G$ and $d_{T_{D}}(G) = 0$. Hence, since $X_{\epsilon_1, \delta_1} \geq 5 > \frac{1}{2}$ on $F_1$, we have
\begin{align*}
&\mathbb{P}\left(\mathcal{L}\left(\thetahat(G), \theta\right) \leq C_2h_{\frac{\epsilon}{4DI_{max}}, \frac{\delta}{I_{max}}}\right)\\
&\geq \mathbb{P}\left(\left\{\mathcal{L}\left(\thetahat(G), \theta\right) \leq C_2h_{\frac{\epsilon}{4DI_{max}}, \frac{\delta}{I_{max}}}\right\} \cap \Omega_0\right)\\
&\geq \mathbb{P}\left(\left\{\mathcal{L}\left(\mathcal{A}_{TC}^{\left(\epsilon/(4D\max\{1/2, X_{\epsilon_1, \delta_1}\}), \delta/\max\{1/2, X_{\epsilon_1, \delta_1}\}\right)}(G), \theta\right) \leq C_2h_{\frac{\epsilon}{4DI_{max}}, \frac{\delta}{I_{max}}}\right\} \cap \Omega_0 \cap F_1\right)\\
&= \mathbb{P}\left(\left\{\mathcal{L}\left(\mathcal{A}_{TC}^{\left(\epsilon/(4DX_{\epsilon_1, \delta_1}), \delta/X_{\epsilon_1, \delta_1}\right)}(G), \theta\right) \leq C_2h_{\frac{\epsilon}{4DI_{max}}, \frac{\delta}{I_{max}}}\right\} \cap \Omega_0 \cap F_1\right)\\
&\geq \int_I\mathbb{P}\left(\mathcal{L}\left(\mathcal{A}_{TC}^{\left(\epsilon/(4DX_{\epsilon_1, \delta_1}), \delta/X_{\epsilon_1, \delta_1}\right)}(G), \theta\right) \leq C_2h_{\frac{\epsilon}{4DI_{max}}, \frac{\delta}{I_{max}}}\middle\vert X_{\epsilon_1, \delta_1} = x\right)f_{X_{\epsilon_1, \delta_1}}(x)dx - \frac{1}{\mathrm{poly}(n)}\\
&= \int_I\mathbb{P}\left(\mathcal{L}\left(\mathcal{A}_{TC}^{\left(\epsilon(4Dx), \delta/x\right)}(G), \theta\right) \leq C_2h_{\frac{\epsilon}{4DI_{max}}, \frac{\delta}{I_{max}}}\right)f_{X_{\epsilon_1, \delta_1}}(x)dx - \frac{1}{\mathrm{poly}(n)}\\
&\geq \int_I\mathbb{P}\left(\mathcal{L}\left(\mathcal{A}_{TC}^{\left(\epsilon/(4Dx), \delta/x\right)}(G), \theta\right) \leq C_2h_{\frac{\epsilon}{4Dx}, \frac{\delta}{x}}\right)f_{X_{\epsilon_1, \delta_1}}(x)dx - \frac{1}{\mathrm{poly}(n)},
\end{align*}
since the function $h_{\frac{\epsilon}{4Dx}, \frac{\delta}{x}}$ is increasing in $x > 0$. We also dropped the conditioning on $\left\{X_{\epsilon_1, \delta_1} = x\right\}$, since the randomness in $X_{\epsilon_1, \delta_1}$ only depends on the $\mathrm{Lap}\left(\frac{8}{\epsilon_1}\right)$ noise, and this is independent of $\mathcal{L}\left(\mathcal{A}_{TC}^{(\epsilon/(4Dx), \delta/x)}(G), \theta\right)$. Since $d \succsim \log(n) \geq C_1$, for $n$ large enough, we have by Lemma \ref{lemma:ChenPriv+Ut} that
\begin{align*}
\mathcal{L}\left(\mathcal{A}_{TC}^{\left(\epsilon/(4Dx), \delta/x\right)}(G), \theta\right) \leq C_2h_{\frac{\epsilon}{4Dx}, \frac{\delta}{x}},
\end{align*}
with probability at least $1 - e^{-C_3n}$. Thus, we have
\begin{align*}
\mathbb{P}\left(\mathcal{L}\left(\thetahat(G), \theta\right) \leq C_2h_{\frac{\epsilon}{4DI_{max}}, \frac{\delta}{I_{max}}}\right) &\geq \int_I\left(1 - e^{-C_3n}\right)f_{X_{\epsilon_1, \delta_1}}(x)dx - \frac{1}{\mathrm{poly}(n)}\\
&\geq 1 - \frac{1}{\mathrm{poly}(n)} - \mathbb{P}(F_1^c) - \frac{1}{\mathrm{poly}(n)}\\
&\geq 1 - 2\sqrt{\delta_1} - \frac{1}{\mathrm{poly}(n)} \geq 1 - \frac{1}{\mathrm{poly}(n)},
\end{align*}
since $\delta_1 \asymp \frac{1}{\mathrm{poly}(n)}$. Recall that $d > \log^3(n)$ and $\epsilon > D^2\log(n)$. Hence, since $I_{max} \asymp \log(n)$, $\delta > \frac{e^{-\Theta(\epsilon)}}{\mathrm{poly}(n)}$, and $\epsilon$ is growing with $n$, we have
\begin{align*}
\mathcal{L}\left(\thetahat(G), \theta\right) &\precsim \frac{1}{\sqrt{d}} + \frac{D^2\log^2(n)\log(\log(n)/\delta)}{d\epsilon^2} \precsim \frac{1}{\sqrt{d}} + \frac{D^2\log^3(n)\epsilon}{d\epsilon^2} = \frac{1}{\sqrt{d}} + \frac{D^2\log^3(n)}{d\epsilon}\\
&< \frac{1}{\sqrt{d}} + \frac{D^2}{\epsilon} < \frac{1}{\sqrt{d}} + \frac{1}{\log(n)} \precsim \frac{1}{\log(n)},
\end{align*}
with probability at least $1 - \frac{1}{\mathrm{poly}(n)}$, as required.


\subsection{Proofs for Section~\ref{SecDeflatePCA}}

\subsubsection{Proof of Lemma~\ref{LemPCADeflate}}
\label{AppLemPCADeflate}

First note that the release of $\hat{\sigma}_i$ certainly satisfies $\left(\frac{\epsilon}{2k}, 0\right)_{2D}$-node DP, since the truncation step guarantees that $|\hat{v}_i^T A_i \hat{v_i}| \le D^2$.

We will use adaptive composition to obtain a bound on the overall privacy after $k$ iterations. Note that
\begin{equation*}
A_{i+1} = A - \hat{\sigma}_1 \hat{v}_1 \hat{v}_1^T - \hat{\sigma}_2 \hat{v}_2 \hat{v}_2^T - \cdots - \hat{\sigma}_i \hat{v}_i \hat{v}_i^T.
\end{equation*}
Note that if we fix the values $\{(\hat{\sigma}_j, \hat{v}_j)\}_{j=1}^i$, the sensitivity of the score function $v^T A_{i+1} v$ (for a fixed unit vector $v$) is clearly equal to the sensitivity of $v^T A^2 v$, hence is bounded by $4D^2$ on graphs in $\mathcal{G}_{n, 2D}$. Thus, by adaptive composition, the output after $k$ iterations indeed satisfies $(\epsilon, 0)_{2D}$-node DP.

\subsubsection{Proof of Theorem~\ref{ThmPCADeflate}}
\label{AppThmPCADeflate}

By Lemma \ref{LemPCADeflate}, we know that $\mathcal{A}_{EigDef}^{(\epsilon/(2k))}(G)$ satisfies $(\epsilon, 0)_{2D}$-node DP. Then by Theorem \ref{theorem:generic_red_pure}, we know that $\thetahat(G)$ satisfies $(\epsilon_1 + \epsilon, e^{\epsilon_1}\delta_1)$-node DP. Since $\epsilon_1 \asymp 1$ and $\delta_1 \asymp \frac{1}{\mathrm{poly}(n)}$, we obtain the desired result.


\section{Proofs for Section~\ref{SecLower}}
\label{AppLower}

\subsection{Proof of Proposition~\ref{prop:LB_Chen_eps_delt}}
\label{AppPropLBApprox}

We relabel the community assignments from $[k]$ to $\{-1, +1\}^n$. By a result of \cite{CheEtal23}, we know that
$\frac{\mathcal{L}(\cdot, \cdot)}{2}$ is a semi-metric over $\{-1, 1\}^n$, so $\mathcal{L}(\cdot, \cdot)$ is a semi-metric over $\{-1, 1\}^n$.

Let $B_{\mathcal{L}}(x, \ell) := \left\{w = (w_1, \dots, w_n)^T \in \{-1, 1\}^n \mid \ \sum_{i = 1}^nw_i = 0, \mathcal{L}(x, w) \leq \ell\right\}$. Pick an arbitrary balanced $\theta \in \{-1, 1\}^n$, and let $M = \{\theta^1, \dots, \theta^m\}$ be a maximal $2\xi$-packing of $B_{\mathcal{L}}(\theta, 4\xi)$. The arguments in \cite{CheEtal23}, after rescaling $\mathcal{L}(\cdot, \cdot)$, since they use $\frac{\mathcal{L}(\cdot, \cdot)}{2}$, show that $2m \geq \left(\frac{1}{4e\xi}\right)^{\xi n}$. For each $i \in [m]$, define $Y_i = \left\{w \in \{-1, 1\}^n \mid \ \mathcal{L}(\theta^i, w) \leq \xi\right\}$. For each $i \in [m]$, let $P_i$ be the SBM distribution for graphs with community assignment $\theta^i$. By our accuracy assumption and the fact that the $Y_i$'s are disjoint, we have
\begin{align}
\label{eq:add_Y_i_approx}
\sum_{i = 2}^m\mathbb{P}_{G \sim P_1}(\mathcal{M}(G) \in Y_i) \leq \eta.
\end{align}
Note that each $P_i$ is a product of $\binom{n}{2}$ independent Bernoulli distributions. If $\theta^i$ and $\theta^j$ are two community assignments which differ on node set $S$, consider a coupling $\omega$ of $P_i$ and $P_j$ which couples all edges not incident to $S$ identically.
If $(G,H)$ are drawn according to the coupling $\omega$, we clearly have $\mathrm{Ham}_{node}(G, H) \leq \mathrm{Ham}(\theta^i, \theta^j)$, where $\mathrm{Ham}_{node}(G, H)$ denotes measures the minimal number of nodes needed to be rewired to go from $G$ to $H$.

We now use the fact that $\mathcal{M}$ is  $(K\epsilon, K\delta e^{K\epsilon})$-DP on groups of size $K \in \mathbb{N}$. For all $i \in [m]\setminus \{1\}$, we then have
\begin{align*}
1 - \eta &\leq \mathbb{P}_{G \sim P_i}(\mathcal{M}(G) \in Y_i) = \mathbb{E}_{(G, H) \sim \omega_{i1}}\left[\mathbb{P}(\mathcal{M}(G) \in Y_i)\right]\\
&\leq \mathbb{E}_{(G, H) \sim \omega_{i1}}\left[e^{\epsilon\text{Ham}_{node}(G, H)}\mathbb{P}(\mathcal{M}(H) \in Y_i) + \delta \cdot \text{Ham}_{node}(G, H)e^{\epsilon\text{Ham}_{node}(G, H)}\right]\\
&\leq e^{4\epsilon\xi n}\mathbb{E}_{(G, H) \sim \omega_{i1}}\left[\mathbb{P}(\mathcal{M}(H) \in Y_i)\right] + 4\delta\xi n e^{4\epsilon\xi n} \\
& = e^{4\epsilon\xi n}\mathbb{P}_{H \sim P_1}(\mathcal{M}(H) \in Y_i) + 4\delta\xi n e^{4\epsilon\xi n},
\end{align*}
implying that
\begin{align*}
\mathbb{P}_{H \sim P_1}(\mathcal{M}(H) \in Y_i) \geq \frac{1 - \eta - 4\delta\xi n e^{4\epsilon\xi n}}{e^{4\epsilon\xi n}},
\end{align*}
for all $i \in [m]\setminus \{1\}$. Combined with inequality~\eqref{eq:add_Y_i_approx}, we then obtain $e^{4\epsilon\xi n} \geq \frac{(1 - \eta - 4\delta\xi n e^{4\epsilon\xi n})(m - 1)}{\eta}$, so
\begin{align*}
\epsilon \geq \frac{\log\left(\max\left\{1, \frac{(1 - \eta - 4\delta\xi n e^{4\epsilon\xi n})(m - 1)}{\eta}\right\}\right)}{4\xi n}.   
\end{align*}
Plugging in $2m = \left(\frac{1}{4e\xi}\right)^{\xi n}$, we obtain the desired result.

For the special choice of $(\xi, \eta, \epsilon, \delta)$, we can make $1 - \eta - 4\delta\xi n e^{4\epsilon\xi n} \ge \frac{1}{4}$.
Furthermore, $m$ increases with $n$ if $\xi \in \left[\frac{1}{n}, \frac{c_0}{n}\right]$, and if $\eta < \frac{1}{2}$, we have $1 < \frac{(1 - \eta - 4\delta\xi n e^{4\epsilon\xi n})(m - 1)}{\eta}$, so $\epsilon \geq \frac{\log\left(\frac{c'(m - 1)}{4\eta}\right)}{4\xi n}$. It is easy to see that $m \asymp \mathrm{poly}(n)$ in this setting, and the lower bound becomes $\epsilon \succsim \log(n)$.


\subsection{Proof of Corollary~\ref{cor:LB_Chen_pure}}
\label{AppPropLBPure}

By Proposition~\ref{prop:LB_Chen_eps_delt}, we have 
\begin{align*}
\epsilon \geq \frac{\log\left(\max\left\{1, \frac{(1 - \eta)(m - 1)}{\eta}\right\}\right)}{4\xi n}.   
\end{align*}
Since $\eta \leq \frac{1}{2}$, we have $1 - \eta \geq \frac{1}{2}$ and $\frac{(m - 1)}{2\eta} \geq 1$. Hence, $\epsilon \geq \frac{\log\left(\frac{m - 1}{2\eta}\right)}{4\xi n}$. For $n$ large enough, we have $m - 1 \geq \frac{m}{2}$, and we finally obtain
\begin{align*}
\epsilon \geq \frac{\log\left(\frac{1}{4e\xi}\right)}{4} + \frac{\log\left(\frac{1}{8\eta}\right)}{4\xi n} \succsim \log\left(\frac{1}{\xi}\right) + \frac{\log\left(\frac{1}{8\eta}\right)}{\xi n},
\end{align*}
as required.

For the special choice of $\xi, \eta \asymp \frac{1}{\mathrm{poly}(n)}$, we have $\log(1/\xi), \log(1/(8\eta)) \asymp \log(n)$. Hence, $\epsilon \succsim \log(n)$. Next, for $\xi \asymp \frac{1}{\log(n)}$, and $\eta \asymp \frac{1}{\mathrm{poly}(n)}$, we obtain $\epsilon \succsim \log(\log(n))$, as required. 


\subsection{Proof of Proposition~\ref{lemma:symmetrization}}
\label{AppLemSym}

Let $L \in \{0, 1\}^{n \times n}$ be an arbitrary deterministic permutation matrix. For any $A$ and $\hat{\theta} \in [k]^n$, we have
\begin{align*}
\mathbb{P}(J_n^T\mathcal{M}(J_nLAL^TJ_n^T) = L\hat{\theta}) = \mathbb{P}(L^TJ_n^T\mathcal{M}(J_nLAL^TJ_n^T) = \hat{\theta}).
\end{align*}
Since $J_n$ is sampled uniformly, we have $J_n \stackrel{d}{=} J_nL$. Furthermore, the randomness in $J_n$ is independent from the randomness in $\mathcal{M}$, so the latter quantity is clearly equal to $\mathbb{P}(J_n^T\mathcal{M}(J_nAJ_n^T) = \hat{\theta})$,
and the definition of symmetry is satisfied.


\subsection{Proof of Proposition~\ref{prop:LBStable}}
\label{AppPropStable}

For every $\hat{\theta} \in [k]^n$, we consider the corresponding matrix $\hat{\Theta} \in \mathbf{M}_{n \times k}$. We also make the following definition:

\begin{definition}
\label{def_stable}
The matrix $\hat\Theta\in\mathbf{M}_{n\times k}$ is a \textit{stable approximation} of $\Theta$ if 
$$ \arg\min_{J\in E_k} || \hat\Theta J - \Theta||_0
$$
does not change when any two rows of $\hat\Theta$ are switched.
\end{definition}

\begin{lemma}
\label{LemStable}
If each community of $\Theta$ has size at least $6$ and $\widetilde{\mathcal{L}}(\hat{\Theta}, \Theta) \leq 1/3$, then $\hat\Theta$ is stable.
\end{lemma}

\begin{proof}
Let $J$ be the best permutation of communities for $\hat\Theta$. Then in a community of size $m$, at least $2m/3$ nodes are correctly classified, so if we swap two rows of $\hat\Theta$, we still have at least $2m/3-1 \geq m/2$ correctly classified nodes. Hence, $J$ is still the best permutation.
\end{proof}

\begin{lemma}
\label{LemPerm}
Suppose $\calM: \{0,1\}^{n\times n}\to \mathbf{M}_{n\times k}$ is a symmetric $(\eps, \delta)$-node DP mechanism. For any adjacency matrix $A\in \{0,1\}^{n\times n}$, any $\hat\Theta\in \mathbf{M}_{n\times k}$, and any permutation matrix $J$ corresponding to a transposition, we have
$$\mathbb{P}(\calM(A) = \hat{\Theta}) \leq e^{2\eps} \mathbb{P}(\calM(A)=J\hat{\Theta}) + 2\delta e^{\epsilon}.
$$
\end{lemma}

\begin{proof}
The node distance between $A$ and $JAJ$ is at most $2$, so by group privacy, we have
$$\mathbb{P}(\calM(A)=\hat\Theta)\leq e^{2\eps} \mathbb{P}\calM(JAJ)=\Theta) + 2\delta e^{\epsilon} = e^{2\eps} \mathbb{P}(\calM(A)=J\Theta) + 2\delta e^{\epsilon}.$$
\end{proof}

\begin{lemma}
\label{LemStableBound}
Suppose $\calM: \{0,1\}^{n\times n}\to \mathbf{M}_{n\times k}$ is a symmetric $(\eps, \delta)$-node DP mechanism. Fix $\ell \in [n]$, and suppose $T\subseteq \mathbf{M}_{n\times k}$ consists only of stable approximations of $\Theta$ such that $(\hat\Theta J_{\hat \Theta})_{\ell} = \Theta_\ell$, where $J_{\hat\Theta} \in \arg\min_{J \in E_k} \|\hat\Theta J - \Theta\|_0$, i.e., the best permutation of communities for $\hat\Theta$ correctly classifies the first node. For any adjacency matrix $A\in \{0,1\}^{n\times n}$, we have
$$\mathbb{P}(\calM(A)\in T) \leq \frac{1}{1+(k-1)e^{-2\eps}} + 2(k - 1)\delta.
$$
\end{lemma}

\begin{proof}
WLOG, suppose $\ell = 1$ and $\Theta$ assigns the first node to community 1. For $i\in \{2, \ldots, k\}$, define $f_i:T\to \mathbf{M}_{n\times k}$ by setting $f_i(\hat\Theta)$ equal to an assignment obtained from taking $\hat\Theta$ and swapping row 1 with a row assigned to community $i$ in $\hat\Theta J_{\hat\Theta}$. Let $F_i$ be the image of $f_i$. Then $f_i$ is a bijection $T\to F_i$ and from Lemma~\ref{LemPerm}, we have
\begin{equation*}
\mathbb{P}(\calM(A)\in T)\leq e^{2\eps} \mathbb{P}(M(A)\in F_i) + 2\delta e^{\epsilon}.
\end{equation*}

Since any $\hat\Theta\in T$ is stable, the best possible permutation for $f_i(\hat\Theta)$ is still $J$, so it classifies node $1$ to community $i$. Hence, the sets $\{T, F_2, \dots, F_k\}$ are pairwise disjoint, so
\begin{align*}
\mathbb{P}(\calM(A)\in T) + (k-1)e^{-2\eps} \left(\mathbb{P}(\calM(A)\in T) - 2\delta e^{\epsilon}\right) & \le \mathbb{P}(\calM(A)\in T) + \sum_{i=2}^k \mathbb{P}(\calM(A)\in F_i) \\
& \leq 1,
\end{align*}
implying that
\begin{align*}
\mathbb{P}(\calM(A)\in T) \leq \frac{1}{1+(k-1)e^{-2\eps}} + \frac{2(k - 1)\delta e^{-\epsilon}}{1 + (k - 1)e^{-2\epsilon}} \leq \frac{1}{1+(k-1)e^{-2\eps}} + 2(k - 1)\delta.
\end{align*}
\end{proof}

We are now ready to establish the statement of the proposition. Pick an arbitrary community and let $m = \frac{n}{k}$ be its size. Let $X$ denote the number of nodes in that community that are correctly classified by $\calM(A)$ (with the best permutation). By Lemma~\ref{LemStable}, we have
$$
\mathbb{P}(\widetilde{\mathcal{L}}(\mathcal{M}(A), \theta) \le \xi) \leq \mathbb{P}(X\geq (1-\xi) m, \calM(A) \text{ stable}).
$$
Applying Lemma~\ref{LemStableBound} to each node in the community and summing up, we have $$\E[X \cdot \mathbf{1}\{\calM(A) \text{ stable}\}]\leq \frac{m}{1+(k-1)e^{-2\eps}} + 2(k - 1)m\delta,$$
so from Markov's inequality, we conclude that
$$
\mathbb{P}(X\geq (1-\xi) m ,\calM(A) \text{ stable}) \leq \frac{1}{(1-\xi)(1+(k-1)e^{-2\eps})} + \frac{2(k - 1)\delta }{1-\xi}.
$$


\subsection{Proof of Corollary~\ref{corr:lb1_lim1}}
\label{AppCorLB}

Taking the result of Proposition~\ref{prop:LBStable} and averaging also over the randomness of $A$, we have
\begin{equation*}
\mathbb{P}(\widetilde{\mathcal{L}}(\mathcal{M}(A), \theta)\leq \xi) \leq \frac{1}{(1-\xi)(1+(k-1)e^{-2\epsilon})} + \frac{2(k - 1)\delta}{1-\xi},
\end{equation*}
implying that
\begin{equation*}
1-\eta-\xi \le (1-\eta)(1-\xi) \le \frac{1}{1+(k-1)e^{-2\epsilon}} + 2k\delta.
\end{equation*}
Rearranging gives the desired lower bound on $\epsilon$, and plugging in $\eta, \xi, \delta = O\left(\frac{1}{\mathrm{poly}(n)}\right)$ easily yields the desired approximation. Additionally, for $\eta = O\left(\frac{1}{\log(n)}\right)$, and $\xi, \delta = O\left(\frac{1}{\mathrm{poly}(n)}\right)$, we have $\epsilon \succsim \log(\log(n))$, as required.


\bibliographystyle{plain}
\bibliography{refs}

\end{document}